\documentclass{article}
\usepackage[utf8]{inputenc}
\usepackage{amsmath,amssymb,amsthm}
\usepackage{parskip}
\usepackage{tikz}
\usepackage{tikz-cd}
\usepackage{float}
\usepackage{imakeidx}
\usepackage{enumerate}
\usepackage{bm}
\usepackage{fancyhdr}
\usepackage[top=3cm,bottom=3cm,right=3.2cm,left=3.2cm]{geometry}
\usepackage{framed}
\usepackage{mdframed}
\usepackage{graphicx}
\usepackage[pdfencoding=auto]{hyperref}
\usepackage{import}
\usepackage{array}
\usepackage{longtable}

\makeindex[intoc]

\allowdisplaybreaks

\usepackage{scalerel,stackengine}
\stackMath
\newcommand\reallywidehat[1]{%
\savestack{\tmpbox}{\stretchto{%
  \scaleto{%
    \scalerel*[\widthof{\ensuremath{#1}}]{\kern-.6pt\bigwedge\kern-.6pt}%
    {\rule[-\textheight/2]{1ex}{\textheight}}%
  }{\textheight}%
}{0.5ex}}%
\stackon[1pt]{#1}{\tmpbox}%
}

\newcounter{braid}
\newcounter{strands}
\pgfkeyssetvalue{/tikz/braid height}{0.5cm}
\pgfkeyssetvalue{/tikz/braid width}{0.5cm}
\pgfkeyssetvalue{/tikz/braid start}{(0,0)}
\pgfkeys{/tikz/strands/.code={\setcounter{strands}{#1}}}

\makeatletter
\def\cross{%
  \@ifnextchar^{\message{Got sup}\cross@sup}{\cross@sub}}

\def\cross@sup^#1_#2{\render@cross{#2}{-#1}}

\def\cross@sub_#1{\@ifnextchar^{\cross@@sub{#1}}{\render@cross{#1}{-1}}}

\def\cross@@sub#1^#2{\render@cross{#1}{-#2}}

\def\render@cross#1#2{
  \def\strand{#1}
  \def\crossing{#2}
  \pgfmathsetmacro{\cross@y}{-\value{braid}*\braid@h}
  \pgfmathtruncatemacro{\nextstrand}{#1+1}
  \foreach \thread in {1,...,\value{strands}}
  {
    \pgfmathsetmacro{\strand@x}{\thread * \braid@w}
    \ifnum\thread=\strand
    \pgfmathsetmacro{\over@x}{\strand * \braid@w + .5*(1 - \crossing) * \braid@w}
    \pgfmathsetmacro{\under@x}{\strand * \braid@w + .5*(1 + \crossing) * \braid@w}
    \draw[braid] \pgfkeysvalueof{/tikz/braid start} +(\under@x pt,\cross@y pt) to[out=-90,in=90] +(\over@x pt,\cross@y pt -\braid@h);
    \draw[braid] \pgfkeysvalueof{/tikz/braid start} +(\over@x pt,\cross@y pt) to[out=-90,in=90] +(\under@x pt,\cross@y pt -\braid@h);
    \else
    \ifnum\thread=\nextstrand
    \else
     \draw[braid] \pgfkeysvalueof{/tikz/braid start} ++(\strand@x pt,\cross@y pt) -- ++(0,-\braid@h);
    \fi
   \fi
  }
  \stepcounter{braid}
}

\tikzset{braid/.style={double=black,double distance=0.5pt,line width=2pt,white}}

\newcommand{\braid}[2][]{%
  \begingroup
  \pgfkeys{/tikz/strands=2}
  \tikzset{#1}
  \pgfkeysgetvalue{/tikz/braid width}{\braid@w}
  \pgfkeysgetvalue{/tikz/braid height}{\braid@h}
  \setcounter{braid}{0}
  \let\sigma=\cross
  #2
  \endgroup
}
\makeatother

\usetikzlibrary{knots,intersections,decorations.pathreplacing,hobby,arrows,patterns}
\tikzset{every path/.style={black,thin}, every node/.style={transform shape, knot crossing, inner sep=2.5pt}}
\newcommand{\mytrefoil}
{\begin{scope}[rotate=  0] \node (a) at (0,-1) {}; \end{scope}
   \begin{scope}[rotate=120] \node (b) at (0,-1) {}; \end{scope}
   \begin{scope}[rotate=240] \node (c) at (0,-1) {}; \end{scope}
   \draw (a) .. controls (a.4 north west) and (c.4 north east) .. (c.center);
   \draw (b) .. controls (b.4 north west) and (a.4 north east) .. (a.center);
   \draw (c) .. controls (c.4 north west) and (b.4 north east) .. (b.center);
   \draw (a.center) .. controls (a.16 south west) and (c.16 south east) .. (c);
   \draw (b.center) .. controls (b.16 south west) and (a.16 south east) .. (a);
   \draw (c.center) .. controls (c.16 south west) and (b.16 south east) .. (b);}

\usepackage{pict2e}
\makeatletter

\newcommand*\@KP@Large@frame[2]{%
    \setlength\unitlength{\fontdimen 22 #1\tw@}%
    \vrule \@width\z@ \@height 4\unitlength \@depth\tw@\unitlength
    \begin{picture}(6,2)(-3,-1)%
        \def\@KP@Radius     {3}%
        \def\@KP@Hole@radius{.5}%
        \def\@KP@Diameter   {6}%
        #2%
    \end{picture}%
}
\newcommand*\@KP@Small@frame[2]{%
    \setlength\unitlength{\fontdimen 22 #1\tw@}%
    \vrule \@width\z@ \@height \thr@@\unitlength \@depth\@ne\unitlength
    \begin{picture}(4,2)(-2,-1)%
        \def\@KP@Radius     {2}%
        \def\@KP@Hole@radius{.5}
        \def\@KP@Diameter   {4}%
        #2%
    \end{picture}%
}

\newcommand*\@KP@Radius     {}
\newcommand*\@KP@Hole@radius{}
\newcommand*\@KP@Diameter   {}
\newcommand*\@KP@Shape@A{%
    \put(0,0){\circle{\@KP@Diameter}}%
}
\newcommand*\@KP@Shape@B{%
    \Line(-\@KP@Radius,\@KP@Radius )(\@KP@Radius,-\@KP@Radius)%
    \Line(-\@KP@Radius,-\@KP@Radius)(-\@KP@Hole@radius,-\@KP@Hole@radius)%
    \Line(\@KP@Radius ,\@KP@Radius )(\@KP@Hole@radius ,\@KP@Hole@radius )%
}
\newcommand*\@KP@Shape@C{%
    \cbezier(-\@KP@Radius,\@KP@Radius )(0,0)(0,0)(\@KP@Radius,\@KP@Radius )%
    \cbezier(-\@KP@Radius,-\@KP@Radius)(0,0)(0,0)(\@KP@Radius,-\@KP@Radius)%
}
\newcommand*\@KP@Shape@D{%
    \cbezier(-\@KP@Radius,-\@KP@Radius)(0,0)(0,0)(-\@KP@Radius,\@KP@Radius)%
    \cbezier(\@KP@Radius ,-\@KP@Radius)(0,0)(0,0)(\@KP@Radius ,\@KP@Radius)%
}
\newcommand*\@KP@Shape@E{%
    \Line(-\@KP@Radius,-\@KP@Radius )(\@KP@Radius,\@KP@Radius)%
    \Line(-\@KP@Radius,\@KP@Radius)(-\@KP@Hole@radius,\@KP@Hole@radius)%
    \Line(\@KP@Radius ,-\@KP@Radius )(\@KP@Hole@radius ,-\@KP@Hole@radius )%
}

\newcommand*\@KP@Atomic@mathpalette[1]{%
    \mathinner{%
        \mathchoice{%
            \linethickness{.6\p@}
            \@KP@Large@frame \textfont {#1}%
        }{%
            \linethickness{.4\p@}
            \@KP@Small@frame \textfont {#1}%
        }{%
            \linethickness{.3\p@}
            \@KP@Small@frame \scriptfont {#1}%
        }{%
            \linethickness{.2\p@}
            \@KP@Small@frame \scriptscriptfont {#1}%
        }%
    }%
}

\newcommand*\KPA{\@KP@Atomic@mathpalette \@KP@Shape@A}
\newcommand*\KPB{\@KP@Atomic@mathpalette \@KP@Shape@B}
\newcommand*\KPC{\@KP@Atomic@mathpalette \@KP@Shape@C}
\newcommand*\KPD{\@KP@Atomic@mathpalette \@KP@Shape@D}
\newcommand*\KPE{\@KP@Atomic@mathpalette \@KP@Shape@E}

\makeatother

\newtheorem{theorem}{Theorem}
\numberwithin{theorem}{section}
\newtheorem{proposition}[theorem]{Proposition}
\newtheorem{lemma}[theorem]{Lemma}
\newtheorem{corollary}[theorem]{Corollary}
\newtheorem{conjecture}[theorem]{Conjecture}
\theoremstyle{definition}
\newtheorem{define}[theorem]{Definition}
\newtheorem{exmpl}[theorem]{Example}

\theoremstyle{remark}
\newtheorem*{remark}{Remark}

\newenvironment{definition}
 {\MakeFramed{\advance \hsize -\width \FrameRestore}\begin{define}}
 {\end{define}\endMakeFramed}

\newenvironment{example}
 {\vspace{0.3cm}\begin{mdframed}\begin{exmpl}}
 {\end{exmpl}\end{mdframed}\vspace{0.3cm}}

\DeclareMathOperator{\tto}{\Rightarrow}
\newcommand{\tangle}[1]{[\![#1]\!]}

\def\id{\text{id}}
\def\BB{\mathbb{B}}
\def\CC{\mathbb{C}}
\def\FF{\mathbb{F}}
\def\SS{\mathbb{S}}
\def\RR{\mathbb{R}}
\def\ZZ{\mathbb{Z}}
\def\Bcal{\mathcal{B}}
\def\Ecal{\mathcal{E}}
\def\Hcal{\mathcal{H}}

\def\TL{\mathcal{TL}}

\def\incl{\text{incl}}
\def\lfincl{\text{lfincl}}
\def\tr{\text{tr}}
\def\RO{\text{RO}}
\def\RI{\text{RI}}
\def\RII{\text{RII}}
\def\RIII{\text{RIII}}
\def\MI{\text{M1}}
\def\MII{\text{M2}}

\newcommand\longmapsfrom{\mathrel{\reflectbox{\ensuremath{\longmapsto}}}}

\newcommand{\knotsinmath}[1]{%
 \tikzset{every path/.style={}}
 \begin{tikzpicture}[baseline=-\dimexpr\fontdimen22\textfont2\relax]
  #1
 \end{tikzpicture}%
}

\newcommand{\signpositive}{%
  \knotsinmath{\begin{scope}[scale=0.4]\draw[->,thin] (-1,-1)--(1,1);
   \draw[ultra thick,double distance=0.5pt,double=black,white] (-1,1)--(1,-1);
   \draw[->] (-1,1)--(1,-1);\end{scope}}%
}

\newcommand{\signnegative}{%
 \knotsinmath{\begin{scope}[scale=0.4]\draw[->] (-1,1)--(1,-1);
   \draw[ultra thick,double distance=0.5pt,double=black,white] (-1,-1)--(1,1);
   \draw[->] (-1,-1)--(1,1);\end{scope}}%
}

\newcommand{\alonecurve}{%
 \knotsinmath{\begin{scope}[scale=0.2]\draw[thin] (0,0) .. controls (1,-0.6) and (2.4,-0.3) .. (3,0);\end{scope}}
}

\newcommand{\alonecurvewithcircle}{%
 \knotsinmath{\begin{scope}[scale=0.005,yshift=-40cm]
   \clip (0,0)--(140,0)--(140,70)--(0,70);
   \draw[thin]    (12,63) .. controls (38.03,43.47) and (103.45,55.93) .. (111,71) ;
   \draw[thin]   (50,26.5) .. controls (50,18.49) and (56.49,12) .. (64.5,12) .. controls (72.51,12) and (79,18.49) .. (79,26.5) .. controls (79,34.51) and (72.51,41) .. (64.5,41) .. controls (56.49,41) and (50,34.51) .. (50,26.5) -- cycle ;
  \end{scope}}%
}

\newcommand{\rimove}{%
 \knotsinmath{\begin{scope}[scale=0.005,yshift=-35cm]
   \draw[thin]    (15,37) .. controls (41.03,17.47) and (95.87,29.72) .. (103.42,44.79) .. controls (110.97,59.85) and (58,68) .. (76,39) ;
   \draw[thin]    (93,24) .. controls (127,9) and (142,18) .. (151,27) ;
  \end{scope}}%
}

\newcommand{\riimove}{%
 \knotsinmath{\begin{scope}[scale=0.005,rotate=180,yshift=-45cm]
   \draw[thin]    (16,36) .. controls (45,67) and (92,67) .. (114,37) ;
   \draw[thin]    (42,44) .. controls (55,29) and (72,30) .. (83,48) ;
   \draw[thin]    (30,57) -- (20,68) ;
   \draw[thin]    (90,63) -- (99,77) ;
  \end{scope}}%
}

\newcommand{\resolveriimove}{%
 \knotsinmath{\begin{scope}[scale=0.005,yshift=-45cm]
   \draw[thin]    (11,63) .. controls (42,48) and (95,61) .. (110,71) ;
   \draw[thin]    (14,29) .. controls (47,43) and (101,41) .. (113,37) ;
  \end{scope}}%
}

\newcommand{\trefoil}{%
 \knotsinmath{\begin{scope}[scale=0.005,yshift=-70cm]
  \clip (0,0)--(160,0)--(160,160)--(0,160)--cycle;
  \draw[thin]    (98.5,132) .. controls (48.07,199.48) and (-22.7,112.44) .. (13.82,70.28) .. controls (26.55,55.59) and (52.31,46.35) .. (96.5,51) ;
  \draw[thin]    (44.5,39) .. controls (78.5,-45) and (130.5,45) .. (114.5,102) ;
  \draw[thin]   (127.5,58) .. controls (195.69,97.4) and (107.74,133.35) .. (64.59,109.76) .. controls (50.78,102.21) and (41.56,88.57) .. (43.5,67) ;
 \end{scope}}%
}

\newcommand{\twopositivetwists}{%
 \knotsinmath{\begin{scope}[scale=0.005,yshift=-80cm]
    \clip (0,0)--(160,0)--(160,160)--(0,160)--cycle;
    \draw[thin]    (98.5,132) .. controls (39.99,210.29) and (-37.73,63.24) .. (36.5,53) .. controls (110.73,42.76) and (63.5,-21) .. (44.5,39) ;
    \draw[thin]    (103.5,114) .. controls (129.5,53) and (172.5,109) .. (143.5,114) .. controls (114.5,119) and (36.41,145.92) .. (43.5,67) ;
   \end{scope}}%
}

\newcommand{\hopflink}{%
 \knotsinmath{\begin{scope}[scale=0.005,yshift=-80cm]
    \clip (0,0)--(155,0)--(155,170)--(0,170);
    \draw[thin]    (78.5,134) .. controls (-15.5,251) and (-34.5,-67) .. (78.5,114) ;
    \draw[thin]    (22.5,47) .. controls (35.5,-33) and (121.5,9) .. (105.5,66) ;
    \draw[thin]    (105.5,66) .. controls (195.5,118) and (13.5,164) .. (21.5,75) ;
   \end{scope}}%
}

\newcommand{\negativetwist}{%
 \knotsinmath{\begin{scope}[scale=0.005,yshift=-40cm]
    \clip (0,0)--(135,0)--(135,80)--(0,80);
    \draw[thin]    (77,56.67) .. controls (123.67,132) and (173.31,-14.5) .. (69.67,45.33) .. controls (-33.98,105.17) and (5.67,-55.33) .. (62.33,36.67) ;
   \end{scope}}%
}

\newcommand{\negativetwistwithcircle}{%
 \knotsinmath{\begin{scope}[scale=0.005,yshift=-70cm]
    \clip (0,0)--(140,0)--(140,130)--(0,130);
    \draw[thin]    (77,56.67) .. controls (123.67,132) and (173.31,-14.5) .. (69.67,45.33) .. controls (-33.98,105.17) and (5.67,-55.33) .. (62.33,36.67) ;
    \draw[thin]   (7,102.5) .. controls (7,87.86) and (19.42,76) .. (34.75,76) .. controls (50.08,76) and (62.5,87.86) .. (62.5,102.5) .. controls (62.5,117.14) and (50.08,129) .. (34.75,129) .. controls (19.42,129) and (7,117.14) .. (7,102.5) -- cycle ;
   \end{scope}}%
}

\newcommand{\negativetwistwithalonecurve}{%
 \knotsinmath{\begin{scope}[scale=0.005,yshift=-45cm]
   \clip (0,0)--(110,0)--(110,80)--(0,80);
   \draw[thin]    (70,30) .. controls (87.08,48.26) and (68.58,74.2) .. (54,76) .. controls (35,80) and (47,42) .. (67,55) ;
   \draw[thin]    (25,32) .. controls (33,40) and (38,62) .. (25,76) ;
   \draw[thin]    (83,63) -- (99,71) ;
  \end{scope}}%
}

\newcommand{\positivetwistinward}{%
 \knotsinmath{\begin{scope}[scale=0.005,yshift=-60cm]
    \clip (0,0)--(150,0)--(150,130)--(0,130);
    \draw[thin]    (89,70) .. controls (138,123) and (67,108) .. (65,88) .. controls (63,68) and (80.97,55.69) .. (108.19,49.03) .. controls (135.41,42.37) and (172.79,111.98) .. (116,122) .. controls (59.21,132.02) and (31,120) .. (16,87) .. controls (1,54) and (10.93,-50.23) .. (76,54) ;
   \end{scope}}%
}

\newcommand{\crossnegpos}{%
 \knotsinmath{\begin{scope}[scale=0.3,yshift=-0.5cm]
  \draw[thin] (0,0)--(1,1);
  \draw[ultra thick,double distance=0.3pt,double=black,white] (1,0)--(0,1);
 \end{scope}}%
}

\newcommand{\lookslikelol}{%
 \knotsinmath{\begin{scope}[scale=0.005,yshift=-40cm]
   \clip (0,0)--(90,0)--(90,80)--(0,80); 
   \draw[thin]    (19.17,16.04) .. controls (33.11,32.02) and (26.82,63.21) .. (20.12,72.68) ;
   \draw[thin]    (68.88,16.32) .. controls (60.37,36.34) and (65.84,66.59) .. (69.83,72.96) ;
   \draw[thin]   (35,45) .. controls (35,38.37) and (39.92,33) .. (46,33) .. controls (52.08,33) and (57,38.37) .. (57,45) .. controls (57,51.63) and (52.08,57) .. (46,57) .. controls (39.92,57) and (35,51.63) .. (35,45) -- cycle ;
  \end{scope}}
}

\newcommand{\tldiagram}[1]{%
 \knotsinmath{\begin{scope}[scale=0.75,yshift=-0.5cm]
  \draw (0.5,0.5) node[scale=1.33] {#1};
  \draw (0,0) rectangle (1,1);
  \draw (0.25,1.5)--(0.25,1);
  \draw (0.25,0)--(0.25,-0.5);
  \draw (0.75,0)--(0.75,-0.5);
  \draw (0.75,1.5)--(0.75,1);
  \draw (0.53,1.25) node {$\cdots$};
  \draw (0.53,-0.25) node {$\cdots$};
 \end{scope}}
}

\newcommand{\tlinclusion}[1]{%
 \knotsinmath{\begin{scope}[scale=0.75,yshift=-0.5cm]
  \draw (0.5,0.5) node[scale=1.33] {#1};
  \draw (0,0) rectangle (1,1);
  \draw (0.25,1.5)--(0.25,1);
  \draw (0.25,0)--(0.25,-0.5);
  \draw (0.75,0)--(0.75,-0.5);
  \draw (0.75,1.5)--(0.75,1);
  \draw (0.53,1.25) node {$\cdots$};
  \draw (0.53,-0.25) node {$\cdots$};
  \draw (1.25,1.5)--(1.25,-0.5);
 \end{scope}}\;
}

\newcommand{\braidinclusion}[1]{%
 \knotsinmath{\begin{scope}[scale=0.75,yshift=-0.75cm]
  \draw (0.75,0.75) node[scale=1.33] {#1};
  \draw (0,0) rectangle (1.5,1.5);
  \draw (0.25,1.5)--(0.25,2);
  \draw (0.25,0)--(0.25,-0.5);
  \draw (1.25,0)--(1.25,-0.5);
  \draw (1.25,1.5)--(1.25,2);
  \draw (0.78,1.75) node {$\cdots$};
  \draw (0.78,-0.25) node {$\cdots$};
  \draw (1.75,2)--(1.75,-0.5);
 \end{scope}}
}

\newcommand{\tlconditional}[1]{
 \knotsinmath{\begin{scope}[scale=0.75,yshift=-0.5cm]
  \draw (0.5,0.5) node[scale=1.33] {#1};
  \draw (0,0) rectangle (1,1);
  \draw (0.25,1.5)--(0.25,1);
  \draw (0.25,0)--(0.25,-0.5);
  \draw (0.75,0)--(0.75,-0.5);
  \draw (0.75,1.5)--(0.75,1);
  \draw (0.9,1)--(0.9,1.5);
  \draw (0.9,-0.5)--(0.9,0);
  \draw[red] (0.9,1.5)--(1.2,1.5)--(1.2,-0.5)--(0.9,-0.5);
  \draw (0.53,1.25) node {$\cdots$};
  \draw (0.53,-0.25) node {$\cdots$};
 \end{scope}}\;
}

\newcommand{\tlconditionalinclusion}[1]{%
 \knotsinmath{\begin{scope}[scale=0.75,yshift=-0.5cm]
  \draw (0.5,0.5) node[scale=1.33] {#1};
  \draw (0,0) rectangle (1,1);
  \draw (0.25,1.5)--(0.25,1);
  \draw (0.25,0)--(0.25,-0.5);
  \draw (0.75,0)--(0.75,-0.5);
  \draw (0.75,1.5)--(0.75,1);
  \draw (0.53,1.25) node {$\cdots$};
  \draw (0.53,-0.25) node {$\cdots$};
  \draw (1.25,1.5)--(1.25,-0.5);
  \draw[red] (1.25,1.5)--(1.5,1.5)--(1.5,-0.5)--(1.25,-0.5);
 \end{scope}}
}

\newcommand{\tltraceplusone}[1]{%
 \knotsinmath{\begin{scope}[scale=0.75,yshift=-0.5cm]
  \draw (0.5,0.5) node[scale=1.33] {#1};
  \draw (0,0) rectangle (1,1);
  \draw (0.25,1.5)--(0.25,1);
  \draw (0.25,0)--(0.25,-0.5);
  \draw (0.75,0)--(0.75,-0.5);
  \draw (0.75,1.5)--(0.75,1);
  \draw[red] (0.75,1.5)--(1.5,1.5)--(1.5,-0.5)--(0.75,-0.5);
  \draw[red] (0.25,1.5)--(0.25,1.75)--(1.75,1.75)--(1.75,-0.75)--(0.25,-0.75)--(0.25,-0.5);
  \draw (0.9,1)--(0.9,1.25);
  \draw (0.9,0)--(0.9,-0.25);
  \draw[red] (0.9,1.25)--(1.25,1.25)--(1.25,-0.25)--(0.9,-0.25);
  \draw (0.53,1.25) node {$\cdots$};
  \draw (0.53,-0.25) node {$\cdots$};
 \end{scope}}
}

\tikzset{
  pt/.style={insert path={node[scale=2]{.}}},
  dnup/.style={insert path={ [pt] .. controls +(0,1) and +(0,-1) .. +(#1,2) [pt]}},
  dndn/.style={insert path={ [pt] .. controls +(0,1) and +(0,1) .. +(#1,0) [pt]}},
  upup/.style={insert path={ [pt] .. controls +(0,-1) and +(0,-1) .. +(#1,0) [pt]}},
}

\usepackage{marginnote}

\makeatletter
\renewcommand*\l@subsection{\@dottedtocline{1}{1.5em}{3.5em}}
\renewcommand*\l@figure{\@dottedtocline{1}{1.5em}{3.5em}}
\makeatother

\begin{document}

\title{\hspace{0pt}
\vskip 1cm
\bfseries On the Jones Polynomial\\
 {\normalsize \textit{constructing link invariants via braids \& von Neumann algebras}}
}
\author{\bfseries A Monica Queen Thesis}
\date{Summer 2021}

\maketitle

\begin{abstract}
 This expository essay is aimed at introducing the Jones polynomial.
 We will see the encapsulation of the Jones polynomial, which will involve topics in functional analysis and geometrical topology; making this essay an interdisciplinary area of mathematics.
 The presentation is based on a lot of different sources of material (check references), but we will mainly be giving an account on Jones' papers (\cite{jones1983},\cite{jones1985},\cite{jones1987},\cite{jonesbook},\cite{jonespoly}) and Kauffman's papers (\cite{kauffman1986},\cite{kauffman1988},\cite{kauffman_stat_mech}).
 A background in undergraduate \textit{Linear Algebra}, \textit{Linear Analysis}, \textit{Abstract Algebra}, \textit{Commutative Algebra}, \textit{\&} \textit{Topology} is essential. It would also be useful if the reader has a background in \textit{Knot Theory} \textit{\&} \textit{Operator Algebras}.
\end{abstract}
\hspace{0pt}
\thispagestyle{empty}

\newpage
\thispagestyle{empty}
\hspace{0pt}
\vskip 5cm
\hspace*{\fill} {\Large \textit{Ars longa, vita brevis.}}
\vskip 5cm
{\large This essay is dedicated to my grandmother \textbf{\textit{Grace}}, who passed away during the start of writing this.}
\vfill
\hspace{0pt}
\newpage%

\newgeometry{bottom=20mm} 
 \tableofcontents
\restoregeometry

\listoffigures

\renewcommand{\thesection}{\Alph{section}}
\renewcommand{\thesubsection}{\thesection.\Roman{subsection}}

\numberwithin{figure}{section}
\renewcommand{\thefigure}{\thesection-\roman{figure}}

\newpage%
\section{Introduction}\label{chapter:intro}

The objective of this M.Sc. thesis, as are most, is to prove to the university of satisfactory scholastic work. However, the author accepts that such egotistical objectives will not keep the writer, not to mention the peruser, engaged. Accordingly, the genuine motivation behind this thesis is to introduce the writer's discoveries and thoughts insightfully and to impart them to other people. Subsequently, ideally, this will add to the assemblage of information regarding the matter.

Thus the aims of this essay are two-fold. The first is to give a solid introduction on knot theory and give the combinatorial method of the Jones polynomial (via Kauffman's bracket \cite{kauffman1986} and skein relation \cite{conway_skein,homfly,jones1985}). This gives us a good base on what the Jones polynomial is, and the importance of the polynomial invariant.
For instance, it was until the Jones polynomial was created that the trefoil knot was proven to be distinct from its reflection.
The second aim is to give a brief explanation on how Vaughan Jones originally came to the realisation that there was a link between the von Neumann algebras he was working on and the construction of a most interesting link invariant.
This is given by first realising that all links can be represented by closed braids; and by defining a trace on the presentation from the braid group $\Bcal_n$ to the quotient algebra of $\Bcal_n$ called the Temperley-Lieb algebras $\TL_n$.
This method is based on Jones' original construction in 1984 \cite{jones1983,jones1985}.

Purely going on the notion of physical space, the ether, and atoms as vortex knots, the mathematical theory of knots was formulated by physicists Peter Tait, William Thomson, and James Maxwell in the late 1800s \cite{cgknott,tait}.
While their motivation for studying knots and links (i.e., atoms being knotted in ether) was widely discredited in 1887 by Michelson and Morley \cite{michealson}, where they proved that the ether does not exist, the amount of contributions made by them was vast enough.
It was until after the inexistence of the ether was proven that mathematicians started studying knots. And in the time of writing, knot theory has major applications linked to chemistry and biology, such as knots in DNA and proteins. For a nice read of introductory knot theory; see Adams' book \cite{adams}.

To put it simply, knot theory\index{knot theory}, is the study of distinguishing links (disjoint union of closed smooth curves in three-dimensional space) up to ambient isotopy (intuitively, this means that we can move the link around without allowing it to pass by itself or cutting it). Distinguishing links is usually done using link invariants (algebraic objects that do not change under ambient isotopy). There are many link invariants such as the coloring invariant and the crossing number invariant - these are easy to visualise, however they are not good enough. The crossing number invariant, for instance, is difficult to determine since its definition states that it has to be over the minimal link diagram (so this is summed over the minimum diagram of all possible diagrams - i.e., infinity). The \textbf{Jones polynomial}, however, is the most \textit{interesting} and one of the most \textit{powerful} knot invariants that are linked to many areas of mathematics and science.
Vaughan Jones constructed his polynomial invariant in 1984 \cite{jones1985}. For that, he won the Fields medal in 1990 and gave his speech wearing his home country's (New Zealand) rugby jersey: a historical day any knot theorist can tell you about.

Jones, however, did not make the connection on his own.
In fact, while Jones was giving a lecture on von Neumann algebras in Geneva in 1982, Didier Hatt-Arnold, a graduate student, was the one who raised the similarities between the algebra he was working on and the braid group \cite{cromwell}.
This then motivated Jones to explore possible representations of the braid group into his algebra he was working on. Finally leading him to the construction of the marvellous link invariant two years later \cite{jones1985}.

The fascination of this link invariant, from a knot theory point of view, is that it could detect mirror links. Since 1928, the only known polynomial link invariant was the Alexander polynomial \cite{alexander}, and while profound in its own way, detecting a link's mirror was not one of its strengths. However, the main fascination with this link invariant is the fact that it came from a completely different field of mathematics. This lead to other link invariants that arose from areas such as category theory.

While, at first, it may have seemed as a surprising result from the subject of von Neumann algebras, it became apparent that this subject has many amazing facts that are linked with other subjects such as Topological Quantum Field Theory.
However, recalling that the mathematical theory of knots was formulated by physicsts, it should not be surprising when physics saves the day again. 
As Jones said,
\begin{quote}
 ``God may or may not play dice, but She sure loves a von Neumann algebra.'' \cite{tqft}
\end{quote}

Despite the ease of defining the Jones polynomial in Chapter \ref{chapter:kauffman} via the Kauffman bracket and the skein relation, we should mention that it is generally known that the Jones polynomial of a link is $NP$-hard \cite{jaeger}. However, when calculating the Jones polynomial of a braid representation of a link, it can be computed in polynomial time \cite{morton}. We will not discuss computation in this paper, but we see in the final chapter that it is in fact a lot easier to calculate the Jones polynomial via its original construction than the more visual and diagrammatic ones. This is not to say that the Kauffman bracket is not useful, in fact the reason it is liked more so today than the original construction is because of its astounding correlation into category and representation theory; see Khovanov's paper \cite{khovanov}.

{\textbf{Outline}}

In Chapter \ref{chapter:intro}, we will start with a brief overview on knot theory which includes all the essential information we need for the remainder of the essay.

In Chapter \ref{chapter:kauffman}, we will give two important definitions of the Jones polynomial: via the Kauffman bracket \cite{kauffman1986} and the Skein relation version, which John Conway conjured in the 1960s for the Alexander polynomial \cite{conway_skein}.

The next chapter (Chapter \ref{chapter:braids}) is on braid theory. We start with the diagrammatic definition of braids and see how it correlates with the Artin braid group, which is the algebraic definition. We also see a correlation between the symmetric group and the braid group.

In Chapter \ref{chapter:braids-knots}, we relate the previous two chapters. We will associate to every knot a braid, and prove Alexander's theorem, which provides a pseudo-converse by showing that every braid closure is equivalent to a link.
We conclude with an overview of Markov's moves and Markov's theorem. This gives an equivalence relation that will help us in making a one-to-one correspondence with braids and knots. However, two braids do not have to be braid equivalent in order to be Markov equivalent.
Markov's theorem states that the braid closures are Reidemeister equivalent if and only if the braids are Markov equivalent.
Thus the one-to-one correspondence will quotient the braid group with the Markov equivalence class.

The second part of the thesis focuses on more of Jones' work: we begin by defining the necessary analytical preliminaries in Chapter \ref{chapter:analysis} and discuss the Temperley-Lieb algebras in algebraic and diagrammatic settings in Chapter \ref{chapter:tl}.

With the foundations set, we now have the necessary background in order to conclude the essay (Chapter \ref{chapter:jones}), i.e., constructing the Jones polynomial via the original method.

\newpage 
\section{Knots \& links}\label{chapter:knots}
 We begin with a brief overview of knot theory. We discuss all preliminary definitions and results needed to define the Jones polynomial and motivate our chosen thesis in the coming chapters.

This chapter is based on the presentations given in \cite{knotpaper}, \cite{enc}, \cite{jones_quantum}, \cite{BurdeZieschang}, \cite{conway_skein}, \cite{cromwell}, \cite{homfly}, \cite{jones1985}, \cite{jonesbook}, \cite{jonespoly}, \cite{kauffman1986}, \cite{kauffman1988}, \cite{kauffman_stat_mech}, \cite{kauffman_lomonaco}, \cite{kawauchi}, \cite{Likorish}, \cite{livingston}, \cite{manturov}, and \cite{murasagibook}.

Most of our discussion is based on that in the author's exposition on knot theory \cite[Ch.~1]{knotpaper}.

\subsection{Defining a knot and link}

\begin{definition}[knot]\label{definition:knot}
 A \textit{knot}\index{knot} is a smooth embedding of a circle $\SS^1$ into $\RR^3$ or $\SS^3$.\footnotemark
\end{definition}
\footnotetext{$\SS^3$ is essentially $\RR^3$ with its one point compactification. We will mostly use $\RR^3$.}
Here, a smooth \textit{embedding}\index{embedding} means a smooth map which is injective and whose differential is nowhere zero.
Pictorially, Kauffman \& Lomonaco Jr. \cite[p.~2]{kauffman_lomonaco} drew the embedded circles in $\RR^3$ as the following.
\begin{figure}[H]
 \centering
 \begin{tikzpicture}[x=0.75pt,y=0.75pt,yscale=-1,xscale=1]
  \draw   (11,138) .. controls (11,124.19) and (22.19,113) .. (36,113) .. controls (49.81,113) and (61,124.19) .. (61,138) .. controls (61,151.81) and (49.81,163) .. (36,163) .. controls (22.19,163) and (11,151.81) .. (11,138) -- cycle ;
  \draw    (69,122) -- (134.04,109.38) ;
  \draw [shift={(136,109)}, rotate = 529.02] [color={rgb, 255:red, 0; green, 0; blue, 0 }  ][line width=0.75]    (10.93,-3.29) .. controls (6.95,-1.4) and (3.31,-0.3) .. (0,0) .. controls (3.31,0.3) and (6.95,1.4) .. (10.93,3.29)   ;
  \draw    (67,151) -- (140.12,177.32) ;
  \draw [shift={(142,178)}, rotate = 199.8] [color={rgb, 255:red, 0; green, 0; blue, 0 }  ][line width=0.75]    (10.93,-3.29) .. controls (6.95,-1.4) and (3.31,-0.3) .. (0,0) .. controls (3.31,0.3) and (6.95,1.4) .. (10.93,3.29)   ;
  \draw   (217.64,66.09) .. controls (231.68,65.52) and (243.53,76.44) .. (244.1,90.47) -- (249.03,211.22) .. controls (249.03,211.22) and (249.03,211.22) .. (249.03,211.22) -- (147.36,215.37) .. controls (133.32,215.95) and (121.48,205.03) .. (120.9,191) -- (115.97,70.25) .. controls (115.97,70.25) and (115.97,70.25) .. (115.97,70.25) -- cycle ;
  \draw    (153,101) .. controls (227,51) and (217,116) .. (181,94) ;
  \draw    (189,103) .. controls (135,148) and (138,52) .. (168,84) ;
  \draw    (145,107) .. controls (117.05,126.47) and (166.16,147.54) .. (189.08,125.77) .. controls (212,104) and (263,62) .. (211,81) ;
  \draw    (202,88) -- (196,93) ;
  \draw    (185,160) .. controls (245,147) and (256,179) .. (214,179) ;
  \draw    (166,162) .. controls (118,170) and (223,225) .. (209,164) ;
  \draw    (209,150) .. controls (207,134) and (188.53,144.78) .. (179,153) .. controls (169.47,161.22) and (174.09,179.7) .. (205,175) ;
 \end{tikzpicture}
 \caption{knot definition with embedded circles in $\RR^3$}
 \label{figure:kauffman_knot_theory}
\end{figure}
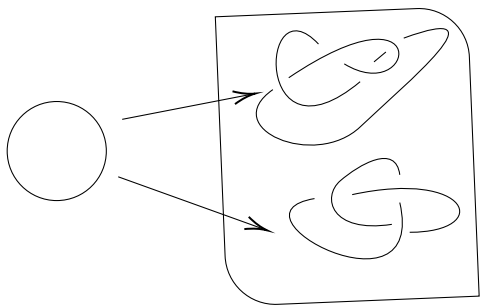
In the above figure, the two embeddings are knots. It is easy to see that they are not the \textit{same}. However, we will need to define what knots being ``the same'' means before claiming such (see next section).

\begin{proposition}\label{proposition:knot_def_eq}
 The definition of a knot is equivalent to the following:
 \begin{itemize}
  \item A \textit{knot} is a subset of $\RR^3$ that consists of piecewise linear simple closed curves.
  \item A \textit{knot} is a subset of $\RR^3$, which is diffeomorphic to $\SS^1$.
 \end{itemize}
\end{proposition}

Here, \textit{piecewise linear}\index{piecewise linear} means that the curve can be written as a finite union of line segments such that the segments are disjoint from one another, except when consecutive, i.e.,
\[K=[a_0,a_1]\cup[a_1,a_2]\cup\cdots\cup[a_n,a_0]\]
Here, \textit{diffeomorphism}\index{diffeomorphism} means that the map and its inverse are both smooth. \textit{Homeomorphism}\index{homeomorphism}, on the other hand, is when the map and its inverse are both continuous.\footnote{So this means that all diffeomorphisms are homeomorphisms, but not the other way around; because, differentiability implies continuity, but continuity does not imply differentiability.}

We can define a knot via homeomorphism (i.e., a subset of $\RR^3$ which is homeomorphic to a circle).
However, this requires us to define \textit{tame knots}\index{tame knots} which are knots that can be defined smoothly or piecewise-linearly; and \textit{wild knots}\index{wild knots} which are knots that are not tame.

To avoid such complications, we will always assume smoothness.




\begin{example}
 The most basic knot is the unknot, written as $0_1$, which is drawn below.
 We also have the trefoil $3_1$ and the trefoil's reflection $\overline{3_1}$. We call $3_1$ the left-handed trefoil, whereas we call its reflection $\overline{3_1}$, the right-handed trefoil (fact: the right-handed trefoil and its reflection are two distinct knots; see Proposition \ref{proposition:two_3_1_knots}).
 \begin{figure}[H]
  \centering
  \begin{center}
   \begin{tabular}{>{\centering\arraybackslash}m{35mm}>{\centering\arraybackslash}m{35mm}>{\centering\arraybackslash}m{35mm}}
    \begin{center}\begin{tikzpicture}[scale=1]
     \draw (0,0) circle (0.8);
    \end{tikzpicture}\end{center}
    &\begin{center}
    \reflectbox{\begin{tikzpicture}[scale=0.5]
     \mytrefoil;
    \end{tikzpicture}}\end{center}
    &\begin{center}
    \begin{tikzpicture}[scale=0.5]
     \mytrefoil;
    \end{tikzpicture}
    \end{center}\\
    $0_1$&$3_1$&$\overline{3_1}$
   \end{tabular}
  \end{center}
  \caption{examples of knots}
  \label{figure:knot_example}
 \end{figure}
\end{example}

\begin{example}
 The following is not a knot, since the differential at $(*)$ is undefined. However, for the sake of completion, we mention that since this is a subset of $\RR^3$ and is homeomorphic to a circle, we can see that it is an example of a \textit{wild knot}.
 \begin{figure}[H]
  \centering
  \begin{tikzpicture}[scale=0.5,x=0.75pt,y=0.75pt,yscale=-1,xscale=1]
   \draw    (26,92) .. controls (3,64) and (-5,37) .. (27,18) .. controls (59,-1) and (85,38) .. (72,52) ;
   \draw    (385,62) .. controls (457,72) and (448,137) .. (417,149) .. controls (386,161) and (329,157) .. (219,157) .. controls (109,157) and (36,167) .. (23,147) .. controls (10,127) and (38.37,80.97) .. (65,61) ;
   \draw    (45,102) .. controls (57.39,110.98) and (102,107) .. (116,95) .. controls (130,83) and (130,84) .. (138,47) ;
   \draw    (57,24) .. controls (33,29) and (38,45) .. (46,51) .. controls (54,57) and (104,60) .. (130,58) ;
   \draw    (132,17) .. controls (108.97,17.49) and (108,27) .. (114,33) .. controls (120,39) and (154,40) .. (197,33) ;
   \draw    (75,17) .. controls (126,-3) and (147,6) .. (140,32) ;
   \draw    (150,14) .. controls (178.67,4.85) and (209,20) .. (210,32) .. controls (211,44) and (188,55) .. (146,57) ;
   \draw    (215,28) .. controls (261,25) and (265,44) .. (317,54) ;
   \draw  [fill={rgb, 255:red, 0; green, 0; blue, 0 }  ,fill opacity=1 ] (397.5,63.25) .. controls (397.5,61.73) and (398.73,60.5) .. (400.25,60.5) .. controls (401.77,60.5) and (403,61.73) .. (403,63.25) .. controls (403,64.77) and (401.77,66) .. (400.25,66) .. controls (398.73,66) and (397.5,64.77) .. (397.5,63.25) -- cycle ;
   \draw (336,39) node [anchor=north west,scale=2][inner sep=0.75pt]   [align=left] {$\displaystyle \dotsc $};
   \draw (395,39) node [anchor=north west,scale=2][inner sep=0.75pt]   [align=left] {$\displaystyle *$};
  \end{tikzpicture}
  \caption{wild knot}
  \label{figure:wild_knot}
 \end{figure}
 Such knots can have pathological behaviour, hence why we mentioned that we will always assume smoothness throughout this paper.
\end{example}

\begin{definition}[link]
 A \textit{link}\index{link} of $m$ components is a smooth embedding of $n$ disjoint circles into $\RR^3$.
 Thus, a knot is just a link with one component.
\end{definition}

\begin{example}\label{example:links}
 An unlink is a finite number of disjointed unknots. We usually write an $m$-unlink as $\KPA^m$ (where $m$ denotes the number of disjointed unknots). The figure to the right is an example of an unlink $\KPA^2$.
 \begin{figure}[H]
  \centering
  \begin{minipage}{0.3\textwidth}
   \begin{center}
    \begin{tikzpicture}[scale=0.01,yscale=-1]
     \draw  [draw opacity=0] (90.01,115.72) .. controls (84.07,117.84) and (77.67,119) .. (71,119) .. controls (40.07,119) and (15,94.15) .. (15,63.5) .. controls (15,32.85) and (40.07,8) .. (71,8) .. controls (101.93,8) and (127,32.85) .. (127,63.5) .. controls (127,79.02) and (120.57,93.06) .. (110.2,103.13) -- (71,63.5) -- cycle ; 
     \draw   (90.01,115.72) .. controls (84.07,117.84) and (77.67,119) .. (71,119) .. controls (40.07,119) and (15,94.15) .. (15,63.5) .. controls (15,32.85) and (40.07,8) .. (71,8) .. controls (101.93,8) and (127,32.85) .. (127,63.5) .. controls (127,79.02) and (120.57,93.06) .. (110.2,103.13) ;
     \draw  [draw opacity=0] (109.49,11.85) .. controls (115.36,9.56) and (121.72,8.22) .. (128.39,8.02) .. controls (159.31,7.13) and (185.09,31.24) .. (185.98,61.88) .. controls (186.86,92.52) and (162.52,118.08) .. (131.61,118.98) .. controls (100.69,119.87) and (74.91,95.76) .. (74.02,65.12) .. controls (73.57,49.61) and (79.6,35.39) .. (89.66,25.02) -- (130,63.5) -- cycle ;
     \draw   (109.49,11.85) .. controls (115.36,9.56) and (121.72,8.22) .. (128.39,8.02) .. controls (159.31,7.13) and (185.09,31.24) .. (185.98,61.88) .. controls (186.86,92.52) and (162.52,118.08) .. (131.61,118.98) .. controls (100.69,119.87) and (74.91,95.76) .. (74.02,65.12) .. controls (73.57,49.61) and (79.6,35.39) .. (89.66,25.02) ;
    \end{tikzpicture}
   \end{center}
  \end{minipage}%
  \begin{minipage}{0.4\textwidth}
   \begin{center}
    \begin{tikzpicture}[scale=0.8]
     \draw (0,0) circle(0.4);
     \draw (0.9,0) circle(0.4);
    \end{tikzpicture}
   \end{center}
  \end{minipage}
  \par{\;\;\;Hopf Link\qquad\qquad\qquad{Unlink of two components}}
  \caption{examples of links}
  \label{figure:link_examples}
 \end{figure}
\end{example}

\begin{remark}
 Note that it is not always the case that a theorem for knots works for links as well. However, for the purposes of this paper, we will use them interchangeably, unless stated otherwise.
\end{remark}

\subsection{Link equivalency}

We now discuss link equivalency. Essentially, we need two links to be equivalent if we can go from one link to the other without forming any self-intersections or breaking the definition of a link. So, we begin by explaining why homotopy equivalency is too weak for links.

Any two knots are \textit{homotopy equivalent}\index{homotopy}, i.e., for knots $K_1\colon{S^1\to\RR^3}$ and $K_2\colon{S^1\to\RR^3}$, there exists a continuous function $H\colon\RR^1\times[0,1]\to\RR^3$ such that $H(x,0)=K_1(x)$ and $H(x,1)=K_2(x)$. So, under homotopy, all knots would be trivial by pulling on the ends until the ``knotted'' part disappears (also known as the gentleman's knot - see Figure \ref{figure:gentleman_knot}).

\begin{figure}[H]
 \centering
 \begin{tikzpicture}[x=0.75pt,y=0.75pt,yscale=-1,xscale=1]
  \draw    (66.74,72.24) .. controls (39.92,85.15) and (19.62,74.23) .. (14.55,57.34) .. controls (9.47,40.45) and (21.21,30.35) .. (37.02,44.92) ;
  \draw    (29.04,67.77) .. controls (39.8,47.18) and (49.34,31.51) .. (67.47,33.5) .. controls (85.59,35.49) and (89.94,55.35) .. (78.34,63.3) ;
  \draw    (45.72,50.88) .. controls (90.66,71.25) and (90.66,99.06) .. (80.51,109.99) .. controls (70.37,120.91) and (53.69,123.89) .. (38.47,117.93) .. controls (23.25,111.97) and (19.62,88.13) .. (25.42,79.19) ;
  \draw    (172.89,32.81) .. controls (167.76,39.55) and (158.35,40.39) .. (152.36,34.49) .. controls (146.38,28.6) and (148.94,21.03) .. (154.93,22.71) ;
  \draw    (154.93,29.44) .. controls (160.92,19.34) and (165.2,18.5) .. (170.33,18.5) .. controls (175.46,18.5) and (181.45,22.71) .. (178.03,26.92) ;
  \draw    (163.48,24.81) .. controls (217.16,45.55) and (205.19,93.53) .. (196.05,109.29) .. controls (186.91,125.04) and (165.84,138.99) .. (141.88,125.52) .. controls (117.93,112.05) and (118.79,74.17) .. (148.09,36.18) ;
  \draw    (281.65,16.41) .. controls (280.44,18.04) and (278.21,18.25) .. (276.8,16.82) .. controls (275.38,15.39) and (275.99,13.55) .. (277.4,13.96) ;
  \draw    (277.4,15.59) .. controls (278.82,13.14) and (279.83,12.93) .. (281.05,12.93) .. controls (282.26,12.93) and (283.68,13.96) .. (282.87,14.98) ;
  \draw    (280.44,14.06) .. controls (287.03,20.82) and (311.42,29.38) .. (324.21,39.22) .. controls (336.99,49.05) and (345.94,68.72) .. (342.1,98.22) .. controls (338.27,127.72) and (317.81,134.74) .. (307.58,137.55) .. controls (297.36,140.36) and (267.95,138.96) .. (255.17,127.72) .. controls (242.38,116.48) and (237.3,103.62) .. (237.27,79.96) .. controls (237.23,56.3) and (257.06,39.42) .. (275.16,18) ;
  \draw   (380.42,75.85) .. controls (380.42,42.74) and (397.01,15.91) .. (417.47,15.91) .. controls (437.93,15.91) and (454.51,42.74) .. (454.51,75.85) .. controls (454.51,108.96) and (437.93,135.8) .. (417.47,135.8) .. controls (397.01,135.8) and (380.42,108.96) .. (380.42,75.85) -- cycle ;
  \draw[fill=black]   (417.47,15.91) .. controls (417.47,15.13) and (417.85,14.51) .. (418.33,14.51) .. controls (418.81,14.51) and (419.19,15.13) .. (419.19,15.91) .. controls (419.19,16.68) and (418.81,17.3) .. (418.33,17.3) .. controls (417.85,17.3) and (417.47,16.68) .. (417.47,15.91) -- cycle ;
  \draw   (495.91,73.16) .. controls (495.91,40.22) and (512.49,13.52) .. (532.95,13.52) .. controls (553.41,13.52) and (570,40.22) .. (570,73.16) .. controls (570,106.1) and (553.41,132.8) .. (532.95,132.8) .. controls (512.49,132.8) and (495.91,106.1) .. (495.91,73.16) -- cycle ;
  \draw (96.7,73) node [anchor=north west][inner sep=0.75pt]    {$\rightarrow$};
  \draw (210,73) node [anchor=north west][inner sep=0.75pt]    {$\rightarrow$};
  \draw (350.7,73) node [anchor=north west][inner sep=0.75pt]    {$\rightarrow$};
  \draw (463.7,73) node [anchor=north west][inner sep=0.75pt]    {$\rightarrow$};
 \end{tikzpicture}
 \caption{the gentleman's knot}
 \label{figure:gentleman_knot}
\end{figure}

So, our definition of link equivalence needs to be stronger than that of homotopy.

\begin{definition}[ambient isotopy]\label{definition:isotopy}
 Two maps $A,B\subseteq\RR^3$ are said to be \textit{ambient isotopic}\index{ambient isotopy} if there exists a continuous map $H\colon\RR^3\times[0,1]\to\RR^3$ such that the map $h_t\colon\RR^3\to\RR^3$ defined by $h_t(x)=H(x,t)$ is a diffeomorphism for all $t\in[0,1]$ with $h_0$ equal to the identity map, and $h_1(A)=B$. We call the map $H$, the \textit{ambient isotopy} map between $A$ and $B$.
\end{definition}

\begin{remark}
 In Definition \ref{definition:isotopy}, ambient isotopy is essentially homotopy where it is a smooth family of diffeomorphisms.
\end{remark}

\begin{definition}[link equivalence] \label{definition:link_equivalence}
 Two links are \textit{equivalent}\index{equivalent!links} if they are
 ambient isotopic.
 When two links $L_1$ and $L_2$ are equivalent, we write $L_1\sim{L_2}$.
\end{definition}

\begin{remark}
 Fisher \cite{fisher} proved that any orientation-preserving\footnotemark diffeomorphism $f$ of $\RR^n$ is smoothly ambient isotopic to the identity. This allows us to characterise ambient isotopies as orientation-preserving diffeomorphisms; in particular, links are equivalent if there exists an orientation $-$ preserving piecewise linear diffeomorphism $h\colon\RR^3\to\RR^3$ such that $h(L_1)=L_2$, where $L_1$ and $L_2$ are links.
\end{remark}
\footnotetext{For those with a background in differential geometry, a diffeomorphism between two oriented manifolds $M$ and $N$ is said to be \textit{orientation-preserving} if the pullback of the orientation on $N$ is the same, up to equivalence, of the orientation on $M$ \cite[Ch.~6]{cam:diffgeo}.}

\begin{definition}[link type]
 Any two links which are equivalent have the same \textit{link type}\index{link type}: the equivalence class of links.
 Similarly, we call the equivalence class of knots \textit{knot type}\index{knot type}.
\end{definition}

\begin{proposition}\label{proposition:ambient_isotopy_eq}
 Ambient isotopy is an equivalence relation. (So, link equivalence is an equivalence relation.)
\end{proposition}
\begin{proof}
 Let $L_1$, $L_2$, and $L_3$ be any links.

 For reflexivity, this is obvious.
 
 For symmetry, we suppose $L_1\sim{L_2}$. Then, we have a function $h_t$ such that $h_0(L_1)=L_1$ and $h_1(L_1)=L_2$ where the map and its inverse are smooth. It follows that $h^{-1}_0(L_2)=L_2$ and $h^{-1}_1(L_2)=L_1$, which means $L_2\sim{L_1}$.
 
 Finally, for transitivity, suppose $L_1\sim{L_2}$ and $L_2\sim{L_3}$. Then, we have two functions $h_t$ and $f_t$ such that $h_t$ is the same as $h_t$ from the symmetry above; and $f_0(L_2)=L_2$ and $f_1(L_2)=L_3$. So, it then follows that $f_t\circ{h_t}(x)=f_t(h_t(x))$ is also a diffeomorphism (map compositions are smooth if both maps are smooth), so then $f_0(h_0(L_1))=f_0(L_1)=L_1$ and $f_1(h_1(L_1))=f_1(L_2)=L_3$. Thus, $L_1\sim{L_3}$.
\end{proof}

Thus we have the following definition of an unlink and unknot.
\begin{definition}[unlink]
 The \textit{unknot}\index{unknot} $0_1$ is a knot that is equivalent to a circle.
 So, the $m$-\textit{unlink}\index{unlink} $\KPA^m$ is a link that is equivalent to $m$ disjointed unknots.
\end{definition}

\subsection{Link diagrams and Reidemeister's theorem}

In this section, we discuss elementary methods of projecting a link onto a two-dimensional plane. The reason for doing so is because it would not be reasonable for us to work in $\RR^3$, as it is harder to imagine three-dimensional objects in our head. Before projecting the link, we first need to define the most elementary method of finding an equivalent link. This elementary method is called the triangle move. Here is a precise definition.

\begin{definition}[triangle-move]\label{definition:triangle-move}
 Consider a link $L$ in $\RR^3$, we find a planar triangle in $\RR^3$ that intersects $L$ in exactly one of the edges of the triangle. We then delete that edge from $L$, and replace it by the other two edges of the triangle. The inverse process of this move is done by replacing the two edges of the triangle with one (literally the opposite). We call this the \textit{triangle-move}\index{triangle-move} and denote it by $\Delta$ and its inverse process by $\Delta^{-1}$, see the figure below.
 \begin{figure}[H]
    \centering
    \begin{tikzpicture}[scale=0.3]
     \begin{scope}
      \draw[ultra thick] (0,0)--(2,0);
      \draw[dashed,thin] (2,0)--(1,2)--(0,0);
     \end{scope}
     \draw (4.5,2) node[scale=3.33] {$\overset{\Delta}{\longmapsto}$};
     \draw (4.5,0) node[scale=3.33] {$\underset{\Delta^{-1}}{\longmapsfrom}$};
     \begin{scope}[xshift=7cm]
      \draw[dashed,thin] (0,0)--(2,0);
      \draw[ultra thick] (2,0)--(1,2)--(0,0);
     \end{scope}
    \end{tikzpicture}
    \caption{triangle-move}
    \label{figure:triangle_move}
 \end{figure}
\end{definition}

Two links are equivalent if and only if they differ by a finite sequence of the triangle-moves or the inverse of the moves. A detailed proof can be found in \cite[pp.~6--8,~Proposition~1.10]{BurdeZieschang}.

\begin{remark}
 It is easy to show that the $\Delta$-move induces an equivalence relation. (It is closely analogous to the proof from Proposition \ref{proposition:ambient_isotopy_eq}.)
\end{remark}

Now we can define the \textit{link projection map}\index{link projection} $\rho\colon\RR^3\to\RR^2$. As explained in \cite[pp.~2--4]{Likorish}, this is done by:
\begin{itemize}
 \item small triangle-moves such that each line segment of the link projects to a line in $\RR^2$,
 \item that the projections of two such segments intersect in at most one point which for disjoint segments is not an end point, and
 \item that no point belongs to the projections of three segments.
\end{itemize}
Given such a situation, the image of the link in $\RR^2$ together with ``over/under'' information (this refers to the relative heights above $\RR^2$ of the two inverse images of a crossing, which is always indicated in diagrams by breaks in the under-passing segments) at the crossings is called a \textit{link diagram} of the given link. So, the only types of crossings in a link diagram that are allowed are essentially shown in Figure \ref{figure:allowed_diagram}. And the crossings that are not allowed are triple points, tangents, and vertical edges (shown in Figure \ref{figure:not_allowed}).


\begin{figure}[H]
    \centering
    \begin{tikzpicture}[scale=0.8]
     \begin{scope}
      \draw (1,-1)--(1,1);
      \draw[ultra thick,draw=white,double=black,double distance=0.5pt] (0,0)--(2,0);
     \end{scope}
     \begin{scope}[xshift=3cm]
      \draw (0,0)--(2,0);
      \draw[ultra thick,draw=white,double=black,double distance=0.5pt] (1,-1)--(1,1);
     \end{scope}
    \end{tikzpicture}
    \caption{crossings that are allowed}
    \label{figure:allowed_diagram}
\end{figure}
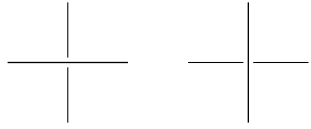
\begin{figure}[H]
    \centering
    \begin{tikzpicture}[scale=0.8]
     \begin{scope}
      \draw (-1,-1) .. controls (0,0) .. (0,1) .. controls (0,0) .. (1,-1);
     \end{scope}
     \begin{scope}[xshift=3cm]
      \draw (-1,-1)--(1,1);
      \draw[ultra thick,draw=white,double=black,double distance=0.5pt] (1,-1)--(-1,1);
      \draw[ultra thick,draw=white,double=black,double distance=0.5pt] (-1,0)--(1,0);
     \end{scope}
     \begin{scope}[xshift=6cm,rotate=30]
      \draw (-1,0) .. controls (0,0.5) .. (1,0);
      \draw (-1,0.4)--(1,0.4);
     \end{scope}
    \end{tikzpicture}
    \caption{vertical edge; triple point; tangent}
    \label{figure:not_allowed}
    \par{crossings that are not allowed}
\end{figure}

\begin{definition}[link universe]
 A \textit{link universe}\index{link universe} is a union of closed piece-wise linear curves in the plane such that there are finitely many self-intersections and that all self-intersections are transverse (i.e., at each crossing, there are exactly four half-curves with that crossing as an end point, so that none of them are tangent to another).
\end{definition}

A \textit{link diagram}\index{link diagram} is essentially a link universe with a choice, i.e., the points of self-intersection are equipped with height information with the lower strand being indicated with a break in the curve.
If the link diagram has one component (i.e., a projection of a knot), then we call it a \textit{knot diagram}\index{knot diagram}.

\begin{example}
 A crossing in a link universe would then have two crossing choices in a link diagram.
 
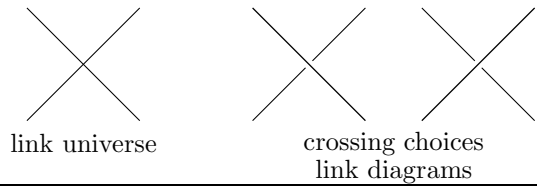
\begin{center}
  \begin{tikzpicture}[scale=0.75]
   \begin{scope}
    \draw (-1,-1)--(1,1);
    \draw (1,-1)--(-1,1);
    \draw (0,-1.4) node[scale=1.33] {link universe};
   \end{scope}
   \begin{scope}[xshift=4cm]
    \draw (-1,-1)--(1,1);
    \draw[white,double=black,ultra thick,double distance=0.5pt] (1,-1)--(-1,1);
   \end{scope}
   \begin{scope}[xshift=7cm]
    \draw (1,-1)--(-1,1);
    \draw[white,double=black,ultra thick,double distance=0.5pt] (-1,-1)--(1,1);
   \end{scope}
   \draw (5.5,-1.4) node[scale=1.33] {crossing choices};
   \draw (5.5,-1.9) node[scale=1.33] {link diagrams};
  \end{tikzpicture}
 \end{center}
\end{example}

If links $L_1$ and $L_2$ are equivalent, then they are related by a sequence of triangle moves, as previously mentioned. After moving all the vertices of all the triangles by a very small amount, it can be assumed that the projections of no three of the vertices lie on the line in $\RR^2$ and the projections of no three edges pass through a single point. So then each triangle projects to a triangle.

The following example shows that using triangle moves is dreary and time consuming. This motivates the usefulness of the upcoming theorem.

\begin{example}
 Consider the following piecewise linear link.
 
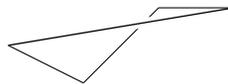
\begin{center}
  \begin{tikzpicture}
   \draw (0,0)--(1,1);
   \draw[double=black,white,ultra thick,double distance=0.5pt] (2,1)--(-1,0.5);
   \draw (2,1)--(1,1);
   \draw (-1,0.5)--(0,0);
  \end{tikzpicture}
 \end{center}
 It is easy to see that this is equivalent to the unknot by a twist. However, with the triangle moves, we first find a planar triangle that intersects our link in one of the edges of the triangle. The intersected edge is thickened for clarity.
 
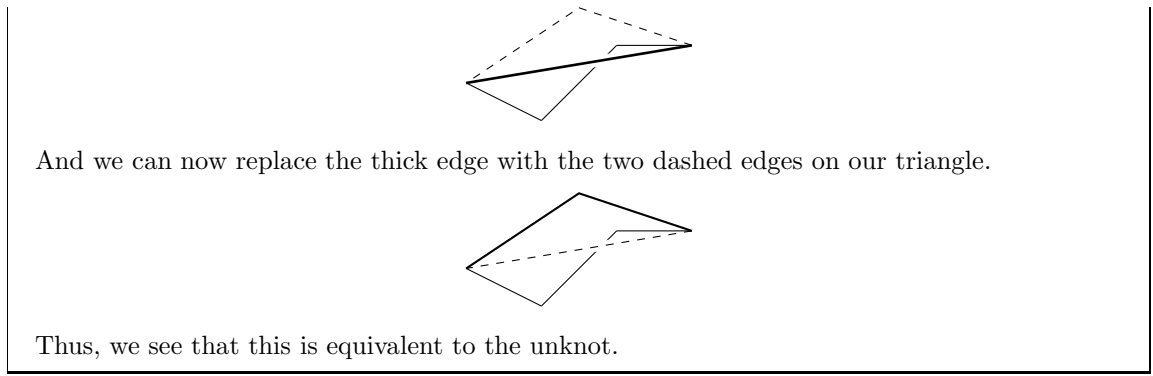
\begin{center}
  \begin{tikzpicture}
   \draw (0,0)--(1,1);
   \draw[double=black,white,ultra thick,double distance=1pt] (2,1)--(-1,0.5);
   \draw (2,1)--(1,1);
   \draw (-1,0.5)--(0,0);
   \draw[dashed] (2,1)--(0.5,1.5)--(-1,0.5);
  \end{tikzpicture}
 \end{center}
 And we can now replace the thick edge with the two dashed edges on our triangle.
 \begin{center}
  \begin{tikzpicture}
   \draw (0,0)--(1,1);
   \draw[double=white,white,ultra thick,double distance=0.5pt] (2,1)--(-1,0.5);
   \draw[dashed] (2,1)--(-1,0.5);
   \draw (2,1)--(1,1);
   \draw (-1,0.5)--(0,0);
   \draw[thick] (2,1)--(0.5,1.5)--(-1,0.5);
  \end{tikzpicture}
 \end{center}
 Thus, we see that this is equivalent to the unknot.
\end{example}

Thankfully, there is an easier method of showing link equivalency with Reidemeister's theorem. The theorem says that two link diagrams are equivalent if and only if they are related by a finite sequence of three moves (see Figure \ref{figure:reidemeiser_moves}). The threorem also says that every link has a link diagram.
This statement is quite fascinating as it simplifies the triangle moves.
Here is a precise statement of the theorem.

\begin{theorem}[Reidemeister's theorem\index{Reidemeister's theorem} {\cite{reidemeister}}]\label{theorem:reidemeisters} {\ }
 \begin{enumerate}
    \item Any link $L$ is equivalent to a link which has a link diagram.
    \item Two links $L_1,L_2$ with diagrams $D_1,D_2$ are equivalent if and only if their diagrams are related by a finite sequence of \textit{Reidemeister moves}\index{Reidemeister moves} and \textit{planar isotopies}\index{planar isotopies}.
 \end{enumerate}
 \begin{figure}[H]
    \centering
    \begin{minipage}{0.3\textwidth}
    \begin{center}
     \begin{tikzpicture}[scale=0.3]
      \begin{scope}
       \draw[black] (0,2) -- (0,-2);
       \draw[white,ultra thick] (0,0) circle(2);
       \draw[black,dashed] (0,0) circle(2);
      \end{scope}
      \begin{scope}[xshift=7cm]
       \node[black,scale=4] at (-3.5,0) {$\bm{\sim}$};
       \node (a) at (0.2,0) {};
       \draw[black] 
        (0,2) .. controls (0,1) and (a.2 north west) ..
        (a.center) .. controls (a.16 south east) and (a.16 north east) ..
        (a) .. controls (a.2 south west) and (0,-1) ..
        (0,-2);
        \draw[white,ultra thick] (0,0) circle(2);
        \draw[black,dashed] (0,0) circle(2);
      \end{scope}
     \end{tikzpicture}
     \par{Type I - $(\RI)$}
    \end{center}
    \end{minipage}%
    \begin{minipage}{0.3\textwidth}
    \begin{center}
        \begin{tikzpicture}[scale=0.3]
            \begin{scope}
            \draw[black] (100:2) -- (260:2);
            \draw[black] ( 80:2) -- (280:2);
            \draw[white,ultra thick] (0,0) circle(2);
            \draw[black,dashed] (0,0) circle(2);
            \end{scope}
            \begin{scope}[xshift=7cm]
            \node[black,scale=4] at (-3.5,0) {$\bm{\sim}$};
            \draw[black] (100:2) .. controls ( 0.6,0) .. (260:2); 
            \draw[draw=white,double=black,ultra thick,,double distance=0.5pt] ( 80:2) .. controls (-0.6,0) .. (280:2);
            \draw[white,ultra thick] (0,0) circle(2);
            \draw[black,dashed] (0,0) circle(2);
            \end{scope}
        \end{tikzpicture}
    \par{Type II - $(\RII)$}
    \end{center}
    \end{minipage}%
    \begin{minipage}{0.3\textwidth}
        \begin{center}
        \begin{tikzpicture}[scale=0.3]
        \begin{scope}
        \draw[draw=white,double=black ,ultra thick,double distance=0.5pt] (  0:2) .. controls (270:0.7) .. (180:2);
        \draw[draw=white,double=black  ,ultra thick,double distance=0.5pt] ( 60:2) .. controls (150:0.7) .. (240:2);
        \draw[draw=white,double=black,ultra thick,double distance=0.5pt] (120:2) .. controls ( 30:0.7) .. (300:2);
        \draw[white,ultra thick] (0,0) circle(2);
        \draw[black,dashed] (0,0) circle(2);
        \end{scope}
        \begin{scope}[xshift=7cm]
        \node[black,scale=4] at (-3.5,0) {$\bm{\sim}$};
        \draw[draw=white,double=black ,ultra thick,double distance=0.5pt] (  0:2) .. controls ( 90:0.7) .. (180:2);
        \draw[draw=white,double=black  ,ultra thick,double distance=0.5pt] ( 60:2) .. controls (330:0.7) .. (240:2);
        \draw[draw=white,double=black,ultra thick,double distance=0.5pt] (120:2) .. controls (210:0.7) .. (300:2);
        \draw[white,ultra thick] (0,0) circle(2);
        \draw[black,dashed] (0,0) circle(2);
        \end{scope}
        \end{tikzpicture}
    \par{Type III - $(\RIII)$}
    \end{center}
    \end{minipage}
 \caption{Reidemeister moves}
 \label{figure:reidemeiser_moves}
 \end{figure}
\end{theorem}

\begin{remark}
 The figures above show regions of the diagram, when the move is made the rest of the diagram is left unchanged (hence why we put dashed lines: to remind ourselves that we should only think of the enclosed region for the moment).
\end{remark}

We note that some texts include planar isotopy as a Type O - $(\RO)$ move. Intuitively, this means moving the knot around by stretching or compressing.

The diagrams are said to be \textit{regularly isotopic}\index{regularly isotopic} if a Type I move is not used. So the theorem shows that there is a bijective correspondence between the set of equivalence classes, i.e.,
(piecewise linear links in $\RR^3$ with $\Delta$-move equivalence) $\leftrightarrow$ (piecewise linear link diagrams in $\RR^2$ with R-move equivalence).

\begin{proof}[Proof (sketch) of Theorem \ref{theorem:reidemeisters}]
 We show that any triangle-move can be reduced to a sequence of Reidemeister moves and planar isotopies. There are two cases: first being a planar isotopy - a regular triangle-move (Figure \ref{figure:triangle_move}); the second is when the strands of the link come in contact with our chosen triangle which gives the three Reidemeister moves as shown in Figure \ref{figure:triangle-move-types}.
 \begin{figure}[H]
    \centering
    \begin{tikzpicture}[scale=0.5]
     \begin{scope}
      \draw[thin,dashed] (0,0)--(2,0)--(1,2)--(0,0);
      \draw[draw=white,double=black,ultra thick,,double distance=1.2pt] (2,0)--(-0.5,2);
      \draw[ultra thick] (0,0)--(2,0);
      \draw[thin,dashed] (2,0)--(1,2);
      \draw (2.8,1.5) node[scale=2] {$\overset{\Delta}{\longmapsto}$};
      \draw (2.8,0.5) node[scale=2] {$\underset{\Delta^{-1}}{\longmapsfrom}$};
      \draw (3,-1) node[scale=2] {$\RI$};
     \end{scope}
     \begin{scope}[xshift=4cm]
      \draw[ultra thick] (0,0)--(1,2);
      \draw[draw=white,double=black,ultra thick,,double distance=1.2pt] (2,0)--(-0.5,2);
      \draw[thin,dashed] (0,0)--(2,0);
      \draw[ultra thick] (1,2)--(2,0);
     \end{scope}
     \begin{scope}[xshift=9cm]
      \draw[thin,dashed] (0,0)--(2,0)--(1,2)--(0,0);
      \draw[ultra thick] (0,0)--(1,2);
      \draw[ultra thick,draw=white,double=black,double distance=1.2pt] (1.6,2.1)--(1.2,-0.3);
      \draw (3,1.5) node[scale=2] {$\overset{\Delta}{\longmapsto}$};
      \draw (3,0.5) node[scale=2] {$\underset{\Delta^{-1}}{\longmapsfrom}$};
      \draw (3,-1) node[scale=2] {$\RII$};
     \end{scope}
     \begin{scope}[xshift=13cm]
      \draw[thin,dashed] (0,0)--(2,0)--(1,2)--(0,0);
      \draw[ultra thick] (1,2)--(2,0);
      \draw[ultra thick] (2,0)--(0,0);
      \draw[ultra thick,draw=white,double=black,double distance=1.2pt] (1.6,2.1)--(1.2,-0.3);
     \end{scope}
     \begin{scope}[xshift=18cm]
      \draw[thin,dashed] (0,0)--(2,0)--(1,2)--(0,0);
      \draw[ultra thick] (0,0)--(1,2);
      \draw[ultra thick,draw=white,double=black,double distance=1.2pt] (1.4,2.1)--(1,-0.3);
      \draw[ultra thick,draw=white,double=black,double distance=1.2pt] (2.1,2.1)--(0.4,-0.3);
      \draw (3,1.5) node[scale=2] {$\overset{\Delta}{\longmapsto}$};
      \draw (3,0.5) node[scale=2] {$\underset{\Delta^{-1}}{\longmapsfrom}$};
      \draw (3,-1) node[scale=2] {$\RIII$};
     \end{scope}
     \begin{scope}[xshift=22cm]
      \draw[thin,dashed] (0,0)--(2,0)--(1,2)--(0,0);
      \draw[ultra thick] (1,2)--(2,0)--(0,0);
      \draw[ultra thick,draw=white,double=black,double distance=1.2pt] (1.4,2.1)--(1,-0.3);
      \draw[ultra thick,draw=white,double=black,double distance=1.2pt] (2.1,2.1)--(0.4,-0.3);
     \end{scope}
    \end{tikzpicture}
    \caption{Triangle moves $\leftrightarrow$ Reidemeister moves}
    \label{figure:triangle-move-types}
 \end{figure}
\end{proof}

Our strategy for proving that a function is a link invariant (see \ref{section:invariants}), is to show that it preserves the Reidemeister moves.

\subsection{Orientation}

We now discuss orientations, and see how the figure eight knot is equivalent to its reflection.

An \textit{orientation}\index{orientation} for a link is a choice of direction along each of the components. This is indicated on the corresponding link diagram by adding an arrow to each component. The components of an $n$-component link can be oriented in $2^n$ ways.
\begin{definition}[sign]
 The \textit{sign}\index{sign} $\epsilon(C)=\pm1$ of an oriented crossing $C$ is given by the figure below.
 \begin{figure}[H]
    \centering
    \begin{tikzpicture}[scale=0.5]
     \begin{scope}
      \draw[thin, ->] (0,0)--(2,2);
      \draw[ultra thick,draw=white,double=black,double distance=0.5pt] (0,2)--(2,0);
      \draw[ultra thin,->] (0,2)--(2,0);
      \draw[thin,->,dashed] (0.2,0) .. controls (1,0.8) .. (1.7,0);
      \draw (1,0.2) node[scale=1.7] {$+$};
      \draw (1,-0.5) node[scale=2] {positive crossing};
     \end{scope}
     \begin{scope}[xshift=7cm]
      \draw[thin,->] (0,2)--(2,0);
      \draw[ultra thick,draw=white,double=black,double distance=0.5pt] (0,0)--(2,2);
      \draw[ultra thin,->] (0,0)--(2,2);
      \draw[thin,->,dashed] (0.2,2) .. controls (1,1.2) .. (1.7,2);
      \draw (1,1.8) node[scale=1.7] {$-$};
      \draw (1,-0.5) node[scale=2] {negative crossing};
     \end{scope}
    \end{tikzpicture}
    \caption{sign of crossing}
    \label{figure:crossings}
 \end{figure}
\end{definition}

\begin{remark}
 Figure \ref{figure:crossings} can be thought of as moving along the orientation of the under-strand until we reach a crossing where we need to ``jump'' to the over strand; if we go to the right at the ``jump'' then the sign is positive, and if we go to the left then it is negative.
\end{remark}

\begin{definition}
 The \textit{writhe}\index{writhe} $\omega(D)$ of a link diagram $D$ is the sum of all the crossing signs, i.e. $\omega(D)=\sum_{C\textnormal{ a crossing in }D}\epsilon(C)$.
\end{definition}

\begin{example}
 The following shows an example of finding the writhe of the left-handed trefoil, Hopf link, and the unlink (with a twist).
 \begin{figure}[H]
    \centering
    \begin{tikzpicture}[scale=0.5]
     \begin{scope}
      \begin{scope}[rotate=  0] \node (a) at (0,-1) {}; \end{scope}
      \begin{scope}[rotate=120] \node (b) at (0,-1) {}; \end{scope}
      \begin{scope}[rotate=240] \node (c) at (0,-1) {}; \end{scope}
      \draw (a) .. controls (a.4 north west) and (c.4 north east) .. (c.center);
      \draw (b) .. controls (b.4 north west) and (a.4 north east) .. (a.center);
      \draw (c) .. controls (c.4 north west) and (b.4 north east) .. (b.center);
      \draw (a.center) .. controls (a.16 south west) and (c.16 south east) .. (c);
      \draw (b.center) .. controls (b.16 south west) and (a.16 south east) .. (a);
      \draw (c.center) .. controls (c.16 south west) and (b.16 south east) .. (b);
      \draw[->] ( 90:2.035) -- +(  0:0.01);
      \draw (0,-2.5) node[scale=1.8] {$\omega(D_1)=-3$};
      \draw (c.east) node[anchor=south east] {$-1$};
      \draw (a.west) node[anchor=north] {$-1$};
      \draw (b.west) node[anchor=south west] {$-1$};
     \end{scope}
     \begin{scope}[xshift=7cm,scale=3]
      \draw (-0.25,0) +( 70: 0.5) arc( 70:410: 0.5);
      \draw ( 0.25,0) +( 70:-0.5) arc( 70:410:-0.5);
      \draw[->] (-0.25, 0.50) -- +(-0.01,0);
      \draw[->] ( 0.25, 0.50) -- +( 0.01,0);
      \draw (0,-.85) node[scale=1.8/3] {$\omega(D_2)=2$};
      \draw (0,0.55) node[scale=1/3] {$+1$};
      \draw (0,-0.55) node[scale=1/3] {$+1$};
     \end{scope}
     \begin{scope}[xshift=12cm]
      \draw (0,0) +(60: 0.8) arc(60:300: 0.8);
      \draw (2,0) +(60:-0.8) arc(60:300:-0.8);
      \draw (2,0) ++( 60:-0.8) -- ( 60:0.8);
      \draw[draw=white,double=black,ultra thick,double distance=0.4pt] (2,0) ++(300:-0.8) -- (300:0.8);
      \draw[->] (0,0.8) -- +(-0.01,0);
      \draw (1,-2.5) node[scale=1.8] {$\omega(D_3)=1$};
      \draw (1,0.5) node {$+1$};
     \end{scope}
    \end{tikzpicture}
    \caption{writhe example}
    \label{example:writhe-example}
\end{figure}
\end{example}

\begin{definition}[reverse]
 The \textit{reverse}\index{reverse} of a link is done by reversing the orientations of all the components in its corresponding diagram. For a link $L$, we denote the reverse by $rL$.
\end{definition}

\begin{remark}
 It is easy to see that the writhe of the diagrams in Example \ref{example:writhe-example} remains the same when reversing the orientation.
\end{remark}

We can generalise this remark. By looking at a single crossing, we can see that changing all orientations of all the components in a link does not change the sign. So the signs of the following are the same.
\[
 \epsilon\left(\signpositive\right)
 = +1 =
 \epsilon\left(\reflectbox{\signnegative}\right)
\]
Similarly,
\[
 \epsilon\left(\signnegative\right) = -1 =
 \epsilon\left(\reflectbox{\signpositive}\right)
\]
Thus, we have $\omega(L)=\omega(rL)$. And so this means the writhe of an unoriented knot diagram is well defined.

\begin{definition}[reflection]
 The \textit{reflection}\index{reflection} of a link is done by replacing all over-strands with under-strands at each crossing and vice versa of the original link's link diagram. For link $L$, we say $\overline{L}$ is the reflection of $L$.
\end{definition}

By the end of the next chapter, we will show that the trefoil and its reflection are in fact two distinct knots.
For now, however, we will show that both the figure eight knot and its reflection are equivalent.

\begin{proposition}
 The figure eight knot $4_1$ is equivalent to its reflection $\overline{4_1}$.
\end{proposition}
\begin{proof}
 We show this using the Reidemeister moves.
 \begin{center}\begin{longtable}{c>{\centering\arraybackslash}m{25mm}c>{\centering\arraybackslash}m{25mm}c>{\centering\arraybackslash}m{28mm}}
  &
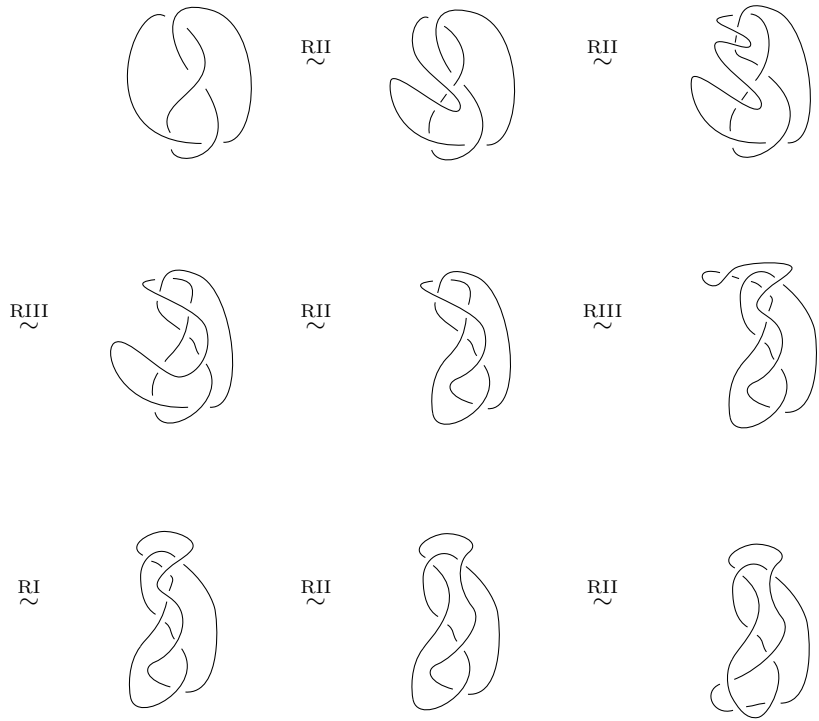
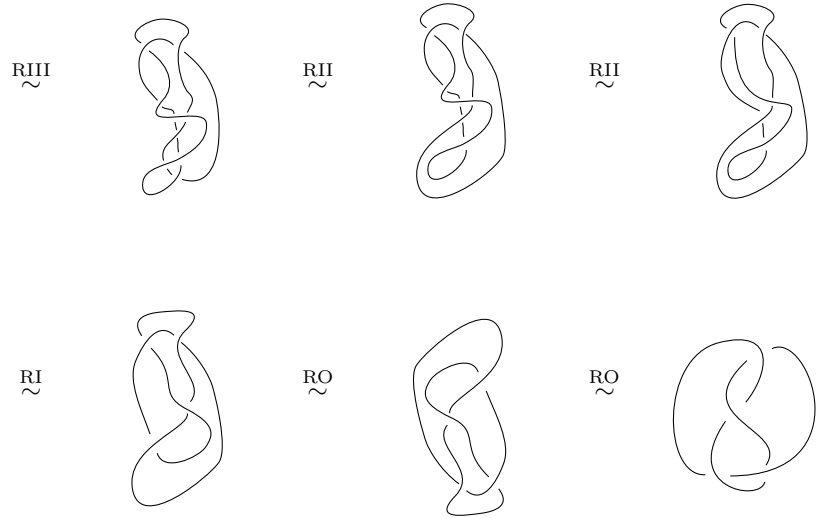
  \begin{center}\begin{tikzpicture}[scale=0.6,x=0.75pt,y=0.75pt,yscale=-1,xscale=1,baseline=-\dimexpr\fontdimen22\textfont2\relax]
   \draw    (100.7,131) .. controls (132.22,133) and (134.84,33) .. (91.08,20) .. controls (47.31,7) and (54.31,47) .. (70.94,70) ;
   \draw    (56.06,123) .. controls (37.68,92) and (117.34,85) .. (70.07,35) ;
   \draw    (51.69,25) .. controls (24.55,9) and (-15.71,132) .. (82.32,132) ;
   \draw    (85.82,86) .. controls (119.96,142) and (60.44,155) .. (56.94,137) ;
  \end{tikzpicture}\end{center}
  &$\overset{\RII}{\sim}$&
  \begin{center}\begin{tikzpicture}[scale=0.6,x=0.75pt,y=0.75pt,yscale=-1,xscale=1,baseline=-\dimexpr\fontdimen22\textfont2\relax]
   \draw    (90.7,125.8) .. controls (122.22,127.8) and (124.84,27.8) .. (81.08,14.8) .. controls (37.31,1.8) and (44.31,41.8) .. (60.94,64.8) ;
   \draw    (63,75.9) .. controls (82,52.9) and (65,31.9) .. (55,25.9) ;
   \draw    (41.69,18.8) .. controls (35.17,14.96) and (23.91,29.46) .. (30,47.9) .. controls (36.09,66.34) and (74,84.9) .. (68,95.9) .. controls (62,106.9) and (13.07,59.9) .. (10.04,71.9) .. controls (7,83.9) and (17,128.9) .. (72.32,125.8) ;
   \draw    (71.57,75.35) .. controls (116.57,125.35) and (52.57,153.35) .. (44.5,130) ;
   \draw    (57,82.9) -- (52,89.9) ;
   \draw    (47,96.8) .. controls (43.5,102) and (41.5,108) .. (42.5,116) ;
  \end{tikzpicture}\end{center}
  &$\overset{\RII}{\sim}$&
  \begin{center}\begin{tikzpicture}[scale=0.6,x=0.75pt,y=0.75pt,yscale=-1,xscale=1,baseline=-\dimexpr\fontdimen22\textfont2\relax]
   \draw    (88.7,124.8) .. controls (120.22,126.8) and (109.59,27.15) .. (79.08,13.8) .. controls (48.57,0.45) and (48.57,21.45) .. (46.57,27.45) ;
   \draw    (61.57,72.45) .. controls (80.57,49.45) and (73.57,17.45) .. (57.57,16.45) ;
   \draw    (42.57,17.45) .. controls (5.57,21.45) and (56.57,30.45) .. (57.57,41.45) .. controls (58.57,52.45) and (21.91,28.46) .. (28,46.9) .. controls (34.09,65.34) and (72,83.9) .. (66,94.9) .. controls (60,105.9) and (11.07,58.9) .. (8.04,70.9) .. controls (5,82.9) and (15,127.9) .. (70.32,124.8) ;
   \draw    (72.57,68.72) .. controls (124.57,112.72) and (49.57,155) .. (42.5,129) ;
   \draw    (55,81.9) -- (50,88.9) ;
   \draw    (45,95.8) .. controls (41.5,101) and (39.5,107) .. (40.5,115) ;
   \draw    (45.57,34.45) -- (44.57,40.45) ;
   \draw    (44.57,46.45) .. controls (45.57,54.45) and (54.57,52.45) .. (63.57,59.45) ;
  \end{tikzpicture}\end{center}\\
  $\overset{\RIII}{\sim}$&
  \begin{center}\begin{tikzpicture}[scale=0.6,x=0.75pt,y=0.75pt,yscale=-1,xscale=1,baseline=-\dimexpr\fontdimen22\textfont2\relax]
   \draw    (90.2,124.8) .. controls (121.72,126.8) and (111.09,27.15) .. (80.58,13.8) .. controls (50.07,0.45) and (50.07,21.45) .. (48.07,27.45) ;
   \draw    (73.95,37.32) .. controls (77.95,23.32) and (75.07,17.45) .. (59.07,16.45) ;
   \draw    (44.07,17.45) .. controls (7.07,21.45) and (80.58,36.85) .. (86.95,58.32) .. controls (93.32,79.78) and (78.95,98.32) .. (64.95,99.32) .. controls (50.95,100.32) and (19.95,62.32) .. (9.54,70.9) .. controls (-0.88,79.48) and (16.5,127.9) .. (71.82,124.8) ;
   \draw    (86.95,88.32) .. controls (108.95,120.32) and (51.07,155) .. (44,129) ;
   \draw    (46.5,95.8) .. controls (43,101) and (41,107) .. (42,115) ;
   \draw    (45.95,35.32) .. controls (43.95,46.32) and (56.07,52.45) .. (65.07,59.45) ;
   \draw    (72.95,66.32) .. controls (78.95,70.32) and (76.95,75.32) .. (80.95,80.32) ;
   \draw    (52.35,86.52) .. controls (55.35,80.52) and (69.95,73.32) .. (72.07,48.45) ;
  \end{tikzpicture}\end{center}
  &$\overset{\RII}{\sim}$&
  \begin{center}\begin{tikzpicture}[scale=0.6,x=0.75pt,y=0.75pt,yscale=-1,xscale=1,baseline=-\dimexpr\fontdimen22\textfont2\relax]
   \draw    (65.2,126.8) .. controls (96.72,128.8) and (86.09,29.15) .. (55.58,15.8) .. controls (25.07,2.45) and (25.07,23.45) .. (23.07,29.45) ;
   \draw    (48.95,39.32) .. controls (52.95,25.32) and (50.07,19.45) .. (34.07,18.45) ;
   \draw    (19.07,19.45) .. controls (-17.93,23.45) and (55.58,38.85) .. (61.95,60.32) .. controls (68.32,81.78) and (52.4,96.13) .. (39.95,101.32) .. controls (27.5,106.5) and (33.5,117.5) .. (52.5,122.5) ;
   \draw    (61.95,90.32) .. controls (83.95,122.32) and (23.5,156.5) .. (19,131) ;
   \draw    (27.35,88.52) .. controls (16.5,102.5) and (17.5,123.5) .. (19,131) ;
   \draw    (20.95,37.32) .. controls (18.95,48.32) and (31.07,54.45) .. (40.07,61.45) ;
   \draw    (47.95,68.32) .. controls (53.95,72.32) and (51.95,77.32) .. (55.95,82.32) ;
   \draw    (27.35,88.52) .. controls (30.35,82.52) and (44.95,75.32) .. (47.07,50.45) ;
  \end{tikzpicture}\end{center}
  &$\overset{\RIII}{\sim}$&
  \begin{center}\begin{tikzpicture}[scale=0.6,x=0.75pt,y=0.75pt,yscale=-1,xscale=1,baseline=-\dimexpr\fontdimen22\textfont2\relax]
   \draw    (78.2,132.8) .. controls (109.72,134.8) and (105.5,71.32) .. (102.5,61.32) .. controls (99.5,51.32) and (91.5,37.32) .. (76.5,25.32) ;
   \draw    (64.5,47.72) .. controls (66.5,41.72) and (69.5,37.72) .. (65.5,34.72) ;
   \draw    (23.5,15.32) .. controls (3.5,10.32) and (6.5,23.32) .. (15.5,26) .. controls (24.5,28.68) and (26.42,18.12) .. (35.46,11.92) .. controls (44.5,5.72) and (89.5,4.72) .. (83.5,13) .. controls (77.5,21.28) and (65.5,24) .. (56.5,37) .. controls (47.5,50) and (69.86,49.15) .. (74.95,66.32) .. controls (80.04,83.49) and (65.4,102.13) .. (52.95,107.32) .. controls (40.5,112.5) and (46.5,123.5) .. (65.5,128.5) ;
   \draw    (74.95,96.32) .. controls (96.95,128.32) and (36.5,162.5) .. (32,137) ;
   \draw    (40.35,94.52) .. controls (29.5,108.5) and (30.5,129.5) .. (32,137) ;
   \draw    (69.5,19.32) .. controls (56.5,7.32) and (40.5,19.32) .. (40.5,35.32) .. controls (40.5,51.32) and (44.5,62.32) .. (53.07,67.45) ;
   \draw    (60.95,74.32) .. controls (66.95,78.32) and (64.95,83.32) .. (68.95,88.32) ;
   \draw    (40.35,94.52) .. controls (43.35,88.52) and (57.5,79.72) .. (62.5,56.72) ;
   \draw    (58.5,27.72) .. controls (55.5,24.72) and (53.5,23.72) .. (49.5,22.72) ;
   \draw    (33.5,18.32) -- (39.5,20.32) ;
  \end{tikzpicture}\end{center}\\
  $\overset{\RI}{\sim}$&
  \begin{center}\begin{tikzpicture}[scale=0.6,x=0.75pt,y=0.75pt,yscale=-1,xscale=1,baseline=-\dimexpr\fontdimen22\textfont2\relax]
   \draw    (55.7,139.8) .. controls (87.22,141.8) and (83,78.32) .. (80,68.32) .. controls (77,58.32) and (69,44.32) .. (54,32.32) ;
   \draw    (42,54.72) .. controls (44,48.72) and (47,44.72) .. (43,41.72) ;
   \draw    (19.9,24.78) .. controls (5.8,17.25) and (23.9,5.78) .. (35.9,4.78) .. controls (47.9,3.78) and (67,11.72) .. (61,20) .. controls (55,28.28) and (43,31) .. (34,44) .. controls (25,57) and (47.36,56.15) .. (52.45,73.32) .. controls (57.54,90.49) and (42.9,109.13) .. (30.45,114.32) .. controls (18,119.5) and (24,130.5) .. (43,135.5) ;
   \draw    (52.45,103.32) .. controls (74.45,135.32) and (14,169.5) .. (9.5,144) ;
   \draw    (17.85,101.52) .. controls (7,115.5) and (8,136.5) .. (9.5,144) ;
   \draw    (47,26.32) .. controls (34,14.32) and (18,26.32) .. (18,42.32) .. controls (18,58.32) and (22,69.32) .. (30.57,74.45) ;
   \draw    (38.45,81.32) .. controls (44.45,85.32) and (42.45,90.32) .. (46.45,95.32) ;
   \draw    (17.85,101.52) .. controls (20.85,95.52) and (35,86.72) .. (40,63.72) ;
   \draw    (36,34.72) .. controls (33,31.72) and (31,30.72) .. (27,29.72) ;
  \end{tikzpicture}\end{center}
  &$\overset{\RII}{\sim}$&
  \begin{center}\begin{tikzpicture}[scale=0.6,x=0.75pt,y=0.75pt,yscale=-1,xscale=1,baseline=-\dimexpr\fontdimen22\textfont2\relax]
   \draw    (62.7,143.8) .. controls (94.22,145.8) and (90,82.32) .. (87,72.32) .. controls (84,62.32) and (76,48.32) .. (61,36.32) ;
   \draw    (26.9,28.78) .. controls (12.8,21.25) and (30.9,9.78) .. (42.9,8.78) .. controls (54.9,7.78) and (74.45,15.45) .. (63.45,24.45) .. controls (52.45,33.45) and (54.45,55.45) .. (66.45,70.45) .. controls (78.45,85.45) and (49.9,113.13) .. (37.45,118.32) .. controls (25,123.5) and (31,134.5) .. (50,139.5) ;
   \draw    (59.45,107.32) .. controls (81.45,139.32) and (21,173.5) .. (16.5,148) ;
   \draw    (24.85,105.52) .. controls (14,119.5) and (15,140.5) .. (16.5,148) ;
   \draw    (54,30.32) .. controls (41,18.32) and (25,30.32) .. (25,46.32) .. controls (25,62.32) and (29,73.32) .. (37.57,78.45) ;
   \draw    (45.45,85.32) .. controls (51.45,89.32) and (49.45,94.32) .. (53.45,99.32) ;
   \draw    (24.85,105.52) .. controls (27.85,99.52) and (67.45,64.45) .. (33.45,35.45) ;
  \end{tikzpicture}\end{center}
  &$\overset{\RII}{\sim}$&
  \begin{center}\begin{tikzpicture}[scale=0.6,x=0.75pt,y=0.75pt,yscale=-1,xscale=1,baseline=-\dimexpr\fontdimen22\textfont2\relax]
   \draw    (69.37,143) .. controls (99.37,143) and (94,82.85) .. (91,72.85) .. controls (88,62.85) and (80,48.85) .. (65,36.85) ;
   \draw    (30.9,29.32) .. controls (16.8,21.78) and (34.9,10.32) .. (46.9,9.32) .. controls (58.9,8.32) and (78.45,15.98) .. (67.45,24.98) .. controls (56.45,33.98) and (58.45,55.98) .. (70.45,70.98) .. controls (82.45,85.98) and (53.37,114.13) .. (30.37,123) ;
   \draw    (63.45,107.85) .. controls (75.37,121) and (60.37,153) .. (50.37,155) .. controls (40.37,157) and (33.37,147) .. (28.37,138) .. controls (23.37,129) and (23.37,114) .. (28.85,106.05) ;
   \draw    (18.37,127) .. controls (3.37,133) and (12.37,154) .. (29.37,149) ;
   \draw    (58,30.85) .. controls (45,18.85) and (29,30.85) .. (29,46.85) .. controls (29,62.85) and (33,73.85) .. (41.57,78.98) ;
   \draw    (49.45,85.85) .. controls (55.45,89.85) and (53.45,94.85) .. (57.45,99.85) ;
   \draw    (28.85,106.05) .. controls (32.37,96) and (71.45,64.98) .. (37.45,35.98) ;
   \draw    (40.37,146) .. controls (48.37,145) and (48.37,144) .. (56.37,143) ;
  \end{tikzpicture}\end{center}\\
  $\overset{\RIII}{\sim}$&
  \begin{center}\begin{tikzpicture}[scale=0.6,x=0.75pt,y=0.75pt,yscale=-1,xscale=1,baseline=-\dimexpr\fontdimen22\textfont2\relax]
   \draw    (52.13,142.82) .. controls (91.23,154.75) and (84.23,84.67) .. (80,71.65) .. controls (75.77,58.63) and (69,47.65) .. (54,35.65) ;
   \draw    (17.9,28.12) .. controls (3.8,20.58) and (21.9,9.12) .. (33.9,8.12) .. controls (45.9,7.12) and (65.45,14.78) .. (54.45,23.78) .. controls (43.45,32.78) and (51.27,63.98) .. (57.45,69.78) .. controls (63.63,75.58) and (60.63,78.58) .. (56.63,86.58) ;
   \draw    (51.23,131.02) .. controls (53.23,146.02) and (17.13,169) .. (19.13,146) ;
   \draw    (36.63,75.42) .. controls (11.23,105.02) and (70.63,79.42) .. (72.63,95.42) .. controls (74.63,111.42) and (58.23,120.82) .. (49.23,125.82) .. controls (40.23,130.82) and (20.13,133) .. (19.13,146) ;
   \draw    (45,29.65) .. controls (32,17.65) and (16,29.65) .. (16,45.65) .. controls (16,61.65) and (23.63,67.42) .. (31.63,76.42) ;
   \draw    (36.63,75.42) .. controls (45.23,66.02) and (44.23,49.02) .. (24.45,34.78) ;
   \draw    (55.23,94.62) .. controls (48.23,110.62) and (32.63,112.58) .. (36.73,127.02) ;
   \draw    (39.73,134.02) .. controls (42.73,138.02) and (41.73,137.02) .. (43.73,139.02) ;
   \draw    (35.43,81.78) .. controls (38.43,84.78) and (45.63,82.42) .. (45.43,87.78) ;
   \draw    (46.23,94.62) -- (47.63,101.42) ;
   \draw    (48.23,108.02) .. controls (49.23,114.02) and (48.23,114.02) .. (49.23,120.02) ;
  \end{tikzpicture}\end{center}
  &$\overset{\RII}{\sim}$&
  \begin{center}\begin{tikzpicture}[scale=0.6,x=0.75pt,y=0.75pt,yscale=-1,xscale=1,baseline=-\dimexpr\fontdimen22\textfont2\relax]
   \draw    (60.23,92.65) .. controls (62.23,105.65) and (7.23,121) .. (14.23,157) .. controls (21.23,193) and (82.13,147.13) .. (86.67,133.81) .. controls (91.2,120.49) and (83.31,75.94) .. (81,68.83) .. controls (78.69,61.72) and (70,44.83) .. (55,32.83) ;
   \draw    (21.9,27.3) .. controls (7.8,19.77) and (25.9,8.3) .. (37.9,7.3) .. controls (49.9,6.3) and (69.45,13.97) .. (58.45,22.97) .. controls (47.45,31.97) and (53.23,55.65) .. (58.23,61.65) .. controls (63.23,67.65) and (60.23,76.65) .. (60.63,85.77) ;
   \draw    (55.23,130.2) .. controls (57.23,145.2) and (21.13,168.18) .. (23.13,145.18) ;
   \draw    (40.63,74.6) .. controls (15.23,104.2) and (74.63,78.6) .. (76.63,94.6) .. controls (78.63,110.6) and (62.23,120) .. (53.23,125) .. controls (44.23,130) and (24.13,132.18) .. (23.13,145.18) ;
   \draw    (49,28.83) .. controls (36,16.83) and (20,28.83) .. (20,44.83) .. controls (20,60.83) and (27.63,66.6) .. (35.63,75.6) ;
   \draw    (40.63,74.6) .. controls (49.23,65.2) and (48.23,48.2) .. (28.45,33.97) ;
   \draw    (39.43,80.97) .. controls (42.43,83.97) and (49.63,81.6) .. (49.43,86.97) ;
   \draw    (50.23,93.8) -- (51.63,100.6) ;
   \draw    (52.23,107.2) .. controls (53.23,113.2) and (52.23,113.2) .. (53.23,119.2) ;
  \end{tikzpicture}\end{center}
  &$\overset{\RII}{\sim}$&
  \begin{center}\begin{tikzpicture}[scale=0.6,x=0.75pt,y=0.75pt,yscale=-1,xscale=1,baseline=-\dimexpr\fontdimen22\textfont2\relax]
   \draw    (59.23,93.65) .. controls (61.23,106.65) and (6.23,122) .. (13.23,158) .. controls (20.23,194) and (77.45,146.78) .. (85.67,134.81) .. controls (93.88,122.83) and (82.31,76.94) .. (80,69.83) .. controls (77.69,62.72) and (69,45.83) .. (54,33.83) ;
   \draw    (20.9,28.3) .. controls (6.8,20.77) and (24.9,9.3) .. (36.9,8.3) .. controls (48.9,7.3) and (68.45,14.97) .. (57.45,23.97) .. controls (46.45,32.97) and (52.23,56.65) .. (57.23,62.65) .. controls (62.23,68.65) and (59.23,77.65) .. (59.63,86.77) ;
   \draw    (54.23,131.2) .. controls (56.23,146.2) and (20.13,169.18) .. (22.13,146.18) ;
   \draw    (27.7,35.83) .. controls (26.7,105.83) and (74.7,84.83) .. (75.63,95.6) .. controls (76.57,106.37) and (61.23,121) .. (52.23,126) .. controls (43.23,131) and (23.13,133.18) .. (22.13,146.18) ;
   \draw    (48,27.83) .. controls (35,15.83) and (24.3,26.83) .. (19,43.83) .. controls (13.7,60.83) and (17.26,66.05) .. (20.91,74.75) .. controls (24.56,83.44) and (31.7,89.83) .. (48.7,97.83) ;
   \draw    (51.23,108.2) .. controls (52.23,114.2) and (51.23,114.2) .. (52.23,120.2) ;
  \end{tikzpicture}\end{center}\\
  $\overset{\RI}{\sim}$&
  \begin{center}\begin{tikzpicture}[scale=0.6,x=0.75pt,y=0.75pt,yscale=-1,xscale=1,baseline=-\dimexpr\fontdimen22\textfont2\relax]
   \draw    (57.23,94.65) .. controls (59.23,107.65) and (4.23,123) .. (11.23,159) .. controls (18.23,195) and (75.45,147.78) .. (83.67,135.81) .. controls (91.88,123.83) and (80.31,77.94) .. (78,70.83) .. controls (75.69,63.72) and (67,46.83) .. (52,34.83) ;
   \draw    (18.9,29.3) .. controls (8.72,14.65) and (22.9,10.3) .. (34.9,9.3) .. controls (46.9,8.3) and (77.18,4.28) .. (55.45,24.97) .. controls (33.72,45.65) and (74.72,67.65) .. (59.72,84.65) ;
   \draw    (26.72,39.65) .. controls (51.72,65.65) and (30.72,77.65) .. (56.55,90.94) .. controls (82.37,104.23) and (78.72,113.65) .. (72.72,122.65) .. controls (66.72,131.65) and (36.88,145) .. (31.88,128) ;
   \draw    (46,28.83) .. controls (33,16.83) and (20.72,33.65) .. (13.72,50.65) .. controls (6.72,67.65) and (18.88,92) .. (25.88,112) ;
  \end{tikzpicture}\end{center}
  &$\overset{\RO}{\sim}$&
  \begin{center}\begin{tikzpicture}[scale=0.6,x=0.75pt,y=0.75pt,yscale=-1,xscale=1,baseline=-\dimexpr\fontdimen22\textfont2\relax]
   \draw    (42.78,84.68) .. controls (40.23,71.78) and (94.52,54.09) .. (85.99,18.43) .. controls (77.46,-17.24) and (22.31,32.38) .. (14.62,44.69) .. controls (6.92,57.01) and (20.44,102.36) .. (23.05,109.37) .. controls (25.67,116.37) and (35.07,132.87) .. (50.57,144.22) ;
   \draw    (83.88,148.33) .. controls (94.68,162.54) and (80.69,167.49) .. (68.74,169) .. controls (56.8,170.51) and (26.71,175.82) .. (47.54,154.23) .. controls (68.37,132.63) and (26.47,112.41) .. (40.73,94.78) ;
   \draw    (75.62,138.33) .. controls (49.53,113.42) and (70,100.53) .. (43.63,88.36) .. controls (17.26,76.18) and (20.51,66.62) .. (26.12,57.37) .. controls (31.73,48.12) and (60.96,33.51) .. (66.68,50.28) ;
   \draw    (56.82,149.96) .. controls (70.32,161.39) and (81.87,144.07) .. (88.14,126.78) .. controls (94.41,109.5) and (81.21,85.69) .. (73.36,66.01) ;
  \end{tikzpicture}\end{center}
  &$\overset{\RO}{\sim}$&
  \begin{center}\begin{tikzpicture}[scale=0.6,x=0.75pt,y=0.75pt,yscale=-1,xscale=1,baseline=-\dimexpr\fontdimen22\textfont2\relax]
   \draw    (35.68,121) .. controls (-0.32,123) and (-3.32,23) .. (46.68,10) .. controls (96.68,-3) and (88.68,37) .. (69.68,60) ;
   \draw    (86.68,113) .. controls (107.68,82) and (16.68,75) .. (70.68,25) ;
   \draw    (91.68,15) .. controls (122.68,-1) and (168.68,122) .. (56.68,122) ;
   \draw    (52.68,76) .. controls (13.68,132) and (81.68,145) .. (85.68,127) ;
  \end{tikzpicture}\end{center}
 \end{longtable}\end{center}
\end{proof}

There exists some knots that have both its reverse and mirror to be distinct to its original knot. One such example is the knot $K=9_{32}$. In fact, $K,rK,\overline{K},\overline{rK}$ are all distinct oriented knots (\cite[p.~4]{Likorish}). However, the reverse of a knot is \textit{often} equivalent to its original knot.

\begin{proposition}
 The reflection of a link diagram negates the writhe of the original link diagram.
\end{proposition}
\begin{proof}[Proof sketch]
 We look at a single crossing and see what happens to the sign when the crossing is reflected,
 \[\epsilon\left(\signpositive\right)=-\epsilon\left(\reflectbox{\signpositive}\right).\]
\end{proof}

\begin{proposition}\label{proposition:writhe_regular_isotopy}
 The writhe of a link preserves under regular isotopy, but not under ambient isotopy.
\end{proposition}
\begin{proof}
 A function that is preserved under regular isotopy but not under ambient isotopy, means that we need to show the writhe does not change under $\RII$ and $\RIII$ moves, but does change under $\RI$ moves.
 \begin{align*}
  \RI: \qquad &\omega\left(\reflectbox{\knotsinmath{\begin{scope}[x=0.75pt,y=0.75pt,yscale=-1,xscale=1,yshift=-0.6cm,scale=0.5]
   \draw    (36,52.23) .. controls (29,84.23) and (4,67.23) .. (15,47.23) .. controls (25.78,27.63) and (47.12,42.61) .. (71.5,53.57) ;
   \draw [shift={(73,54.23)}, rotate = 203.75] [color={rgb, 255:red, 0; green, 0; blue, 0 }  ][line width=0.75]    (10.93,-3.29) .. controls (6.95,-1.4) and (3.31,-0.3) .. (0,0) .. controls (3.31,0.3) and (6.95,1.4) .. (10.93,3.29)   ;
   \draw    (47,11.23) -- (42,30.23) ;
  \end{scope}}}\!\right) =
  \omega\left(\knotsinmath{\begin{scope}[scale=0.2]
   \draw[black,->] (0,2) -- (0,-2);
  \end{scope}}\right) + 1\\
  &\omega\left(\knotsinmath{\begin{scope}[x=0.75pt,y=0.75pt,yscale=-1,xscale=1,yshift=-0.6cm,scale=0.5]
   \draw    (36,52.23) .. controls (29,84.23) and (4,67.23) .. (15,47.23) .. controls (25.78,27.63) and (47.12,42.61) .. (71.5,53.57) ;
   \draw [shift={(73,54.23)}, rotate = 203.75] [color={rgb, 255:red, 0; green, 0; blue, 0 }  ][line width=0.75]    (10.93,-3.29) .. controls (6.95,-1.4) and (3.31,-0.3) .. (0,0) .. controls (3.31,0.3) and (6.95,1.4) .. (10.93,3.29)   ;
   \draw    (47,11.23) -- (42,30.23) ;
  \end{scope}}\right) =
  \omega\left(\knotsinmath{\begin{scope}[scale=0.2]
   \draw[black,->] (0,2) -- (0,-2);
  \end{scope}}\right) - 1
 \end{align*}
 For $\RII$, given any orientation, we can see
 \begin{align*}
  \omega\left(\knotsinmath{\begin{scope}[x=0.75pt,y=0.75pt,yscale=-1,xscale=1,yshift=-0.4cm,scale=0.5]
    \draw    (14,48) .. controls (53.6,18.3) and (74.58,18.23) .. (112.84,47.12) ;
    \draw [shift={(114,48)}, rotate = 217.35] [color={rgb, 255:red, 0; green, 0; blue, 0 }  ][line width=0.75]    (10.93,-3.29) .. controls (6.95,-1.4) and (3.31,-0.3) .. (0,0) .. controls (3.31,0.3) and (6.95,1.4) .. (10.93,3.29)   ;
    \draw    (39,41.23) .. controls (52,53.23) and (69,51.23) .. (83,37.23) ;
    \draw    (20,20.23) -- (29,29.23) ;
    \draw    (94,24.23) -- (106.77,7.81) ;
    \draw [shift={(108,6.23)}, rotate = 487.87] [color={rgb, 255:red, 0; green, 0; blue, 0 }  ][line width=0.75]    (10.93,-3.29) .. controls (6.95,-1.4) and (3.31,-0.3) .. (0,0) .. controls (3.31,0.3) and (6.95,1.4) .. (10.93,3.29)   ;
   \end{scope}}\right) = \omega\left(\knotsinmath{\begin{scope}[x=0.75pt,y=0.75pt,yscale=-1,xscale=1,yshift=-0.4cm,scale=0.5]
    \draw    (14,48) .. controls (53.6,18.3) and (74.58,18.23) .. (112.84,47.12) ;
    \draw [shift={(114,48)}, rotate = 217.35] [color={rgb, 255:red, 0; green, 0; blue, 0 }  ][line width=0.75]    (10.93,-3.29) .. controls (6.95,-1.4) and (3.31,-0.3) .. (0,0) .. controls (3.31,0.3) and (6.95,1.4) .. (10.93,3.29)   ;
    \draw    (39,41.23) .. controls (52,53.23) and (69,51.23) .. (83,37.23) ;
    \draw    (29,29.23) -- (16.32,14.74) ;
    \draw [shift={(15,13.23)}, rotate = 408.81] [color={rgb, 255:red, 0; green, 0; blue, 0 }  ][line width=0.75]    (10.93,-3.29) .. controls (6.95,-1.4) and (3.31,-0.3) .. (0,0) .. controls (3.31,0.3) and (6.95,1.4) .. (10.93,3.29)   ;
    \draw    (94,24.23) -- (102,12.23) ;
   \end{scope}}\right) &= \omega\left(\knotsinmath{\begin{scope}[x=0.75pt,y=0.75pt,yscale=-1,xscale=1,yshift=-0.4cm,scale=0.5]
    \draw    (107,44.23) .. controls (76.31,17.5) and (48.56,15.28) .. (15.02,50.16) ;
    \draw [shift={(14,51.23)}, rotate = 313.36] [color={rgb, 255:red, 0; green, 0; blue, 0 }  ][line width=0.75]    (10.93,-3.29) .. controls (6.95,-1.4) and (3.31,-0.3) .. (0,0) .. controls (3.31,0.3) and (6.95,1.4) .. (10.93,3.29)   ;
    \draw    (39,41.23) .. controls (52,53.23) and (69,51.23) .. (83,37.23) ;
    \draw    (20,20.23) -- (29,29.23) ;
    \draw    (94,24.23) -- (106.77,7.81) ;
    \draw [shift={(108,6.23)}, rotate = 487.87] [color={rgb, 255:red, 0; green, 0; blue, 0 }  ][line width=0.75]    (10.93,-3.29) .. controls (6.95,-1.4) and (3.31,-0.3) .. (0,0) .. controls (3.31,0.3) and (6.95,1.4) .. (10.93,3.29)   ;
   \end{scope}}\right) = \omega\left(\knotsinmath{\begin{scope}[x=0.75pt,y=0.75pt,yscale=-1,xscale=1,yshift=-0.4cm,scale=0.5]
    \draw    (107,44.23) .. controls (76.31,17.5) and (48.56,15.28) .. (15.02,50.16) ;
    \draw [shift={(14,51.23)}, rotate = 313.36] [color={rgb, 255:red, 0; green, 0; blue, 0 }  ][line width=0.75]    (10.93,-3.29) .. controls (6.95,-1.4) and (3.31,-0.3) .. (0,0) .. controls (3.31,0.3) and (6.95,1.4) .. (10.93,3.29)   ;
    \draw    (39,41.23) .. controls (52,53.23) and (69,51.23) .. (83,37.23) ;
    \draw    (17.26,14.79) -- (29,29.23) ;
    \draw [shift={(16,13.23)}, rotate = 50.91] [color={rgb, 255:red, 0; green, 0; blue, 0 }  ][line width=0.75]    (10.93,-3.29) .. controls (6.95,-1.4) and (3.31,-0.3) .. (0,0) .. controls (3.31,0.3) and (6.95,1.4) .. (10.93,3.29)   ;
    \draw    (94,24.23) -- (102,12.23) ;
   \end{scope}}\right)\\
   &= \omega\left(\resolveriimove\right) = 0 \text{, for the given orientation.}
 \end{align*}
 And it should be clear that it is also true for its reflection, i.e., $\omega\left(\riimove\right)=0=\omega\left(\resolveriimove\right)$, given any orientation. 
 Now $\RIII$ is analogous to what we did for $\RII$. So the details are left for the reader, i.e., given any orientation we have,
 \[
  \omega\left(\knotsinmath{\begin{scope}[scale=0.2]
    \draw[draw=white,double=black ,ultra thick,double distance=0.5pt] (  0:2) .. controls (270:0.7) .. (180:2);
    \draw[draw=white,double=black  ,ultra thick,double distance=0.5pt] ( 60:2) .. controls (150:0.7) .. (240:2);
    \draw[draw=white,double=black,ultra thick,double distance=0.5pt] (120:2) .. controls ( 30:0.7) .. (300:2);
   \end{scope}
  }\right) = 0 = \omega\left(\knotsinmath{
   \begin{scope}[scale=0.2]
    \draw[draw=white,double=black ,ultra thick,double distance=0.5pt] (  0:2) .. controls ( 90:0.7) .. (180:2);
    \draw[draw=white,double=black  ,ultra thick,double distance=0.5pt] ( 60:2) .. controls (330:0.7) .. (240:2);
    \draw[draw=white,double=black,ultra thick,double distance=0.5pt] (120:2) .. controls (210:0.7) .. (300:2);
   \end{scope}
  }\right)
 \]
\end{proof}

\subsection{Invariants}\label{section:invariants}

By only considering the Reidemeister moves, it is difficult to distinguish links. This is because there is no method that can tell us whether we have not made the right number of Reidemeister moves yet. So the best way to distinguish links is by using invariants, which are algebraic ways of characterising these topological objects.

\begin{definition}[invariant]
 An \textit{invariant}\index{invariant} $f$ is a function
 \[f\colon\{\text{equivalence classes of topological objects}\}\to{S}\]
 for some set $S$, with the property $f(A)=f(B)$ whenever $A$ and $B$ are equivalent.
\end{definition}

So a \textit{link invariant} is just an invariant function $f\colon\{\text{link type}\}\to{S}$ for some set $S$.

Thus, to prove that two links are in fact distinct (the main problem of knot theory), we need to show that they produce different outputs in a link invariant.

\begin{remark}
 Invariant functions can only tell the objects apart - they can not show equivalency.
\end{remark}

\begin{definition}[crossing number]
 The \textit{crossing number}\index{crossing number} is a link invariant, such that it is defined as the minimal number of crossings needed for a diagram of the link $L$, where the minimum is taken over the infinity of possible diagrams of a link. We write $c(L)$ for the link $L$.
\end{definition}

For low crossing numbers, we can find the knots by tabulating all knot diagrams; see the following result. In general, however, it is difficult to know which knots have the same knot type. This is solved by associating a knot to an invariant.

\begin{proposition}\label{proposition:no_knot_1_2}
 There are no knots with crossing number one or two.
\end{proposition}
\begin{proof}
 We first draw the link universe with one component and one crossing and then show all the possible choices. The strategy here is to draw a link universe with the number of crossings we need without having to create more than one component.
 
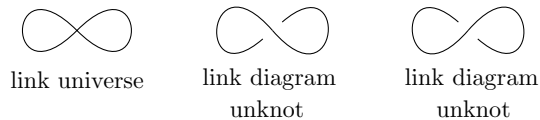
\begin{center}
  \begin{tikzpicture}[scale=1.3]
   \begin{scope}
    \node (a) at (0,0) {};
    \draw (a.center) .. controls (a.8 north west) and (a.8 south west) .. (a.center);
    \draw (a.center) .. controls (a.8 north east) and (a.8 south east) .. (a.center);
    \draw (0,-0.5) node[scale=0.7] {link universe};
   \end{scope}
   \begin{scope}[xshift=2cm,rotate=-4]
    \node (a) at (0,0) {};
    \draw (a.center) .. controls (a.8 north west) and (a.8 south west) .. (a);
    \draw (a) .. controls (a.8 north east) and (a.8 south east) .. (a.center);
    \draw (0,-0.5) node[scale=0.7,rotate=4] {link diagram};
    \draw (0,-0.8) node[scale=0.7,rotate=4] {unknot};
   \end{scope}
   \begin{scope}[xshift=4cm,rotate=4]
    \node (a) at (0,0) {};
    \draw (a) .. controls (a.8 north west) and (a.8 south west) .. (a.center);
    \draw (a.center) .. controls (a.8 north east) and (a.8 south east) .. (a);
    \draw (0,-0.5) node[scale=0.7,rotate=-4] {link diagram};
    \draw (0,-0.8) node[scale=0.7,rotate=-4] {unknot};
   \end{scope}
  \end{tikzpicture}
 \end{center}
 We clearly see that the first and second choices are unknots (via $\RI$). Therefore, there are no knots with crossing number one.
 We leave the details for crossing number two to the reader, as it follows the same argument.
\end{proof}

\begin{definition}[unlinking number]
 The \textit{unlinking number}\index{unlinking number} is another link invariant. It is defined to be the minimum number of crossing changes (from ``over'' to ``under'' or vice versa) needed to change link $L$ to the unlink, where the minimum is taken over all possible sets of crossing changes in all possible diagrams of $L$. We write this as $u(L)$.
\end{definition}

So, like the crossing number invariant, the unlinking number is very hard to calculate.

\begin{definition}[complexity {\cite{knotes}}]\label{definition:complexity}
 We define the \textit{complexity}\index{complexity} of a link $\kappa(L)$ to be the minimal ordered pair of integers $(c(L),u(L))$, where $c(L)$ is the crossing number of link $L$ and $u(L)$ is the unlinking number of $L$.
 We order complexities \textit{lexocographically}, i.e., $(a,b)<(c,d)$ if and only if $(a<c)$ or ($a=c$ and $b<d$).
\end{definition}

\begin{example}
 The most basic complexity, is that of the unknot, i.e., $\kappa\left(\KPA\right)=(0,0)$.
 
 For the following trefoil,
 
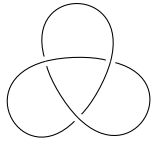
\begin{center}
  \begin{tikzpicture}[scale=0.5]
   \begin{scope}
    \begin{scope}[rotate=  0]
     \node (a) at (0,-1) {};
    \end{scope}
    \begin{scope}[rotate=120]
     \node (b) at (0,-1) {};
    \end{scope}
    \begin{scope}[rotate=240]
     \node (c) at (0,-1) {};
    \end{scope}
    \draw (a.center) .. controls (a.4 north west) and (c.4 north east) .. (c);
    \draw (b.center) .. controls (b.4 north west) and (a.4 north east) .. (a);
    \draw (c.center) .. controls (c.4 north west) and (b.4 north east) .. (b);
    \draw (a) .. controls (a.16 south west) and (c.16 south east) .. (c.center);
    \draw (b) .. controls (b.16 south west) and (a.16 south east) .. (a.center);
    \draw (c) .. controls (c.16 south west) and (b.16 south east) .. (b.center);
   \end{scope}
  \end{tikzpicture}
 \end{center}
 We can see that it has complexity $(3,1)$, by changing one of its crossings as shown below,
 
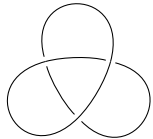
\begin{center}
  \begin{tikzpicture}[scale=0.5]
   \begin{scope}
    \begin{scope}[rotate=  0] \node (a) at (0,-1) {}; \end{scope}
    \begin{scope}[rotate=120] \node (b) at (0,-1) {}; \end{scope}
    \begin{scope}[rotate=240] \node (c) at (0,-1) {}; \end{scope}
    \draw (a) .. controls (a.4 north west) and (c.4 north east) .. (c);
    \draw (b.center) .. controls (b.4 north west) and (a.4 north east) .. (a.center);
    \draw (c.center) .. controls (c.4 north west) and (b.4 north east) .. (b);
    \draw (a.center) .. controls (a.16 south west) and (c.16 south east) .. (c.center);
    \draw (b) .. controls (b.16 south west) and (a.16 south east) .. (a);
    \draw (c) .. controls (c.16 south west) and (b.16 south east) .. (b.center);
   \end{scope}
  \end{tikzpicture}
 \end{center}
 For the figure-eight knot $4_1$, it is not hard to see that it has complexity $(4,1)$ by also changing one of its crossings.
 
 So, $\kappa(0_1)<\kappa(3_1)<\kappa(4_1)$.
\end{example}

\begin{lemma}
 Every link $L$ with a corresponding diagram $D$ has an unlinking number.
\end{lemma}
\begin{proof}[Proof sketch]
 We begin by choosing a starting point on all of the components on the link diagram. Following a clockwise orientation, we follow each component of the link diagram from the starting point.
 When a crossing appears, if it is an overpass, we change it to an underpass, otherwise we keep it as it is.
 Once we are done, we have an unlink. This shows us that the unlinking number $u(L)\leq{c(L)}$.
\end{proof}

There are many other link invariants, but in this paper we will mainly be focusing on the Jones polynomial and the Kauffman bracket.

\newpage
\section{The Kauffman bracket construction}\label{chapter:kauffman}
 We will now define the Jones polynomial. This will be based on the Kauffman bracket \cite{kauffman1986,kauffman1988,kauffman_stat_mech}.
We will then show that this admits a skein relation \cite{conway_skein,homfly,jones1985}. We will not go deep into the theory of skein relations. The interested reader is directed to \cite{conway_skein} and \cite{homfly}.

Our discussion is based on that in Kauffman's paper \cite{kauffman1986}, Jones' paper \cite{jones1985}, Jones' book \cite{jonesbook}, Conway's paper \cite{conway_skein}, and HOMFLY-PT's paper \cite{homfly,pt}.

\begin{remark}
 We note that HOMFLY-PT is an acronym for the eight people who wrote the two papers, namely, Hoste, Ocneanu, Millett, Freyd, Lickorish, and Yetter. The PT are the other two who wrote a different paper on the same subject: Pyztycki and Traczyk.
\end{remark}

\textbf{Notation \cite[p.~330]{conway_skein}.}
 Formally, when we look at an enclosed region of a link such that there are four distinct points (NW, NE, SE, SW), then we call the enclosed region a \textit{tangle}\index{tangle}.
 \begin{figure}[H]
  \vskip 0.5cm
  \centering
  \begin{tikzpicture}
   \draw[dashed] (0,0) circle (0.5);
   \draw (0,0) node {$L$};
   \draw (0.4,0.4)--(0.7,0.7);
   \draw (0.7,0.7) node[anchor=south west] {NE};
   \draw (-0.4,0.4)--(-0.7,0.7);
   \draw (-0.7,0.7) node[anchor=south east] {NW};
   \draw (0.4,-0.4)--(0.7,-0.7);
   \draw (0.7,-0.7) node[anchor=north west] {SE};
   \draw (-0.4,-0.4)--(-0.7,-0.7);
   \draw (-0.7,-0.7) node[anchor=north east] {SW};
  \end{tikzpicture}
  \vskip 0.5cm
  \caption{tangle}
  \label{figure:tangle}
 \end{figure}
 We call it an \textit{elementary $2$-tangle}\index{tangle!elementary $2$-tangle} if the tangle consists of exactly one crossing or none. Such as the following.
 \begin{figure}[H]
  \centering
  \begin{tikzpicture}[scale=0.6]
   \begin{scope}
    \draw[dashed] (0,0) circle(1.5);
    \draw[-] (-1,-1) to [bend left=40] (1,-1);
    \draw[-] (1,1) to [bend left=40] (-1,1);
   \end{scope}
   \draw (0,-2) node[scale=1.5] {$\tangle{\infty}$};
   \begin{scope}[xshift=5cm]
    \draw[dashed] (0,0) circle(1.5);
    \draw[-] (-1,-1)--(1,1);
    \draw[ultra thick, draw=white,double=black,double distance=0.5pt] (1,-1)--(-1,1);
    \draw[-] (-0.9,0.9)--(-1,1);
    \draw (0,-2) node[scale=1.5] {$\tangle{-1}$};
   \end{scope}
   \begin{scope}[xshift=10cm]
    \draw[dashed] (0,0) circle(1.5);
    \draw[-] (1,-1)--(-1,1);
    \draw[ultra thick,draw=white,double=black,double distance=0.5pt] (-1,-1)--(1,1);
    \draw[-] (-0.9,-0.9)--(1,1);
    \draw (0,-2) node[scale=1.5] {$\tangle{1}$};
   \end{scope}
   \begin{scope}[xshift=15cm]
    \draw[dashed] (0,0) circle(1.5);
    \draw[-] (-1,-1) to [bend right=40] (-1,1);
    \draw[-] (1,-1) to [bend left=40] (1,1);
    \draw (0,-2) node[scale=1.5] {$\tangle{0}$};
   \end{scope}
  \end{tikzpicture}
  \caption{elementary $2$-tangles}
  \label{figure:elementary_2_tangles}
 \end{figure}
 These elementary $2$-tangles are usually recognised as the skein relation of a link (see \S\ref{section:skein}).
 For more details on tangles, the interested reader is directed to Conway's paper \cite{conway_skein}.
 
\begin{definition}[Laurent polynomial]
 A \textit{Laurent polynomial}\index{Laurent polynomial} is an expression of the form $p=\sum_{k=-N}^Na_kA^k$ for some natural number $N$ and some list of coefficients $a_{-N},\dotsc,a_N\in\ZZ$.
\end{definition}

The aim of the Kauffman bracket is to define an easier, more combinatorial approach to the Jones polynomial.

\subsection{`states model' and the Kauffman bracket polynomial}


Our first definition of the Jones polynomial will require us to define the Kauffman bracket. We will begin by defining a state of a link diagram (as was presented in Kauffman's paper \cite{kauffman1988}).

Given a non-oriented link diagram, we notice that each crossing is divided into four regions. By convention, we label the regions of each crossing $A$ and $B$, by starting from the over-strand in a counterclockwise fashion. Both possible crossings are labelled below.
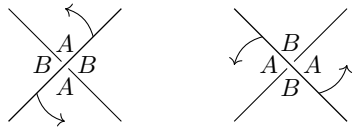
\begin{center}
 \begin{tikzpicture}[scale=0.75]
  \begin{scope}
   \draw (-1,1)--(1,-1);
   \draw[ultra thick,draw=white,double=black,double distance=0.5pt] (-1,-1)--(1,1);
   \draw (0,0) node[scale=1.25,anchor=north] {$A$};
   \draw (0,0) node[scale=1.25,anchor=west] {$B$};
   \draw (0,0) node[scale=1.25,anchor=south] {$A$};
   \draw (0,0) node[scale=1.25,anchor=east] {$B$};
   \draw[->] (0.5,0.5) to [bend right=30] (0,1);
   \draw[->] (-0.5,-0.5) to [bend right=30] (0,-1);
  \end{scope}
  \begin{scope}[xshift=4cm]
   \draw (-1,-1)--(1,1);
   \draw[ultra thick,draw=white,double=black,double distance=0.5pt] (-1,1)--(1,-1);
   \draw (0,0) node[scale=1.25,anchor=west] {$A$};
   \draw (0,0) node[scale=1.25,anchor=south] {$B$};
   \draw (0,0) node[scale=1.25,anchor=east] {$A$};
   \draw (0,0) node[scale=1.25,anchor=north] {$B$};
   \draw[->] (-0.5,0.5) to [bend right=30] (-1,0);
   \draw[->] (0.5,-0.5) to [bend right=30] (1,0);
  \end{scope}
 \end{tikzpicture}
\end{center}

\begin{definition}[splitting marker]
 Given a non-oriented link diagram, a \textit{splitting marker}\index{splitting marker} is defined as a marker that essentially smooths the crossing. We colour the splitting marker in red in the figure below, and show how we can get two different split diagrams for any crossing.
 \begin{figure}[H]
  \centering
  \begin{tikzpicture}[scale=0.5]
   \begin{scope}
    \draw (-1,-1)--(1,1);
    \draw (-1,1)--(1,-1);
    \draw[draw=red] (-0.5,0)--(0.5,0);
    \draw[->] (1,0)--(2,0);
    \draw (2,1) .. controls (3,0) .. (4,1);
    \draw (2,-1) .. controls (3,0) .. (4,-1);
   \end{scope}
   \begin{scope}[xshift=6.5cm]
    \draw (-1,-1)--(1,1);
    \draw (-1,1)--(1,-1);
    \draw[draw=red] (0,0.5)--(0,-0.5);
    \draw[->] (1,0)--(2,0);
    \draw (2,1) .. controls (3,0) .. (2,-1);
    \draw (4,1) .. controls (3,0) .. (4,-1);
   \end{scope}
  \end{tikzpicture}
  \caption{splitting marker}
  \label{figure:splitting_marker}
 \end{figure}
\end{definition}

\begin{definition}[state]\label{definition:state}
 A \textit{state}\index{state} of a link $L$ is a choice of the splitting marker for each crossing in $L$.
 So, by convention, we define an $A$-smoothing and a $B$-smoothing as the following.
 \begin{figure}[H]
  \centering
  \begin{tikzpicture}[scale=0.5]
   \draw (-1,-1)--(1,1);
   \draw[ultra thick,draw=white,double=black] (-1,1)--(1,-1);
   \draw (0,0) node[scale=1.75,anchor=west] {$A$};
   \draw (0,0) node[scale=1.75,anchor=south] {$B$};
   \draw (0,0) node[scale=1.75,anchor=east] {$A$};
   \draw (0,0) node[scale=1.75,anchor=north] {$B$};
   \draw[->] (0,-1)--(-2,-3);
   \draw[->] (0,-1)--(2,-3);
   \begin{scope}[xshift=-4cm,yshift=-5cm]
    \draw (0,1) .. controls (1,0) .. (2,1);
    \draw (0,-1) .. controls (1,0) .. (2,-1);
    \draw[draw=red] (1.5,0)--(0.5,0);
    \draw (1,-2) node[scale=2] {$A$-smoothing};
   \end{scope}
   \begin{scope}[xshift=2cm,yshift=-5cm]
    \draw (0,1) .. controls (1,0) .. (0,-1);
    \draw (2,1) .. controls (1,0) .. (2,-1);
    \draw[draw=red] (1,0.5)--(1,-0.5);
    \draw (1,-2) node[scale=2] {$B$-smoothing};
   \end{scope}
   \begin{scope}[xshift=12cm]
    \draw (-1,1)--(1,-1);
    \draw[ultra thick,draw=white,double=black] (-1,-1)--(1,1);
    \draw (0,0) node[scale=2,anchor=north] {$A$};
    \draw (0,0) node[scale=2,anchor=west] {$B$};
    \draw (0,0) node[scale=2,anchor=south] {$A$};
    \draw (0,0) node[scale=2,anchor=east] {$B$};
    \draw[->] (0,-1)--(-2,-3);
    \draw[->] (0,-1)--(2,-3);
   \end{scope}
   \begin{scope}[xshift=8cm,yshift=-5cm]
    \draw (0,1) .. controls (1,0) .. (0,-1);
    \draw (2,1) .. controls (1,0) .. (2,-1);
    \draw[draw=red] (1,0.5)--(1,-0.5);
    \draw (1,-2) node[scale=2] {$A$-smoothing};
   \end{scope}
   \begin{scope}[xshift=14cm,yshift=-5cm]
    \draw (0,1) .. controls (1,0) .. (2,1);
    \draw (0,-1) .. controls (1,0) .. (2,-1);
    \draw[draw=red] (1.5,0)--(0.5,0);
    \draw (1,-2) node[scale=2] {$B$-smoothing};
   \end{scope}
  \end{tikzpicture}
  \caption{states}
  \label{figure:states}
 \end{figure}
 So for an $n$ crossing diagram $L$ we have a state $s$ is a bijective map $s\mapsto\{A,B\}^n$. This means we have $2^n$ states for an $n$ crossing diagram.
\end{definition}

\textbf{Notation.}
When we smooth out a crossing on a diagram $D$, we call the $A$-smoothed diagram $D_{\{A\}}$, and we call the $B$-smoothed diagram $D_{\{B\}}$. The subscripted set represents the smoothings of the respected crossings on $D$. For example, if we have two crossings $c_1,c_2$ on a diagram $D$, then the four states would be $D_{\{A,A\}}$, $D_{\{A,B\}}$, $D_{\{B,A\}}$, and $D_{\{B,B\}}$.

\begin{remark}
 Since these are tangles, we can instead define the splitting of a single crossing. We immediately get the other crossings, by rotating by $\pi/2$.
 For instance, if we define $\tangle{-1}\mapsto\{\tangle{\infty}_A,\tangle{0}_B\}$, then by rotating all the tangles and keeping $A$ and $B$ in place, we get $\tangle{1}\mapsto\{\tangle{0}_A,\tangle{\infty}_B\}$.
\end{remark}

\begin{example}
 Consider the trefoil knot diagram $D$ drawn below. To find all the states of $D$, we need to resolve all three crossings $c_1$, $c_2$, and $c_3$ to find all possible $2^3$ states.
 
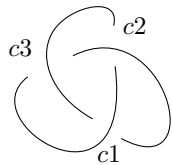
\begin{center}
  \begin{tikzpicture}[every path/.style={},scale=0.02,yscale=-1]
   \clip (10,120)--(10,11)--(120,11)--(120,120)--cycle;
   \draw    (81,52) .. controls (94,158) and (-17,89) .. (23,59) ;
   \draw    (85,100) .. controls (145,132) and (115,14) .. (53,45) ;
   \draw    (80,25) .. controls (85,5) and (53,11) .. (41,27) .. controls (29,43) and (36,70) .. (66,87) ;
   \draw (77,110) node[scale=48,yscale=-1] {$c1$};
   \draw (18,40) node[scale=48,yscale=-1] {$c3$};
   \draw (94,25) node[scale=48,yscale=-1] {$c2$};
  \end{tikzpicture}
 \end{center}
 So, the first two possible diagrams are obtained by considering crossing $c_1$ and finding the $A$-smoothing and $B$-smoothing.
 
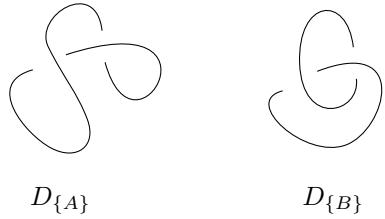
\begin{center}
  \begin{tikzpicture}[scale=0.02,yscale=-1]
   \draw    (101,52.5) .. controls (121,125.5) and (186,11.5) .. (75,47.5) ;
   \draw    (53,57.5) .. controls (20,66.5) and (48.03,100.44) .. (65.02,109.47) .. controls (82,118.5) and (96,108.5) .. (89.21,90.43) .. controls (82.42,72.36) and (73.63,63.29) .. (63.32,42.39) .. controls (53,21.5) and (95,-4.5) .. (99,31.5) ;
   \draw    (219,71.5) .. controls (191,78.5) and (234,117.5) .. (261,107.5) .. controls (288,97.5) and (304,34.5) .. (242,58.5) ;
   \draw    (268,63.5) .. controls (271,86.5) and (232,96.5) .. (230.32,53.39) .. controls (228.63,10.29) and (262,6.5) .. (266,42.5) ;
  \end{tikzpicture}\\
  {$D_{\{A\}}$\qquad\qquad\qquad\qquad$D_{\{B\}}$}
 \end{center}
 The next possible diagrams are obtained by resolving the second crossing $c_2$ of $D_{\{A\}}$ and $D_{\{B\}}$.
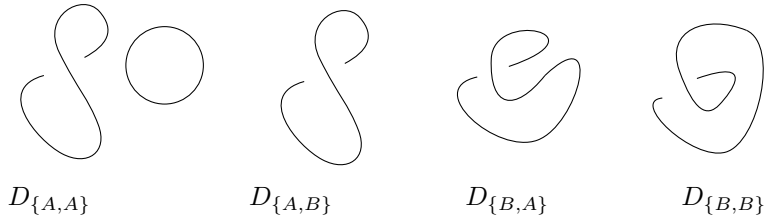
 \begin{center}
  \begin{tikzpicture}[scale=0.02,yscale=-1]
   \draw    (53,57.5) .. controls (20,66.5) and (48.03,100.44) .. (65.02,109.47) .. controls (82,118.5) and (96,108.5) .. (89.21,90.43) .. controls (82.42,72.36) and (75.63,67.29) .. (63.32,42.39) .. controls (51,17.5) and (78.67,2.2) .. (89.83,14.85) .. controls (101,27.5) and (95,37.5) .. (80,45.5) ;
   \draw    (341,58.5) .. controls (304,66.5) and (354,110.5) .. (381,100.5) .. controls (408,90.5) and (426,19.5) .. (386,57.5) ;
   \draw    (386,57.5) .. controls (366,82.5) and (349,79.5) .. (350.32,46.39) .. controls (351.63,13.29) and (427,29.5) .. (362,51.5) ;
   \draw   (107.5,50.75) .. controls (107.5,36.53) and (119.03,25) .. (133.25,25) .. controls (147.47,25) and (159,36.53) .. (159,50.75) .. controls (159,64.97) and (147.47,76.5) .. (133.25,76.5) .. controls (119.03,76.5) and (107.5,64.97) .. (107.5,50.75) -- cycle ;
   \draw    (226,61.5) .. controls (193,70.5) and (221.03,104.44) .. (238.02,113.47) .. controls (255,122.5) and (269,112.5) .. (262.21,94.43) .. controls (255.42,76.36) and (248.63,71.29) .. (236.32,46.39) .. controls (224,21.5) and (251.67,6.2) .. (262.83,18.85) .. controls (274,31.5) and (268,41.5) .. (253,49.5) ;
   \draw    (464,72.5) .. controls (442,74.5) and (484,114.5) .. (506,108.5) .. controls (528,102.5) and (542,37.5) .. (520,27.5) ;
   \draw    (487,59.5) .. controls (527,45.5) and (510,68.5) .. (501.85,77.01) .. controls (493.7,85.52) and (485,81.5) .. (475.32,54.39) .. controls (465.63,27.29) and (493,16.5) .. (520,27.5) ;
  \end{tikzpicture}\\
  {$D_{\{A,A\}}$\qquad\qquad\qquad$D_{\{A,B\}}$\quad\qquad\qquad$D_{\{B,A\}}$\quad\qquad\qquad$D_{\{B,B\}}$}
 \end{center}
 And so the $8$ states are thus obtained by resolving the last crossing $c_3$ of $D_{\{A,A\}}$, $D_{\{A,B\}}$, $D_{\{B,A\}}$, and $D_{\{B,B\}}$.
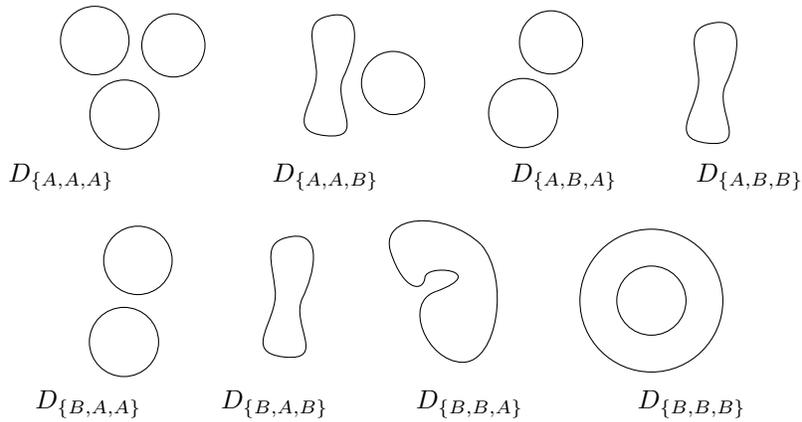
 \begin{center}
  \begin{tikzpicture}[scale=0.02,yscale=-1]
   \draw   (84.5,51) .. controls (84.5,39.4) and (93.9,30) .. (105.5,30) .. controls (117.1,30) and (126.5,39.4) .. (126.5,51) .. controls (126.5,62.6) and (117.1,72) .. (105.5,72) .. controls (93.9,72) and (84.5,62.6) .. (84.5,51) -- cycle ;
   \draw   (50,97) .. controls (50,84.3) and (60.3,74) .. (73,74) .. controls (85.7,74) and (96,84.3) .. (96,97) .. controls (96,109.7) and (85.7,120) .. (73,120) .. controls (60.3,120) and (50,109.7) .. (50,97) -- cycle ;
   \draw   (30.5,47.75) .. controls (30.5,35.19) and (40.69,25) .. (53.25,25) .. controls (65.81,25) and (76,35.19) .. (76,47.75) .. controls (76,60.31) and (65.81,70.5) .. (53.25,70.5) .. controls (40.69,70.5) and (30.5,60.31) .. (30.5,47.75) -- cycle ;
   \draw    (218.5,72) .. controls (210.5,89) and (235.5,114) .. (206.5,111) .. controls (177.5,108) and (202.5,87) .. (200.5,67) .. controls (198.5,47) and (191.17,33.7) .. (211.83,30.85) .. controls (232.5,28) and (226.5,56) .. (218.5,72) -- cycle ;
   \draw   (230.5,76) .. controls (230.5,64.4) and (239.9,55) .. (251.5,55) .. controls (263.1,55) and (272.5,64.4) .. (272.5,76) .. controls (272.5,87.6) and (263.1,97) .. (251.5,97) .. controls (239.9,97) and (230.5,87.6) .. (230.5,76) -- cycle ;
   \draw   (335.5,49) .. controls (335.5,37.4) and (344.9,28) .. (356.5,28) .. controls (368.1,28) and (377.5,37.4) .. (377.5,49) .. controls (377.5,60.6) and (368.1,70) .. (356.5,70) .. controls (344.9,70) and (335.5,60.6) .. (335.5,49) -- cycle ;
   \draw   (315,96) .. controls (315,83.3) and (325.3,73) .. (338,73) .. controls (350.7,73) and (361,83.3) .. (361,96) .. controls (361,108.7) and (350.7,119) .. (338,119) .. controls (325.3,119) and (315,108.7) .. (315,96) -- cycle ;
   \draw    (472.5,77) .. controls (464.5,94) and (489.5,119) .. (460.5,116) .. controls (431.5,113) and (456.5,92) .. (454.5,72) .. controls (452.5,52) and (445.17,38.7) .. (465.83,35.85) .. controls (486.5,33) and (480.5,61) .. (472.5,77) -- cycle ;
  \end{tikzpicture}\\
  {$D_{\{A,A,A\}}$\qquad\qquad\qquad$D_{\{A,A,B\}}$\quad\qquad\qquad$D_{\{A,B,A\}}$\quad\qquad$D_{\{A,B,B\}}$}
 \end{center}
 \begin{center}
  \begin{tikzpicture}[scale=0.02,yscale=-1]
   \draw    (242.68,57.31) .. controls (219.64,64.53) and (236.2,104.04) .. (259.85,104.52) .. controls (283.5,105) and (291.5,49) .. (272.5,26.47) ;
   \draw    (242.68,57.31) .. controls (278.82,43.82) and (235.91,38.52) .. (235,48.15) .. controls (234.1,57.79) and (221.17,57.89) .. (212.42,31.77) .. controls (203.67,5.66) and (248.5,3) .. (272.5,26.47) ;
   \draw   (12,91) .. controls (12,78.3) and (22.3,68) .. (35,68) .. controls (47.7,68) and (58,78.3) .. (58,91) .. controls (58,103.7) and (47.7,114) .. (35,114) .. controls (22.3,114) and (12,103.7) .. (12,91) -- cycle ;
   \draw   (21.5,36.75) .. controls (21.5,24.19) and (31.69,14) .. (44.25,14) .. controls (56.81,14) and (67,24.19) .. (67,36.75) .. controls (67,49.31) and (56.81,59.5) .. (44.25,59.5) .. controls (31.69,59.5) and (21.5,49.31) .. (21.5,36.75) -- cycle ;
   \draw    (153.5,62) .. controls (145.5,79) and (170.5,104) .. (141.5,101) .. controls (112.5,98) and (137.5,77) .. (135.5,57) .. controls (133.5,37) and (126.17,23.7) .. (146.83,20.85) .. controls (167.5,18) and (161.5,46) .. (153.5,62) -- cycle ;
   \draw   (362.5,63.5) .. controls (362.5,50.8) and (372.8,40.5) .. (385.5,40.5) .. controls (398.2,40.5) and (408.5,50.8) .. (408.5,63.5) .. controls (408.5,76.2) and (398.2,86.5) .. (385.5,86.5) .. controls (372.8,86.5) and (362.5,76.2) .. (362.5,63.5) -- cycle ;
   \draw   (338,63.5) .. controls (338,37.27) and (359.27,16) .. (385.5,16) .. controls (411.73,16) and (433,37.27) .. (433,63.5) .. controls (433,89.73) and (411.73,111) .. (385.5,111) .. controls (359.27,111) and (338,89.73) .. (338,63.5) -- cycle ;
  \end{tikzpicture}\\
  {\!\!\!\!\!\!\!$D_{\{B,A,A\}}$\qquad\qquad\!\!\!\!\!\!$D_{\{B,A,B\}}$\qquad\qquad\!\!\!\!$D_{\{B,B,A\}}$\qquad\qquad\,\,$D_{\{B,B,B\}}$}
 \end{center}
\end{example}

\begin{definition}[Kauffman bracket]\label{definition:Kauffman_bracket}
 The \textit{Kauffman bracket}\index{Kauffman bracket} is a function,
 \[\langle{D}\rangle\colon\{\text{unoriented link diagrams D}\}\to\ZZ[A,B,d],\]
 that is an invariant under $\RII$ and $\RIII$, and is uniquely defined by the following properties,
 \begin{enumerate}[(i)]
  \item $\langle{\bigcirc}\rangle=1$,
  \item $\langle{D\cup\bigcirc}\rangle=d\langle{D}\rangle$,
  \item $\left\langle\KPB\right\rangle=A\left\langle\KPC\right\rangle+B\left\langle\KPD\right\rangle$,
  \item $\left\langle\KPE\right\rangle=A\left\langle\KPD\right\rangle+B\left\langle\KPC\right\rangle$
 \end{enumerate}
\end{definition}

 The bracket is well defined with states and splitting markers. After labelling the regions, when we split the diagram, we multiply the bracket of the split diagram with the region the splitting marker is coinciding with.
 \[
  \left\langle\begin{tikzpicture}[every path/.style={},scale=0.5,baseline=-\dimexpr\fontdimen22\textfont2\relax]
   \clip (-1,-1)--(-1,1)--(1,1)--(1,-1)--cycle;
   \draw (-1,-1)--(1,1);
    \draw[ultra thick,draw=white,double=black] (-1,1)--(1,-1);
    \draw (0,0) node[scale=1.4,anchor=west] {$A$};
    \draw (0,0) node[scale=1.4,anchor=south] {$B$};
    \draw (0,0) node[scale=1.4,anchor=east] {$A$};
    \draw (0,0) node[scale=1.4,anchor=north] {$B$};
    \draw[thick,red,dashed] (-0.8,0)--(0.8,0);
    \draw[blue,thick,dashed] (0,-0.8)--(0,0.8);
  \end{tikzpicture}\right\rangle
  =
  \textcolor{red}{A}\left\langle\KPC\right\rangle+\textcolor{blue}{B}\left\langle\KPD\right\rangle
 \]

\begin{remark}
 In tangle notation, (iii) would be $\left\langle\tangle{-1}\right\rangle=A\left\langle\tangle{\infty}_A\right\rangle+B\left\langle\tangle{0}_B\right\rangle$ and by rotating, we get (iv). Thus, defining both is not necessary, but is convenient.
\end{remark}

\begin{proposition}
 The Kauffman bracket is well-defined, i.e., we have the Kauffman bracket defined by its \textit{state-sum}\index{Kauffman bracket!state-sum}
 \[\langle{D}\rangle=\sum_{\text{states }s\text{ of }D}A^{|a(s)|}B^{|b(s)|}d^{|s|-1},\]
 where $a(s)$ is the number of $A$-smoothings in the state $s$, $b(s)$ is the number of $B$-smoothings in the state $s$, and $|s|$ is the number of components in the state $s$.
\end{proposition}
\begin{proof}
 We show the uniqueness of the polynomial bracket by induction on the number of crossings in the link diagram.
 
 For a diagram with no crossings, we clearly have $\left\langle\KPA\right\rangle=1$, since the only state would be the unknot itself.
 Now if the link diagram $D$ is empty, then we need $\left\langle{D\cup\bigcirc}\right\rangle=1$, which is clearly satisfied.
 
 If the link diagram $D$ is not empty, we see that any state $s$ of $D$ with $\bigcirc$ is the same as the state of $D\cup\bigcirc$.
 We note
 $\left|a\left(s\cup\KPA\right)\right|=|a(s)|$ and $\left|b\left(s\cup\KPA\right)\right|=|b(s)|$,
 since they have the same crossings and so the smoothings will be the same (except for the extra circle which will not affect $a(s)$ or $b(s)$). So
 \begin{align*}
  \left\langle{D\cup\KPA}\right\rangle
  &=
  \sum_{s}{A^{|a(s)|}B^{|b(s)|}d^{|s|+1-1}}\\
  &= \sum_{s}{A^{|a(s)|}B^{|b(s)|}d^{|s|-1}d},
 \end{align*}
 which is clearly $d\left\langle{D}\right\rangle$. Thus for an $m$-unlink (no crossings), we have,
 \[\left\langle{\KPA^m}\right\rangle=d^{m-1}.\]
 We now suppose that the Kauffman bracket is in fact unique for all link diagrams with less than $c(D)$ crossings (inductive hypothesis).
 Thus, we want to show that for any arbitrary link diagram $D$, we have $\left\langle{D}\right\rangle=A\left\langle{D_A}\right\rangle+B\left\langle{D_B}\right\rangle$. And by the inductive hypothesis,
 \begin{align*}
  A\langle{D_A}\rangle+B\langle{D_B}\rangle &= A\sum_{s\in{D_A}}A^{|a(s)|}B^{|b(s)|}d^{|s|-1}+B\sum_{s\in{D_B}}A^{|b(S)|}B^{|b(s)|}d^{|s|-1}\\
  &= \sum_{s\in{D_A}}A^{|a(s)|+1}B^{|b(s)|}d^{|s|-1}+\sum_{s\in{D_B}}A^{|a(s)|}B^{|b(s)|+1}d^{|s|-1}
 \end{align*}
 And since $\sum_{s\in{D_A}}A^{|a(s)|+1}B^{|b(s)|}d^{|s|-1}$ is the same as $\tangle{1}\mapsto\{\tangle{0}_A,\tangle{\infty}_B\}$ and, similarly, for the other sum we have it being $\tangle{-1}\mapsto\{\tangle{\infty}_A,\tangle{0}_B\}$. Adding the states of $\tangle{1}$ with the states of $\tangle{-1}$ gives us the states of $D$. Thus, we have
 \[= \sum_{s\in{D}}A^{|a(s)|}B^{|b(s)|}d^{|s|-1} = \langle{D}\rangle\]
 And we are done.
\end{proof}

\begin{proposition}\label{proposition:Kauffman_regular_isotopy}
 The Kauffman bracket is invariant under regular isotopy if we have $B=A^{-1}$ and $d=-A^{2}-A^{-2}$.
\end{proposition}
\begin{proof}
 For $\RII$,
 \begin{align*}
  \left\langle\riimove\right\rangle
  &=
  A\left\langle\crossnegpos\right\rangle
  +
  B\left\langle\negativetwistwithalonecurve\right\rangle\\
  &=
  A\left(A\left\langle\rotatebox{85}{\!\!\!\!\!\resolveriimove}\right\rangle
  +
  B\left\langle\resolveriimove\right\rangle\right)
  +
  B\left(A\left\langle\lookslikelol\right\rangle
  +
  B\left\langle\rotatebox{85}{\!\!\!\!\!\resolveriimove}\right\rangle\right)\\
  &=
  A^2\left\langle\rotatebox{85}{\!\!\!\!\!\resolveriimove}\right\rangle
  +
  B^2\left\langle\rotatebox{85}{\!\!\!\!\!\resolveriimove}\right\rangle
  +
  AB\left\langle\resolveriimove\right\rangle
  +
  ABd\left\langle\rotatebox{85}{\!\!\!\!\!\resolveriimove}\right\rangle\\
  &=
  (A^2+B^2+ABd)\left\langle\rotatebox{85}{\!\!\!\!\!\resolveriimove}\right\rangle
  +
  AB\left\langle\resolveriimove\right\rangle
 \end{align*}
 We need $(A^2+B^2+ABd)\left\langle\rotatebox{85}{\!\!\!\!\!\resolveriimove}\right\rangle+AB\left\langle\resolveriimove\right\rangle=\left\langle\resolveriimove\right\rangle$, so $AB=1$ and $A^2+B^2+ABd=0$.
 And we are done.
 As for $\RIII$, we have
 \begin{align*}
  \left\langle\knotsinmath{\begin{scope}[scale=0.15]
    \draw[draw=white,double=black ,ultra thick,double distance=0.5pt] (  0:2) .. controls (270:0.7) .. (180:2);
    \draw[draw=white,double=black  ,ultra thick,double distance=0.5pt] ( 60:2) .. controls (150:0.7) .. (240:2);
    \draw[draw=white,double=black,ultra thick,double distance=0.5pt] (120:2) .. controls ( 30:0.7) .. (300:2);
   \end{scope}}\right\rangle
   &= A\left\langle\knotsinmath{\begin{scope}[x=0.75pt,y=0.75pt,yscale=-1,xscale=1,scale=0.3,yshift=-0.9cm]
    \draw    (32,59.98) .. controls (49,23.98) and (95,23.98) .. (117,60.98) ;
    \draw    (49,52.15) .. controls (63,59.15) and (75,60.15) .. (93,52.15) ;
    \draw    (51,10.15) .. controls (63,22.15) and (80,23.15) .. (96,11.15) ;
    \draw    (16,35.15) .. controls (30,42.15) and (16,35.15) .. (32,43) ;
    \draw    (110,42.15) .. controls (120,34.15) and (114,39.15) .. (125,30.15) ;
   \end{scope}}\right\rangle + B\left\langle\knotsinmath{\begin{scope}[x=0.75pt,y=0.75pt,yscale=-1,xscale=1,scale=0.3,yshift=-0.9cm]
    \draw    (96,10.65) .. controls (84,17.65) and (80.99,24.66) .. (89,32.66) .. controls (97,40.65) and (100,44.32) .. (118,64.65) ;
    \draw    (49,52.15) .. controls (63,59.15) and (75,60.15) .. (93,52.15) ;
    \draw    (56,9.65) .. controls (65,17.65) and (65,21.65) .. (55,32.65) .. controls (45,43.65) and (39,49.65) .. (25,64.65) ;
    \draw    (16,35.15) .. controls (30,42.15) and (16,35.15) .. (32,43) ;
    \draw    (110,42.15) .. controls (120,34.15) and (114,39.15) .. (125,30.15) ;
   \end{scope}}\right\rangle
  \end{align*}
  By using the fact that the bracket is invariant under $\RII$, we find,
  \begin{align*}
   &= A\left\langle\knotsinmath{\begin{scope}[x=0.75pt,y=0.75pt,yscale=-1,xscale=1,scale=0.3,yshift=-0.9cm]
    \draw    (21.72,26.85) .. controls (42.43,49.7) and (97.9,36.72) .. (118.45,24.93) ;
    \draw    (18.45,55.93) .. controls (30.45,51.93) and (54.45,47.93) .. (76.45,48.93) .. controls (98.45,49.93) and (129.45,52.93) .. (137.45,55.93) ;
    \draw    (51,10.15) .. controls (63,22.15) and (80,23.15) .. (96,11.15) ;
   \end{scope}}\right\rangle + B\left\langle\rotatebox{180}{\knotsinmath{\begin{scope}[x=0.75pt,y=0.75pt,yscale=-1,xscale=1,scale=0.3,yshift=-0.3cm]
    \draw    (96,10.65) .. controls (84,17.65) and (80.99,24.66) .. (89,32.66) .. controls (97,40.65) and (100,44.32) .. (118,64.65) ;
    \draw    (49,52.15) .. controls (63,59.15) and (75,60.15) .. (93,52.15) ;
    \draw    (56,9.65) .. controls (65,17.65) and (65,21.65) .. (55,32.65) .. controls (45,43.65) and (39,49.65) .. (25,64.65) ;
    \draw    (16,35.15) .. controls (30,42.15) and (16,35.15) .. (32,43) ;
    \draw    (110,42.15) .. controls (120,34.15) and (114,39.15) .. (125,30.15) ;
   \end{scope}}}\right\rangle\\
   &= A\left\langle\rotatebox{180}{\knotsinmath{\begin{scope}[x=0.75pt,y=0.75pt,yscale=-1,xscale=1,scale=0.3,yshift=-0.3cm]
    \draw    (32,59.98) .. controls (49,23.98) and (95,23.98) .. (117,60.98) ;
    \draw    (49,52.15) .. controls (63,59.15) and (75,60.15) .. (93,52.15) ;
    \draw    (51,10.15) .. controls (63,22.15) and (80,23.15) .. (96,11.15) ;
    \draw    (16,35.15) .. controls (30,42.15) and (16,35.15) .. (32,43) ;
    \draw    (110,42.15) .. controls (120,34.15) and (114,39.15) .. (125,30.15) ;
   \end{scope}}}\right\rangle + B\left\langle\rotatebox{180}{\knotsinmath{\begin{scope}[x=0.75pt,y=0.75pt,yscale=-1,xscale=1,scale=0.3,yshift=-0.3cm]
    \draw    (96,10.65) .. controls (84,17.65) and (80.99,24.66) .. (89,32.66) .. controls (97,40.65) and (100,44.32) .. (118,64.65) ;
    \draw    (49,52.15) .. controls (63,59.15) and (75,60.15) .. (93,52.15) ;
    \draw    (56,9.65) .. controls (65,17.65) and (65,21.65) .. (55,32.65) .. controls (45,43.65) and (39,49.65) .. (25,64.65) ;
    \draw    (16,35.15) .. controls (30,42.15) and (16,35.15) .. (32,43) ;
    \draw    (110,42.15) .. controls (120,34.15) and (114,39.15) .. (125,30.15) ;
   \end{scope}}}\right\rangle\\
   &= \left\langle\knotsinmath{\begin{scope}[scale=0.15]
    \draw[draw=white,double=black ,ultra thick,double distance=0.5pt] (  0:2) .. controls ( 90:0.7) .. (180:2);
    \draw[draw=white,double=black  ,ultra thick,double distance=0.5pt] ( 60:2) .. controls (330:0.7) .. (240:2);
    \draw[draw=white,double=black,ultra thick,double distance=0.5pt] (120:2) .. controls (210:0.7) .. (300:2);
   \end{scope}}\right\rangle
 \end{align*}
 Thus, the bracket is invariant under regular isotopy.
\end{proof}

And by the uniqueness of the Kauffman bracket, we then have the following state sum, where $B$ and $d$ are replaced with $A^{-1}$ and $(-A^{2}-A^{-2})$, respectively, we have
\[\langle{D}\rangle=\sum_{\text{states }s\text{ of }D}A^{\alpha(s)}(-A^2-A^{-2})^{|\text{components of state }s|-1}\]
where $\alpha(s)=|a(s)|+|b(s)|$.

\begin{example} \label{example:left-handed-trefoil-kauffman} Consider the left-handed trefoil (the trefoil with the negative writhe). We will calculate its bracket polynomial using its inductive and state-sum definition. We begin with the inductive definition.
 \begin{align*}
  \left\langle\trefoil\right\rangle 
  &=
  A\left\langle\twopositivetwists\right\rangle
  +
  A^{-1}\left\langle\hopflink\!\!\right\rangle\\
  &=
  A\left(A\left\langle\reflectbox{\negativetwistwithcircle}\right\rangle
  +
  A^{-1}\left\langle\reflectbox{\negativetwist}\!\!\right\rangle\right)
  +
  A^{-1}\left(A\left\langle\reflectbox{\negativetwist}\!\!\right\rangle
  +
  A^{-1}\left\langle\reflectbox{\positivetwistinward}\!\!\right\rangle
  \right)\\
  &=
  A^2(-A^2-A^{-2})\left\langle\reflectbox{\negativetwist}\!\!\right\rangle
  +
  A\langle\bigcirc\bigcirc\rangle + A^{-1}\langle\bigcirc\rangle\\
  &+
  A^{-2}\left(A\langle\bigcirc\rangle+A^{-1}\langle\bigcirc\bigcirc\rangle\right) + A\langle\bigcirc\bigcirc\rangle + A^{-1}\langle\bigcirc\rangle\\
  &=
  (-A^4-1)(A\langle\bigcirc\bigcirc\rangle+A^{-1}\langle\bigcirc\rangle) + 2A\langle\bigcirc\bigcirc\rangle + 3A^{-1}\\
  &+
  A^{-3}\langle\bigcirc\bigcirc\rangle\\
  &=
  (-A^5+A^{-3}+A)(-A^2-A^{-2})\langle\bigcirc\rangle - A^3 + 2A^{-1}\\
  &= -A^{-5} - A^3 + A^7
 \end{align*}
 We now use the state-sum definition.\newpage
 \begin{center}
  \begin{longtable}{c>{\centering\arraybackslash}m{25mm}||c||c}
   \multicolumn{2}{c}{state}&\multicolumn{1}{c}{$\alpha(s)$}&\multicolumn{1}{c}{$|\text{components of state}|$}\\\hline
   $\{A,A,A\}$&
   \begin{center}\begin{tikzpicture}[scale=0.6,x=0.75pt,y=0.75pt,yscale=-1,xscale=1,baseline=-\dimexpr\fontdimen22\textfont2\relax]
    \draw   (84.5,51) .. controls (84.5,39.4) and (93.9,30) .. (105.5,30) .. controls (117.1,30) and (126.5,39.4) .. (126.5,51) .. controls (126.5,62.6) and (117.1,72) .. (105.5,72) .. controls (93.9,72) and (84.5,62.6) .. (84.5,51) -- cycle ;
    \draw   (50,97) .. controls (50,84.3) and (60.3,74) .. (73,74) .. controls (85.7,74) and (96,84.3) .. (96,97) .. controls (96,109.7) and (85.7,120) .. (73,120) .. controls (60.3,120) and (50,109.7) .. (50,97) -- cycle ;
    \draw   (30.5,47.75) .. controls (30.5,35.19) and (40.69,25) .. (53.25,25) .. controls (65.81,25) and (76,35.19) .. (76,47.75) .. controls (76,60.31) and (65.81,70.5) .. (53.25,70.5) .. controls (40.69,70.5) and (30.5,60.31) .. (30.5,47.75) -- cycle ;
   \end{tikzpicture}\end{center}
   &$3$&$3$\\
   $\{A,A,B\}$&
   \begin{center}\begin{tikzpicture}[scale=0.6,x=0.75pt,y=0.75pt,yscale=-1,xscale=1,baseline=-\dimexpr\fontdimen22\textfont2\relax]
   \draw    (36.5,53) .. controls (28.5,70) and (53.5,95) .. (24.5,92) .. controls (-4.5,89) and (20.5,68) .. (18.5,48) .. controls (16.5,28) and (9.17,14.7) .. (29.83,11.85) .. controls (50.5,9) and (44.5,37) .. (36.5,53) -- cycle ;
   \draw   (48.5,57) .. controls (48.5,45.4) and (57.9,36) .. (69.5,36) .. controls (81.1,36) and (90.5,45.4) .. (90.5,57) .. controls (90.5,68.6) and (81.1,78) .. (69.5,78) .. controls (57.9,78) and (48.5,68.6) .. (48.5,57) -- cycle ;
  \end{tikzpicture}\end{center}
   &$1$&$2$\\
   $\{A,B,A\}$&
   \begin{center}\begin{tikzpicture}[scale=0.6,x=0.75pt,y=0.75pt,yscale=-1,xscale=1,baseline=-\dimexpr\fontdimen22\textfont2\relax]
    \draw   (38.5,27) .. controls (38.5,15.4) and (47.9,6) .. (59.5,6) .. controls (71.1,6) and (80.5,15.4) .. (80.5,27) .. controls (80.5,38.6) and (71.1,48) .. (59.5,48) .. controls (47.9,48) and (38.5,38.6) .. (38.5,27) -- cycle ;
    \draw   (18,74) .. controls (18,61.3) and (28.3,51) .. (41,51) .. controls (53.7,51) and (64,61.3) .. (64,74) .. controls (64,86.7) and (53.7,97) .. (41,97) .. controls (28.3,97) and (18,86.7) .. (18,74) -- cycle ;
   \end{tikzpicture}\end{center}
   &$1$&$2$\\
   $\{A,B,B\}$&
   \begin{center}\begin{tikzpicture}[scale=0.6,x=0.75pt,y=0.75pt,yscale=-1,xscale=1,baseline=-\dimexpr\fontdimen22\textfont2\relax]
    \draw    (41.5,55) .. controls (33.5,72) and (58.5,97) .. (29.5,94) .. controls (0.5,91) and (25.5,70) .. (23.5,50) .. controls (21.5,30) and (14.17,16.7) .. (34.83,13.85) .. controls (55.5,11) and (49.5,39) .. (41.5,55) -- cycle ;
   \end{tikzpicture}\end{center}
   &$-1$&$1$\\
   $\{B,A,A\}$&
   \begin{center}\begin{tikzpicture}[scale=0.6,x=0.75pt,y=0.75pt,yscale=-1,xscale=1,baseline=-\dimexpr\fontdimen22\textfont2\relax]
    \draw   (7,85) .. controls (7,72.3) and (17.3,62) .. (30,62) .. controls (42.7,62) and (53,72.3) .. (53,85) .. controls (53,97.7) and (42.7,108) .. (30,108) .. controls (17.3,108) and (7,97.7) .. (7,85) -- cycle ;
    \draw   (16.5,30.75) .. controls (16.5,18.19) and (26.69,8) .. (39.25,8) .. controls (51.81,8) and (62,18.19) .. (62,30.75) .. controls (62,43.31) and (51.81,53.5) .. (39.25,53.5) .. controls (26.69,53.5) and (16.5,43.31) .. (16.5,30.75) -- cycle ;
   \end{tikzpicture}\end{center}
   &$1$&$2$\\
   $\{B,A,B\}$&
   \begin{center}\begin{tikzpicture}[scale=0.6,x=0.75pt,y=0.75pt,yscale=-1,xscale=1,baseline=-\dimexpr\fontdimen22\textfont2\relax]
    \draw    (38.5,57) .. controls (30.5,74) and (55.5,99) .. (26.5,96) .. controls (-2.5,93) and (22.5,72) .. (20.5,52) .. controls (18.5,32) and (11.17,18.7) .. (31.83,15.85) .. controls (52.5,13) and (46.5,41) .. (38.5,57) -- cycle ;
   \end{tikzpicture}\end{center}
   &$-1$&$1$\\
   $\{B,B,A\}$&
   \begin{center}\begin{tikzpicture}[scale=0.6,x=0.75pt,y=0.75pt,yscale=-1,xscale=1,baseline=-\dimexpr\fontdimen22\textfont2\relax]
    \draw    (45.68,54.79) .. controls (22.64,62.01) and (39.2,101.52) .. (62.85,102) .. controls (86.5,102.48) and (94.5,46.48) .. (75.5,23.95) ;
    \draw    (45.68,54.79) .. controls (81.82,41.3) and (38.91,36) .. (38,45.63) .. controls (37.1,55.27) and (24.17,55.37) .. (15.42,29.25) .. controls (6.67,3.14) and (51.5,0.48) .. (75.5,23.95) ;
   \end{tikzpicture}\end{center}
   &$-1$&$1$\\
   $\{B,B,B\}$&
   \begin{center}\begin{tikzpicture}[scale=0.6,x=0.75pt,y=0.75pt,yscale=-1,xscale=1,baseline=-\dimexpr\fontdimen22\textfont2\relax]
    \draw   (52.5,54.5) .. controls (52.5,41.8) and (62.8,31.5) .. (75.5,31.5) .. controls (88.2,31.5) and (98.5,41.8) .. (98.5,54.5) .. controls (98.5,67.2) and (88.2,77.5) .. (75.5,77.5) .. controls (62.8,77.5) and (52.5,67.2) .. (52.5,54.5) -- cycle ;
    \draw   (28,54.5) .. controls (28,28.27) and (49.27,7) .. (75.5,7) .. controls (101.73,7) and (123,28.27) .. (123,54.5) .. controls (123,80.73) and (101.73,102) .. (75.5,102) .. controls (49.27,102) and (28,80.73) .. (28,54.5) -- cycle ;
   \end{tikzpicture}\end{center}
   &$-3$&$2$
  \end{longtable}
 \end{center}
 Thus, the polynomial becomes
 \begin{align*}
  \left\langle{D}\right\rangle &= \sum_{\text{states }s\text{ of }D}A^{\alpha(s)}(-A^2-A^{-2})^{|\text{components of state }s|-1}\\
  &= A^3d^2 + Ad + Ad + A^{-1} + Ad + A^{-1} + A^{-1} + A^{-3}d\\
  &= d(A^3d+3A+A^{-3}) + 3A^{-1}\\
  &= (-A^2-A^{-2})(-A^5-A+3A+A^{-3}) + 3A^{-1}\\
  &= A^7+A^3-2A^3-2A^{-1}-A^{-1}-A^{-5}+3A^{-1} = A^7-A^3-A^{-5}
 \end{align*}
\end{example}

Let us see what happens to the Kauffman bracket when we apply the Reidemeister moves.

\begin{lemma}\label{lemma:Kauffman_RI}
 The Kauffman bracket is not invariant under $\RI$.
\end{lemma}
\begin{proof}
 \begin{align*}
  \left\langle\rimove\right\rangle
  &=
  A\left\langle\alonecurve\right\rangle
  +
  A^{-1}\left\langle\alonecurvewithcircle\!\right\rangle\\
  &=
  A\left\langle\alonecurve\right\rangle
  +
  A^{-1}(-A^2-A^{-2})\left\langle\alonecurve\right\rangle
  =
  -A^{-3}\left\langle\alonecurve\right\rangle
 \end{align*}
 and
 \begin{align*}
  \left\langle\reflectbox{\rimove}\!\right\rangle
  &=
  A\left\langle\alonecurvewithcircle\!\right\rangle
  +
  A^{-1}\left\langle\alonecurve\right\rangle\\
  &=
  -A^3\left\langle\alonecurve\right\rangle
 \end{align*}
 So, the Kauffman bracket is not invariant under $\RI$.
\end{proof}



\begin{example}
 We can use the facts we proved in Lemma \ref{lemma:Kauffman_RI} for a faster calculation. For instance, let us see the trefoil again.
 \[
  \left\langle\trefoil\right\rangle 
  =
  A\left\langle\twopositivetwists\right\rangle
  +
  A^{-1}\left\langle\hopflink\!\!\right\rangle
 \]
 we now apply two $\RI$ moves and multiply $(-A^3)^2=A^6$. So, we would get
 \[
  = A^7+A^{-1}\left\langle\hopflink\!\!\right\rangle
  =
  A^7 + A^{-1}\left(A\left\langle\reflectbox{\negativetwist}\!\!\right\rangle
  +
  A^{-1}\left\langle\reflectbox{\positivetwistinward}\!\!\right\rangle\right)
 \]
 we then apply an $\RI$ move again to both and find,
 \[
  = A^7+A^{-1}(A(-A^3)+A^{-1}(-A^{-3}))=A^7-A^3-A^{-5}
 \]
 Hence we have simplified the calculation from Example \ref{example:left-handed-trefoil-kauffman} by threefold.
\end{example}

\subsection{Kauffman's bracket polynomial definition of the Jones polynomial}

We have already seen from Lemma \ref{lemma:Kauffman_RI} that the bracket is not invariant under $\RI$. We now discuss the method of `normalising' the bracket with respect to its writhe in order to make it invariant under ambient isotopy. Here is a precise definition.

\begin{definition}[The Jones polynomial - via Kauffman bracket]\label{definition:Jones-polynomial-bracket}
 The \textit{Jones polynomial}\index{Jones polynomial!via Kauffman bracket} is the function
 \[V(L)\colon\{\text{oriented links L with corresponding diagrams D}\}\to\ZZ[t^{-1/2},t^{1/2}],\]
 given by $V(L)=((-A)^{-3\omega(D)}\langle{|D|}\rangle)_{t^{1/2}=A^{-2}}$, where $|D|$ is the link diagram $D$ without orientation.
\end{definition}

\begin{proposition}
 The Jones polynomial is a link invariant.
\end{proposition}
\begin{proof}
 We check what happens to the Jones polynomial under the three Reidemeister moves (ambient isotopy).
 We already know that the bracket polynomial is invariant under regular isotopy from Proposition \ref{proposition:Kauffman_regular_isotopy}, and we also know from Proposition \ref{proposition:writhe_regular_isotopy} that the writhe is invariant under regular isotopy.
 Thus, we only need to show that the Jones polynomial is invariant under $\RI$.
 
 Now remember from Lemma \ref{lemma:Kauffman_RI}, if we apply an $\RI$ move to the diagram, then we need to multiply by a factor of $-A^{\pm3}$ depending on its writhe.
 
 We now consider the positive twist, i.e.,
 \[ \omega\left(\reflectbox{\knotsinmath{\begin{scope}[x=0.75pt,y=0.75pt,yscale=-1,xscale=1,yshift=-0.35cm,scale=0.3]
   \draw    (36,52.23) .. controls (29,84.23) and (4,67.23) .. (15,47.23) .. controls (25.78,27.63) and (47.12,42.61) .. (71.5,53.57) ;
   \draw [shift={(73,54.23)}, rotate = 203.75] [color={rgb, 255:red, 0; green, 0; blue, 0 }  ][line width=0.75]    (10.93,-3.29) .. controls (6.95,-1.4) and (3.31,-0.3) .. (0,0) .. controls (3.31,0.3) and (6.95,1.4) .. (10.93,3.29)   ;
   \draw    (47,11.23) -- (42,30.23) ;
  \end{scope}}}\!\right)=+1 \qquad\text{ and }\qquad \left\langle\reflectbox{\knotsinmath{\begin{scope}[x=0.75pt,y=0.75pt,yscale=-1,xscale=1,yshift=-0.35cm,scale=0.3]
   \draw    (36,52.23) .. controls (29,84.23) and (4,67.23) .. (15,47.23) .. controls (25.78,27.63) and (47.12,42.61) .. (71.5,53.57) ;
   \draw    (47,11.23) -- (42,30.23) ;
  \end{scope}}}\right\rangle=-A^3\left\langle\alonecurve\right\rangle \]
 So,
 \begin{align*}
  V\left(\reflectbox{\knotsinmath{\begin{scope}[x=0.75pt,y=0.75pt,yscale=-1,xscale=1,yshift=-0.35cm,scale=0.3]
   \draw    (36,52.23) .. controls (29,84.23) and (4,67.23) .. (15,47.23) .. controls (25.78,27.63) and (47.12,42.61) .. (71.5,53.57) ;
   \draw [shift={(73,54.23)}, rotate = 203.75] [color={rgb, 255:red, 0; green, 0; blue, 0 }  ][line width=0.75]    (10.93,-3.29) .. controls (6.95,-1.4) and (3.31,-0.3) .. (0,0) .. controls (3.31,0.3) and (6.95,1.4) .. (10.93,3.29)   ;
   \draw    (47,11.23) -- (42,30.23) ;
  \end{scope}}}\!\right) &= (-A)^{-3\omega\left(\reflectbox{\knotsinmath{\begin{scope}[x=0.75pt,y=0.75pt,yscale=-1,xscale=1,yshift=-0.35cm,scale=0.3]
   \draw    (36,52.23) .. controls (29,84.23) and (4,67.23) .. (15,47.23) .. controls (25.78,27.63) and (47.12,42.61) .. (71.5,53.57) ;
   \draw [shift={(73,54.23)}, rotate = 203.75] [color={rgb, 255:red, 0; green, 0; blue, 0 }  ][line width=0.75]    (10.93,-3.29) .. controls (6.95,-1.4) and (3.31,-0.3) .. (0,0) .. controls (3.31,0.3) and (6.95,1.4) .. (10.93,3.29)   ;
   \draw    (47,11.23) -- (42,30.23) ;
  \end{scope}}}\!\right)}\left\langle\reflectbox{\knotsinmath{\begin{scope}[x=0.75pt,y=0.75pt,yscale=-1,xscale=1,yshift=-0.35cm,scale=0.3]
   \draw    (36,52.23) .. controls (29,84.23) and (4,67.23) .. (15,47.23) .. controls (25.78,27.63) and (47.12,42.61) .. (71.5,53.57) ;
   \draw    (47,11.23) -- (42,30.23) ;
  \end{scope}}}\!\right\rangle\\
  &= (-A)^{-3\left(\omega\left(\alonecurve\right)+1\right)}\left(-A^{3}\right)\left\langle\alonecurve\right\rangle\\
  &= (-A)^{-3\omega\left(\alonecurve\right)}\left\langle\alonecurve\right\rangle = V\left(\alonecurve\right)
 \end{align*}
 And similarly,
 \begin{align*}
  V\left(\knotsinmath{\begin{scope}[x=0.75pt,y=0.75pt,yscale=-1,xscale=1,yshift=-0.35cm,scale=0.3]
   \draw    (36,52.23) .. controls (29,84.23) and (4,67.23) .. (15,47.23) .. controls (25.78,27.63) and (47.12,42.61) .. (71.5,53.57) ;
   \draw [shift={(73,54.23)}, rotate = 203.75] [color={rgb, 255:red, 0; green, 0; blue, 0 }  ][line width=0.75]    (10.93,-3.29) .. controls (6.95,-1.4) and (3.31,-0.3) .. (0,0) .. controls (3.31,0.3) and (6.95,1.4) .. (10.93,3.29)   ;
   \draw    (47,11.23) -- (42,30.23) ;
  \end{scope}}\right) &= (-A)^{-3\omega\left(\knotsinmath{\begin{scope}[x=0.75pt,y=0.75pt,yscale=-1,xscale=1,yshift=-0.35cm,scale=0.3]
   \draw    (36,52.23) .. controls (29,84.23) and (4,67.23) .. (15,47.23) .. controls (25.78,27.63) and (47.12,42.61) .. (71.5,53.57) ;
   \draw [shift={(73,54.23)}, rotate = 203.75] [color={rgb, 255:red, 0; green, 0; blue, 0 }  ][line width=0.75]    (10.93,-3.29) .. controls (6.95,-1.4) and (3.31,-0.3) .. (0,0) .. controls (3.31,0.3) and (6.95,1.4) .. (10.93,3.29)   ;
   \draw    (47,11.23) -- (42,30.23) ;
  \end{scope}}\right)}\left\langle\knotsinmath{\begin{scope}[x=0.75pt,y=0.75pt,yscale=-1,xscale=1,yshift=-0.35cm,scale=0.3]
   \draw    (36,52.23) .. controls (29,84.23) and (4,67.23) .. (15,47.23) .. controls (25.78,27.63) and (47.12,42.61) .. (71.5,53.57) ;
   \draw    (47,11.23) -- (42,30.23) ;
  \end{scope}}\right\rangle\\
  &= (-A)^{-3\omega\left(\alonecurve\right)+3}(-A^{-3})\left\langle\alonecurve\right\rangle\\
  &= (-A)^{-3\omega\left(\alonecurve\right)}\left\langle\alonecurve\right\rangle = V\left(\alonecurve\right)
 \end{align*}
 Thus, it is also invariant under $\RI$, making it a link invariant under ambient isotopy.
\end{proof}

\begin{example}\label{example:3_1_reflection_jones}
 Let us calculate the Jones polynomial for the trefoil from our examples above,
 \[V(\overline{3_1})=-A^{-3(-3)}(-A^{-5}-A^3+A^7)=A^4+A^{12}-A^{16}=-t^{-4}+t^{-3}+t^{-1}\]
\end{example}

\begin{proposition}\label{proposition:two_3_1_knots}
 There exists two distinct knots with crossing number three.
\end{proposition}
\begin{proof}[Proof sketch]
 This also follows an argument analogous to the one in Proposition \ref{proposition:no_knot_1_2}. We begin by drawing all possible link universes with three crossings such that
 \begin{enumerate}
  \item there are no twists (since an $\RI$ can undo it thus leaving it with a lower crossing number)
  \item if there is a loop, it has to be connected to the whole of the knot, since we are trying to draw a link universe of one component.
 \end{enumerate}
 So, we have the following link universe.
 
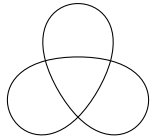
\begin{center}
  \begin{tikzpicture}[scale=0.5]
   \begin{scope}
    \begin{scope}[rotate=  0] \node (a) at (0,-1) {}; \end{scope}
    \begin{scope}[rotate=120] \node (b) at (0,-1) {}; \end{scope}
    \begin{scope}[rotate=240] \node (c) at (0,-1) {}; \end{scope}
    \draw (a.center) .. controls (a.4 north west) and (c.4 north east) .. (c.center);
    \draw (b.center) .. controls (b.4 north west) and (a.4 north east) .. (a.center);
    \draw (c.center) .. controls (c.4 north west) and (b.4 north east) .. (b.center);
    \draw (a.center) .. controls (a.16 south west) and (c.16 south east) .. (c.center);
    \draw (b.center) .. controls (b.16 south west) and (a.16 south east) .. (a.center);
    \draw (c.center) .. controls (c.16 south west) and (b.16 south east) .. (b.center);
   \end{scope}
  \end{tikzpicture}
 \end{center}
 By tabulating all $2^3$ knot diagrams, we find that there are only in fact two that are not the unknot.
 \begin{center}
  \reflectbox
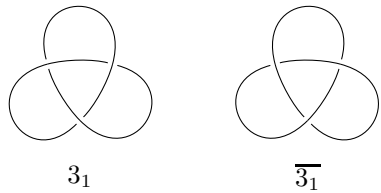
{\begin{tikzpicture}[scale=0.5]
   \mytrefoil;
   \draw (0,-2.5) node[scale=2,xscale=-1] {$3_1$};
  \end{tikzpicture}}
  \begin{tikzpicture}[scale=0.5]
    \mytrefoil;
    \draw (0,-2.5) node[scale=2] {$\overline{3_1}$};
  \end{tikzpicture}
 \end{center}
 So, we only need to show that $V(3_1)\neq{V(\overline{3_1})}$.
 
 We have already calculated $V(\overline{3_1})$ form Example \ref{example:3_1_reflection_jones}. And for the sake of not repeating onesself, we will soon see from Example \ref{example:og_Jones_3_1} that we have $V(3_1)=A^{-4}+A^{-12}-A^{-16}$. Thus, $V(3_1)\neq{V(\overline{3_1})}$, and so $3_1\not\sim\overline{3_1}$.
\end{proof}

So, now we know that the trefoil and its reflection are not equivalent knots. And so we have the following result.

\begin{proposition}
 The Jones polynomial of a mirror link $\overline{L}$ of link $L$ has the polynomial $V(\overline{L})=(V(L))_{t=t^{-1}}$.
\end{proposition}
\begin{proof}
 We know that the writhe of the reflection of the link $L$ is the negation of the writhe of $L$.
 As for the Kauffman bracket, it is not hard to see that $A=B$ when looking at the reflection, since, $\left\langle\KPB\right\rangle=A\left\langle\KPD\right\rangle+B\left\langle\KPC\right\rangle$, whereas $\left\langle\KPE\right\rangle=A\left\langle\KPC\right\rangle+B\left\langle\KPD\right\rangle$.
 Thus, $\langle\overline{D}\rangle=(\langle{D}\rangle)_{A=A^{-1}}$.
 And so, $V(\overline{D})=(-A^{3\omega(D)}(\langle{D}\rangle)_{A=A^{-1}}=V(D)_{A=A^{-1}}=V(D)_{t=t^{-1}}$
\end{proof}








\subsection{The skein relation definition of the Jones polynomial}\label{section:skein}


In this section, we will begin by exploring the Jones polynomial defined by its skein relation and we will also see the Kauffman bracket polynomial and how its normalisation (with respect to the writhe) is equivalent to the Jones polynomial. 

The \textit{skein relation} was shown by John Conway in the 1960's \cite{conway_skein}. He defined the Alexander polynomial via the skein relation (now sometimes known as the Alexander-Conway polynomial). So what is a skein relation?

\begin{definition}[skein relation]
 The \textit{skein relation}\index{skein relation} is defined as the relation between three link diagrams which we call $L_+$, $L_-$ and $L_0$. These diagrams are the same, except in the neighbourhood of a single crossing. See the below figure.
 \begin{figure}[H]
 \centering
 \begin{tikzpicture}[scale=0.6]
  \begin{scope}
   \draw[dashed] (0,0) circle(1.5);
   \draw[->] (1,-1)--(-1,1);
   \draw[ultra thick, draw=white,double=black,double distance=0.5pt] (-1,-1)--(1,1);
   \draw[->] (0.9,0.9)--(1,1);
   \draw (0,-2) node[scale=1.5] {$L_+$};
  \end{scope}
  \begin{scope}[xshift=5cm]
   \draw[dashed] (0,0) circle(1.5);
   \draw[->] (-1,-1)--(1,1);
   \draw[ultra thick, draw=white,double=black,double distance=0.5pt] (1,-1)--(-1,1);
   \draw[->] (-0.9,0.9)--(-1,1);
   \draw (0,-2) node[scale=1.5] {$L_-$};
  \end{scope}
  \begin{scope}[xshift=10cm]
   \draw[dashed] (0,0) circle(1.5);
   \draw[->] (-1,-1) to [bend right=40] (-1,1);
   \draw[->] (1,-1) to [bend left=40] (1,1);
   \draw (0,-2) node[scale=1.5] {$L_0$};
  \end{scope}
 \end{tikzpicture}
 \caption{skein relation}
 \label{figure:skein_relation}
 \end{figure}
 In other words, they are the three elementary $2$-tangles $\tangle{1}$, $\tangle{-1}$ and $\tangle{0}$.
\end{definition}

To define an invariant function that admits a skein relation, we have to notice that the skein relation is related to the unknotting number, in that when we have a skein relation, the complexity of the other two link diagrams are strictly less than that of the original link diagram.

We now define an invariant function $P$ that admits a skein relation, and is unique. We will not go too deep into this theory. The interested reader is directed to \cite{homfly} and \cite{Likorish}.
\begin{definition}[skein relation invariant polynomial]\label{definition:skein_relation_invariant}
 It is said that a link invariant $P\colon\{\text{Oriented links in }\RR^3\}\to{A}$, where $A$ is a commutative ring, admits a \textit{skein relation invariant}\index{skein relation invariant} if
 \begin{enumerate}[(i)]
  \item $P(0_1)=P(\KPA)=1$ and
  \item there exists three invertible elements $a_+$, $a_-$ and $a_0$ in $A$, such that whenever we have three links $L_+$, $L_-$ and $L_0$ that are skein related, then $a_+P(L_+)+a_-P(L_-)+a_0P(L_0)=0$.
 \end{enumerate}
\end{definition}

The following theorem says that if the skein relation invariant polynomial is well defined (i.e., it is a link invariant), then it is uniquely determined by $a_+$, $a_-$ and $a_0$.
\begin{theorem}
 If the skein relation invariant polynomial $P$ defined above is in fact an invariant, then it is uniquely determined by $a_+$, $a_-$ and $a_0$.
\end{theorem}
The proof of the uniqueness of the skein relation invariant is closely analogous to the uniqueness of the Jones polynomial via skein theory, Proposition \ref{proposition:Jones-polynomial-skein}, and so is left as an exercise for the reader. The existence proof of this polynomial can be found in \cite{Likorish}.

\begin{proposition}[skein relation invariant is homogeneous {\cite[p.~72]{knotes}}]
 The skein relation invariant is homogeneous in $a_+$, $a_-$, and $a_0$, i.e., all of the terms have a total degree of zero of the variables.
\end{proposition}


\begin{proposition}[The Jones polynomial - skein theory\index{Jones polynomial!via skein relation}]\label{proposition:Jones-polynomial-skein}
 The \textit{Jones polynomial} invariant $V(L)$ admits a skein relation such that it is a function $V\colon\{\text{Oriented links in }\RR^3\}\to\ZZ[t^{-1/2},t^{1/2}]$ such that $V(0_1)=V(\KPA)=1$, and whenever three oriented links $L_+,L_-,L_0$ are skein related, then
 \[t^{-1}V(L_+)-tV(L_-)+(t^{-1/2}-t^{1/2})V(L_0)=0.\]
 And it is uniquely determined by those two conditions.
\end{proposition}
\begin{proof}
 We first show that the Jones polynomial admits a skein relation, then we will show that it can be uniquely determined by the skein relation conditions.
 The first property $V(\KPA)=1$ is trivial.
 For the second property we need to remove $D_\infty$ terms found in the Kauffman bracket. So consider the diagrams below.
 
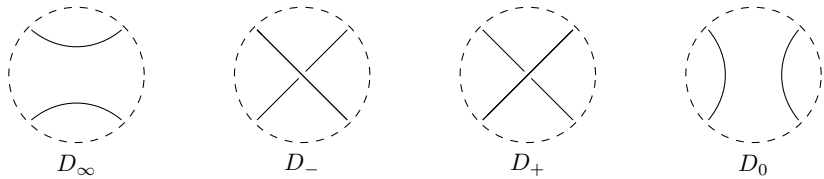
\begin{center}
  \begin{tikzpicture}[scale=0.6]
   \begin{scope}
    \draw[dashed] (0,0) circle(1.5);
    \draw[-] (-1,-1) to [bend left=40] (1,-1);
    \draw[-] (1,1) to [bend left=40] (-1,1);
   \end{scope}
   \draw (0,-2) node[scale=1.5] {$D_\infty$};
   \begin{scope}[xshift=5cm]
    \draw[dashed] (0,0) circle(1.5);
    \draw[-] (-1,-1)--(1,1);
    \draw[ultra thick, draw=white,double=black,double distance=0.5pt] (1,-1)--(-1,1);
    \draw[-] (-0.9,0.9)--(-1,1);
    \draw (0,-2) node[scale=1.5] {$D_-$};
   \end{scope}
   \begin{scope}[xshift=10cm]
    \draw[dashed] (0,0) circle(1.5);
    \draw[-] (1,-1)--(-1,1);
    \draw[ultra thick,draw=white,double=black,double distance=0.5pt] (-1,-1)--(1,1);
    \draw[-] (-0.9,-0.9)--(1,1);
    \draw (0,-2) node[scale=1.5] {$D_+$};
   \end{scope}
   \begin{scope}[xshift=15cm]
    \draw[dashed] (0,0) circle(1.5);
    \draw[-] (-1,-1) to [bend right=40] (-1,1);
    \draw[-] (1,-1) to [bend left=40] (1,1);
    \draw (0,-2) node[scale=1.5] {$D_0$};
   \end{scope}
  \end{tikzpicture}
 \end{center}
 So we have $\left\langle{D_-}\right\rangle=\left\langle{\KPB}\right\rangle=A\left\langle{D_\infty}\right\rangle+A^{-1}\left\langle{D_0}\right\rangle$ and $\left\langle{D_+}\right\rangle=\left\langle\KPE\right\rangle=A\left\langle{D_0}\right\rangle+A^{-1}\left\langle{D_\infty}\right\rangle$. And so noting the substitution $t^{1/2}=A^{-2}$, we have
 \[
  A^{4}(-A)^{-3\omega(D_+)}\left\langle{D_+}\right\rangle-A^{-4}(-A)^{-3\omega(D_-)}\left\langle{D_-}\right\rangle+(A^2-A^{-2})(-A)^{-3\omega(D_0)}\left\langle{D_0}\right\rangle
 \]
 Now $\omega(D_0)+1=\omega(D_+)$ and $\omega(D_-)=\omega(D_0)-1$. So the above formula becomes
 \[
  (-A)^{-3\omega(D_0)}(-A\left\langle{D_+}\right\rangle+A^{-1}\left\langle{D_-}\right\rangle+(A^2-A^{-2})\langle{D_0}\rangle)
 \]
 Now we substitute the relations we wrote above $\left\langle{D_-}\right\rangle$ and $\left\langle{D_+}\right\rangle$, so
 \[
  -A(A\left\langle{D_0}\right\rangle+A^{-1}\left\langle{D_\infty}\right\rangle)+A^{-1}(A\left\langle{D_\infty}\right\rangle+A^{-1}\left\langle{D_0}\right\rangle)+(A^2-A^{-2})\left\langle{D_0}\right\rangle=0.
 \]
 So we are done.
 
 We now show the uniqueness by induction on the number of crossings $n$ on an arbitrary link $L$.
 
 We note that we already know $V(\KPA)=1$. And by considering any arbitrary link $L$, we will now see that for an $m$ component unlink $V(\KPA^m)=(t^{1/2}+t^{-1/2})^{m-1}$.
 
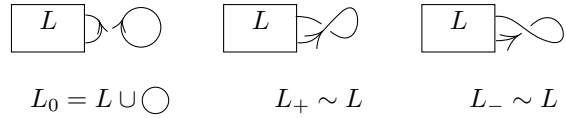
\begin{center}
  \begin{tikzpicture}[x=0.75pt,y=0.75pt,yscale=-1,xscale=1]
   \draw   (16.17,17) -- (52.98,17) -- (52.98,40.93) -- (16.17,40.93) -- cycle ;
   \draw    (53.47,22.28) .. controls (64.33,22.38) and (64.33,36.38) .. (53.47,36.28) ;
   \draw   (71.52,28.48) .. controls (71.52,22.96) and (75.99,18.48) .. (81.52,18.48) .. controls (87.04,18.48) and (91.52,22.96) .. (91.52,28.48) .. controls (91.52,34.01) and (87.04,38.48) .. (81.52,38.48) .. controls (75.99,38.48) and (71.52,34.01) .. (71.52,28.48) -- cycle ;
   \draw   (123.17,17) -- (159.98,17) -- (159.98,40.93) -- (123.17,40.93) -- cycle ;
   \draw    (177.02,30.47) .. controls (194.02,44.47) and (194.53,9.8) .. (182.53,20.8) .. controls (170.53,31.8) and (170.53,36.8) .. (160.47,36.28) ;
   \draw    (172.53,26.8) .. controls (169.02,24.47) and (166.53,23) .. (159.53,23) ;
   \draw   (223.17,17) -- (259.98,17) -- (259.98,40.93) -- (223.17,40.93) -- cycle ;
   \draw    (278.02,27.78) .. controls (292.23,13.95) and (301.23,38.95) .. (288.23,35.95) .. controls (275.23,32.95) and (271,21.91) .. (260.23,22.95) ;
   \draw    (273.53,31.44) .. controls (270.02,33.76) and (267.53,35.23) .. (260.53,35.23) ;
   \draw   (164,32.83) .. controls (167.65,32.82) and (170.61,32.09) .. (172.88,30.63) .. controls (171.29,32.82) and (170.4,35.73) .. (170.18,39.37) ;
   \draw   (262.45,32.13) .. controls (265.92,33.23) and (268.96,33.44) .. (271.57,32.74) .. controls (269.39,34.34) and (267.65,36.84) .. (266.34,40.24) ;
   \draw   (59.59,32.06) .. controls (60.44,29.71) and (60.73,27.51) .. (60.45,25.48) .. controls (61.29,27.35) and (62.69,29.06) .. (64.66,30.62) ;
   \draw   (67.01,29) .. controls (69.47,27.83) and (71.19,26.41) .. (72.17,24.74) .. controls (71.93,26.66) and (72.45,28.83) .. (73.71,31.25) ;
   \draw (29.17,19) node [anchor=north west][inner sep=0.75pt][yscale=-1]   [align=left] {$\displaystyle L$};
   \draw (136.17,19) node [anchor=north west][inner sep=0.75pt]   [align=left][yscale=-1] {$\displaystyle L$};
   \draw (236.17,19) node [anchor=north west][inner sep=0.75pt]   [align=left][yscale=-1] {$\displaystyle L$};
  \end{tikzpicture}
  \par{$L_0=L\cup\KPA$\qquad\qquad$L_+\sim{L}$\qquad\qquad$L_-\sim{L}$}
 \end{center}
 By the skein relation definition, we have $t^{-1}V(L)-tV(L)+(t^{-1/2}-t^{1/2})V(L\cup\KPA)=0$, and thus
 \[V\left(L\cup\KPA\right)=\frac{t-t^{-1}}{t^{-1/2}-t^{1/2}}V(L)=-(t^{1/2}+t^{-1/2})V(L).\]
 So by induction on $m$, we find $V(L\cup\KPA^m)=(-t^{1/2}-t^{-1/2})^mV(L)$.
 Now we know that the link $L_0$ will have one less crossing than $L_+$ and $L_-$. So, without loss of generality, if our oriented link $L$ is $L_+$, then we have complexity $(c(L_+),u(L_+))$.
 Since $c(L_0)<c(L_+)$, we have $\kappa(L_0)<\kappa(L_+)$. Now $c(L_-)=c(L_+)$, but $u(L_-)=c(L_+)-1$. So we have the complexities of the two links $L_-$ and $L_0$, strictly less than the complexity of $L_+$.
 Now, by induction on the complexity, we can see that the Jones polynomial is in fact unique.
\end{proof}

\begin{remark}
 It can be seen that by substituting $a_+$ for $t^{-1}$, $a_-$ for $-t$, and $a_0$ for $t^{-1/2}-t^{1/2}$, the skein relation invariant $P$ becomes the Jones polynomial $V$.
\end{remark}

We have shown in the above proof what the Jones polynomial of the disjoint union of a link and the $m$-unlink is, so we shall write it as the following corollary.
\begin{corollary}
 The disjoint union of a link $L$ and the $m$-unlink is given by $V\left(L\cup\KPA^m\right)=(-t^{1/2}-t^{-1/2})^mV(L)$.
\end{corollary}

We then have the following theorem, which we will prove at the end of the essay using nicer results with braid representations.
\begin{theorem}\label{theorem:disjoint_links}
 The two disjoint links $L_1$ and $L_2$ has the following Jones polynomial,
 \[V(L_1\cup{L_2})=(-t^{1/2}-t^{-1/2})V(L_1)V(L_2)\]
\end{theorem}

\begin{example}\label{example:jones_skein_4_1}
 Let us calculate the Jones polynomial of the figure eight knot $4_1$ drawn below via skein theory. We let the knot $4_1$ be the knot $L_+$ where we have circled our chosen neighbourhood.
 \begin{center}
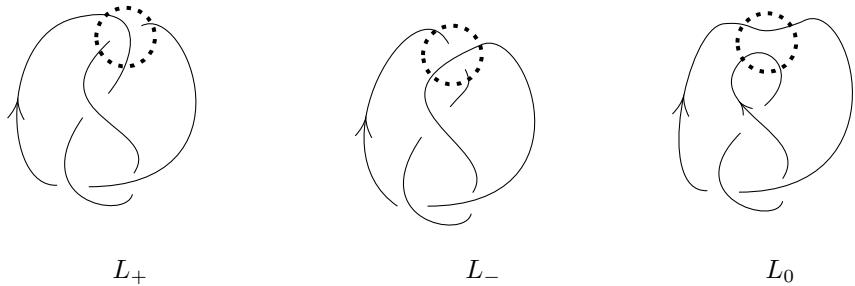
 \begin{minipage}{0.3\textwidth}
  \begin{center}
   \begin{tikzpicture}[scale=0.02,yscale=-1]
    \draw    (41,121) .. controls (5,123) and (2,23) .. (52,10) .. controls (102,-3) and (94,37) .. (75,60) ;
    \draw    (92,113) .. controls (113,82) and (22,75) .. (76,25) ;
    \draw    (97,15) .. controls (128,-1) and (174,122) .. (62,122) ;
    \draw    (58,76) .. controls (19,132) and (87,145) .. (91,127) ;
    \draw   (8.04,76.36) .. controls (11.54,72.43) and (13.82,68.37) .. (14.86,64.2) .. controls (15.05,68.5) and (16.46,72.93) .. (19.09,77.49) ;
    \draw  [dash pattern={on 1.69pt off 2.76pt}][line width=1.5]  (67,22.5) .. controls (67,11.73) and (75.73,3) .. (86.5,3) .. controls (97.27,3) and (106,11.73) .. (106,22.5) .. controls (106,33.27) and (97.27,42) .. (86.5,42) .. controls (75.73,42) and (67,33.27) .. (67,22.5) -- cycle ;
   \end{tikzpicture}
  \end{center}
 \end{minipage}%
 \begin{minipage}{0.3\textwidth}
  \begin{center}
   \begin{tikzpicture}[scale=0.02,yscale=-1]
    \draw    (42,122) .. controls (-15,83) and (53,-29) .. (76,14) ;
    \draw    (92,113) .. controls (113,82) and (7,53) .. (97,15) ;
    \draw    (97,15) .. controls (128,-1) and (174,122) .. (62,122) ;
    \draw    (58,76) .. controls (19,132) and (87,145) .. (91,127) ;
    \draw   (13.95,75.26) .. controls (17.51,71.39) and (19.85,67.37) .. (20.96,63.21) .. controls (21.08,67.51) and (22.42,71.97) .. (24.98,76.56) ;
    \draw    (77,56) .. controls (87,44) and (93,43) .. (87,31) ;
    \draw  [dash pattern={on 1.69pt off 2.76pt}][line width=1.5]  (59,21.5) .. controls (59,10.73) and (67.73,2) .. (78.5,2) .. controls (89.27,2) and (98,10.73) .. (98,21.5) .. controls (98,32.27) and (89.27,41) .. (78.5,41) .. controls (67.73,41) and (59,32.27) .. (59,21.5) -- cycle ;
   \end{tikzpicture}
  \end{center}
 \end{minipage}%
 \begin{minipage}{0.3\textwidth}
  \begin{center}
   \begin{tikzpicture}[scale=0.02,yscale=-1]
    \draw    (41,121) .. controls (5,123) and (27,12.4) .. (52,10) .. controls (77,7.6) and (68,22.6) .. (104,8.6) ;
    \draw    (91,113) .. controls (110.13,84.76) and (44.94,68.16) .. (58.97,40.88) .. controls (73,13.6) and (108,39.6) .. (79,64.6) ;
    \draw    (104,8.6) .. controls (137,-11.4) and (174,122) .. (62,122) ;
    \draw    (63,82.6) .. controls (20,125.6) and (87,146) .. (91,128) ;
    \draw   (17.74,70.02) .. controls (21.45,66.29) and (23.93,62.35) .. (25.2,58.24) .. controls (25.16,62.55) and (26.33,67.05) .. (28.71,71.74) ;
    \draw  [dash pattern={on 1.69pt off 2.76pt}][line width=1.5]  (61,22.5) .. controls (61,11.73) and (69.73,3) .. (80.5,3) .. controls (91.27,3) and (100,11.73) .. (100,22.5) .. controls (100,33.27) and (91.27,42) .. (80.5,42) .. controls (69.73,42) and (61,33.27) .. (61,22.5) -- cycle ;
    \draw   (64.85,70.95) .. controls (64.86,67.63) and (64.17,64.8) .. (62.8,62.48) .. controls (64.87,64.29) and (67.62,65.6) .. (71.07,66.4) ;
   \end{tikzpicture}
  \end{center}
 \end{minipage}
 \par{$L_+$\qquad\qquad\qquad\qquad\qquad\qquad$L_-$\qquad\qquad\qquad\qquad\qquad$L_0$}
 \end{center}
 Note that the knot $L_-$ is Reidemeiser equivalent to the unknot $0_1$, so $V(L_-)=1$. Whereas $L_0$ is Reidemeister equivalent to the Hopf link with the given orientations.
 So, by the definition of the Jones polynomial, we have
 \begin{align*}
  t^{-1}V(L_+)-tV(L_-)+(t^{-1/2}-t^{1/2})V(L_0)&=0\\
  t^{-1}V(4_1)-t+(t^{-1/2}-t^{1/2})V(L_0)&=0 \tag{*}\label{equation:Jonespolynomialof4_1}
 \end{align*}
 Now, we need to find what $V(L_0)$ is. So, we let $L_0$ be the Hopf link $H_-$ (since they are Reidemeister equivalent).
 
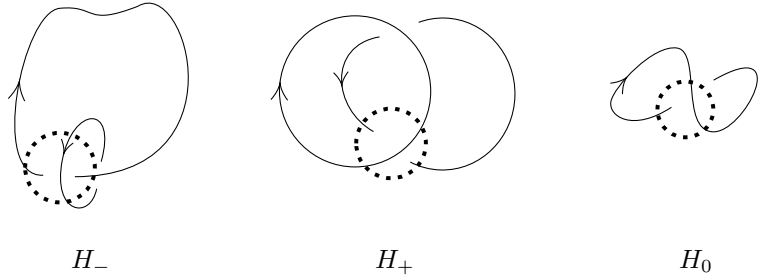
\begin{center}
  \begin{tikzpicture}[scale=0.02,yscale=-1]
   \draw    (41,121) .. controls (5,123) and (27,12.4) .. (52,10) .. controls (77,7.6) and (68,22.6) .. (104,8.6) ;
   \draw    (104,8.6) .. controls (137,-11.4) and (174,122) .. (62,122) ;
   \draw   (17.74,70.02) .. controls (21.45,66.29) and (23.93,62.35) .. (25.2,58.24) .. controls (25.16,62.55) and (26.33,67.05) .. (28.71,71.74) ;
   \draw  [dash pattern={on 1.69pt off 2.76pt}][line width=1.5]  (29,115.95) .. controls (29,103.16) and (39.36,92.8) .. (52.15,92.8) .. controls (64.94,92.8) and (75.3,103.16) .. (75.3,115.95) .. controls (75.3,128.74) and (64.94,139.1) .. (52.15,139.1) .. controls (39.36,139.1) and (29,128.74) .. (29,115.95) -- cycle ;
   \draw   (60.95,99.62) .. controls (57.97,101.67) and (55.87,104.02) .. (54.65,106.68) .. controls (54.99,103.71) and (54.46,100.45) .. (53.05,96.88) ;
   \draw    (76.5,130) .. controls (72.5,151) and (46.78,147.2) .. (53.64,111.1) .. controls (60.5,75) and (90.5,73) .. (79.5,112) ;
   \draw   (198,65.25) .. controls (198,37.5) and (220.5,15) .. (248.25,15) .. controls (276,15) and (298.5,37.5) .. (298.5,65.25) .. controls (298.5,93) and (276,115.5) .. (248.25,115.5) .. controls (220.5,115.5) and (198,93) .. (198,65.25) -- cycle ;
   \draw  [draw opacity=0] (290.52,19.35) .. controls (295.91,17.34) and (301.73,16.31) .. (307.79,16.47) .. controls (334.3,17.16) and (355.2,40.35) .. (354.47,68.25) .. controls (353.74,96.16) and (331.66,118.22) .. (305.14,117.52) .. controls (297.87,117.33) and (291.01,115.45) .. (284.91,112.24) -- (306.46,67) -- cycle ;
   \draw   (290.52,19.35) .. controls (295.91,17.34) and (301.73,16.31) .. (307.79,16.47) .. controls (334.3,17.16) and (355.2,40.35) .. (354.47,68.25) .. controls (353.74,96.16) and (331.66,118.22) .. (305.14,117.52) .. controls (297.87,117.33) and (291.01,115.45) .. (284.91,112.24) ;
   \draw    (260.5,92) .. controls (233.5,77) and (229.5,35) .. (264.5,29) ;
   \draw   (192.31,70.71) .. controls (195.56,67.08) and (197.58,63.4) .. (198.35,59.67) .. controls (198.82,63.46) and (200.52,67.29) .. (203.45,71.18) ;
   \draw   (243.72,52.13) .. controls (241.21,55.28) and (239.79,58.34) .. (239.45,61.34) .. controls (238.7,58.42) and (236.86,55.58) .. (233.93,52.82) ;
   \draw  [dash pattern={on 1.69pt off 2.76pt}][line width=1.5]  (249,99.85) .. controls (249,87.06) and (259.36,76.7) .. (272.15,76.7) .. controls (284.94,76.7) and (295.3,87.06) .. (295.3,99.85) .. controls (295.3,112.64) and (284.94,123) .. (272.15,123) .. controls (259.36,123) and (249,112.64) .. (249,99.85) -- cycle ;
   \draw    (458.5,75.8) .. controls (432.5,92.8) and (416.5,81.8) .. (418.26,72.06) .. controls (420.01,62.32) and (446.5,33.8) .. (462.35,36.45) .. controls (478.19,39.11) and (464.33,89.99) .. (485.47,92.12) .. controls (506.6,94.25) and (540.5,26.8) .. (486.5,57.8) ;
   \draw   (417.89,59.79) .. controls (422.38,59.34) and (425.98,58.07) .. (428.69,55.99) .. controls (426.88,58.88) and (425.96,62.59) .. (425.95,67.11) ;
   \draw  [dash pattern={on 1.69pt off 2.76pt}][line width=1.5]  (449.2,77.25) .. controls (449.2,67.01) and (457.51,58.7) .. (467.75,58.7) .. controls (477.99,58.7) and (486.3,67.01) .. (486.3,77.25) .. controls (486.3,87.49) and (477.99,95.8) .. (467.75,95.8) .. controls (457.51,95.8) and (449.2,87.49) .. (449.2,77.25) -- cycle ;
  \end{tikzpicture}
  \par{$H_-$\qquad\qquad\qquad\qquad\qquad$H_+$\qquad\qquad\qquad\qquad\qquad$H_0$}
 \end{center}
 We note that the knot $L_0\sim_R0_1$. So, from the definition of the Jones polynomial we get,
 \begin{equation}
  t^{-1}V(H_+)-tV(H_-)+(t^{-1/2}-t^{1/2})V(H_0)=0 \tag{**}\label{equation:Hopflinkskein}
 \end{equation}
 and so we need to calculate the Jones polynomial of the unlink: $V(L_+)$. We let $L_+$ be $M_0$ since they are R-equivalent.
 
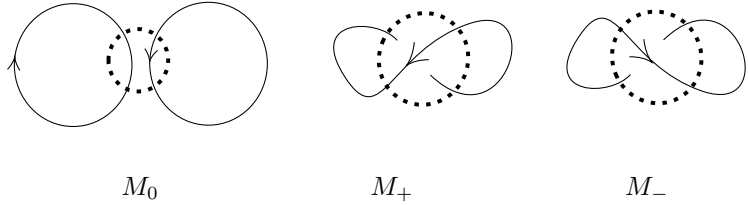
\begin{center}
  \begin{tikzpicture}[scale=0.02,yscale=-1]
   \draw   (17.06,75.55) .. controls (17.06,52.99) and (34.56,34.7) .. (56.14,34.7) .. controls (77.72,34.7) and (95.22,52.99) .. (95.22,75.55) .. controls (95.22,98.11) and (77.72,116.4) .. (56.14,116.4) .. controls (34.56,116.4) and (17.06,98.11) .. (17.06,75.55) -- cycle ;
   \draw   (107.04,74.65) .. controls (107.04,52.09) and (124.54,33.8) .. (146.12,33.8) .. controls (167.7,33.8) and (185.2,52.09) .. (185.2,74.65) .. controls (185.2,97.21) and (167.7,115.5) .. (146.12,115.5) .. controls (124.54,115.5) and (107.04,97.21) .. (107.04,74.65) -- cycle ;
   \draw   (12.5,78.36) .. controls (15.03,75.65) and (16.64,72.86) .. (17.34,70) .. controls (17.56,72.94) and (18.71,75.98) .. (20.77,79.09) ;
   \draw   (112.22,65.77) .. controls (109.53,68.12) and (107.77,70.59) .. (106.96,73.19) .. controls (106.83,70.46) and (105.76,67.59) .. (103.74,64.59) ;
   \draw  [dash pattern={on 1.69pt off 2.76pt}][line width=1.5]  (79.62,72.37) .. controls (79.62,60.91) and (88.5,51.63) .. (99.46,51.63) .. controls (110.42,51.63) and (119.3,60.91) .. (119.3,72.37) .. controls (119.3,83.82) and (110.42,93.11) .. (99.46,93.11) .. controls (88.5,93.11) and (79.62,83.82) .. (79.62,72.37) -- cycle ;
   \draw    (426.5,81.8) .. controls (413.5,96.8) and (366.5,96.8) .. (391.5,57.8) .. controls (416.5,18.8) and (421.51,75.14) .. (474.5,92.97) .. controls (527.5,110.8) and (498.5,10.8) .. (448.5,58.8) ;
   \draw   (435.4,58.27) .. controls (436.13,64.73) and (437.91,69.88) .. (440.76,73.72) .. controls (436.85,71.19) and (431.88,69.96) .. (425.84,70.04) ;
   \draw  [dash pattern={on 1.69pt off 2.76pt}][line width=1.5]  (417.36,83.07) .. controls (410.11,67.75) and (416.14,49.69) .. (430.85,42.72) .. controls (445.55,35.76) and (463.35,42.52) .. (470.61,57.84) .. controls (477.87,73.16) and (471.83,91.22) .. (457.12,98.19) .. controls (442.42,105.15) and (424.62,98.38) .. (417.36,83.07) -- cycle ;
   \draw    (271.5,59.11) .. controls (258.5,44.11) and (211.5,44.11) .. (236.5,83.11) .. controls (261.5,122.11) and (266.51,65.77) .. (319.5,47.94) .. controls (372.5,30.11) and (343.5,130.11) .. (293.5,82.11) ;
   \draw   (282.87,60.3) .. controls (282.21,66.77) and (280.48,71.93) .. (277.67,75.8) .. controls (281.55,73.23) and (286.51,71.96) .. (292.55,71.97) ;
   \draw  [dash pattern={on 1.69pt off 2.76pt}][line width=1.5]  (262.36,58.84) .. controls (255.11,74.16) and (261.14,92.22) .. (275.85,99.19) .. controls (290.55,106.15) and (308.35,99.38) .. (315.61,84.07) .. controls (322.87,68.75) and (316.83,50.69) .. (302.12,43.72) .. controls (287.42,36.76) and (269.62,43.52) .. (262.36,58.84) -- cycle ;
  \end{tikzpicture}
  \par{$M_0$\qquad\qquad\qquad\qquad$M_+$\qquad\qquad\qquad\qquad$M_-$}
 \end{center}
 Note that the knots $M_+$ and $M_-$ are R-equivalent to $0_1$, so we have $V(M_+)=V(M_-)=1$. So then using the skein theory definition of the Jones polynomial, we have $t^{-1}-t+(t^{-1/2}-t^{1/2})V(M_0)=0$. And so
 \begin{align*}
  V(M_0) &= \frac{t-t^{-1}}{t^{-1/2}-t^{1/2}}\\
  &= \frac{(t^{-1/2}-t^{1/2})(-t^{1/2}-t^{-1/2})}{(t^{-1/2}-t^{1/2})}=-t^{1/2}-t^{-1/2}
 \end{align*}
 Now noting that we have $H_+\sim_RM_0$, we consider Equation (\ref{equation:Hopflinkskein}),
 \begin{align*}
  \qquad\qquad&\!\!\!\!\!\!\!\!\!\!\!\!\!\!\!\!\!\!\!\!\!\!\!\!\!t^{-1}(-t^{1/2}-t^{-1/2})-tV(H_-)+t^{-1/2}-t^{1/2}=0\\
  V(H_-) &= t^{-1}(-t^{-3/2}-t^{1/2})\\
  &= -t^{-5/2}-t^{-1/2}
 \end{align*}
 Now we can calculate the Jones polynomial of our figure eight knot. We consider Equation (\ref{equation:Jonespolynomialof4_1}),
 \begin{align*}
  V(4_1) &= t(t-(t^{-1/2}-t^{1/2})V(H_-))\\
  &= t^2-(t^{1/2}-t{3/2})(-t^{-5/2}-t^{-1/2})\\
  &= t^2+t^{-2}+1-t^{-1}-t
 \end{align*}\qed
\end{example}




\begin{remark}
 \textit{The Alexander polynomial}\index{The Alexander polynomial} \cite{alexander} also admits a skein relation \cite{conway_skein}. From the skein relation invariant polynomial, we can get the Alexander polynomial by substituting $a_+=1$, $a_-=-1$ and $a_0=t^{-1/2}-t^{1/2}$ \cite{homfly}. That is,
 \[\nabla(L_+)-\nabla(L_-)+(t^{-1/2}-t^{1/2})\nabla(L_0)=0.\]
 While the skein relation invariant polynomial determines both Alexander's and Jones' polynomial, there are no other relations between them, for instance
 \begin{itemize}
  \item $V(4_1)=V(11_{388})$, but $\nabla(4_1)\neq\nabla(11_{388})$.
  \item $V(11_{388})=V(\overline{11_{388}})$ and $\nabla(11_{388})=\nabla(\overline{11_{388}})$, but $P(11_{388})\neq{P(\overline{11_{388}})}$.
  \item $\nabla(11_{471})=1=\nabla(0_1)$, but $V(11_{471})\neq1$.
 \end{itemize}
 On another note, we notice that for any knot $K$, $V(K)_{t=-1}=\nabla(K)_{t=-1}$. This is the determinant of $K$.
\end{remark}

\begin{remark}
 The Jones polynomial is as strong as the 3-colouring invariant polynomial is weak.
\end{remark}

The following conjecture is still open. Khovanov homology \cite{khovanov}, however, does detect the unknot.
\begin{conjecture}[Jones \cite{jones1985}]
 Does there exist a non-trivial knot $K$ such that $V(K)=1$?
\end{conjecture}


From the author's exposition on knot theory \cite{knotpaper}, it was shown that the Alexander polynomial can be defined using a method known as ``colouring'', and thus we have the following conjecture.
\begin{conjecture}
 Can we define the Jones polynomial via colouring?
\end{conjecture}

While these two definitions were amongst the last to be created with respect to its history, they are amongst the easiest since it is more of a combinatorial calculation, requiring little to no knowledge of more results.

The next chapters will introduce some results and notation we need, to define the original definition of the Jones polynomial. It is in the author's opinion that the original definition of the Jones polynomial (presented in the final chapter), is by far the easiest in terms of both comprehension and calculation.

\section{The braid group \texorpdfstring{$\Bcal_n$}{B}}\label{chapter:braids}
 Thus far, we only know of topological methods for studying links, which motivates the next two chapters where we will learn algebraic methods for studying links. This is because braids can be defined both algebraically and geometrically. So, the aim of the next two chapters will be to discuss the connection between braids and links (Alexander's and Markov's theorem), in hopes of constructing a link invariant. In fact, the original construction of the Jones polynomial \cite{jones1985} (see Chapter \ref{chapter:jones}) utilises these algebraic methods.


We begin with introducing the notion of braids, which happen to be closely related to links with the exception that it cannot be tangled with itself (and is not closed).
Braid equivalence, similarly to link equivalence, is defined with ambient isotopy. Thus, allowing us to define invariant functions on braids.

The aim of this chapter is to form the braid group from a geometrical perspective onto an algebraic presentation. This is done by defining a multiplication (braid concatenation) on the set of braids $\Bcal_n$, thus giving us the group structure we need.
By defining a homomorphism from the braid group onto the symmetric group, we find that the kernel is an interesting group called the pure braid group.

This chapter is based on the presentations given in \cite{enc}, \cite{jones_quantum}, \cite{birman}, \cite{birman_survey}, \cite{BurdeZieschang}, \cite{cromwell}, \cite{jordan}, \cite{roger_fenn}, \cite{jones1987}, \cite{jonesbook}, \cite{jonespoly}, \cite{kassel_turaev}, \cite{kauffman_lomonaco}, \cite{kawauchi}, \cite{livingston}, \cite{manturov}, \cite{murasagibook}, \cite{rolfsen_braids}, and \cite{wilson}.

\subsection{Braids}

Informally, braids consist of a top and a bottom bar, with strings tied to each bar, such that the strings always go downwards. There exists a permutation from the $n$ points on the top bar to the $n$ points on the bottom bar. Recall that a \textit{permutation}\index{permutation} of a set $S$ is a bijective map $f\colon{S\to{S}}$. Here is a precise definition.

\begin{definition}[braid]
 An $n$-\textit{braid}\index{braid} is a set of $n$ injective (non-intersecting) smooth paths (called strands) in $\RR^2\times[0,1]$ connecting the points on the top bar with the points on the bottom bar (i.e., there exists a permutation from the $n$ points on the top bar to the $n$ points on the bottom bar).
 Such that the projection of each of the paths onto the $\RR^2$ plane (shadow - over/under strands) is a diffeomorphism (i.e., no self-tangles: the strings never have a horizontal tangent vector).
 An example of a 3-braid can be seen below in Figure \ref{figure:3-braid}.
\end{definition}

\begin{figure}[H]
 \centering
 \begin{tikzpicture}
  \braid[strands=3,braid start={(0,0)}]
   {\sigma_1 \sigma_2 \sigma_1}
  \draw (0.25,0)--(1.75,0);
  \draw (0.25,-1.5)--(1.75,-1.5);
 \end{tikzpicture}
 \caption{Example of a 3-braid}\label{figure:3-braid}
\end{figure}

\begin{remark}
 Note that the end points of the braids are fixed. And a strand is always going down. So equivalency is essentially ambient isotopy while fixing the end points.
\end{remark}


\begin{definition}[braid equivalence]
 Similarly to links, two braids are \textit{equivalent}\index{equivalent!braids} if they are \textit{ambient isotopic}.
\end{definition}

Similarly to links and knots, we call the equivalence class of braids, \textit{braid type}\index{braid type}.

\begin{proposition}
 Braid equivalence in the above definition is an equivalence relation.
\end{proposition}
\begin{proof}
 It is obviously reflexive. It is symmetric since if braid $\alpha$ can be transformed into braid $\beta$, then by looking over the `moves' of the transformation, we can just go backwards. Finally, it is transitive since if braid $\alpha$ can be transformed into braid $\beta$ and then to braid $\gamma$, by looking at the `moves' of the two transformations, we can concatenate them and form one so that we can get a transformation from $\alpha$ to $\gamma$.
\end{proof}

We can, just as with links, perform elementary moves, the $\Delta$-moves (Definition \ref{definition:triangle-move}), on braids to show that two braids are equivalent. So by fixing the end points of the braids, we can move the strings however we would like, ensuring that we still satisfy the requirements of a braid and without breaking the strands.

  
  

\begin{example}
 The $\RII$ move on a braid via the triangle move:
 
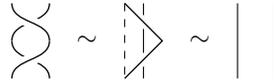
\begin{center}
  \begin{tikzpicture}
   \braid[braid start={(0,0)},strands=2] {\sigma_1^{-1}\sigma_1}
   
   \draw (1.5,-0.5) node {$\sim$};
   
   \draw[thin] (2.25,0)--(2.25,-1);
   \draw[ultra thick,white,double=black,double distance=0.5] (2,0)--(2.5,-0.5)--(2,-1);
   \draw[thin,dashed] (2,0)--(2,-1);
   
   \draw (3,-0.5) node {$\sim$};
   
   \draw[thin] (3.5,0)--(3.5,-1);
   \draw[thin] (4,0)--(4,-1);
  \end{tikzpicture}
 \end{center}
 It is easy to see the $\RIII$ move on a braid via the triangle move, and is left as an exercise for the reader.
\end{example}

\begin{remark}
 The bars on the braids are usually drawn to emphasise that the diagram represents a braid. So this is not always necessary.
\end{remark}

Now that we know we can perform the $\RII$ and $\RIII$ moves on a braid, it is typical to wonder if we can perform an $\RI$ move. Obviously we cannot since by definition, the strand cannot self-tangle. However, in the next chapter, we will see that when relating braids to knots (via closure), there exists a fascinating move that mimics the $\RI$ move when closing the braid.

\begin{definition}[braid permutation]\label{definition:braid permutation}
 Suppose we have an $n$-braid $\alpha$, with $A_1$ connected to $A^\prime_{1^\prime}$, $A_2$ connected to $A^\prime_{2^\prime}$, $\ldots$, $A_n$ connected to $A^\prime_{n^\prime}$. Then, we can assign to $\alpha$ a permutation,
 \[\begin{pmatrix}1&2&\cdots&n\\1^\prime&2^\prime&\cdots&n^\prime\end{pmatrix}.\]
 This is called the \textit{braid permutation}\index{braid permutation}.
\end{definition}

The \textit{trivial} $n$-braid, which we get by connecting $A_1$ to $A^\prime_1$, $A_2$ to $A^\prime_2$, $\ldots$, $A_n$ to $A^\prime_n$ corresponds to the identity permutation, $\begin{pmatrix}1&2&\cdots&n\\1&2&\cdots&n\end{pmatrix}$.
\begin{figure}[H]
 \centering
 \begin{tikzpicture}[scale=0.5]
  \draw[thin] (0.5,0)--(4.5,0);
  \draw[thin] (1,0)--(1,-3);
  \draw[thin] (1.5,0)--(1.5,-3);
  \draw[thin] (2,0)--(2,-3);
  \draw[thin] (3,-1.5) node[scale=1.3] {$\bm{\ldots}$};
  \draw[thin] (4,0)--(4,-3);
  \draw[thin] (0.5,-3)--(4.5,-3);
 \end{tikzpicture}
 \caption{the trivial n-braid}
 \label{figure:}
\end{figure}

The braid permutation for Figure \ref{figure:3-braid} is
\[\begin{pmatrix}1&2&3\\3&2&1\end{pmatrix}=(1\;3).\]

\begin{proposition}
 The braid permutation is invariant under ambient isotopy on braids.
\end{proposition}
\begin{proof}
 If two braids are equivalent, their braid permutations are equal (this is clear since by definition, in order to transform one braid into another, we have to keep the end points fixed). Thus, this is a braid invariant.
\end{proof}

\subsection{Representations of groups}

This section will give a brief on representations of groups.

\begin{definition}[representation]
 A \textit{representation}\index{representation} of a finite group $G$ on an $n$-dimensional $\FF$-vector space $V$ is a group homomorphism $\phi\colon{G\to\text{GL}_n(V)}$.
\end{definition}

Given a group $G$ that acts on set $X=\{x_1,x_2,\ldots,x_n\}$. We call the vector space $\FF(X)$ a \textit{permutation representation} and consists of linear combinations.

So, recall that the group algebra $\CC[S_n]$ is generated by $a_i=(i\;i+1)$ for $i\leq{n-1}$. And has the following relations.
\begin{enumerate}[(i)]
 \item $a_i^2=1$
 \item $a_ia_j=a_ja_i$ if $|i-j|>1$
 \item $a_ia_{i+1}a_i=a_{i+1}a_ia_{i+1}$
\end{enumerate}
This is called the \textit{presentation}\index{presentation} of $\CC[S_n]$. 

\subsection{The braid group}

In this section we will give the set of braids a group structure by defining a multiplication. This is useful because we can use algebraic instead of geometrical structures.

We let $\Bcal_n$ be the set of the equivalence classes of all $n$-braids.

\begin{definition}[braid multiplication]
 \textit{Braid multiplication}\index{braid multiplication} of $n$-braids is defined by concatenation of the braids by putting one braid on top of the other, i.e., $\Bcal_n\times\Bcal_n\to\Bcal_n$. By convention, $\alpha\beta$ means $\alpha$ on top of $\beta$.
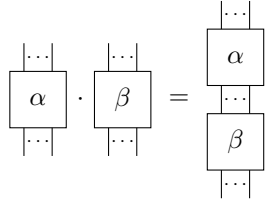
 \begin{center}
  \begin{tikzpicture}[scale=0.75]
   \begin{scope}[yshift=-0.75cm,xshift=-3.5cm]
    \draw (0,0) rectangle (1,1);
    \draw (0.25,1.5)--(0.25,1);
    \draw (0.25,0)--(0.25,-0.5);
    \draw (0.75,0)--(0.75,-0.5);
    \draw (0.75,1.5)--(0.75,1);
    \draw (0.53,1.25) node {$\cdots$};
    \draw (0.53,-0.25) node {$\cdots$};
    \draw (0.5,0.5) node[scale=1.33] {$\alpha$};
   \end{scope}
   \begin{scope}[yshift=-0.75cm,xshift=-2cm]
    \draw (-0.25,0.5) node[scale=1.33] {$\cdot$};
    \draw (0,0) rectangle (1,1);
    \draw (0.25,1.5)--(0.25,1);
    \draw (0.25,0)--(0.25,-0.5);
    \draw (0.75,0)--(0.75,-0.5);
    \draw (0.75,1.5)--(0.75,1);
    \draw (0.53,1.25) node {$\cdots$};
    \draw (0.53,-0.25) node {$\cdots$};
    \draw (0.5,0.5) node[scale=1.33] {$\beta$};
   \end{scope}
   \draw (-0.5,-0.25) node[scale=1.33] {$=$};
   \draw (0,0) rectangle (1,1);
   \draw (0.25,1.5)--(0.25,1);
   \draw (0.25,0)--(0.25,-0.5);
   \draw (0.75,0)--(0.75,-0.5);
   \draw (0.75,1.5)--(0.75,1);
   \draw (0.53,1.25) node {$\cdots$};
   \draw (0.53,-0.25) node {$\cdots$};
   \draw (0,-0.5) rectangle (1,-1.5);
   \draw (0.25,-1.5)--(0.25,-2);
   \draw (0.75,-1.5)--(0.75,-2);
   \draw (0.53,-1.75) node {$\cdots$};
   \draw (0.5,0.5) node[scale=1.33] {$\alpha$};
   \draw (0.5,-1) node[scale=1.33] {$\beta$};
  \end{tikzpicture}
 \end{center}
\end{definition}

It is typical to ask if (how) we can multiply an $n$-braid with an $m$-braid, for $n\neq{m}$. In fact we will need to do this in the next chapter when relating a braid to a knot.
It turns out that there exists a natural inclusion from $\Bcal_n$ to $\Bcal_{n+1}$ by adding a vertical string to the right of the braid. So, we have the inclusions,
\[\Bcal_1\subset\Bcal_2\subset\cdots\]
This is quite similar to the Temperley-Lieb algebra planar diagram relations in Chapter \ref{chapter:tl}.
For the sake of completeness and less ambiguity, we shall define precise definitions we can use.
\begin{definition}[braid inclusion]
 The \textit{braid inclusion map}\index{linear operations on $\Bcal_n$!braid inclusion map $\incl$} $\incl_n\colon\Bcal_n\to\Bcal_{n+1}$ is defined by,
 \begin{figure}[H]
 \centering
 \[\incl_n\left(\knotsinmath{\begin{scope}[yshift=-0.35cm,scale=0.75]
    \draw (0,0) rectangle (1,1);
    \draw (0.25,1.5)--(0.25,1);
    \draw (0.25,0)--(0.25,-0.5);
    \draw (0.75,0)--(0.75,-0.5);
    \draw (0.75,1.5)--(0.75,1);
    \draw (0.53,1.25) node {$\cdots$};
    \draw (0.53,-0.25) node {$\cdots$};
    \draw (0.5,0.5) node[scale=1.33] {$\alpha$};
   \end{scope}}\right) =\,\, \knotsinmath{\begin{scope}[yshift=-0.35cm,scale=0.75]
    \draw (0,0) rectangle (1,1);
    \draw (0.25,1.5)--(0.25,1);
    \draw (0.25,0)--(0.25,-0.5);
    \draw (0.75,0)--(0.75,-0.5);
    \draw (0.75,1.5)--(0.75,1);
    \draw (0.53,1.25) node {$\cdots$};
    \draw (0.53,-0.25) node {$\cdots$};
    \draw (0.5,0.5) node[scale=1.33] {$\alpha$};
    \draw (1.25,1.5)--(1.25,-0.5);
   \end{scope}}\]
  \caption{braid inclusion}
 \end{figure}
\end{definition}

So, we can multiply an $n$-braid with an $m$-braid for $n<m$, means that we have to use the inclusion map on the $n$-braid recursively until $n=m$. We have the following example.
\begin{example}
 Let $\alpha\in\Bcal_n$ and $\varrho\in\Bcal_{n+3}$, so to multiply them we need to perform the braid inclusion map on $\alpha$ thrice, i.e.,
 \[\incl_n(\alpha)\in\Bcal_{n+1}\tto\incl_{n+1}(\incl_{n}(\alpha))\in\Bcal_{n+2}\tto\incl_{n+2}(\incl_{n+1}(\incl_{n}(\alpha)))\in\Bcal_{n+3}\]
 visually we have,
 \begin{align*}
  \knotsinmath{\begin{scope}[yshift=-0.35cm,scale=0.75]
    \draw (0,0) rectangle (1,1);
    \draw (0.25,1.5)--(0.25,1);
    \draw (0.25,0)--(0.25,-0.5);
    \draw (0.75,0)--(0.75,-0.5);
    \draw (0.75,1.5)--(0.75,1);
    \draw (0.53,1.25) node {$\cdots$};
    \draw (0.53,-0.25) node {$\cdots$};
    \draw (0.5,0.5) node[scale=1.33] {$\alpha$};
   \end{scope}} \cdot \knotsinmath{\begin{scope}[yshift=-0.35cm,scale=0.75]
    \draw (0,0) rectangle (2,1);
    \draw (0.25,1.5)--(0.25,1);
    \draw (0.25,0)--(0.25,-0.5);
    \draw (1.25,1.5)--(1.25,1);
    \draw (1.25,0)--(1.25,-0.5);
    \draw (1.5,0)--(1.5,-0.5);
    \draw (1.5,1.5)--(1.5,1);
    \draw (1.75,0)--(1.75,-0.5);
    \draw (1.75,1.5)--(1.75,1);
    \draw (0.85,1.25) node {$\cdots$};
    \draw (0.85,-0.25) node {$\cdots$};
    \draw (1,0.5) node[scale=1.33] {$\varrho$};
   \end{scope}} &= \incl_{n+2}\circ\incl_{n+1}\circ\incl_{n}\left(\knotsinmath{\begin{scope}[yshift=-0.35cm,scale=0.75]
    \draw (0,0) rectangle (1,1);
    \draw (0.25,1.5)--(0.25,1);
    \draw (0.25,0)--(0.25,-0.5);
    \draw (0.75,0)--(0.75,-0.5);
    \draw (0.75,1.5)--(0.75,1);
    \draw (0.53,1.25) node {$\cdots$};
    \draw (0.53,-0.25) node {$\cdots$};
    \draw (0.5,0.5) node[scale=1.33] {$\alpha$};
   \end{scope}}\right) \cdot \knotsinmath{\begin{scope}[yshift=-0.35cm,scale=0.75]
    \draw (0,0) rectangle (2,1);
    \draw (0.25,1.5)--(0.25,1);
    \draw (0.25,0)--(0.25,-0.5);
    \draw (0.75,0)--(0.75,-0.5);
    \draw (0.75,1.5)--(0.75,1);
    \draw (1.25,1.5)--(1.25,1);
    \draw (1.25,0)--(1.25,-0.5);
    \draw (1.5,0)--(1.5,-0.5);
    \draw (1.5,1.5)--(1.5,1);
    \draw (1.75,0)--(1.75,-0.5);
    \draw (1.75,1.5)--(1.75,1);
    \draw (0.53,1.25) node {$\cdots$};
    \draw (0.53,-0.25) node {$\cdots$};
    \draw (1,0.5) node[scale=1.33] {$\varrho$};
   \end{scope}}\\
  &= \knotsinmath{\begin{scope}[yshift=-0.35cm,scale=0.75]
    \draw (0,0) rectangle (1,1);
    \draw (0.25,1.5)--(0.25,1);
    \draw (0.25,0)--(0.25,-0.5);
    \draw (0.75,0)--(0.75,-0.5);
    \draw (0.75,1.5)--(0.75,1);
    \draw (0.53,1.25) node {$\cdots$};
    \draw (0.53,-0.25) node {$\cdots$};
    \draw (0.5,0.5) node[scale=1.33] {$\alpha$};
    \draw (1.25,1.5)--(1.25,-0.5);
    \draw (1.5,1.5)--(1.5,-0.5);
    \draw (1.75,1.5)--(1.75,-0.5);
   \end{scope}} \cdot \knotsinmath{\begin{scope}[yshift=-0.35cm,scale=0.75]
    \draw (0,0) rectangle (2,1);
    \draw (0.25,1.5)--(0.25,1);
    \draw (0.25,0)--(0.25,-0.5);
    \draw (0.75,0)--(0.75,-0.5);
    \draw (0.75,1.5)--(0.75,1);
    \draw (1.25,1.5)--(1.25,1);
    \draw (1.25,0)--(1.25,-0.5);
    \draw (1.5,0)--(1.5,-0.5);
    \draw (1.5,1.5)--(1.5,1);
    \draw (1.75,0)--(1.75,-0.5);
    \draw (1.75,1.5)--(1.75,1);
    \draw (0.53,1.25) node {$\cdots$};
    \draw (0.53,-0.25) node {$\cdots$};
    \draw (1,0.5) node[scale=1.33] {$\varrho$};
   \end{scope}} = \knotsinmath{\begin{scope}[yshift=0.45cm,scale=0.75]
    \draw (0,0) rectangle (1,1);
    \draw (0.25,1.5)--(0.25,1);
    \draw (0.25,0)--(0.25,-0.5);
    \draw (0.75,0)--(0.75,-0.5);
    \draw (0.75,1.5)--(0.75,1);
    \draw (0.53,1.25) node {$\cdots$};
    \draw (0.53,-0.5) node {$\cdots$};
    \draw (0.5,0.5) node[scale=1.33] {$\alpha$};
    \draw (1.25,1.5)--(1.25,-0.5);
    \draw (1.5,1.5)--(1.5,-0.5);
    \draw (1.75,1.5)--(1.75,-0.5);
   \end{scope}
   \begin{scope}[yshift=-1.05cm,scale=0.75]
    \draw (0,0) rectangle (2,1);
    \draw (0.25,1.5)--(0.25,1);
    \draw (0.25,0)--(0.25,-0.5);
    \draw (0.75,0)--(0.75,-0.5);
    \draw (0.75,1.5)--(0.75,1);
    \draw (1.25,1.5)--(1.25,1);
    \draw (1.25,0)--(1.25,-0.5);
    \draw (1.5,0)--(1.5,-0.5);
    \draw (1.5,1.5)--(1.5,1);
    \draw (1.75,0)--(1.75,-0.5);
    \draw (1.75,1.5)--(1.75,1);
    \draw (0.53,-0.25) node {$\cdots$};
    \draw (1,0.5) node[scale=1.33] {$\varrho$};
   \end{scope}}
 \end{align*}
\end{example}

In general, the $n$-braid $\alpha\beta$ is not equivalent to $\beta\alpha$. However, braid multiplication is associative, i.e., $\alpha(\beta\gamma)=(\alpha\beta)\gamma$ (see proof of Theorem \ref{theorem:braid_group}).

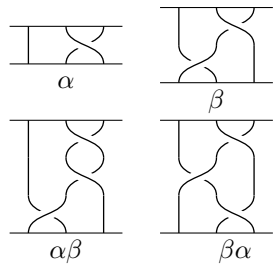
\begin{figure}[H]
 \centering
 \begin{tikzpicture}
  \braid[strands=3,braid start={(0,-0.25)}]
   {\sigma_2^{-1}}
  \draw (0.25,-0.25)--(1.75,-0.25);
  \draw (0.25,-0.75)--(1.75,-0.75);
  \draw (1,-1) node {$\alpha$};
  
  \braid[strands=3,braid start={(2,0)}]
   {\sigma_2^{-1}\sigma_1}
  \draw (2.25,0)--(3.75,0);
  \draw (2.25,-1)--(3.75,-1);
  \draw (3,-1.25) node {$\beta$};
  
  \braid[strands=3,braid start={(0,-3/2)}] {\sigma_2^{-1}\sigma_2^{-1}\sigma_1}
  \draw (0.25,-1.5)--(1.75,-1.5);
  \draw (0.25,-3)--(1.75,-3);
  \draw (1,-3.25) node {$\alpha\beta$};
  
  \braid[strands=3,braid start={(2,-3/2)}] {\sigma_2^{-1}\sigma_1\sigma_2^{-1}}
  \draw (4.5/2,-3/2)--(7.5/2,-3/2);
  \draw (4.5/2,-6/2)--(7.5/2,-6/2);
  \draw (6.5/2,-6.5/2) node {$\beta\alpha$};
 \end{tikzpicture}
 \caption{(\cite[Exercise 10.2.1,~p.~202]{murasagibook}) braid multiplication is not commutative, i.e. $\alpha\beta$ is not equivalent to $\beta\alpha$ (their braid permutations are different, hence they are not equivalent)}
 \label{figure:braid_mult_comm}
\end{figure}

The braid identity $\Xi$ is the trivial braid such that $\Xi\alpha=\alpha=\alpha\Xi$.
\begin{figure}[H]
 \centering
 \begin{tikzpicture}
  \draw (0.1,0)--(0.1,-1);
  \draw (0.2,0)--(0.2,-1);
  \draw (0.3,0)--(0.3,-1);
  \draw (0.55,-0.5) node {$\ldots$};
  \draw (0.8,0)--(0.8,-1);
  \draw (0.9,0)--(0.9,-1);
  \draw (0,0)--(1,0);
  \draw (0,-1)--(1,-1);
 \end{tikzpicture}
 \caption{trivial braid}
 \label{figure:trivial_braid}
\end{figure}

Finally, the inverse of a braid is the reflection of a braid, such that $\alpha^{-1}\cdot\alpha=\Xi=\alpha\cdot\alpha^{-1}$.
 \begin{figure}[H]
  \centering
  \[\left(\knotsinmath{\begin{scope}[yshift=-0.35cm,scale=0.75]
    \draw (0,0) rectangle (1,1);
    \draw (0.25,1.5)--(0.25,1);
    \draw (0.25,0)--(0.25,-0.5);
    \draw (0.75,0)--(0.75,-0.5);
    \draw (0.75,1.5)--(0.75,1);
    \draw (0.53,1.25) node {$\cdots$};
    \draw (0.53,-0.25) node {$\cdots$};
    \draw (0.5,0.5) node[scale=1.33] {$\alpha$};
   \end{scope}}\right)^{-1} = \,\,\reflectbox{\knotsinmath{\begin{scope}[yshift=-0.35cm,scale=0.75]
    \draw (0,0) rectangle (1,1);
    \draw (0.25,1.5)--(0.25,1);
    \draw (0.25,0)--(0.25,-0.5);
    \draw (0.75,0)--(0.75,-0.5);
    \draw (0.75,1.5)--(0.75,1);
    \draw (0.53,1.25) node {$\cdots$};
    \draw (0.53,-0.25) node {$\cdots$};
    \draw (0.5,0.5) node[scale=1.33] {$\alpha$};
   \end{scope}}}\]
  \caption{braid inverse}
 \end{figure}
and so,
   \[\knotsinmath{\begin{scope}[yshift=0.45cm,scale=0.75]
    \draw (0,0) rectangle (1,1);
    \draw (0.25,1.5)--(0.25,1);
    \draw (0.25,0)--(0.25,-0.5);
    \draw (0.75,0)--(0.75,-0.5);
    \draw (0.75,1.5)--(0.75,1);
    \draw (0.53,1.25) node {$\cdots$};
    \draw (0.53,-0.55) node {$\cdots$};
    \draw (0.5,0.5) node[scale=1.33] {$\alpha$};
    \draw (-1,-0.5) node[scale=1.33] {$\alpha\cdot\alpha^{-1}=$};
   \end{scope}
   \begin{scope}[yshift=-1.05cm,xshift=0.75cm,scale=0.75,xscale=-1]
    \draw (0,0) rectangle (1,1);
    \draw (0.25,1.5)--(0.25,1);
    \draw (0.25,0)--(0.25,-0.5);
    \draw (0.75,0)--(0.75,-0.5);
    \draw (0.75,1.5)--(0.75,1);
    \draw (0.53,-0.25) node {$\cdots$};
    \draw (0.5,0.5) node[scale=1.33] {$\alpha$};
   \end{scope}} = \knotsinmath{\begin{scope}[yshift=0.5cm]
    \draw (0.1,0)--(0.1,-1);
    \draw (0.2,0)--(0.2,-1);
    \draw (0.3,0)--(0.3,-1);
    \draw (0.55,-0.5) node {$\ldots$};
    \draw (0.8,0)--(0.8,-1);
    \draw (0.9,0)--(0.9,-1);
   \end{scope}}=\Xi\]

\begin{example}
 The following shows an example of multiplying a braid by its inverse.
 \begin{figure}[H]
  \centering
  \begin{tikzpicture}
   \braid[strands=2,braid start={(0,0)}]
   {\sigma_1}
   \draw (0.5/2,0)--(2.5/2,0);
   \draw (0.5/2,-1/2)--(2.5/2,-1/2);
   \draw (1.5/2,-1.5/2) node {$\alpha$};
   \braid[strands=2,braid start={(3/2,0)}]
   {\sigma_1^{-1}}
   \draw (3.5/2,0)--(5.5/2,0);
   \draw (3.5/2,-1/2)--(5.5/2,-1/2);
   \draw (4.5/2,-1.5/2) node {$\alpha^{-1}$};
   \braid[strands=2,braid start={(1.5/2,-2/2)}] {\sigma_1 \sigma_1^{-1}}
   \draw (2/2,-2/2)--(4/2,-2/2);
   \draw (2/2,-4/2)--(4/2,-4/2);
   \draw (3/2,-4.5/2) node {$\alpha\alpha^{-1}$};
  \end{tikzpicture}
  \caption{braid inverse example}
  \label{figure:braid_inverse}
 \end{figure}
 The braid $\alpha\alpha^{-1}$ is clearly equivalent to the trivial braid after an $\RII$ move.
\end{example}

We are done. Let us summarize the braid group with the following statement.
\begin{theorem}[the geometrical braid group]\index{braid group $\Bcal_n$!geometrical}\label{theorem:braid_group}
 The set of all $n$-braids $\Bcal_n$ forms a group under braid composition (concatenation) defined below. Given two $n$-braids $\alpha$ and $\beta$,
 \begin{center}
  \begin{tikzpicture}
   \draw (0.5,-0.5) node {$\alpha$};
   \draw (0,0)--(1,0)--(1,-1)--(0,-1)--cycle;
   \draw (1.5,-0.5) node {$\odot$};
   \draw (2.5,-0.5) node {$\beta$};
   \draw (2,0)--(3,0)--(3,-1)--(2,-1)--cycle;
   \draw (3.5,-0.5) node {$=$};
   \draw (4.5,0) node {$\alpha$};
   \draw (4.5,-1) node {$\beta$};
   \draw (4,0.5)--(5,0.5)--(5,-1.5)--(4,-1.5)--cycle;
   \draw[dashed] (4,-0.5)--(5,-0.5);
   \draw (5.5,-0.5) node {$=$};
   \draw (6.5,-0.5) node {$\alpha\beta$};
   \draw (6,0)--(7,0)--(7,-1)--(6,-1)--cycle;
  \end{tikzpicture}
 \end{center}
\end{theorem}
\begin{proof}
 Let $\alpha,\beta,\gamma$ be $n$-braids. Composing two $n$-braids, by definition, will give a new $n$-braid. So we have closure: $\odot\colon\Bcal_n\odot\Bcal_n\to\Bcal_n$. We also have associativity, i.e. $\alpha(\beta\gamma)=(\alpha\beta)\gamma$ (see figure below).
 \begin{center}
  \begin{tikzpicture}
   \draw (0.5,0) node {$\alpha$};
   \draw (0.5,-1) node {$\beta\gamma$};
   \draw (0,0.5)--(1,0.5)--(1,-1.5)--(0,-1.5)--cycle;
   \draw[dashed] (0,-0.5)--(1,-0.5);
   \draw (1.5,-0.5) node {$=$};
   \draw (2.5,0.5) node {$\alpha$};
   \draw (2.5,-0.5) node {$\beta$};
   \draw (2.5,-1.5) node {$\gamma$};
   \draw (2,1)--(3,1)--(3,-2)--(2,-2)--cycle;
   \draw[dashed] (2,0)--(3,0);
   \draw[dashed] (2,-1)--(3,-1);
   \draw (3.5,-0.5) node {$=$};
   \draw (4.5,0) node {$\alpha\beta$};
   \draw (4.5,-1) node {$\gamma$};
   \draw (4,0.5)--(5,0.5)--(5,-1.5)--(4,-1.5)--cycle;
   \draw[dashed] (4,-0.5)--(5,-0.5);
  \end{tikzpicture}
 \end{center}
 We have an identity $\Xi$ which is the trivial $n$-braid. So, $\alpha\Xi=\alpha=\Xi\alpha$. The figure below shows $\alpha\Xi=\alpha$ (from this, it then follows that $\Xi\alpha=\alpha$).
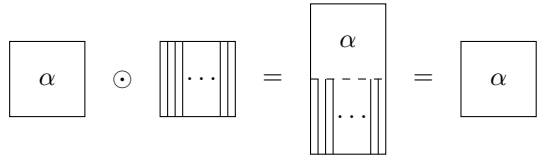
 \begin{center}
  \begin{tikzpicture}
   \draw (0.5,-0.5) node {$\alpha$};
   \draw (0,0)--(1,0)--(1,-1)--(0,-1)--cycle;
   \draw (1.5,-0.5) node {$\odot$};
   \draw (2.1,0)--(2.1,-1);
   \draw (2.2,0)--(2.2,-1);
   \draw (2.3,0)--(2.3,-1);
   \draw (2.55,-0.5) node {$\ldots$};
   \draw (2.8,0)--(2.8,-1);
   \draw (2.9,0)--(2.9,-1);
   \draw (2,0)--(3,0)--(3,-1)--(2,-1)--cycle;
   \draw (3.5,-0.5) node {$=$};
   \draw (4.5,0) node {$\alpha$};
   \draw (4.1,-0.5)--(4.1,-1.5);
   \draw (4.2,-0.5)--(4.2,-1.5);
   \draw (4.3,-0.5)--(4.3,-1.5);
   \draw (4.55,-1) node {$\ldots$};
   \draw (4.8,-0.5)--(4.8,-1.5);
   \draw (4.9,-0.5)--(4.9,-1.5);
   \draw (4,0.5)--(5,0.5)--(5,-1.5)--(4,-1.5)--cycle;
   \draw[dashed] (4,-0.5)--(5,-0.5);
   \draw (5.5,-0.5) node {$=$};
   \draw (6.5,-0.5) node {$\alpha$};
   \draw (6,0)--(7,0)--(7,-1)--(6,-1)--cycle;
  \end{tikzpicture}
 \end{center}
 Finally, we have inverses. By definition, there exists an $n$-braid, called $\alpha^{-1}$ for the sake of argument, where the composition of this with another $n$-braid, say $\alpha$, would produce a braid which is equivalent to the trivial braid by elementary triangle moves. So, $\alpha\alpha^{-1}=\Xi=\alpha^{-1}\alpha$. The inverse braid is the mirror image of the braid.
\end{proof}

So, similarly to how there exists a natural inclusion of the set $\Bcal_n$, the braid group $\Bcal_n$ is embedded as a subgroup of $\Bcal_{n+1}$ (with the braid inclusion map).

There exists a representation of the group by writing all the braids as words on some generators. The generators $\sigma_i$ is defined by connecting $A^\prime_i$ to $A_{i+1}$ and $A^\prime_{i+1}$ to $A_i$ and is seen in Figure \ref{figure:braid_generator}.

\begin{figure}[H]
 \centering
 \begin{tikzpicture}
  \draw (-0.5/2,-0.5/2) node {$\sigma_i =$};
  \braid[braid start={(0,0)},strands=8] {\sigma_4}
  \draw (1/2,-1.5/2) node {1};
  \draw (2/2,-1.5/2) node {2};
  \draw (4/2,-1.5/2) node {$i$};
  \draw (5/2,-1.53/2) node {$i+1$};
  \draw (8/2,-1.5/2) node {$n$};
  \draw (0.5/2,0)--(8.5/2,0);
  \draw (0.5/2,-1/2)--(8.5/2,-1/2);
 \end{tikzpicture}
 \caption{braid generator $\sigma_i$}
 \label{figure:braid_generator}
\end{figure}

And the inverse, $\sigma_i^{-1}$ is just
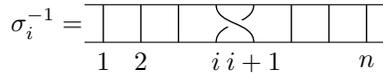
\begin{center}
 \begin{tikzpicture}
  \draw (-0.5/2,-0.5/2) node {$\sigma_i^{-1}=$};
  \braid[braid start={(0,0)},strands=8] {\sigma_4^{-1}}
  \draw (1/2,-1.5/2) node {1};
  \draw (2/2,-1.5/2) node {2};
  \draw (4/2,-1.5/2) node {$i$};
  \draw (5/2,-1.53/2) node {$i+1$};
  \draw (8/2,-1.5/2) node {$n$};
  \draw (0.5/2,0)--(8.5/2,0);
  \draw (0.5/2,-1/2)--(8.5/2,-1/2);
 \end{tikzpicture}
\end{center}

These generators are also known as \textit{elementary transpositions}\index{elementary transposition}.

We can use these generators to express any element in the braid group.

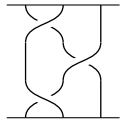
\begin{figure}[H]
 \centering
 \begin{tikzpicture}
  \braid[braid start={(0,0)},strands=3] {\sigma_1\sigma_2\sigma_1^{-1}}
  \draw (0.5/2,0)--(3.5/2,0);
  \draw (0.5/2,-3/2)--(3.5/2,-3/2);
 \end{tikzpicture}
 \caption{$\sigma_1\sigma_2\sigma_1^{-1}$}
\end{figure}

We can also express any element in the braid group using these generators. We can do this by dividing our braid diagram with lines parallel to the bars. The rule is that we can only have one crossing in each divided rectangle, if there are more crossings, then we can just shift one crossing downwards or upwards. See below figure.

\begin{figure}[H]
 \centering
 \begin{tikzpicture}
  \braid[braid start={(0,0)},strands=3] {\sigma_2\sigma_2\sigma_1^{-1}}
  \draw (0.5/2,0)--(3.5/2,0);
  \draw[dashed] (0.5/2,-1/2)--(3.5/2,-1/2);
  \draw[dashed] (0.5/2,-2/2)--(3.5/2,-2/2);
  \draw (0.5/2,-3/2)--(3.5/2,-3/2);
  
  \draw (4.5/2,-0.5/2) node {$\sigma_2$};
  \draw (4.5/2,-1.5/2) node {$\sigma_2$};
  \draw (4.5/2,-2.5/2) node {$\sigma_1^{-1}$};
 \end{tikzpicture}
 \caption{$\sigma_2\sigma_2\sigma_1^{-1}$}
\end{figure}
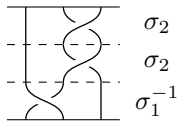

Thus the braid group $\Bcal_n$ can be described by those generators and three relations. It is usually referred to the Artin braid group, since Artin introduced it \cite{artin1925,artin1947,birman_survey}.
\begin{definition}[Artin's braid group $\Bcal_n$ {\cite{artin1925,artin1947}}]
 The \textit{Artin braid group $\Bcal_n$}\index{braid group $\Bcal_n$!Artin's presentation} is the group generated by $\sigma_i$ for $i=1,\ldots,n-1$,
 \begin{center}
  \begin{tikzpicture}
   \draw (-0.5/2,-0.5/2) node {$\sigma_i =$};
   \braid[braid start={(0,0)},strands=8] {\sigma_4}
   \draw (1/2,-1.5/2) node {1};
   \draw (2/2,-1.5/2) node {2};
   \draw (4/2,-1.5/2) node {$i$};
   \draw (5/2,-1.5/2) node {$i+1$};
   \draw (8/2,-1.5/2) node {$n$};
   \draw (0.5/2,0)--(8.5/2,0);
   \draw (0.5/2,-1/2)--(8.5/2,-1/2);
  \end{tikzpicture}
 \end{center}
 and has the following relations
 \begin{enumerate}[(i)]
  \item $\sigma_i\sigma_j=\sigma_j\sigma_i \mbox{ if } |i-j|\geq2$
  \item $\sigma_i\sigma_{i+1}\sigma_i=\sigma_{i+1}\sigma_i\sigma_{i+1} \mbox{ for } i=1,2,\ldots,n-2$
  \item $\sigma_i\sigma_i^{-1}=\Xi, \mbox{ for } i=1,\ldots{n-1}$
 \end{enumerate}
\end{definition}

Notice that the first relation just means that we can shift the crossings if they are ``separated.''
\begin{figure}[H]
 \centering
 \begin{tikzpicture}
  \braid[braid start={(0,0)},strands=4] {\sigma_1\sigma_3}
  \draw (0.5/2,0)--(4.5/2,0);
  \draw (0.5/2,-2/2)--(4.5/2,-2/2);
  
  \draw (5/2,-1/2) node {$=$};
  
  \braid[braid start={(5/2,0)},strands=4] {\sigma_3\sigma_1}
  \draw (5.5/2,0)--(9.5/2,0);
  \draw (5.5/2,-2/2)--(9.5/2,-2/2);
 \end{tikzpicture}
 \caption{$\sigma_i\sigma_j=\sigma_j\sigma_i$, if $|i-j|\geq2$}
\end{figure}

The second relation, here, is essentially the Redemeister $3$ move.
\begin{figure}[H]
 \centering
 \begin{tikzpicture}
  \braid[braid start={(0,0)},strands=3] {\sigma_1\sigma_2\sigma_1}
  \draw (0.5/2,0)--(3.5/2,0);
  \draw (0.5/2,-3/2)--(3.5/2,-3/2);
  
  \draw (4/2,-1.5/2) node {$=$};
  
  \braid[braid start={(4/2,0)},strands=3] {\sigma_2\sigma_1\sigma_2}
  \draw (4.5/2,0)--(7.5/2,0);
  \draw (4.5/2,-3/2)--(7.5/2,-3/2);
 \end{tikzpicture}
 \caption{$\sigma_i\sigma_{i+1}\sigma_i=\sigma_{i+1}\sigma_i\sigma_{i+1}$}
\end{figure}
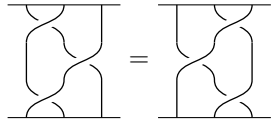

The third relation is just the Redeimeister 2 move.
\begin{figure}[H]
 \centering
 \begin{tikzpicture}
  \braid[braid start={(0,0)},strands=6] {\sigma_3\sigma_3^{-1}}
  \draw (0.5/2,0)--(6.5/2,0);
  \draw (0.5/2,-2/2)--(6.5/2,-2/2);
  \draw (1.5/2,-1/2) node[scale=0.6] {$\ldots$};
  \draw (5.5/2,-1/2) node[scale=0.6] {$\ldots$};
  \draw (7/2,-1/2) node {$=$};
  \draw (7.5/2,0)--(9/2,0);
  \draw (7.5/2,-2/2)--(9/2,-2/2);
  \draw (7.6/2,0)--(7.6/2,-2/2);
  \draw (7.9/2,0)--(7.9/2,-2/2);
  \draw (8.25/2,-1/2) node[scale=0.6] {$\ldots$};
  \draw (8.6/2,0)--(8.6/2,-2/2);
  \draw (8.9/2,0)--(8.9/2,-2/2);
 \end{tikzpicture}
 \caption{$\sigma_i\sigma_i^{-1}=\Xi$}
\end{figure}


\begin{example}
 $\Bcal_1$ is generated by $\sigma_1$\\
 $\Bcal_2$ is generated by $\sigma_1,\sigma_2$ and has the relation $\sigma_1\sigma_2\sigma_1=\sigma_2\sigma_1\sigma_2$\\
 $\Bcal_3$ is generated by $\sigma_1,\sigma_2,\sigma_3$ and has (in addition to the relations in $\Bcal_2$) $\sigma_2\sigma_3\sigma_2=\sigma_3\sigma_2\sigma_3$ and $\sigma_1\sigma_3=\sigma_3\sigma_1$.
\end{example}

\begin{remark}
 The braid group $\Bcal_2$ is isomorphic to $\ZZ$.
\end{remark}

We have already seen the braid permutation from Definition \ref{definition:braid permutation}. So it is not hard to see that there exists a natural homomorphism from the Artin braid group $\Bcal_n$ into the symmetric group $S_n$.
\begin{corollary}
 There is a group homomorphism $\varphi$, from $\Bcal_n$ onto the symmetric group $S_n$.
 We have the exact sequence
 \[1\to{P\Bcal_n}\to\Bcal_n\to{S_n}\to1\]
 where the kernel of the homomorphism is known as the pure braid group $P\Bcal_n$.
\end{corollary}

So the end points of elements in the pure braid group are not permuted. We will give one example. The interested reader is directed to \cite{birman_survey}.
\begin{example}
 An example of an element in the pure braid group $P\Bcal_n$ is shown below,
 
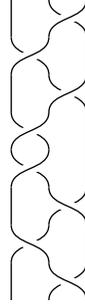
\begin{center}
  \begin{tikzpicture}
   \braid[strands=3,braid start={(0,0)},braid width=0.5cm,braid height=0.5cm] {\sigma_2\sigma_1\sigma_2\sigma_1\sigma_1\sigma_2^{-1}\sigma_1^{-1}\sigma_2^{-1}}
   \draw (1,-4.5) node {$\sigma_2\sigma_1\sigma_2\sigma_1^2\sigma_2^{-1}\sigma_1^{-1}\sigma_2^{-1}$};
  \end{tikzpicture}
 \end{center}
 The generators of the pure braid group are given by
 \[{\text{pb}}_{i,j}=(\sigma_{j-1}\sigma_{j-2}\cdots\sigma_{i+1})\sigma_i^2(\sigma_{i+1}^{-1}\cdots\sigma_{j-2}^{-1}\sigma_{j-1}^{-1}),\]
 for $i<j$. And so, visually, we have
 \begin{figure}[H]
  \centering
  \begin{tikzpicture}
   \braid[strands=6,braid start={(0,0)},braid width=0.5cm,braid height=0.5cm] {\sigma_2\sigma_3\sigma_4^{-1}\sigma_4^{-1}\sigma_3^{-1}\sigma_2^{-1}}
   \draw (1,-3.25) node {$i$};
   \draw (2.5,-3.25) node {$j$};
   \draw (-0.75,-1.625) node {$\text{pb}_{i,j}=$};
  \end{tikzpicture}
  \caption{pure braid generator}
  \label{figure:pure_braid}
 \end{figure}
\end{example}

Now we have an easier, purely algebraic, method of showing that two braids are equivalent.

\begin{proposition}[{\cite[Exercise~10.2.4,~p.~207]{murasagibook}}]
 The following relations hold for any $\Bcal_n$ for $n\geq3$.
 \[
  \sigma_1\sigma_2\sigma_1^{-1}=\sigma_2^{-1}\sigma_1\sigma_2 \mbox{ and } \sigma_2\sigma_1^{-1}\sigma_2^{-1}=\sigma_1^{-1}\sigma_2^{-1}\sigma_1
 \]
\end{proposition}
\begin{proof} {\ }
 \begin{align*}
  \sigma_1\sigma_2\sigma_1^{-1} &= \sigma_2^{-1}\sigma_2\sigma_1\sigma_2\sigma_1^{-1}\\
   &= \sigma_2^{-1}\sigma_1\sigma_2\sigma_1\sigma_1^{-1} = \sigma_2^{-1}\sigma_1\sigma_2
 \end{align*}
 \begin{align*}
  \sigma_2\sigma_1^{-1}\sigma_2^{-1} &= \sigma_2\sigma_1^{-1}\sigma_2^{-1}\sigma_1^{-1}\sigma_1\\
   &= \sigma_2\sigma_2^{-1}\sigma_1^{-1}\sigma_2^{-1}\sigma_1 = \sigma_1^{-1}\sigma_2^{-1}\sigma_1
 \end{align*}
 These will look like
 \begin{center}
  \begin{tikzpicture}
   \braid[strands=3,braid start={(0,0)},braid width=0.5cm,braid height=0.5cm] {\sigma_1\sigma_2\sigma_1^{-1}}
   \draw (1,-2) node {$\sigma_1\sigma_2\sigma_1^{-1}$};
   \braid[strands=3,braid start={(2,0)},braid width=0.5cm,braid height=0.5cm] {\sigma_2^{-1}\sigma_1\sigma_2}
   \draw (3,-2) node {$\sigma_2^{-1}\sigma_1\sigma_2$};
   \draw (2,-0.75) node {$\sim$};
  \end{tikzpicture}
 \end{center}
 And
 \begin{center}
  \begin{tikzpicture}
   \braid[strands=3,braid start={(0,0)},braid width=0.5cm,braid height=0.5cm] {\sigma_2\sigma_1^{-1}\sigma_2^{-1}}
   \draw (1,-2) node {$\sigma_2\sigma_1^{-1}\sigma_2^{-1}$};
   \braid[strands=3,braid start={(2,0)},braid width=0.5cm,braid height=0.5cm] {\sigma_1^{-1}\sigma_2^{-1}\sigma_1}
   \draw (3,-2) node {$\sigma_1^{-1}\sigma_2^{-1}\sigma_1$};
   \draw (2,-0.75) node {$\sim$};
  \end{tikzpicture}
 \end{center}
 Which are clearly just the Reidemeister type 3 move.
\end{proof}

%
To motivate our next example, we will very briefly discuss the conjugacy problem. Our discussion is based on that in Birman's survey \cite[Ch.~5,~p.~60]{birman_survey}.
The \textit{conjugacy problem}\index{conjugacy problem} asks for an algorithm to decide whether $\{a\}=\{a'\}$, where $\{a\},\{a'\}$ are the conjugacy classes of the elements $a,a'\in\Bcal_n$. A solution was first given in 1965 by Garside \cite{garside}. While, there does exist a finite set of conjugates of a representation of an arbitrary element $a\in\Bcal_n$, it is exponential in both $n$ and $|a|$. This makes it difficult for us to understand the combinatorial solution sometimes.
\begin{example}
 The following two braids are known to be equivalent. So we will show that these two braids are equivalent by using the braid group relations and generators. In other words, we are finding a solution that gives us an algorithm to go from the first braid to the second braid (the conjugacy problem). The reader is advised to attempt it.
 \begin{center}
  \begin{tikzpicture}
   \begin{scope}[xshift=-2cm]
    \braid[braid start={(-3,-0.5)},braid width=0.5cm,braid height=0.5cm,strands=4] {\sigma_1^{-1}\sigma_2^{-1}\sigma_3\sigma_3\sigma_2\sigma_1\sigma_3^{-1}\sigma_1\sigma_2\sigma_3^{-1}}
    \draw (-1.75,-6.5) node {$\sigma_1^{-1}\sigma_2^{-1}\sigma_3^2\sigma_2\sigma_1\sigma_3^{-1}\sigma_1\sigma_2\sigma_3^{-1}$};
   \end{scope}
   \begin{scope}[xshift=-1.5cm]
    \braid[braid start={(3,0)},braid width=0.5cm,braid height=0.5cm,strands=4] {\sigma_2\sigma_1^{-1}\sigma_2^{-1}\sigma_3\sigma_3\sigma_1^{-1}\sigma_2\sigma_1\sigma_1\sigma_3^{-1}\sigma_2\sigma_3^{-1}}
    \draw (4.25,-6.5) node {$\sigma_2\sigma_1^{-1}\sigma_2^{-1}\sigma_3^2\sigma_1^{-1}\sigma_2\sigma_1^2\sigma_3^{-1}\sigma_2\sigma_3^{-1}$};
   \end{scope}
  \end{tikzpicture}
 \end{center}
 \begin{proof} {\ }
  \begin{align*}
   \sigma_1^{-1}\sigma_2^{-1}\sigma_3^2\sigma_2\sigma_1\sigma_3^{-1}\sigma_1\sigma_2\sigma_3^{-1} &= \sigma_2\sigma_2^{-1}\sigma_1^{-1}\sigma_2^{-1}\sigma_3^2\sigma_2\sigma_1\sigma_3^{-1}\sigma_1\sigma_2\sigma_3^{-1}\sigma_2\sigma_2^{-1}\\
   &= \sigma_2\sigma_1^{-1}\sigma_2^{-1}\sigma_1^{-1}\sigma_3^2\sigma_2\sigma_1\sigma_3^{-1}\sigma_1\sigma_3^{-1}\sigma_2\sigma_3^{-1}\sigma_2^{-1}\\
   &= \sigma_2\sigma_1^{-1}\sigma_2^{-1}\sigma_3\sigma_1^{-1}\sigma_3\sigma_2\sigma_3^{-1}\sigma_1^2\sigma_3^{-1}\sigma_2\sigma_3^{-1}\sigma_2^{-1}\\
   &= \sigma_2\sigma_1^{-1}\sigma_2^{-1}\sigma_3^2\sigma_1^{-1}\sigma_2\sigma_1^{2}\sigma_3^{-2}\sigma_2\sigma_3^{-1}\sigma_2^{-1}\\
   &= \sigma_2\sigma_1^{-1}\sigma_2^{-1}\sigma_3^2\sigma_1^{-1}\sigma_2\sigma_1^2\sigma_3^{-1}\sigma_2\sigma_3^{-1}\sigma_2\sigma_2^{-1}\\
   &= \sigma_2\sigma_1^{-1}\sigma_2^{-1}\sigma_3^2\sigma_1^{-1}\sigma_2\sigma_1^2\sigma_3^{-1}\sigma_2\sigma_3^{-1}
  \end{align*}
 \end{proof}
 We will show each relation diagrammatically.
 \begin{align*}
  \knotsinmath{\braid[braid start={(-1,2)},braid width=0.5cm,braid height=0.5cm,strands=4] {\sigma_1^{-1}\sigma_2^{-1}\sigma_3\sigma_3\sigma_2\sigma_1\sigma_3^{-1}\sigma_1\sigma_2\sigma_3^{-1}}} \quad\sim \knotsinmath{\braid[braid start={(0,3)},braid width=0.5cm,braid height=0.5cm,strands=4]{\sigma_2\sigma_2^{-1}\sigma_1^{-1}\sigma_2^{-1}\sigma_3\sigma_3\sigma_2\sigma_1\sigma_3^{-1}\sigma_1\sigma_2\sigma_3^{-1}\sigma_2\sigma_2^{-1}}}
  \quad\sim \knotsinmath{\braid[braid start={(0,3)},braid width=0.5cm,braid height=0.5cm,strands=4]{\sigma_2\sigma_1^{-1}\sigma_2^{-1}\sigma_1^{-1}\sigma_3\sigma_3\sigma_2\sigma_1\sigma_3^{-1}\sigma_1\sigma_3^{-1}\sigma_2\sigma_3^{-1}\sigma_2^{-1}}}
  \quad\sim \knotsinmath{\braid[braid start={(0,3)},braid width=0.5cm,braid height=0.5cm,strands=4]{\sigma_2\sigma_1^{-1}\sigma_2^{-1}\sigma_3\sigma_1^{-1}\sigma_3\sigma_2\sigma_3^{-1}\sigma_1\sigma_1\sigma_3^{-1}\sigma_2\sigma_3^{-1}\sigma_2^{-1}}}\\\\
  \quad\sim \knotsinmath{\braid[braid start={(0,3)},braid width=0.5cm,braid height=0.5cm,strands=4]{\sigma_2\sigma_1^{-1}\sigma_2^{-1}\sigma_3\sigma_3\sigma_1^{-1}\sigma_2\sigma_1\sigma_1\sigma_3^{-1}\sigma_3^{-1}\sigma_2\sigma_3^{-1}\sigma_2^{-1}}}
  \quad\sim \knotsinmath{\braid[braid start={(0,3)},braid width=0.5cm,braid height=0.5cm,strands=4]{\sigma_2\sigma_1^{-1}\sigma_2^{-1}\sigma_3\sigma_3\sigma_1^{-1}\sigma_2\sigma_1\sigma_1\sigma_3^{-1}\sigma_2\sigma_3^{-1}\sigma_2\sigma_2^{-1}}}
  \quad\sim \knotsinmath{\braid[braid start={(0,2.5)},braid width=0.5cm,braid height=0.5cm,strands=4]{\sigma_2\sigma_1^{-1}\sigma_2^{-1}\sigma_3\sigma_3\sigma_1^{-1}\sigma_2\sigma_1\sigma_1\sigma_3^{-1}\sigma_2\sigma_3^{-1}}}
 \end{align*}
\end{example}

\section{The link between links \textit{\&} braids}\label{chapter:braids-knots}

We will now see how to associate links with braids and vice versa.

This chapter is based on the presentations given in \cite{adams}, \cite{enc}, \cite{jones_quantum}, \cite{birman}, \cite{birman_survey}, \cite{BurdeZieschang}, \cite{cromwell}, \cite{jordan}, \cite{roger_fenn}, \cite{gadgil}, \cite{jones1985}, \cite{jones1987}, \cite{jonesbook}, \cite{jonespoly}, \cite{kauffman1986}, \cite{kauffman1988}, \cite{kawauchi}, \cite{manturov}, \cite{murasagibook}, \cite{rolfsen_braids}, \cite{speicher}, and \cite{wilson}.

However, mainly our discussions will be based on that in Birman's book \cite{birman} and Birman's survey \cite{birman_survey}.

\begin{definition}[braid closure]
 Consider any $n$-braid $\alpha$. We define the \textit{closure}\index{linear operations on $\Bcal_n$!braid closure} of the braid $\alpha$ by connecting the parallel arcs on the right from the top bar to the bottom bar by starting with the far right, and we denote it as $\reallywidehat{\alpha}$.
 \begin{figure}[H] \centering\[
  \knotsinmath{\begin{scope}[xshift=-0.5cm]\draw (0,0.5) node {$\reallywidehat{\alpha}=$};\end{scope}\begin{scope}[yshift=0.1cm,scale=0.75]
  \draw (0.5,0.5) node[scale=1.33] {$\alpha$};
  \draw (0,0) rectangle (1,1);
  \draw (0.25,1.5)--(0.25,1);
  \draw (0.25,0)--(0.25,-0.5);
  \draw (0.75,0)--(0.75,-0.5);
  \draw (0.75,1.5)--(0.75,1);
  \draw[red] (0.75,1.5)--(1.25,1.5)--(1.25,-0.5)--(0.75,-0.5);
  \draw[red] (0.25,1.5)--(0.25,1.75)--(1.5,1.75)--(1.5,-0.75)--(0.25,-0.75)--(0.25,-0.5);
  \draw (0.53,1.25) node {$\cdots$};
  \draw (0.53,-0.25) node {$\cdots$};
 \end{scope}}
 \]\caption{braid closure}\end{figure}
\end{definition}

\begin{example}
 The closure of the braid $\sigma_1\sigma_2\sigma_1$ is shown below and is Reidemeister equivalent to a link (it should be fairly easy to recognise which link).
 
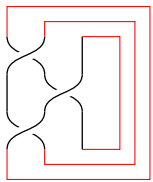
\begin{center}
  \begin{tikzpicture}
   \braid[strands=3,braid start={(0,0)},braid width=0.5cm,braid height=0.5cm] {\sigma_1\sigma_2\sigma_1}
   \draw[red] (1.5,0)--(2,0)--(2,-1.5)--(1.5,-1.5);
   \draw[red] (1,0)--(1,0.2)--(2.2,0.2)--(2.2,-1.7)--(1,-1.7)--(1,-1.5);
   \draw[red] (1/2,0)--(1/2,0.4)--(2.4,0.4)--(2.4,-1.9)--(1/2,-1.9)--(1/2,-1.5);
  \end{tikzpicture}
 \end{center}
\end{example}

The braid closure is a regular diagram of a link. We can assign an orientation by going from the top to the bottom each time we add an arc, i.e.,
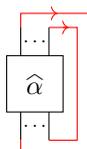
\begin{center}
 \knotsinmath{\begin{scope}[scale=0.75]
  \draw (0.5,0.5) node[scale=1.33] {$\reallywidehat{\alpha}$};
  \draw (0,0) rectangle (1,1);
  \draw (0.25,1)--(0.25,1.5);
  \draw (0.25,0)--(0.25,-0.5);
  \draw (0.75,0)--(0.75,-0.5);
  \draw (0.75,1)--(0.75,1.5);
  \draw[red,->] (1,1.5)--(1.1,1.5);
  \draw[red,->] (0.875,1.75)--(0.9,1.75);
  \draw[red] (0.75,1.5)--(1.25,1.5)--(1.25,-0.5)--(0.75,-0.5);
  \draw[red] (0.25,1.5)--(0.25,1.75)--(1.5,1.75)--(1.5,-0.75)--(0.25,-0.75)--(0.25,-0.5);
  \draw (0.53,1.25) node {$\cdots$};
  \draw (0.53,-0.25) node {$\cdots$};
 \end{scope}}
\end{center}
Thinking of the braid closure as a map from braid types to link types. It is typical to question whether this map is injective or surjective. We can easily show that it is in fact not injective (for instance, $\reallywidehat{\sigma_1\sigma_2\sigma_1^{-1}}=\reallywidehat{\sigma_2}$, but $\sigma_1\sigma_2\sigma_1^{-1}\not\sim\sigma_2$). Fascinatingly enough, it is, however, surjective. This was first shown in 1923 by Alexander \cite{alexander1923}.


\begin{theorem}[Alexander's theorem \cite{alexander1923}]\index{Alexander's theorem}\label{theorem:alexander}
 Every link is ambient isotopic to a braid closure.
\end{theorem}

Before we begin to sketch the proof. We will need to define Seifert circles.

\begin{definition}[Seifert circles]
 Let $D$ be an oriented diagram for the oriented link $L$ and let $D'$ be modified as shown in Figure \ref{figure:seiferts_algorithm}.
 \begin{figure}[H]
    \centering
    \begin{tikzpicture}[scale=0.3]
     \begin{scope}
      \draw[->] (2,0)--(0,2);
      \draw[ultra thick,draw=white,double=black,double distance=0.5pt] (0,0)--(2,2);
      \draw[->] (1.9,1.9)--(2,2);
     \end{scope}
     \begin{scope}[xshift=2cm]
      \draw[|->] (1,1)--(2,1);
     \end{scope}
     \begin{scope}[xshift=5cm]
      \draw[->] (0,0) .. controls (0.65,1) .. (0,2);
      \draw[->] (1.3,0) .. controls (0.65,1) .. (1.3,2);
     \end{scope}
     \begin{scope}[xshift=6.3cm]
      \draw[|->] (2,1)--(1,1);
     \end{scope}
     \begin{scope}[xshift=9.3cm]
      \draw[->] (0,0)--(2,2);
      \draw[ultra thick, draw=white,double=black,double distance=0.5pt] (2,0)--(0,2);
      \draw[->] (0.1,1.9)--(0,2);
     \end{scope}
    \end{tikzpicture}
    \caption{Seifert's algorithm}
    \label{figure:seiferts_algorithm}
 \end{figure}
 Diagram $D'$ is the same as $D$ except in a small neighbourhood of each crossing where the crossing has been removed in the only way compatible with the orientation.
 This $D'$ is just a disjoint union of oriented simple closed curves in $\RR^2$, called \textit{Seifert circles}\index{Seifert circles}.
\end{definition}

\begin{example} {\ }
 \begin{figure}[H]
  \centering
  \begin{tikzpicture}[scale=0.3]
  \begin{scope}[xshift=6cm]
    \mytrefoil
    \draw (0,-2.5) node[scale=3.2] {trefoil knot};
  \end{scope}
  \begin{scope}[yshift=-6cm]
    \mytrefoil
    \draw[->] ( 90:2) -- +(  0:0.2);
    \draw (0,-2.5) node[scale=3.2] {add orientation};
  \end{scope}
  \begin{scope}[yshift=-6cm,xshift=12cm,scale=2]
    \draw[->] ( 90:1) -- +(  0:0.1);
    \draw[->] ( 90:0.2) -- +(0:0.05);
    \draw (0,0) circle(1);
    \draw (0,0) circle(0.2);
    \draw (0,-1.5) node[scale=1.6] {form Seifert circles};
  \end{scope}
  \end{tikzpicture}
  \caption{trefoil knot Seifert circles}
 \end{figure}
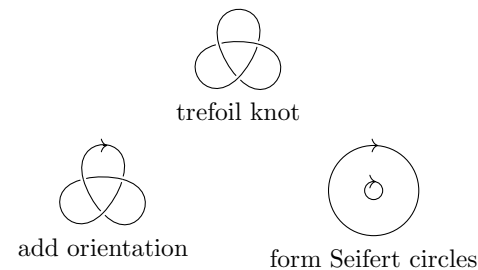
\end{example}

\begin{definition}[braided around]
 We say that a link is \textit{braided around}\index{braided around} a point if the Seifert circles of the link are nested and coherent (i.e., they all orient clockwise around the point).
\end{definition}

We have the following equivalent definition for a link being braided around a point.
\begin{proposition}
 A link is braided around a point if each edge of the link is visible from the point and is oriented clockwise.
\end{proposition}

\begin{remark}
 It is in the author's opinion that checking Seifert circles is easier at times. It also motivates Vogel's proof of Alexander's theorem in 1990 \cite{vogel} which was based on Yamada's proof in 1987 \cite{yamada} using Seifert circles and ``infinity moves''.
 Vogel's proof turns out to terminate quicker than the original proof given by Alexander. The interested reader is directed to \cite{birman_survey}.
\end{remark}

Thus it is easy to see that the only three knots in our knot table from the \hyperlink{appendix:knot_table}{Appendix} that are already braided around a point are $3_1$, $5_1$, and $7_1$.

\begin{proof}[Proof sketch of Theorem \ref{theorem:alexander}]
 We will show that there is a method that can show that any link can be braided around a point.
 Consider a link diagram $L$ and a point $O$ on the plane $P$ of the diagram (the point should not touch the link). If there exists a point on the plane such that the link diagram is braided around the point, then we ``cut'' the diagram along that point to get the braid.
 \begin{figure}[H]
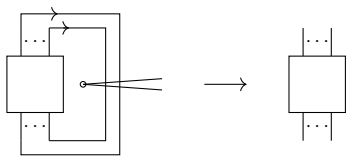

  \centering
  \knotsinmath{\begin{scope}[scale=0.75]
   \draw (0,0) rectangle (1,1);
   \draw (0.25,1)--(0.25,1.5);
   \draw (0.25,0)--(0.25,-0.5);
   \draw (0.75,0)--(0.75,-0.5);
   \draw (0.75,1)--(0.75,1.5);
   \draw[->] (1,1.5)--(1.1,1.5);
   \draw[->] (0.875,1.75)--(0.9,1.75);
   \draw (0.75,1.5)--(1.75,1.5)--(1.75,-0.5)--(0.75,-0.5);
   \draw (0.25,1.5)--(0.25,1.75)--(2,1.75)--(2,-0.75)--(0.25,-0.75)--(0.25,-0.5);
   \draw (0.53,1.25) node {$\cdots$};
   \draw (0.53,-0.25) node {$\cdots$};
   \draw (1.35,0.5) circle (0.05);
   \draw (1.35,0.5)--(2.75,0.6);
   \draw (1.35,0.5)--(2.75,0.4);
  \end{scope}\begin{scope}[xshift=3.75cm,scale=0.75]
   \draw[->] (-1.5,0.5)--(-0.75,0.5);
   \draw (0,0) rectangle (1,1);
   \draw (0.25,1)--(0.25,1.5);
   \draw (0.25,0)--(0.25,-0.5);
   \draw (0.75,0)--(0.75,-0.5);
   \draw (0.75,1)--(0.75,1.5);
   \draw (0.53,1.25) node {$\cdots$};
   \draw (0.53,-0.25) node {$\cdots$};
  \end{scope}}
  \caption{cutting a braided around link}
  \label{figure:cutting_braided_around}
 \end{figure}
 Now consider an arbitrary link diagram. We need to construct this, so that it will be braided around some point. We need to first set a point $O$ on the plane $P$. If there happens to be an edge, say $AB$, that is oriented counterclockwise, then we apply the triangle method by finding some point $C$ on $P$, such that $ABC$ contains $O$, by replacing $AB$ by $AC$ and $BC$. Now $AC$ and $BC$ are both oriented clockwise: the orientation we were looking for. We continue doing this until we get a link diagram that is braided around our point $O$.
 \begin{figure}[H]
  \centering
  \begin{tikzpicture}
   \draw[ultra thick] (0,0)--(-1,1);
   \draw[dashed] (0,0)--(2,1.25)--(-1,1);
   \draw[->,thick] (-0.5,0.5)--(-0.25,0.25);
   \draw (0,0)--(-0.3,-0.3);
   \draw (-1,1)--(-0.7,1.3);
   \draw (0.35,0.75) circle (0.05);
   \draw (-1,1) node[anchor=south east] {$A$};
   \draw (0,0) node[anchor=north west] {$B$};
   \draw (2,1.25) node[anchor=west] {$C$};
   \begin{scope}[xshift=6cm]
    \draw[->] (-3,0.75)--(-2.5,0.75);
    \draw[dashed] (0,0)--(-1,1);
    \draw (0,0)--(2,1.25)--(-1,1);
    \draw[->] (0.8,0.5)--(0.5,0.3125);
    \draw[->] (0.2,1.1)--(0.5,1.125);
    \draw (0,0)--(-0.3,-0.3);
    \draw (-1,1)--(-0.7,1.3);
    \draw (0.35,0.75) circle (0.05);
    \draw (-1,1) node[anchor=south east] {$A$};
    \draw (0,0) node[anchor=north west] {$B$};
    \draw (2,1.25) node[anchor=west] {$C$};
   \end{scope}
  \end{tikzpicture}
  \caption{Alexander's method}
  \label{figure:alexander_method}
 \end{figure}
 We know that this method will eventually terminate given any point $O$, since with every triangle move we make, the number of counterclockwise edges decreases by one, thus eventually giving us a link that is braided around a point $O$.
\end{proof}

\begin{example}
 We have already mentioned above that the trefoil is already braided around. So, we have the following diagram of the left-handed trefoil.
 
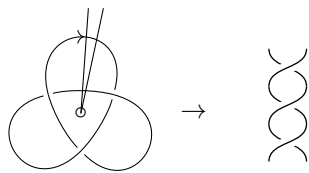
\begin{center}
  \begin{tikzpicture}
   \begin{scope}[scale=0.5]
    \mytrefoil;
    \draw[->] ( 88:2.04) -- +(  0:0.01);
   \end{scope}
   \draw (0,0) node {$\circ$};
   \draw[thin] (0,0)--(0.1,1.4);
   \draw[thin] (0,0)--(0.3,1.4);
   \draw (1.5,0) node {$\to$};
   \begin{scope}[xshift=2cm,yshift=0.86cm]
    \braid[strands=2,braid start={(0,0)},braid width=0.5cm,braid height=0.5cm] {\sigma_1\sigma_1\sigma_1}
   \end{scope}
  \end{tikzpicture}
 \end{center}
 After cutting it, we see that the braid representation is given by $\sigma_1^{-3}$. It is easy to see that the right-handed trefoil's braid representation is $\sigma_1^3$.
 
 Now let us find the braid representation of the following figure-eight knot using Alexander's method.
 
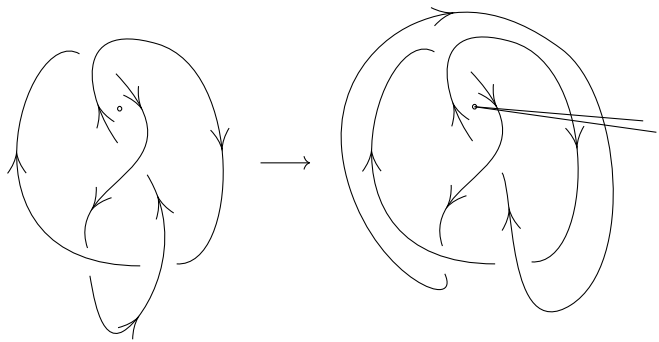
\begin{center}
  \begin{tikzpicture}[x=0.75pt,y=0.75pt,yscale=-1,xscale=1]
   \draw    (120.7,151) .. controls (152.22,153) and (154.84,53) .. (111.08,40) .. controls (67.31,27) and (74.31,67) .. (90.94,90) ;
   \draw    (75.73,143) .. controls (63.73,107) and (137.34,105) .. (90.07,55) ;
   \draw    (71.69,45) .. controls (44.55,29) and (4.29,152) .. (102.32,152) ;
   \draw    (105.82,106) .. controls (126.75,140.33) and (105.73,178) .. (93.73,185) .. controls (81.73,192) and (78.3,163.97) .. (76.94,157) ;
   \draw   (147.14,82.45) .. controls (144.92,86.16) and (143.74,89.75) .. (143.59,93.23) .. controls (142.69,89.86) and (140.76,86.61) .. (137.79,83.47) ;
   \draw   (35.4,105.33) .. controls (37.96,101.84) and (39.48,98.38) .. (39.95,94.93) .. controls (40.53,98.36) and (42.15,101.78) .. (44.81,105.18) ;
   \draw   (101.46,61.27) .. controls (101.08,65.48) and (101.55,69.2) .. (102.87,72.4) .. controls (100.7,69.7) and (97.67,67.5) .. (93.8,65.8) ;
   \draw   (87.85,117.58) .. controls (83.73,118.89) and (80.46,120.81) .. (78.06,123.32) .. controls (79.61,120.2) and (80.3,116.48) .. (80.15,112.16) ;
   \draw   (110.38,128.56) .. controls (111.65,124.43) and (111.94,120.66) .. (111.26,117.25) .. controls (112.93,120.3) and (115.58,122.99) .. (119.22,125.33) ;
   \draw   (80.44,82.58) .. controls (81.67,78.44) and (81.93,74.67) .. (81.21,71.26) .. controls (82.91,74.3) and (85.59,76.96) .. (89.25,79.27) ;
   \draw   (91.13,183.21) .. controls (95.32,182.14) and (98.69,180.42) .. (101.24,178.05) .. controls (99.51,181.07) and (98.6,184.74) .. (98.5,189.06) ;
   \draw   (90.89,72.67) .. controls (90.89,72.02) and (91.36,71.5) .. (91.94,71.5) .. controls (92.53,71.5) and (93,72.02) .. (93,72.67) .. controls (93,73.31) and (92.53,73.83) .. (91.94,73.83) .. controls (91.36,73.83) and (90.89,73.31) .. (90.89,72.67) -- cycle ;
   \draw[->] (163,100)--(188,100);
   \draw    (299.7,150) .. controls (331.22,152) and (333.84,52) .. (290.08,39) .. controls (246.31,26) and (253.31,66) .. (269.94,89) ;
   \draw    (254.73,142) .. controls (242.73,106) and (316.34,104) .. (269.07,54) ;
   \draw    (250.69,44) .. controls (223.55,28) and (183.29,151) .. (281.32,151) ;
   \draw    (284.82,105) .. controls (290.87,114.91) and (287,192.25) .. (319.5,171.75) .. controls (352,151.25) and (345,65.25) .. (315,42.75) .. controls (285,20.25) and (261,20.75) .. (240.5,29.75) .. controls (220,38.75) and (193,66.25) .. (208,112.75) .. controls (223,159.25) and (264.5,175.25) .. (255.94,156) ;
   \draw   (326.14,81.45) .. controls (323.92,85.16) and (322.74,88.75) .. (322.59,92.23) .. controls (321.69,88.86) and (319.76,85.61) .. (316.79,82.47) ;
   \draw   (214.4,104.33) .. controls (216.96,100.84) and (218.48,97.38) .. (218.95,93.93) .. controls (219.53,97.36) and (221.15,100.78) .. (223.81,104.18) ;
   \draw   (280.46,60.27) .. controls (280.08,64.48) and (280.55,68.2) .. (281.87,71.4) .. controls (279.7,68.7) and (276.67,66.5) .. (272.8,64.8) ;
   \draw   (266.85,116.58) .. controls (262.73,117.89) and (259.46,119.81) .. (257.06,122.32) .. controls (258.61,119.2) and (259.3,115.48) .. (259.15,111.16) ;
   \draw   (284.81,134.18) .. controls (287.04,130.48) and (288.25,126.9) .. (288.41,123.42) .. controls (289.29,126.79) and (291.21,130.04) .. (294.17,133.2) ;
   \draw   (259.44,81.58) .. controls (260.67,77.44) and (260.93,73.67) .. (260.21,70.26) .. controls (261.91,73.3) and (264.59,75.96) .. (268.25,78.27) ;
   \draw   (248.96,21.62) .. controls (252.84,23.54) and (256.51,24.44) .. (259.99,24.31) .. controls (256.71,25.47) and (253.62,27.66) .. (250.73,30.87) ;
   \draw   (269.89,71.67) .. controls (269.89,71.02) and (270.36,70.5) .. (270.94,70.5) .. controls (271.53,70.5) and (272,71.02) .. (272,71.67) .. controls (272,72.31) and (271.53,72.83) .. (270.94,72.83) .. controls (270.36,72.83) and (269.89,72.31) .. (269.89,71.67) -- cycle ;
   \draw    (270.94,71.67) -- (362.65,84.57) ;
   \draw    (270.94,71.67) -- (355.99,78.82) ;
  \end{tikzpicture}
 \end{center}
 After cutting the diagram from our center point, we find the braid representation to be $\sigma_1\sigma_2^{-1}\sigma_1\sigma_2^{-1}$,
 \begin{center}
  \begin{tikzpicture}
   \braid[strands=3,braid height=0.5cm,braid width=0.5cm,braid start={(0,0)}] {\sigma_1\sigma_2^{-1}\sigma_1\sigma_2^{-1}}
  \end{tikzpicture}
 \end{center}
\end{example}


If two braids are equivalent, it is obvious that their closures are also equivalent. However, as we have mentioned above, the map of the braid closure is non-injective. The following are more examples showing equivalent braid closures but non-equivalent braids.
\begin{example}
 Let us show that the closure of the $2$-braid $\sigma_1^{-1}$, the $3$-braid $\sigma_1\sigma_2^{-1}$, and the $4$-braids $\sigma_1\sigma_2\sigma_3^{-1}$ and $\sigma_1\sigma_2^{-1}\sigma_3^{-1}$ are all equivalent to the trivial knot.
 \begin{center}
  \begin{tikzpicture}
   \begin{scope}
    \braid[strands=2,braid start={(0,0)},braid width=0.5cm,braid height=0.5cm] {\sigma_1^{-1}}
    \draw (1,0)--(1.2,0)--(1.2,-0.5)--(1,-0.5);
    \draw (0.5,0)--(0.5,0.2)--(1.4,0.2)--(1.4,-0.7)--(0.5,-0.7)--(0.5,-0.5);
   \end{scope}
   
   \begin{scope}[xshift=2cm]
    \braid[strands=3,braid start={(0,0)},braid width=0.5cm,braid height=0.5cm] {\sigma_1\sigma_2^{-1}}
    \draw (1.5,0)--(1.7,0)--(1.7,-1)--(1.5,-1);
    \draw (1,0)--(1,0.2)--(1.9,0.2)--(1.9,-1.2)--(1,-1.2)--(1,-1);
    \draw (0.5,0)--(0.5,0.4)--(2.1,0.4)--(2.1,-1.4)--(0.5,-1.4)--(0.5,-1);
   \end{scope}
   
   \begin{scope}[xshift=5cm]
    \braid[strands=4,braid start={(0,0)},braid width=0.5cm,braid height=0.5cm] {\sigma_1\sigma_2\sigma_3^{-1}}
    \draw (2,0)--(2.2,0)--(2.2,-1.5)--(2,-1.5);
    \draw (1.5,0)--(1.5,0.2)--(2.4,0.2)--(2.4,-1.7)--(1.5,-1.7)--(1.5,-1.5);
    \draw (1,0)--(1,0.4)--(2.6,0.4)--(2.6,-1.9)--(1,-1.9)--(1,-1.5);
    \draw (0.5,0)--(0.5,0.6)--(2.8,0.6)--(2.8,-2.1)--(0.5,-2.1)--(0.5,-1.5);
   \end{scope}
   
   \begin{scope}[xshift=9cm]
    \braid[strands=4,braid start={(0,0)},braid width=0.5cm,braid height=0.5cm] {\sigma_1\sigma_2^{-1}\sigma_3^{-1}}
    \draw (2,0)--(2.2,0)--(2.2,-1.5)--(2,-1.5);
    \draw (1.5,0)--(1.5,0.2)--(2.4,0.2)--(2.4,-1.7)--(1.5,-1.7)--(1.5,-1.5);
    \draw (1,0)--(1,0.4)--(2.6,0.4)--(2.6,-1.9)--(1,-1.9)--(1,-1.5);
    \draw (0.5,0)--(0.5,0.6)--(2.8,0.6)--(2.8,-2.1)--(0.5,-2.1)--(0.5,-1.5);
   \end{scope}
  \end{tikzpicture}
 \end{center}
 They are all equivalent to the trivial knot via a finite number of $\RI$ moves.
\end{example}

We now discuss how to form a bijective braid closure map. In 1935, Markov found three moves on the braid group, that when related by them, their closures are link equivalent \cite{markov}. Four years later, in 1939, Weinberg reduced the three moves that Markov found, to two moves \cite{weinberg}. These moves are called Markov's moves.

\begin{definition}[Markov's moves]
 Let $\Bcal_\infty=\bigcup_{k\geq1}\Bcal_k$, i.e. the infinite union of the braid groups. There are two operations, called \textit{Markov's moves}\index{Markov's moves}, which may be performed in $\Bcal_\infty$.
 \begin{description}
  \item[M1] If $\beta\in\Bcal_n$, then $\beta\xrightarrow[]{\text{M1}}\gamma\beta\gamma^{-1}$, for some $\gamma\in\Bcal_n$. The braid $\gamma\beta\gamma^{-1}$ is said to be the \textit{conjugate} of $\beta$.
  \begin{figure}[H]
   \centering
   \knotsinmath{
   \begin{scope}[yshift=-1.5cm,scale=0.75]
    \draw (0,0) rectangle (1,1);
    \draw (0.25,1.5)--(0.25,1);
    \draw (0.25,0)--(0.25,-0.5);
    \draw (0.75,0)--(0.75,-0.5);
    \draw (0.75,1.5)--(0.75,1);
    \draw (0.53,1.55) node {$\cdots$};
    \draw (0.53,-0.25) node {$\cdots$};
    \draw (0.5,0.5) node[scale=1.33] {$\gamma$};
   \end{scope}
   \begin{scope}[yshift=1.5cm,scale=0.75]
    \draw (0,0) rectangle (1,1);
    \draw (0.25,1.5)--(0.25,1);
    \draw (0.25,0)--(0.25,-0.5);
    \draw (0.75,0)--(0.75,-0.5);
    \draw (0.75,1.5)--(0.75,1);
    \draw (0.53,-0.55) node {$\cdots$};
    \draw (0.53,1.25) node {$\cdots$};
    \draw (0.5,0.5) node[scale=1.33] {$\gamma^{-1}$};
   \end{scope}
   \begin{scope}[scale=0.75]
    \draw (0,0) rectangle (1,1);
    \draw (0.25,1.5)--(0.25,1);
    \draw (0.25,0)--(0.25,-0.5);
    \draw (0.75,0)--(0.75,-0.5);
    \draw (0.75,1.5)--(0.75,1);
    \draw (0.5,0.5) node[scale=1.33] {$\beta$};
   \end{scope}
   \begin{scope}[scale=0.75,xshift=-2.5cm]
    \draw (1.75,0.5) node[scale=1.33] {$\overset{\MI}{\longrightarrow}$};
    \draw (0,0) rectangle (1,1);
    \draw (0.25,1.5)--(0.25,1);
    \draw (0.25,0)--(0.25,-0.5);
    \draw (0.75,0)--(0.75,-0.5);
    \draw (0.75,1.5)--(0.75,1);
    \draw (0.53,-0.25) node {$\cdots$};
    \draw (0.53,1.25) node {$\cdots$};
    \draw (0.5,0.5) node[scale=1.33] {$\beta$};
   \end{scope}
   }
   \caption{Markov $\MI$ move}
  \end{figure}
  It is easy to see that taking the closure after an M1 move gives the same link, since $\gamma$ and $\gamma^{-1}$ both \textit{cancel} out.
  \item[M2] If $\beta\in\Bcal_n$, then $\beta\xrightarrow[]\MII\incl_n(\beta)\sigma_n$ or $\beta\xrightarrow[]\MII\incl_n(\beta)\sigma_n^{-1}$, for the generator $\sigma_n\in\Bcal_{n+1}$.
  We call the braid $\incl_n(\beta)\sigma_n^{\pm1}$ the \textit{stabilisation}\index{stabilisation} of $\beta$.
  \begin{figure}[H]
   \centering
   \knotsinmath{\begin{scope}[scale=0.75,yshift=0.25cm]
    \draw (0,0) rectangle (1,1);
    \draw (0.25,1.5)--(0.25,1);
    \draw (0.25,0)--(0.25,-0.5);
    \draw (0.75,0)--(0.75,-0.5);
    \draw (0.75,1.5)--(0.75,1);
    \draw (1.25,1.5)--(1.25,-0.5);
    \draw (0.53,-0.25) node {$\cdots$};
    \draw (0.53,1.25) node {$\cdots$};
    \draw (0.5,0.5) node[scale=1.33] {$A$};
    \braid[strands=3,braid width=0.5cm,braid height=0.5cm,braid start={(-0.25,-0.5)},line width=0.25pt] {\sigma_2}
   \end{scope}
   \begin{scope}[xshift=-2cm,scale=0.75,yshift=0cm]
    \draw (1.75,0.5) node[scale=1.33] {$\overset{\MII}{\longrightarrow}$};
    \draw (0,0) rectangle (1,1);
    \draw (0.25,1.5)--(0.25,1);
    \draw (0.25,0)--(0.25,-0.5);
    \draw (0.75,0)--(0.75,-0.5);
    \draw (0.75,1.5)--(0.75,1);
    \draw (0.53,-0.25) node {$\cdots$};
    \draw (0.53,1.25) node {$\cdots$};
    \draw (0.5,0.5) node[scale=1.33] {$\beta$};
   \end{scope}
   \begin{scope}[xshift=2.5cm,scale=0.75,yshift=0.25cm]
    \draw (-1,0.25) node[scale=1.33] {or};
    \draw (0,0) rectangle (1,1);
    \draw (0.25,1.5)--(0.25,1);
    \draw (0.25,0)--(0.25,-0.5);
    \draw (0.75,0)--(0.75,-0.5);
    \draw (0.75,1.5)--(0.75,1);
    \draw (1.25,1.5)--(1.25,-0.5);
    \draw (0.53,-0.25) node {$\cdots$};
    \draw (0.53,1.25) node {$\cdots$};
    \draw (0.5,0.5) node[scale=1.33] {$\beta$};
    \braid[strands=3,braid width=0.5cm,braid height=0.5cm,braid start={(-0.25,-0.5)},line width=0.25pt] {\sigma_2^{-1}}
   \end{scope}
   }
   \caption{Markov $\MII$ move}
  \end{figure}
  It is easy to see that taking the closure after an $\MII$ move produces the same link since the move mimics an $\RI$ move.
 \end{description}
 We say that the two braids $\alpha,\beta\in\Bcal_\infty$ are \textit{Markov equivalent}\index{Markov equivalent} if you can transform one into the other by a finite sequence of the above moves (and their inverses). We denote it as $\alpha\sim_M\beta$.
\end{definition}

For clarity, the closure after an $\MI$ move will look like,
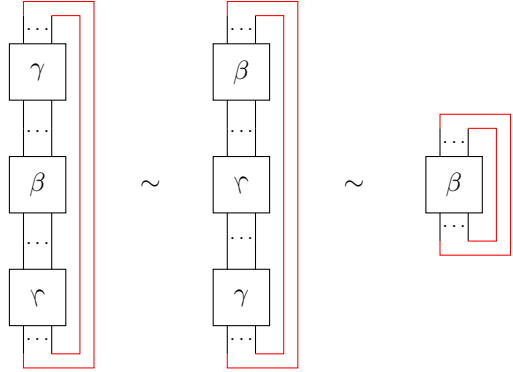
\begin{center}
 \knotsinmath{\begin{scope}[yshift=1.5cm,scale=0.75]
    \draw (0,0) rectangle (1,1);
    \draw (0.25,1.5)--(0.25,1);
    \draw (0.25,0)--(0.25,-0.5);
    \draw (0.75,0)--(0.75,-0.5);
    \draw (0.75,1.5)--(0.75,1);
    \draw (0.53,1.25) node {$\cdots$};
    \draw (0.53,-0.55) node {$\cdots$};
    \draw (0.5,0.5) node[scale=1.33] {$\gamma$};
   \end{scope}
   \begin{scope}[scale=0.75]
    \draw (0,0) rectangle (1,1);
    \draw (0.25,1.5)--(0.25,1);
    \draw (0.25,0)--(0.25,-0.5);
    \draw (0.75,0)--(0.75,-0.5);
    \draw (0.75,1.5)--(0.75,1);
    \draw (0.53,-0.55) node {$\cdots$};
    \draw (0.5,0.5) node[scale=1.33] {$\beta$};
    \draw[red] (0.75,3.5)--(1.25,3.5)--(1.25,-2.5)--(0.75,-2.5);
    \draw[red] (0.25,3.5)--(0.25,3.75)--(1.5,3.75)--(1.5,-2.75)--(0.25,-2.75)--(0.25,-2.5);
    \draw (2.5,0.5) node[anchor=center,scale=1.33] {$\sim$};
   \end{scope}
   \begin{scope}[xscale=-1,yshift=-1.5cm,xshift=-0.75cm,scale=0.75]
    \draw (0,0) rectangle (1,1);
    \draw (0.25,1.5)--(0.25,1);
    \draw (0.25,0)--(0.25,-0.5);
    \draw (0.75,0)--(0.75,-0.5);
    \draw (0.75,1.5)--(0.75,1);
    \draw (0.53,-0.25) node {$\cdots$};
    \draw (0.5,0.5) node[scale=1.33] {$\gamma$};
   \end{scope}}\qquad
   \knotsinmath{\begin{scope}[yshift=-1.5cm,scale=0.75]
    \draw (0,0) rectangle (1,1);
    \draw (0.25,1.5)--(0.25,1);
    \draw (0.25,0)--(0.25,-0.5);
    \draw (0.75,0)--(0.75,-0.5);
    \draw (0.75,1.5)--(0.75,1);
    \draw (0.53,1.55) node {$\cdots$};
    \draw (0.53,-0.25) node {$\cdots$};
    \draw (0.5,0.5) node[scale=1.33] {$\gamma$};
   \end{scope}
   \begin{scope}[yshift=1.5cm,scale=0.75]
    \draw (0,0) rectangle (1,1);
    \draw (0.25,1.5)--(0.25,1);
    \draw (0.25,0)--(0.25,-0.5);
    \draw (0.75,0)--(0.75,-0.5);
    \draw (0.75,1.5)--(0.75,1);
    \draw (0.53,-0.55) node {$\cdots$};
    \draw (0.53,1.25) node {$\cdots$};
    \draw (0.5,0.5) node[scale=1.33] {$\beta$};
   \end{scope}
   \begin{scope}[xscale=-1,xshift=-0.75cm,scale=0.75]
    \draw (0,0) rectangle (1,1);
    \draw (0.25,1.5)--(0.25,1);
    \draw (0.25,0)--(0.25,-0.5);
    \draw (0.75,0)--(0.75,-0.5);
    \draw (0.75,1.5)--(0.75,1);
    \draw (0.5,0.5) node[scale=1.33] {$\gamma$};
   \end{scope}
   \begin{scope}[scale=0.75]
    \draw[red] (0.75,3.5)--(1.25,3.5)--(1.25,-2.5)--(0.75,-2.5);
    \draw[red] (0.25,3.5)--(0.25,3.75)--(1.5,3.75)--(1.5,-2.75)--(0.25,-2.75)--(0.25,-2.5);
    \draw (2.5,0.5) node[scale=1.33] {$\sim$};
   \end{scope}}
   \qquad\knotsinmath{\begin{scope}[scale=0.75]
  \draw (0.5,0.5) node[scale=1.33] {$\beta$};
  \draw (0,0) rectangle (1,1);
  \draw (0.25,1.5)--(0.25,1);
  \draw (0.25,0)--(0.25,-0.5);
  \draw (0.75,0)--(0.75,-0.5);
  \draw (0.75,1.5)--(0.75,1);
  \draw[red] (0.75,1.5)--(1.25,1.5)--(1.25,-0.5)--(0.75,-0.5);
  \draw[red] (0.25,1.5)--(0.25,1.75)--(1.5,1.75)--(1.5,-0.75)--(0.25,-0.75)--(0.25,-0.5);
  \draw (0.53,1.25) node {$\cdots$};
  \draw (0.53,-0.25) node {$\cdots$};
 \end{scope}}
\end{center}
Similarly, the closure after an $\MII$ move will look like,
\begin{center}
 \knotsinmath{
  \begin{scope}[scale=0.75]
   \draw (0,0) rectangle (1,1);
   \draw (0.25,1.5)--(0.25,1);
   \draw (0.25,0)--(0.25,-0.5);
   \draw (0.75,0)--(0.75,-0.5);
   \draw (0.75,1.5)--(0.75,1);
   \draw (1.25,1.5)--(1.25,-0.5);
   \draw (0.53,-0.25) node {$\cdots$};
   \draw (0.53,1.25) node {$\cdots$};
   \draw (0.5,0.5) node[scale=1.33] {$\beta$};
   \braid[strands=3,braid width=0.5cm,braid height=0.5cm,braid start={(-0.25,-0.5)},line width=0.25pt] {\sigma_2}
   \draw[red] (1.25,1.5)--(1.5,1.5)--(1.5,-1)--(1.25,-1);
   \draw[red] (0.75,1.5)--(0.75,1.75)--(1.75,1.75)--(1.75,-1.25)--(0.75,-1.25)--(0.75,-1);
   \draw[red] (0.25,1.5)--(0.25,2)--(2,2)--(2,-1.5)--(0.25,-1.5)--(0.25,-1);
   \draw (2.75,0.5) node[scale=1.33] {$\sim$};
  \end{scope}
 }\qquad\knotsinmath{\begin{scope}[scale=0.75]
  \draw (0.5,0.5) node[scale=1.33] {$\beta$};
  \draw (0,0) rectangle (1,1);
  \draw (0.25,1.5)--(0.25,1);
  \draw (0.25,0)--(0.25,-0.5);
  \draw (0.75,0)--(0.75,-0.5);
  \draw (0.75,1.5)--(0.75,1);
  \draw[red] (0.75,1.5)--(1.25,1.5)--(1.25,-0.5)--(0.75,-0.5);
  \draw[red] (0.25,1.5)--(0.25,1.75)--(1.5,1.75)--(1.5,-0.75)--(0.25,-0.75)--(0.25,-0.5);
  \draw (0.53,1.25) node {$\cdots$};
  \draw (0.53,-0.25) node {$\cdots$};
 \end{scope}}
 \qquad\qquad\qquad
  \knotsinmath{
  \begin{scope}[scale=0.75]
   \draw (0,0) rectangle (1,1);
   \draw (0.25,1.5)--(0.25,1);
   \draw (0.25,0)--(0.25,-0.5);
   \draw (0.75,0)--(0.75,-0.5);
   \draw (0.75,1.5)--(0.75,1);
   \draw (1.25,1.5)--(1.25,-0.5);
   \draw (0.53,-0.25) node {$\cdots$};
   \draw (0.53,1.25) node {$\cdots$};
   \draw (0.5,0.5) node[scale=1.33] {$\beta$};
   \braid[strands=3,braid width=0.5cm,braid height=0.5cm,braid start={(-0.25,-0.5)},line width=0.25pt] {\sigma_2^{-1}}
   \draw[red] (1.25,1.5)--(1.5,1.5)--(1.5,-1)--(1.25,-1);
   \draw[red] (0.75,1.5)--(0.75,1.75)--(1.75,1.75)--(1.75,-1.25)--(0.75,-1.25)--(0.75,-1);
   \draw[red] (0.25,1.5)--(0.25,2)--(2,2)--(2,-1.5)--(0.25,-1.5)--(0.25,-1);
   \draw (2.75,0.5) node[scale=1.33] {$\sim$};
  \end{scope}
 }\qquad\knotsinmath{\begin{scope}[scale=0.75]
  \draw (0.5,0.5) node[scale=1.33] {$\beta$};
  \draw (0,0) rectangle (1,1);
  \draw (0.25,1.5)--(0.25,1);
  \draw (0.25,0)--(0.25,-0.5);
  \draw (0.75,0)--(0.75,-0.5);
  \draw (0.75,1.5)--(0.75,1);
  \draw[red] (0.75,1.5)--(1.25,1.5)--(1.25,-0.5)--(0.75,-0.5);
  \draw[red] (0.25,1.5)--(0.25,1.75)--(1.5,1.75)--(1.5,-0.75)--(0.25,-0.75)--(0.25,-0.5);
  \draw (0.53,1.25) node {$\cdots$};
  \draw (0.53,-0.25) node {$\cdots$};
 \end{scope}}
\end{center}

\begin{theorem}[Markov's theorem \cite{markov}]\index{Markov's theorem}\label{theorem:markov}
 Let $K_1,K_2$ be two oriented links which can be formed from the braids $\beta_1,\beta_2$, respectively. Then $K_1\sim{K_2}\Leftrightarrow\beta_1\sim_M\beta_2$.
\end{theorem}

For a proof of this; see Birman's survey \cite[pp.~18--26]{birman_survey} or Birman's original paper \cite{birman}.

\subsection{Motivation}\label{section:motivation}

Since we have reached the end of the first part, we will summarise and thus motivate the second part with all of what we learned here.

Alexander's theorem \ref{theorem:alexander} says that all links have a braid representation, and Markov's theorem \ref{theorem:markov} says that two links are equivalent if and only if their braid representations are Markov equivalent (i.e., they can be transformed by $\MI$ and $\MII$ moves). So, then we call the \textit{Markov braid type} the infinite union of the braid group $\Bcal_\infty$ that is quotiented by the two Markov moves (Markov equivalence). And so we have,
\[\{\text{link type}\}\simeq\{\text{Markov braid type}\}\]
So if we prove that a function is invariant under the Markov moves, then by Markov's theorem we have that it is a link invariant.
And this will be the strategy for the upcoming original construction of the Jones polynomial (See Chapter \ref{chapter:jones}).

More precisely, for $a\in\Bcal_n$ and any $i<n$, we have $\reallywidehat{a}=\reallywidehat{\sigma_ia\sigma_i^{-1}}$ and $\reallywidehat{a}=\reallywidehat{\incl_n(a)\sigma_n^{\pm1}}$. So our strategy of constructing the Jones polynomial will be to define a tracial function on the Temperley-Lieb algebra. Where the Temperley-Lieb algebra is seen as a quotient algebra of the group algebra $\CC[\Bcal_n]$. There exists other similar constructions, however for the purposes of this essay we will use the Temperley-Lieb algebras since it has a very comprehensible diagrammatic definition. To see other similar constructions; a very nice read is \cite{jones1987}. The following chapter will briefly touch on the work that Jones did in 1983 \cite{jones1983}. We just want to see the correlation he made between the von Neumann algebras and constructing a link invariant.

\section{von Neumann algebras}\label{chapter:analysis}


The aim of this chapter is to give the essential knowledge in order to understand the construction and encapsulation of the Jones polynomial. So this chapter will give a brief overview on operator algebras. 

The interested reader is directed to \cite{jones1983}.


This chapter is based on the presentations given in \cite{conway_book}, \cite{coxeter}, \cite{jones1983}, \cite{jonesbook}, \cite{kadison_ringrose}, \cite{murphy}, \cite{penneys_subfactors}, \cite{penneys}, \cite{speicher}, and \cite{wenzl}.

\subsection{Background}

We shall first review some background we will need. The discussion in this section is mainly based on that in the author's exposition on analysis \textit{\&} its foundations \cite{analfoundations}.

A \textit{norm}\index{norm} over a vector space $V$ over $\CC$ is a map $\|\cdot\|:V\to\RR$ that satisfies
\begin{enumerate}[(i)]
 \item $\|v\|\geq0$, with equality if and only if $v=0$ for all $v\in{V}$
 \item $\|\lambda{v}\|=|\lambda|\|v\|$ for all $v\in{V}$ and $\lambda\in\CC$
 \item $\|v+w\|\leq\|v\|+\|w\|$ for all $v,w\in{V}$
\end{enumerate}
A \textit{normed vector space}\index{normed vector space} is a vector space $V$ over $\CC$ with a norm $\|\cdot\|$.
We say that a normed vector space is a \textit{Banach space}\index{Banach space} if it is complete (i.e., any Cauchy sequence converges).
We call a map $\varphi:V\to\CC$ a \textit{linear functional}\index{linear functional} if it satisfies $\varphi(\lambda{v}+\mu{w})=\lambda\varphi(v)+\mu\varphi(w)$, for all $\lambda,\mu\in\CC$ and $v,w\in{V}$. We say the linear functional is \textit{bounded}\index{bounded!linear functional} if $\|\varphi\|=\sup_{\|x\|\leq1}|\varphi(x)|<\infty$. The \textit{dual}\index{dual} of the normed vector space $V$, is the normed vector space of all bounded linear functionals $\varphi:V\to\CC$, and we denote it as $V^*$.

Now, we call a map $T:V\to{V}$ \textit{linear}\index{linear} if it satisfies $T(\lambda{v}+\mu{w})=\lambda{T(v)}+\mu{T(w)}$ for all $v,w\in\CC$ and $v,w\in{V}$. We call the linear map \textit{bounded}\index{bounded!linear map} if $\|T\|=\sup_{\|v\|\leq1}\|T(v)\|<\infty$. We let $\BB(V)$\index{$\BB(V)$} be the normed vector space of all bounded linear maps $T:V\to{V}$.

A \textit{Hilbert space}\index{Hilbert space} is essentially a complete inner product on a vector space $V$ over $\CC$, i.e.,
\[\langle\cdot,\cdot\rangle\colon{V\times{V}\to\CC}\]
satisfying the following for all $x,y,z\in{V}$ and $\lambda\in\CC$,
\begin{enumerate}[(i)]
 \item $\langle{x,y}\rangle=\overline{\langle{y,x}\rangle}$
 \item $\langle\lambda{x}+y,z\rangle=\lambda\langle{x,z}\rangle+\langle{y,z}\rangle$
 \item $\langle{x,x}\rangle\geq0$ with equality if and only if $x=0$
\end{enumerate}
Notice that if we let the norm $\|x\|=\sqrt{\langle{x,x}\rangle}$ for any $x\in{V}$, then we can see $|\langle{x,y}\rangle|\leq\|x\|\|y\|$ for any $x,y\in{V}$. This inequality is known as the \textit{Cauchy-Shwarz inequality}\index{Cauchy-Shwarz inequality} (it is easy to show that this is indeed satisfied). So, in fact, an inner product induces a norm.

\textbf{Notation.}
Note that the notation used for the inner product $\langle\cdot,\cdot\rangle$ should not be confused with the Kauffman bracket $\langle\cdot\rangle$ from the previous chapters. In fact, all notations used in this chapter are self-contained, unless stated otherwise.

The \textit{adjoint}\index{adjoint} of an operator $T:\Hcal\to\Hcal^\prime$ is the unique linear map $T^*:\Hcal\to\Hcal$ such that for all $x\in\Hcal^\prime$ and $y\in\Hcal$ we have $\langle{T^*x,y}\rangle=\langle{x,Ty}\rangle$.

An \textit{algebra}\index{algebra} is a vector space $V$ over $\CC$ if there is a multiplication operation $V\times{V}\to{V}$, given by $(u,v)\mapsto{uv}$ such that for all $u,v,w\in{V}$ and $\lambda,\mu\in\CC$, we have $u(v+w)=uv+uw$, $(u+v)w=uw+vw$ and $(\lambda{u})(\mu{v})=(\lambda\mu)(uv)$. We call an algebra $V$ over $\CC$ a \textit{normed algebra}\index{normed algebra} if for all $u,v\in{V}$ we have $\|uv\|\leq\|u\|\|v\|$. We call a normed algebra a \textit{Banach algebra}\index{Banach algebra} if it is a Banach space. If $V$ is a normed algebra with a multiplicative identity, then we call $V$ \textit{unital}\index{unital}.
Given $V$ is a an algebra, and a vector subspace $W$, then we call $B$ a \textit{subalgebra}\index{subalgebra} if $v,w\in{W}$ implies $vw\in{W}$.

And if $V$ is an algebra that is commutative as a ring, then it is a \textit{commutative algebra}\index{commutative algebra}.

We shall assume that our algebras are unital.

We define the \textit{spectrum}\index{spectrum} of $x$, denoted as $\sigma(x)$, in a Banach algebra $V$ as the set of complex numbers $\lambda$ such that $x-\lambda{I}$ is not invertible, i.e., $\sigma(x)=\{\lambda\in\CC:x-\lambda{I}\text{ is not invertible}\}$.

\begin{proposition}
 The spectrum of an element in a Banach algebra $A$ is a non-empty compact subset of $\CC$.
\end{proposition}
\begin{proof}[Proof sketch]
 First start with the \textit{Murray-von Neumann criterion}\index{Murray-von Neumann criterion}, which states that for $x$ in a Banach algebra $A$, if $\|x\|<1$ then $1-x$ is invertible.
 It then follows that the spectrum of $x$ is in fact closed.
 Then it follows from Heine-Borel, $\sigma(x)$ is compact since it is closed and bounded.
 
 To show that it is non-empty, we need to define the resolvent set of $x$ $\rho(x)$. This is the complement of the spectrum $\rho(x)=\CC\setminus\sigma(x)$. And we also define the resolvent map of $x$ for $\lambda\in\rho(x)$ as $R_\lambda(x)=(\lambda-x)^{-1}$.
 
 Then it is easy to see that as $\lambda\to\infty$ we have $\|R_\lambda(x)\|\to0$.
 It is also easy to see that for $\lambda,\mu\in\rho(x)$, we have $R_\lambda-R_\mu=(\mu-\lambda)R_\lambda{R_\mu}$.
 
 There is an analogue to Louiville's theorem for complex functions, so it then follows that the spectrum is non-empty.
\end{proof}

\subsection{\texorpdfstring{${C^*}$}{C*}-algebras}

\begin{definition}[$^*$-algebra]
 A $^*$\textit{-algebra}\index{$^*$-algebra} is an algebra $A$ with an involution map $^*\colon{x}\mapsto{x^*}$ which for $x,y\in{A}$ and $\lambda\in\CC$ satisfies
 \begin{enumerate}[(i)]
  \item $(xy)^*=y^*x^*$
  \item $(x+y)^*=x^*+y^*$
  \item $\lambda^*=\bar{\lambda}$
  \item $x^{**}=x$
 \end{enumerate}
 We say a $^*$-algebra is an \textit{abstract} $C^*$\textit{-algebra}\index{$C^*$-algebra!abstract} if it is a Banach $^*$-algebra where the norm satisfies the $C^*$\textit{-identity}\index{$C^*$-identity} $\|x\|^2=\|x^*x\|$.
\end{definition}

The completion of a Banach algebra with an involution is an abstract $C^*$-algebra.

\begin{example}
 The complex numbers $\CC$ forms a $C^*$-algebra, by defining the involution $\lambda^*=\overline{\lambda}$ and where the norm is $\|\lambda\|=|\lambda|$.
\end{example}

We say that a $^*$-algebra is a \textit{concrete $C^*$-algebra}\index{$C^*$-algebra!concrete} if it is a closed subalgebra of $\BB(\Hcal)$.

\begin{example}
 So the set of bounded operators on a Hilbert space $\Hcal$, $\BB(\Hcal)$, is one example of a concrete $C^*$-algebra.
\end{example}

Given a $C^*$-algebra $A$, if $A$ has a unit $1\in{A}$, i.e., an identity element. Then it can easily be shown $1^*=1$. This can be seen by $1=(1^*)^*=(1^*1)^*=(1^*(1^*)^*)=1^*1=1^*$.

\begin{proposition}
 Given a $C^*$-algebra $A$, the involution is continuous.
\end{proposition}
\begin{proof}
 Let $x_n\in{A}$ such that $x_n\to{x\in{A}}$. So, $\|x_n^*-x^*\|=\|(x_n-x)^*\|$.
 Now we want to show $\|(x_n-x)^*\|=\|x_n-x\|$. So $\|x_n-x\|^2=\|(x_n-x)^*(x_n-x)\|\leq\|(x_n-x)^*\|\|x_n-x\|$, so if $x_n-x\neq0$ then $\|x_n-x\|\leq\|(x_n-x)^*\|$ (otherwise if $x_n-x=0$, then we are done). So, $\|(x_n-x)^*\|\leq\|((x_n-x)^*)^*\|=\|x_n-x\|$. So, then, $\|(x_n-x)^*\|=\|x_n-x\|$.\\
 So, we now have $\|x_n^*-x^*\|=\|x_n-x\|\to0$ since $x_n\to{x}$. Therefore, the involution is continuous.
\end{proof}

\begin{proposition}
 If $A$ is a $C^*$-algebra, then $\sigma(x^*)=\sigma(x)$ for all $x\in{A}$.
\end{proposition}

Given a $C^*$-algebra $A$, then we have
\begin{itemize}
 \item $x\in{A}$ is \textit{normal}\index{normal} if $xx^*=x^*x$.
 \item $y\in{A}$ is \textit{Hermitian}\index{Hermitian} if $x=x^*$.
 \item $u\in{A}$ is \textit{unitary}\index{unitary} if $uu^*=u^*u=1$.
 \item $p\in{A}$ is a \textit{projection}\index{projection} if $p=p^*=p^2$.
\end{itemize}

The following result is left as an exercise to the reader.
\begin{proposition}
 If $x$ is an element of a $C^*$-algebra $A$. Then $x^*x$ is Hermitian.
\end{proposition}

\begin{definition}
 Given $C^*$-algebras $A$ and $B$, a \textit{$^*$-homomorphism}\index{*-homomorphism} is an algebra homomorphism $\alpha:A\to{B}$ such that $\alpha(x^*)=\alpha(x)^*$ for all $x\in{A}$. A $^*$-homomorphism, $\alpha$, between unital $C^*$-algebras is itself \textit{unital}\index{$C^*$-algebra!unital} if $\alpha(1)=1$.
\end{definition}

\begin{theorem}
 Any $^*$-homomorphism is continuous, with closed image, and is bounded with norm at most one. Any injective $^*$-homomorphism is an isometry.
\end{theorem}

So, the norm on a $C^*$-algebra is determined by the rest of the algebraic structure.

A $^*$-homomorphism is an \textit{isomorphism} of $C^*$-algebras if it is bijective.
\begin{proposition}
 The inverse of a bijective $^*$-homomorphism is also a $^*$-homomorphism.
\end{proposition}


\subsection{Three topologies on \texorpdfstring{$\BB(\Hcal)$}{B(H)}}

\begin{definition}[WOT, SOT, \& norm topology] {\ }
 \begin{description}
  \item[norm] The \textit{norm topology}\index{norm topology} on $\BB(\Hcal)$ is given by the norm $\|T\|=\sup_{\|v\|\leq1}\|T(v)\|$.
  \item[WOT]
   The \textit{weak operator topology}\index{weak operator topology (WOT)} on $\BB(\Hcal)$ is given by \[\{b\in\BB(\Hcal):|\langle{ax,y}\rangle-\langle{bx,y}\rangle|<\epsilon \text{ where } a\in\BB(\Hcal), x,y\in\Hcal\text{ and }\epsilon>0\}\]
  \item[SOT]
   The \textit{strong operator topology}\index{strong operator topology (SOT)} (SOT) on $\BB(\Hcal)$ is given by
   \[\{b\in\BB(\Hcal):\|ax-bx\|<\epsilon\text{ where }a\in\BB(\Hcal), x\in\Hcal\text{ and }\epsilon>0\}\]
 \end{description}
\end{definition}

So, we then have that the norm topology contains the strong and weak operator topologies, i.e., $\text{WOT}\prec\text{SOT}\prec\text{the norm topology}$. And so for any $A\subset\BB(\Hcal)$, we have $\overline{A}^{\|\cdot\|}\subset\overline{A}^\text{SOT}\subset\overline{A}^\text{WOT}$.

\subsection{von Neumann algebras}

\begin{definition}[commutant]
 We define the \textit{commutant}\index{commutant} of $A\subseteq\BB(\Hcal)$ as
 \[A^\prime=\{y\in\BB(\Hcal) : xy=yx\text{ for all }x\in{A}\},\]
 and the \textit{bicommutant}\index{bicommutant} by $A^{\prime\prime}=(A^\prime)^\prime$.
\end{definition}

The following theorem is known as the \textit{bicommutant theorem}\index{bicommutant theorem}.
\begin{theorem}\label{theorem:vonNeumann}
 Let $M\subset\BB(\Hcal)$ be a unital $^*$-subalgebra. Then the following three are equivalent.
 \begin{enumerate}[(i)]
  \item $M$ is closed in WOT
  \item $M$ is closed in SOT
  \item $M=M^{\prime\prime}=(M^\prime)^\prime$
 \end{enumerate}
\end{theorem}

\begin{definition}[von Neumann algebra]
 A unital $^*$-subalgebra $M\subseteq\BB(\Hcal)$ is a \textit{von Neumann algebra}\index{von Neumann algebra} if it satisfies any of the properties in the above theorem.
\end{definition}

It is clear that any von Neumann algebra is a $C^*$-algebra.

\begin{example}
 The set of bounded linear maps on the Hilbert space $\BB(\Hcal)$ is a von Neumann algebra.
 We also have $M_n(\CC)$ as a von Neumann algebra.
\end{example}
We note that von Neumann algebras come in pairs $M,M^\prime$.

%

\subsection{Factors}

\begin{definition}[factor]
 We call a von Neumann algebra $M$ a \textit{factor}\index{factor} if $M\cap{M^\prime}=\CC\cdot1_{\BB(\Hcal)}$ (i.e., the center is trivial).
\end{definition}

\begin{example}
 The set of bounded linear maps on the Hilbert space $\BB(\Hcal)$ is thus also a factor.
\end{example}

The only commutative factor is $\CC$.

Given $M$ is a von Neumann algebra, then
\begin{itemize}
 \item let $e,f\in{M}$ be two projections, then $e$ and $f$ are \textit{equivalent}\index{projections!equivalent} if there exists $u\in{M}$ such that $u^*u=e$ and $uu^*=f$. We write $e\sim{f}$ if they are equivalent.
 \item let $e,f\in{M}$ be two projections, then if there exists $u\in{M}$ such that $u^*u=e$ and $u^*uf=u^*u$. We write this as $e\preceq{f}$.
 \item we say a projection $p\in{M}$ such that $p\neq0$ is \textit{minimal}\index{projections!minimal} if for all projections $a\in{M}$ we have that if $a\leq{p}$ then either $a=0$ or $a=p$.
 \item we say a projection $e\in{M}$ is \textit{finite}\index{projections!finite} if we have that for all projections $f\in{M}$, if $e\sim{f}\leq{e}$ then $f=e$.
\end{itemize}

The types of factors $M\subseteq\BB(\Hcal)$:
\begin{itemize}
 \item \textit{type} I \textit{factor}\index{factor!type I} if there exists a minimal projection in $M$.
 \item we call the factor $M\subseteq\BB(\Hcal)$ a \textit{type} II \textit{factor}\index{factor!type II} if there exists a finite projection in $M$ and there does not exist a minimal projection in $M$.
 \begin{itemize}
  \item given $M\subseteq\BB(\Hcal)$ is a type II factor and $1_{\BB(\Hcal)}$ is finite, then $M$ is a \textit{type} II$_1$ \textit{factor}\index{factor!type II$_1$}.
  \item if $1_{\BB(\Hcal)}$ is infinite and there exists finite projections, then it is said to be a \textit{type} II$_\infty$ \textit{factor}\index{factor!type II$_\infty$}.
 \end{itemize}
 \item finally, we have that the factor $M\subseteq\BB(\Hcal)$ is a \textit{type} III \textit{factor}\index{factor!type III} if there exists no finite projections in $M$.
\end{itemize}

We will only discuss type II$_1$ factors in this paper. So to simplify, a factor is of type II$_1$ if all projections are finite (i.e., $1_{\BB(\Hcal)}$ is finite) and there does not exist a minimal projection.

\subsection{Trace}



We call a linear functional $f$ on a von Neumann algebra $M$,
\begin{itemize}
 \item \textit{positive}\index{linear functional!positive} if $f(x^*x)\geq0$ for all $x\in{M}$, 
 \item \textit{faithful}\index{linear functional!faithful} if when we have $f(x^*x)=0$, then $x=0$,
 \item \textit{state}\index{linear functional!state} if $f(1)=1$, and
 \item \textit{tracial}\index{linear functional!tracial} if $f(xy)=f(yx)$ for all $x,y\in{M}$.
\end{itemize}


Type II$_1$ factors $M$ admit a unique faithful tracial state $\tr\colon{M\to\CC}$, so
\begin{enumerate}[(i)]
 \item $\tr(xy)=\tr(yx)$ for all $x,y\in{M}$
 \item $\tr(1)=1$
 \item $\tr(xx^*)>0$ for all $x\in{M}$ where $x\neq0$.
\end{enumerate}

We call a linear functional $\tr\colon{M\to\CC}$ a \textit{trace}\index{trace} if $\tr(xy)=\tr(yx)$ for all $x,y\in{M}$.

Given $M$ is a finite factor, now we can say that two projections $e,f\in{M}$ are equivalent if $\tr(e)=\tr(f)$.


%


\subsection{The index theorem}


We define a Hilbert space $L^2(M,\tr_M)$ to be the completion of $M$ with respect to the inner product $\langle{a,b}\rangle=\tr_M(b^*a)$.

\begin{definition}[Jones index]\index{Jones index}
 Given $N$ is a subfactor of $M$, we define the \textit{Jones index}\index{Jones index} of $N$ as $[M:N]=\dim_N{L^2(M,\tr_M)}$.
\end{definition}

\begin{theorem}[index theorem]\index{Jones index theorem}\label{theorem:index}
 Given $N$ is a II$_1$ subfactor of $M$, then we have
 \[[M:N]\in\left\{4\cos^2(\pi/(n+2)):n\geq1\right\}\cup[4,\infty).\]
\end{theorem}

Jones proved this using the basic construction explained in the following section.

\subsection{Jones' basic construction}

In this section we will discuss Jones' basic construction for type II$_1$ subfactors. This discussion is majorly based on that of Jones' paper \cite{jones1983}, Penneys' lecture notes \cite{penneys}, and Speicher's lecture notes \cite{speicher}.

\begin{theorem}[GNS - Gelfand-Naimark-Segal]\index{GNS}
 Every abstract $C^*$-algebra is isometrically $^*$-isomorphic to a concrete $C^*$-algebra.
\end{theorem}

\begin{theorem}[GNS construction]\index{GNS construction}
 Given a unital $C^*$-algebra $A$ and a state $f\colon{A\to\CC}$. There exists a Hilbert space $\Hcal^\prime$, a $^*$-homomorphism $\pi\colon{A\to\BB(\Hcal^\prime)}$, and a unit vector $h\in\Hcal^\prime$ such that for all $a\in{A}$ we have $f(a)=\langle\pi(a)h,h\rangle$.
\end{theorem}

We have the GNS representation of a factor $M$ on $L^2(M,\tr_M)$, which is given by $\langle{{x},{y}}\rangle=\tr_M(y^*x)$.

Since the entire chapter was filled with definition and results, it would be reasonable to summarise what it is exactly we need in order to discuss Jones' basic construction.

So a von Neumann algebra $A$ is a unital $^*$-subalgebra of $\BB(\Hcal)$ such that $A=A^{\prime\prime}$; see Theorem \ref{theorem:vonNeumann}. A factor $M$ is a von Neumann algebra such that the center is trivial (i.e., $M\cap{M^\prime}=\CC\cdot1_{\BB(\Hcal)}$). A factor is of type II$_1$ if it is of infinite dimension and it admits a unique tracial state. We then define a Hilbert space $L^2(M,\tr_M)$ to be the completion of $M$ with respect to the inner product $\langle{a,b}\rangle=\tr_M(b^*a)$. So, $M$ acts on $L^2(M,\tr_M)$ by both left and right multiplication. So, for a projection $p\in{M}$, we have the representation $L^2(M,\tr_M)p$. We then have that for II$_1$ factors $N\subset{M}$, we have $[M:N]=\dim_N(L^2(M,\tr_M))$.

We have the conjugate-linear unitary functional $J\colon{L^2(M,\tr_M)\to{L^2(M,\tr_M)}}$ given by ${x}\mapsto{{x^*}}$. Then we have $JM^{\prime}J=M$.


Now consider a subfactor $N\subset{M}$, and the Hilbert space $L^2(M,\tr_M)$. We see that are already five II$_1$ factors, namely: $M$, $M^\prime$, $N$, $N^\prime$, and $JNJ$.

\begin{proposition}[{\cite[pp.~7,8]{jones1983}}]\label{proposition:conditional_expectation_factors}
 Given finite von Neumann algebras $N\subset{M}$ with a faithful normal trace $\tr_N$, with $1_N=1_M$. Then there exists a normal linear map $E_N\colon{M\to{N}}$ defined by $\tr_N(E_N(x)y)=\tr_N(xy)$ for $x\in{M}$ and $y\in{N}$, and satisfies the following properties
 \begin{enumerate}[(i)]
  \item $E_N(axb)=aE_N(x)b$ for $x\in{M}$ and $a,b\in{N}$
  \item $E_N(x^*)=E_N(x)^*$ for all $x\in{M}$
  \item $E_N(x^*)E_N(x)\leq{E_N(x^*x)}$ for all $x\in{M}$
  \item $E_N(x^*x)=0$ implies $x=0$ for all $x\in{M}$
 \end{enumerate}
\end{proposition}

The above proposition essentially says that $E_N$ is a trace-preserving, faithful and positive linear map. We call $E_N$ the \textit{conditional expectation}\index{conditional expectation} of $M$ to $N$.

\begin{proposition}[{\cite[Proposition~5.4]{speicher}}]\label{proposition:Jonesprojections}
 Given II$_1$ factors $N\subset{M}$, there exists the conditional expectation $E_N\colon{M\to{N}}$, and an orthogonal projection $e_N\colon{L^2(M,\tr_M)\to{L^2(N,\tr_N)}}$ such that
 \begin{enumerate}[(i)]
  \item it is uniquely determined by $e_Nxe_N=E_N(x)e_N$ for $x\in{M}$
  \item $xe_N=e_Nx$ if and only if $x\in{N}$
  \item $N=M\cap\{e_N\}^\prime$
 \end{enumerate}
\end{proposition}


\textbf{Basic construction.}\index{basic construction}
 Given II$_1$ factors $N\subset{M}$.
 Let $e_N\colon{L^2(M,\tr_M)}\to{L^2(N,\tr_N)}$ be an orthogonal projection. Then we have $N=M\cap\{e_N\}^\prime$. So, a direct corollary to this is that we have the von Neumann algebra $\langle{M,e_N}\rangle$ on $L^2(M,\tr_M)$ which is defined by $(M\cup\{e_N\})^{\prime\prime}=JN^\prime{J}$. Thus we also have $[M:N]=\tr(e_N)^{-1}=[\langle{M,e_N}\rangle:M]$, where the trace is defined on $\langle{M,e_N}\rangle$.


And so analogously to how for a von Neumann algebra $M$ there exists a von Neumann algebra $M^\prime$;
The basic construction tells us that for a subfactor $N\subset{M}$ there exists subfactors $JNJ$ and $JN^\prime{J}$ on $L^2(M,\tr_M)$.



Thus, by iterating the basic construction, we find a tower $M_i$ of II$_1$ factors with $M_0=N$, $M_1=M$ and $M_{i+1}=\langle{M_i,e_{M_{i-1}}}\rangle$.
\[N\subset{M}\overset{e_{M}}{\subset}M_2\overset{e_{M_2}}{\subset}\cdots\overset{e_{M_n}}{\subset}\langle{M_n,e_{M_n}}\rangle=M_{n+1}\subset\cdots\]
For better readability, we let $e_i=e_{M_i}$, and call the sequence of $(e_i)$ the \textit{Jones projections}\index{Jones projections}.

\begin{theorem}
 The Jones projections $(e_i)$ satisfy the following properties.
 \begin{enumerate}[(i)]
  \item $\left[M:N\right]e_ie_{i\pm1}e_i=e_i$
  \item $e_ie_j=e_je_i \text{ if }|i-j|\geq2$
  \item $e_i^2=e_i^*=e_i$
  \item $\left[M:N\right]\tr_{n+1}(xe_{n})=tr_{n}(x) \text{ where }x\text{ is a word on }1,e_1,e_2,\ldots,e_{n-1}$
 \end{enumerate}
\end{theorem}

These relations are similar to those of the braid group and the Markov moves. This is the similarity that Didier Hatt-Arnold pointed out to Vaughan Jones in 1982 (see Introduction \ref{chapter:intro}).
The reader should instantly notice that the Jones projections $(e_i)$ satisfy the properties of the Temperley-Lieb algebras from Definition \ref{definition:TLn} given $d^{2}=[M:N]$. For instance, for part (i) we have,
\begin{align*}
 [M:N]e_ne_{n-1}e_n &= [M:N]E_n(e_{n-1})e_n\\
 &= [M:N]\tr_{n+1}(e_{n-1})e_n\\
 &= [M:N][M:N]^{-1}e_n = e_n
\end{align*}

\section{The Temperley-Lieb algebras}\label{chapter:tl}
The Temperley-Lieb algebra was first introduced in 1971 by Lieb and Temperley \cite{lieb}.
In this chapter, we will see the relation between the von Neumann algebra that we constructed in the previous chapter (via the Jones' projections) and the Temperley-Lieb algebra. This will help our understanding since it can be explained in a diagrammatic setting.

Our discussion will mainly depend on that found in \cite{penneys}.






\subsection{Quantum integers}

\begin{definition}
 A \textit{qunatum integer}\index{quantum integer} $[n]_q$ for $n\in\ZZ$ and $q\in\{Q\cup-Q\}$ where $Q:=\{e^{i\theta}:\theta\in(0,\pi/2)\}\cup[1,\infty)$, is defined as $\dfrac{q^n-q^{-n}}{q-q^{-1}}$.
\end{definition}

A quantum integer is essentially a deformation of an integer by $q$.

\begin{example}
 So the quantum integer $[1]_q$ is $1$, while $[2]_q=q+q^{-1}$.
\end{example}

\begin{proposition}\label{proposition:doubling_quantum_integer}
 We can double a quantum integer, $[2]_q[n]_q=[n+1]_q+[n-1]_q$.
\end{proposition}
\begin{proof}
 \begin{align*}
  [2]_q[n]_q &= \frac{(q+q^{-1})(q^n-q^{-n})}{q-q^{-1}}\\
  &= \frac{q^{n+1}-q^{1-n}+q^{n-1}-q^{-n-1}}{q-q^{-1}}\\
  &= \frac{q^{n+1}-q^{-(n+1)}}{q-q^{-1}}+\frac{q^{n-1}-q^{-(n-1)}}{q-q^{-1}}
 \end{align*}
 And so the result then follows.
\end{proof}

\begin{proposition}
 Given $m\geq{a}$, we have $[m-a]_q=[m]_q[a+1]_q-[m+1]_q[a]_q$.
\end{proposition}
\begin{proof} {\ }
 \begin{align*}
  [m-a]_q &= \frac{q^{m-a}-q^{a-m}}{q-q^{-1}}\\
  &= \frac{(q^{m-a}-q^{a-m})(q-q^{-1})}{(q-q^{-1})(q-q^{-1})}\\
  &= \frac{q^{m-a+1}-q^{m-a-1}-q^{a-m+1}+q^{a-m-1}}{q^2+q^{-2}}\\
  &= \frac{q^{m+a+1}-q^{m+a+1}+q^{-m-a-1}-q^{-a-m-1}+q^{m-a+1}-q^{m-a-1}-q^{a-m+1}+q^{a-m-1}}{q^2+q^{-2}}\\
  &= \frac{q^{m+(a+1)}-q^{m-(a+1)}-q^{(a+1)-m}+q^{-m-(a+1)}}{q^2+q^{-2}}\\&\qquad\quad-\frac{q^{a+(m+1)}-q^{-a+(m+1)}-q^{a-(m+1)}+q^{-a-(m+1)}}{q^2+q^{-2}}\\
  &= \left(\frac{q^m-q^{-m}}{q-q^{-1}}\right)\frac{q^{a+1}-q^{-(a+1)}}{q-q^{-1}}-\left(\frac{q^{m+1}-q^{-(m+1)}}{q-q^{-1}}\right)\frac{q^a-q^{-a}}{q-q^{-1}}\\
  &= [m]_q[a+1]_q-[m+1]_q[a]_q
 \end{align*}
\end{proof}

\subsection{Defining the Temperley-Lieb algebra \texorpdfstring{${TL_n}$}{TLn}}

\begin{definition}\label{definition:TLn}
 For $d\in\{Q\cup-Q\}$, $TL_{n}(d)$\index{Temperley-Lieb algebra!$TL_n$} is the associative unital (i.e., has an identity) $*$-algebra with generators $e_1,e_2,\ldots,e_{n-1}$ and relations
 \begin{enumerate}[(i)]
  \item $e_i^2=e_i=e_i^*$
  \item $e_ie_j=e_je_i$ for $|i-j|\geq2$
  \item $e_ie_{i\pm1}e_i=d^{-2}{e_i}$.
 \end{enumerate}
\end{definition}

So we can reduce each word in the algebra with the relations.

\begin{example}
 So, $TL_1$ is generated by $1$, $TL_2$ is generated by $1$ and $e_1$, $TL_3$ is generated by $1$, $e_1$, $e_2$, $e_1e_2$ and $e_2e_1$, and so on.
\end{example}

\begin{definition}
 The $n$-th \textit{Catalan number}\index{Catalan number} $c_n$ is given by $\dfrac{1}{n+1}\left(\!\!\!\begin{array}{c}2n\\n\end{array}\!\!\!\right)$.
\end{definition}

In our above example, we saw $\dim(TL_1)=1=c_1$, $\dim(TL_2)=2=c_2$ and $\dim(TL_3)=5=c_3$. So, we have the following result.

\begin{proposition}
 $TL_n$ is finite dimensional.
\end{proposition}
\begin{proof}[Sketch of Proof]
 Essentially, since this is a finitely generated group. We have any word in the algebra can be written as
 \[(e_{i_1}e_{i_1-1}\cdots{e_{k_1}})(e_{i_2}e_{i_2-1}\cdots{e_{k_2}})\cdots(e_{i_p}e_{i_p-1}\cdots{e_{k_p}})\]
 for $1\leq{i_1}<i_2<\cdots<i_p$ and for $1\leq{k_1}<k_2<\cdots<k_p$.
 Thus the dimension is less than or equal to the $n^\text{th}$ Catalan number. \[\dim(TL_n)\leq{c_n}\]
\end{proof}

It is easy to realise that this is essentially the type II$_1$ algebra formed from the Jones projections $(e_i)$ in the previous chapter.

\subsection{Diagrammatic definition of the Temperley-Lieb algebra \texorpdfstring{$gTL_n$}{gTLn}}

We now discuss the diagrammatic definition of the Temperley-Lieb algebras. This will help us in the construction of the Jones polynomial in the next chapter.
\begin{definition}\index{Temperley--Lieb algebra!diagrammatic $gTL_n$} 
 For $d\in\{Q\cup-Q\}$, we call the algebra $gTL_{n}(d)$ the \textit{Temperley-Lieb algebra}, whose basis is the set of non-intersecting string diagrams on a rectangle with $n$ boundary points on the top and bottom.
 The multiplication is defined by concatenation $gTL_n\times{gTL_n}\to{gTL_n}$ (analogously to braid concatenation)  and is shown below.
 
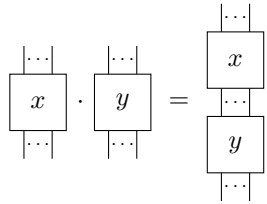
\begin{center}
  \begin{tikzpicture}[scale=0.75]
   \begin{scope}[yshift=-0.75cm,xshift=-3.5cm]
    \draw (0,0) rectangle (1,1);
    \draw (0.25,1.5)--(0.25,1);
    \draw (0.25,0)--(0.25,-0.5);
    \draw (0.75,0)--(0.75,-0.5);
    \draw (0.75,1.5)--(0.75,1);
    \draw (0.53,1.25) node {$\cdots$};
    \draw (0.53,-0.25) node {$\cdots$};
    \draw (0.5,0.5) node[scale=1.33] {$x$};
   \end{scope}
   \begin{scope}[yshift=-0.75cm,xshift=-2cm]
    \draw (-0.25,0.5) node[scale=1.33] {$\cdot$};
    \draw (0,0) rectangle (1,1);
    \draw (0.25,1.5)--(0.25,1);
    \draw (0.25,0)--(0.25,-0.5);
    \draw (0.75,0)--(0.75,-0.5);
    \draw (0.75,1.5)--(0.75,1);
    \draw (0.53,1.25) node {$\cdots$};
    \draw (0.53,-0.25) node {$\cdots$};
    \draw (0.5,0.5) node[scale=1.33] {$y$};
   \end{scope}
   \draw (-0.5,-0.25) node[scale=1.33] {$=$};
   \draw (0,0) rectangle (1,1);
   \draw (0.25,1.5)--(0.25,1);
   \draw (0.25,0)--(0.25,-0.5);
   \draw (0.75,0)--(0.75,-0.5);
   \draw (0.75,1.5)--(0.75,1);
   \draw (0.53,1.25) node {$\cdots$};
   \draw (0.53,-0.25) node {$\cdots$};
   \draw (0,-0.5) rectangle (1,-1.5);
   \draw (0.25,-1.5)--(0.25,-2);
   \draw (0.75,-1.5)--(0.75,-2);
   \draw (0.53,-1.75) node {$\cdots$};
   \draw (0.5,0.5) node[scale=1.33] {$x$};
   \draw (0.5,-1) node[scale=1.33] {$y$};
  \end{tikzpicture}
 \end{center}
 If a closed loop appears in the diagram, we can remove it by multiplying the diagram by the weight of the closed loop (which is $d$). 
\end{definition}

We have that two elements in $gTL_n$ are equivalent if the strings are ambient isotopic. So by holding all the boundary points in place, we can move the strings around, given that they do not intersect or break the definition.

\begin{example}
 An example of an element in $gTL_5$ would be the following
 
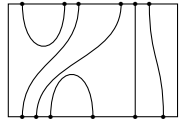
\begin{center}
  \begin{tikzpicture}[scale=0.75]
   \draw (0,0) rectangle (3,2);
   \draw (0.25,2) [upup=0.75];
   \draw (0.25,0) [dnup=1];
   \draw (0.75,0) [dndn=0.75];
   \draw (0.5,0) [dnup=1.5];
   \draw (2.75,0) [dnup=-0.25];
   \draw (2.25,0) [dnup=0];
  \end{tikzpicture}
 \end{center}
\end{example}

\begin{example}
 The elements of $gTL_1$ is just a diagram with a single string. The elements of $gTL_2$ is the following
 
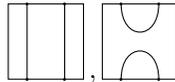
\begin{center}
  \begin{tikzpicture}[scale=0.5]
   \draw (0,0) rectangle (2,2);
   \draw (0.5,0) [dnup=0];
   \draw (1.5,0) [dnup=0];
   \draw (2.25,0) node[scale=2] {$,$};
   \begin{scope}[xshift=2.5cm]
    \draw (0,0) rectangle (2,2);
    \draw (0.5,0) [dndn=1];
    \draw (0.5,2) [upup=1];
   \end{scope}
  \end{tikzpicture}
 \end{center}
 We also have the following elements of $gTL_3$
 \begin{center}
  \begin{tikzpicture}[scale=0.5]
   \draw (0,0) rectangle (2,2);
   \draw (0.5,0) [dnup=0];
   \draw (1,0) [dnup=0];
   \draw (1.5,0) [dnup=0];
   \draw (2.25,0) node[scale=2] {$,$};
   \begin{scope}[xshift=2.5cm]
    \draw (0,0) rectangle (2,2);
    \draw (0.25,0) [dndn=0.75];
    \draw (0.25,2) [upup=0.75];
    \draw (1.5,0) [dnup=0];
    \draw (2.25,0) node[scale=2] {$,$};
   \end{scope}
   \begin{scope}[xshift=5cm]
    \draw (0,0) rectangle (2,2);
    \draw (1,0) [dndn=0.75];
    \draw (1,2) [upup=0.75];
    \draw (0.5,0) [dnup=0];
    \draw (2.25,0) node[scale=2] {$,$};
   \end{scope}
   \begin{scope}[xshift=7.5cm]
    \draw (0,0) rectangle (2,2);
    \draw (0.25,0) [dndn=0.75];
    \draw (1,2) [upup=0.75];
    \draw (1.75,0) [dnup=-1.5];
    \draw (2.25,0) node[scale=2] {$,$};
   \end{scope}
   \begin{scope}[xshift=10cm]
    \draw (0,0) rectangle (2,2);
    \draw (0.25,2) [upup=0.75];
    \draw (1,0) [dndn=0.75];
    \draw (0.25,0) [dnup=1.5];
   \end{scope}
  \end{tikzpicture}
 \end{center}
\end{example}

So, we have the generators for $gTL_n$,
\begin{figure}[H]
 \centering
 \begin{tikzpicture}[scale=0.75]
  \begin{scope}[xshift=-3.75cm]
   \draw (0,0) rectangle (3,2);
   \draw (0.25,0) [dnup=0];
   \draw (2.75,0) [dnup=0];
   \draw (1.5,1) node {$\cdots$};
   \draw (1.5,-0.35) node[scale=1.33] {$1_n$};
   \draw (3.375,0) node[scale=1.33] {$,$};
  \end{scope}
  \begin{scope}
   \draw (0,0) rectangle (3,2);
   \draw (0.25,2) [upup=0.75];
   \draw (0.25,0) [dndn=0.75];
   \draw (1.25,0) [dnup=0];
   \draw (1.875,1) node {$\ldots$};
   \draw (2.5,0) [dnup=0];
   \draw (2.75,0) [dnup=0];
   \draw (1.5,-0.35) node[scale=1.33] {$E_1$};
   \draw (3.375,0) node[scale=1.33] {$,$};
  \end{scope}
  \begin{scope}[xshift=3.75cm]
   \draw (0,0) rectangle (3,2);
   \draw (0.25,0) [dnup=0];
   \draw (0.5,2) [upup=0.75];
   \draw (0.5,0) [dndn=0.75];
   \draw (1.5,0) [dnup=0];
   \draw (2,1) node {$\ldots$};
   \draw (2.5,0) [dnup=0];
   \draw (2.75,0) [dnup=0];
   \draw (1.5,-0.35) node[scale=1.33] {$E_2$};
   \draw (3.625,0) node[scale=1.33] {$,$};
  \end{scope}
  \begin{scope}[xshift=8cm]
   \draw (0,0) rectangle (3,2);
   \draw (0.25,0) [dnup=0];
   \draw (0.5,0) [dnup=0];
   \draw (0.75,1) node {$\ldots$};
   \draw (1,0) [dnup=0];
   \draw (1.25,2) [upup=0.75];
   \draw (1.25,0) [dndn=0.75];
   \draw (2.25,0) [dnup=0];
   \draw (2.5,1) node {$\ldots$};
   \draw (2.75,0) [dnup=0];
   \draw (1.5,-0.35) node[scale=1.33] {$E_i$};
   \draw (3.625,0) node[scale=1.33] {$,$};
  \end{scope}
  \begin{scope}[xshift=12.25cm]
   \draw (0,0) rectangle (3,2);
   \draw (0.25,0) [dnup=0];
   \draw (0.5,0) [dnup=0];
   \draw (1.125,1) node {$\ldots$};
   \draw (1.75,0) [dnup=0];
   \draw (2,0) [dndn=0.75];
   \draw (2,2) [upup=0.75];
   \draw (1.5,-0.35) node[scale=1.33] {$E_{n-1}$};
  \end{scope}
 \end{tikzpicture}
 \caption{the $gTL_n$ generators}
\end{figure}

To see the relations, we begin by multiplying $E_i^2$,
\begin{figure}[H]
 \centering
 \begin{tikzpicture}[scale=0.75]
  \begin{scope}
   \draw (0,0) rectangle (3,2);
   \draw (0.25,0) [dnup=0];
   \draw (0.5,0) [dnup=0];
   \draw (0.75,1) node {$\ldots$};
   \draw (1,0) [dnup=0];
   \draw (1.25,2) [upup=0.75];
   \draw (1.25,0) [dndn=0.75];
   \draw (2.25,0) [dnup=0];
   \draw (2.5,1) node {$\ldots$};
   \draw (2.75,0) [dnup=0];
   \draw (3.375,1) node[scale=1.33] {$\cdot$};
  \end{scope}
  \begin{scope}[xshift=3.75cm]
   \draw (0,0) rectangle (3,2);
   \draw (0.25,0) [dnup=0];
   \draw (0.5,0) [dnup=0];
   \draw (0.75,1) node {$\ldots$};
   \draw (1,0) [dnup=0];
   \draw (1.25,2) [upup=0.75];
   \draw (1.25,0) [dndn=0.75];
   \draw (2.25,0) [dnup=0];
   \draw (2.5,1) node {$\ldots$};
   \draw (2.75,0) [dnup=0];
  \end{scope}
  \begin{scope}[xshift=8cm]
   \draw (-0.625,1) node[scale=1.33] {$=$};
   \draw (0,3) rectangle (3,-1);
   \draw (0.25,1) [dnup=0];
   \draw (0.5,1) [dnup=0];
   \draw (0.75,2) node {$\ldots$};
   \draw (1,1) [dnup=0];
   \draw (1.25,3) [upup=0.75];
   \draw (1.25,1) [dndn=0.75];
   \draw (2.25,1) [dnup=0];
   \draw (2.5,2) node {$\ldots$};
   \draw (2.75,1) [dnup=0];
   \draw[dashed] (0,1)--(3,1);
   \draw (0.25,-1) [dnup=0];
   \draw (0.5,-1) [dnup=0];
   \draw (0.75,0) node {$\ldots$};
   \draw (1,-1) [dnup=0];
   \draw (1.25,1) [upup=0.75];
   \draw (1.25,-1) [dndn=0.75];
   \draw (2.25,-1) [dnup=0];
   \draw (2.5,0) node {$\ldots$};
   \draw (2.75,-1) [dnup=0];
  \end{scope}
  \begin{scope}[xshift=12.3cm]
   \draw (-0.8,1) node[scale=1.33] {$=$};
   \draw (-0.3,1) node[scale=1.33] {$d$};
   \draw (0,0) rectangle (3,2);
   \draw (0.25,0) [dnup=0];
   \draw (0.5,0) [dnup=0];
   \draw (0.75,1) node {$\ldots$};
   \draw (1,0) [dnup=0];
   \draw (1.25,2) [upup=0.75];
   \draw (1.25,0) [dndn=0.75];
   \draw (2.25,0) [dnup=0];
   \draw (2.5,1) node {$\ldots$};
   \draw (2.75,0) [dnup=0];
  \end{scope}
 \end{tikzpicture}
 \caption{$E_i^2=dE_i$}
\end{figure}
Therefore, we have $E_i^2=d{E_i}$.

We also have $E_iE_j=E_jE_i$ where $|i-j|>1$,
\begin{center}
 \begin{tikzpicture}[scale=0.75]
  \begin{scope}
   \draw (0,0) rectangle (4,2);
   \draw (0.25,0) [dnup=0];
   \draw (0.5,1) node {$\ldots$};
   \draw (0.75,0) [dnup=0];
   \draw (1,2) [upup=0.75];
   \draw (1,0) [dndn=0.75];
   \draw (2,0) [dnup=0];
   \draw (2.75,0) [dnup=0];
   \draw (3,0) [dnup=0];
   \draw (3.375,1) node {$\ldots$};
   \draw (3.75,0) [dnup=0];
   \draw (4.125,1) node[scale=1.33] {$\cdot$};
  \end{scope}
  \begin{scope}[xshift=4.25cm]
   \draw (0,0) rectangle (4,2);
   \draw (0.25,0) [dnup=0];
   \draw (0.5,1) node {$\ldots$};
   \draw (0.75,0) [dnup=0];
   \draw (1,0) [dnup=0];
   \draw (1.75,0) [dnup=0];
   \draw (2,2) [upup=0.75];
   \draw (2,0) [dndn=0.75];
   \draw (3,0) [dnup=0];
   \draw (3.375,1) node {$\ldots$};
   \draw (3.75,0) [dnup=0];
   \draw (4.3,1) node[scale=1.33] {$=$};
  \end{scope}
  \begin{scope}[xshift=8.8cm]
   \draw (0,3) rectangle (4,-1);
   \draw (0.25,1) [dnup=0];
   \draw (0.5,2) node {$\ldots$};
   \draw (0.75,1) [dnup=0];
   \draw (1,3) [upup=0.75];
   \draw (1,1) [dndn=0.75];
   \draw (2,1) [dnup=0];
   \draw (2.75,1) [dnup=0];
   \draw (3,1) [dnup=0];
   \draw (3.375,2) node {$\ldots$};
   \draw (3.75,1) [dnup=0];
   \draw[dashed] (0,1)--(4,1);
   \draw (0.25,-1) [dnup=0];
   \draw (0.5,0) node {$\ldots$};
   \draw (0.75,-1) [dnup=0];
   \draw (1,-1) [dnup=0];
   \draw (1.75,-1) [dnup=0];
   \draw (2,1) [upup=0.75];
   \draw (2,-1) [dndn=0.75];
   \draw (3,-1) [dnup=0];
   \draw (3.375,0) node {$\ldots$};
   \draw (3.75,-1) [dnup=0];
  \end{scope}
  \begin{scope}[xshift=13.4cm]
   \draw (-0.3,1) node[scale=1.33] {$=$};
   \draw (0,3) rectangle (4,-1);
   \draw (0.25,1) [dnup=0];
   \draw (0.5,2) node {$\ldots$};
   \draw (0.75,1) [dnup=0];
   \draw (1,1) [dnup=0];
   \draw (1.75,1) [dnup=0];
   \draw (2,3) [upup=0.75];
   \draw (2,1) [dndn=0.75];
   \draw (3,1) [dnup=0];
   \draw (3.375,2) node {$\ldots$};
   \draw (3.75,1) [dnup=0];
   \draw[dashed] (0,1)--(4,1);
   \draw (0.25,-1) [dnup=0];
   \draw (0.5,0) node {$\ldots$};
   \draw (0.75,-1) [dnup=0];
   \draw (1,1) [upup=0.75];
   \draw (1,-1) [dndn=0.75];
   \draw (2,-1) [dnup=0];
   \draw (2.75,-1) [dnup=0];
   \draw (3,-1) [dnup=0];
   \draw (3.375,0) node {$\ldots$};
   \draw (3.75,-1) [dnup=0];
  \end{scope}
 \end{tikzpicture}
\end{center}

\begin{figure}[H]
 \centering
 \begin{tikzpicture}[scale=0.75]
  \begin{scope}[xshift=8.8cm]
   \draw (-0.25,1) node[scale=1.33] {$=$};
   \draw (0,0) rectangle (4,2);
   \draw (0.25,0) [dnup=0];
   \draw (0.5,1) node {$\ldots$};
   \draw (0.75,0) [dnup=0];
   \draw (1,2) [upup=0.75];
   \draw (1,0) [dndn=0.75];
   \draw (2,0) [dnup=0];
   \draw (2.75,0) [dnup=0];
   \draw (3,0) [dnup=0];
   \draw (3.375,1) node {$\ldots$};
   \draw (3.75,0) [dnup=0];
   \draw (4.125,1) node[scale=1.5] {$\cdot$};
   \draw (4.25,0) rectangle (8.25,2);
   \draw (4.5,0) [dnup=0];
   \draw (4.75,1) node {$\ldots$};
   \draw (5,0) [dnup=0];
   \draw (5.25,0) [dnup=0];
   \draw (6,0) [dnup=0];
   \draw (6.25,2) [upup=0.75];
   \draw (6.25,0) [dndn=0.75];
   \draw (7.25,0) [dnup=0];
   \draw (7.625,1) node {$\ldots$};
   \draw (8,0) [dnup=0];
  \end{scope}
 \end{tikzpicture}
 \caption{$E_iE_j=E_jE_i$ where $|i-j|>1$}
\end{figure}

Our final relation is $E_iE_{i+1}E_i=E_i$.
\begin{figure}[H]
 \centering
 \begin{tikzpicture}[scale=0.75]
  \begin{scope}
   \draw (0,0) rectangle (4,2);
   \draw (0.25,0) [dnup=0];
   \draw (0.5,1) node {$\ldots$};
   \draw (0.75,0) [dnup=0];
   \draw (1,2) [upup=0.75];
   \draw (1,0) [dndn=0.75];
   \draw (2,0) [dnup=0];
   \draw (2.75,0) [dnup=0];
   \draw (3,0) [dnup=0];
   \draw (3.375,1) node {$\ldots$};
   \draw (3.75,0) [dnup=0];
   \draw (4.125,1) node[scale=1.33] {$\cdot$};
  \end{scope}
  \begin{scope}[xshift=4.25cm]
   \draw (0,0) rectangle (4,2);
   \draw (0.25,0) [dnup=0];
   \draw (3.375,1) node {$\ldots$};
   \draw (0.75,0) [dnup=0];
   \draw (1,0) [dnup=0];
   \draw (1.25,0) [dnup=0];
   \draw (1.75,2) [upup=0.75];
   \draw (1.75,0) [dndn=0.75];
   \draw (0.5,1) node {$\ldots$};
   \draw (2.75,0) [dnup=0];
   \draw (3,0) [dnup=0];
   \draw (3.75,0) [dnup=0];
   \draw (4.125,1) node[scale=1.33] {$\cdot$};
  \end{scope}
  \begin{scope}[xshift=8.5cm]
   \draw (0,0) rectangle (4,2);
   \draw (0.25,0) [dnup=0];
   \draw (0.5,1) node {$\ldots$};
   \draw (0.75,0) [dnup=0];
   \draw (1,2) [upup=0.75];
   \draw (1,0) [dndn=0.75];
   \draw (2,0) [dnup=0];
   \draw (2.75,0) [dnup=0];
   \draw (3,0) [dnup=0];
   \draw (3.375,1) node {$\ldots$};
   \draw (3.75,0) [dnup=0];
  \end{scope}
  \begin{scope}[yshift=-3cm,xshift=3cm]
   \draw (-0.25,-1) node[scale=1.33] {$=\;$};
   \draw (0,2) rectangle (4,-4);
   \draw (0.25,0) [dnup=0];
   \draw (0.5,1) node {$\ldots$};
   \draw (0.75,0) [dnup=0];
   \draw (1,2) [upup=0.75];
   \draw (1,0) [dndn=0.75];
   \draw (2.5,0) [dnup=0];
   \draw (2.75,0) [dnup=0];
   \draw (3,0) [dnup=0];
   \draw (3.375,1) node {$\ldots$};
   \draw (3.75,0) [dnup=0];
   \draw[dashed] (0,0)--(4,0);
   \draw (0.25,-2) [dnup=0];
   \draw (0.5,-1) node {$\ldots$};
   \draw (0.75,-2) [dnup=0];
   \draw (1,-2) [dnup=0];
   \draw (1.75,0) [upup=0.75];
   \draw (1.75,-2) [dndn=0.75];
   \draw (2.75,-2) [dnup=0];
   \draw (3,-2) [dnup=0];
   \draw (3.375,-1) node {$\ldots$};
   \draw (3.75,-2) [dnup=0];
   \draw[dashed](0,-2)--(4,-2);
   \draw (0.25,-4) [dnup=0];
   \draw (0.5,-3) node {$\ldots$};
   \draw (0.75,-4) [dnup=0];
   \draw (1,-2) [upup=0.75];
   \draw (1,-4) [dndn=0.75];
   \draw (2.5,-4) [dnup=0];
   \draw (2.75,-4) [dnup=0];
   \draw (3,-4) [dnup=0];
   \draw (3.375,-3) node {$\ldots$};
   \draw (3.75,-4) [dnup=0];
   \draw (4.3,-1) node[scale=1.33] {$\;=$};
  \end{scope}
  \begin{scope}[yshift=-5cm,xshift=7.8cm]
   \draw (0,0) rectangle (4,2);
   \draw (0.25,0) [dnup=0];
   \draw (0.5,1) node {$\ldots$};
   \draw (0.75,0) [dnup=0];
   \draw (1,2) [upup=0.75];
   \draw (1,0) [dndn=0.75];
   \draw (2,0) [dnup=0];
   \draw (2.75,0) [dnup=0];
   \draw (3,0) [dnup=0];
   \draw (3.375,1) node {$\ldots$};
   \draw (3.75,0) [dnup=0];
  \end{scope}
 \end{tikzpicture}
 \caption{$E_iE_{i+1}E_i=E_i$}
\end{figure}
We leave the reader to check $E_iE_{i-1}E_i=E_i$.

The involution $^*$ of a diagram would be the reflection about the horizontal line. Thus, $E_i=E_i^*$. However, for non-generators we have,
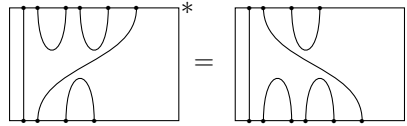
\begin{figure}[H]
 \centering
 \begin{tikzpicture}[scale=0.75]
  \draw (0,0) rectangle (3,2);
  \draw (0.25,0) [dnup=0];
  \draw (0.5,2) [upup=0.5];
  \draw (1.25,2) [upup=0.5];
  \draw (0.5,0) [dnup=1.75];
  \draw (1,0) [dndn=0.5];
  \draw (3.15,2) node[scale=1.33] {$*$};
  \draw (3.45,1) node[scale=1.33] {$=$};
  \begin{scope}[yscale=-1,xshift=4cm,yshift=-2cm]
   \draw (0,0) rectangle (3,2);
   \draw (0.25,0) [dnup=0];
   \draw (0.5,2) [upup=0.5];
   \draw (1.25,2) [upup=0.5];
   \draw (0.5,0) [dnup=1.75];
   \draw (1,0) [dndn=0.5];
  \end{scope}
 \end{tikzpicture}
 \caption{involution example on $gTL_n$}
\end{figure}

Thus the relations for our generators in $gTL_n$ are
\begin{enumerate}[(i)]
 \item $E_i^2=dE_i=dE_i^*$
 \item $E_iE_j=E_jE_i$ for $|i-j|\geq2$
 \item $E_iE_{i\pm1}E_i=E_i$
\end{enumerate}

These relations are closely related to that of $TL_n$. One can define a homomorphism by mapping $e_i$ to $d^{-1}E_i$. In fact, $gTL_n$ is $^*$-isomorphic to $TL_n$.
\begin{theorem}
 There exists a $^*$-isomorphism from $TL_n$ to $gTL_n$.
\end{theorem}

So from this point forth, we denote $\TL_n$ as our Temperley-Lieb algebra, where we will use the diagrammatic definition. So we have the generators $E_i$ for $i=1,2,\ldots,n-1$ and the relations above.


\subsection{Linear operations on \texorpdfstring{$\TL_n$}{\TL n}}

We now discuss some sub-planar algebra on $\TL_n$ which will make things easier to comprehend.

\begin{definition}\label{definition:linear_operations_tl}
 Consider the following diagram in $\TL_n$,
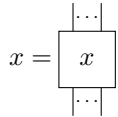
 \begin{center}
  \begin{tikzpicture}[scale=0.75]
   \draw (-0.5,0.5) node[scale=1.33] {$\displaystyle x=$};
   \draw (0,0) rectangle (1,1);
   \draw (0.25,1.5)--(0.25,1);
   \draw (0.25,0)--(0.25,-0.5);
   \draw (0.75,1.5)--(0.75,1);
   \draw (0.75,0)--(0.75,-0.5);
   \draw (0.53,1.25) node {$\cdots$};
   \draw (0.53,-0.25) node {$\cdots$};
   \draw (0.5,0.5) node[scale=1.33] {$x$};
  \end{tikzpicture}
 \end{center}
 we then have the following linear operations on $\TL_n$
 \begin{enumerate}[(i)]
  \item The \textit{inclusion}\index{linear operations on $\TL_n$!inclusion} $i_n\colon\TL_n\to\TL_{n+1}$ is a unital, injective $^*$-homomorphism. We essentially add a string to the right of the diagram.
   \begin{figure}[H]
   \centering
   \begin{tikzpicture}[scale=0.75]
    \draw (-1,0.5) node[scale=1.33] {$\displaystyle i_n(x)=$};
    \draw (0,0) rectangle (1,1);
    \draw (0.25,1.5)--(0.25,1);
    \draw (0.25,0)--(0.25,-0.5);
    \draw (0.75,1.5)--(0.75,1);
    \draw (0.75,0)--(0.75,-0.5);
    \draw (0.53,1.25) node {$\cdots$};
    \draw (0.53,-0.25) node {$\cdots$};
    \draw (0.5,0.5) node[scale=1.33] {$x$};
    \draw (1.25,1.5)--(1.25,-0.5);
   \end{tikzpicture}
   \caption{inclusion on $\TL_n$}
   \end{figure}
  \item The \textit{conditional expectation}\index{linear operations on $\TL_n$!conditional expectation} $\Ecal_{n+1}\colon\TL_{n+1}\to\TL_n$ is a surjective $^*$-map of $\CC$-vector spaces. We essentially close the last string on the right.
   \begin{figure}[H]
   \centering
   \begin{tikzpicture}[scale=0.75]
    \draw (-1.25,0.5) node[scale=1.33] {$\displaystyle \Ecal_{n+1}(x)=$};
    \draw (0.5,0.5) node[scale=1.33] {$x$};
    \draw (0,0) rectangle (1,1);
    \draw (0.25,1.5)--(0.25,1);
    \draw (0.25,0)--(0.25,-0.5);
    \draw (0.75,0)--(0.75,-0.5);
    \draw (0.75,1.5)--(0.75,1);
    \draw (0.9,1)--(0.9,1.5);
    \draw (0.9,-0.5)--(0.9,0);
    \draw[red] (0.9,1.5)--(1.2,1.5)--(1.2,-0.5)--(0.9,-0.5);
    \draw (0.53,1.25) node {$\cdots$};
    \draw (0.53,-0.25) node {$\cdots$};
   \end{tikzpicture}
   \caption{conditional expectation $\Ecal_{n+1}$ on $\TL_{n+1}$}
   \end{figure}
  \item The \textit{trace}\index{linear operations on $\TL_n$!trace} $\tr_n\colon\TL_n\to\TL_0$ is a linear $^*$-map of $\CC$-vector spaces. Essentially, it is the closure of the diagram.
   \begin{figure}[H]
   \centering
   \begin{tikzpicture}[scale=0.75]
    \draw (-1,0.5) node[scale=1.33] {$\displaystyle \tr_{n}(x)=$};
    \draw (0.5,0.5) node[scale=1.33] {$x$};
    \draw (0,0) rectangle (1,1);
    \draw (0.25,1.5)--(0.25,1);
    \draw (0.25,0)--(0.25,-0.5);
    \draw (0.75,0)--(0.75,-0.5);
    \draw (0.75,1.5)--(0.75,1);
    \draw[red] (0.75,1.5)--(1.25,1.5)--(1.25,-0.5)--(0.75,-0.5);
    \draw[red] (0.25,1.5)--(0.25,1.75)--(1.5,1.75)--(1.5,-0.75)--(0.25,-0.75)--(0.25,-0.5);
    \draw (0.53,1.25) node {$\cdots$};
    \draw (0.53,-0.25) node {$\cdots$};
   \end{tikzpicture}
   \caption{trace $\tr_n$ on $\TL_n$}
   \end{figure}
  \item The \textit{identity}\index{linear operations on $\TL_n$!identity map} $\id_n\colon\TL_n\to\TL_n$ which is given by
   \begin{center}
   \begin{tikzpicture}[scale=0.75]
    \draw (-1,0.5) node[scale=1.33] {$\displaystyle \id_{n}(x)=$};
    \draw (0.5,0.5) node[scale=1.33] {$x$};
    \draw (0,0) rectangle (1,1);
    \draw (0.25,1.5)--(0.25,1);
    \draw (0.25,0)--(0.25,-0.5);
    \draw (0.75,0)--(0.75,-0.5);
    \draw (0.75,1.5)--(0.75,1);
    \draw (0.53,1.25) node {$\cdots$};
    \draw (0.53,-0.25) node {$\cdots$};
   \end{tikzpicture}
   \end{center}
 \end{enumerate}
\end{definition}

It is typical to wonder whether the conditional expectation on $\TL_n$ defined above is related to the conditional expectation on factors in the previous chapter. This would make sense as to why we started with von Neumann algebras before defining the Temperley-Lieb algebras. The following proposition will answer our question.
\begin{proposition}\label{proposition:tldiagramrelations}
 There exists the following relations,
 \begin{enumerate}[(i)]
  \item $\Ecal_{n+1}\circ{i_n}=d\cdot\id_n$
  \item $\tr_{n+1}=tr_n\circ\Ecal_{n+1}$
  \item $(i_n\circ{i_{n-1}}\circ\Ecal_n(x))E_n=E_ni_n(x)E_n$ for all $x\in\TL_n$
  \item $\tr_n(xy)=\tr_n(yx)$ for all $x,y\in\TL_n$
  \item $\tr_{n+1}(i_n(x)\cdot{E_n})=\tr_n(x)$ for all $x\in\TL_n$
  \item $\tr_n(\Ecal_{n+1}(x)\cdot{y})=\tr_{n+1}(x\cdot{i_n(y)})$ for all $x\in\TL_{n+1}$ and $y\in\TL_n$
  \item $\tr_n(1_n)=d^n$
 \end{enumerate}
\end{proposition}
\begin{proof} {\ }
 \begin{enumerate}[(i)]
  \item $\Ecal_{n+1}\circ{i_n}\colon\TL_n\to\TL_n$,
   \[
    \Ecal_{n+1}\circ{i_n}\left(\tldiagram{$x$}\right) = \Ecal_{n+1}\left(\tlinclusion{$x$}\right) = \tlconditionalinclusion{$x$} = d\cdot\id_n\left(\tldiagram{$x$}\right)
   \]
  \item $\tr_{n+1}\colon\TL_{n+1}\to\TL_0$,
   \[
    \tr_{n+1}\left(\tldiagram{$x$}\right) = \tltraceplusone{$x$} = \tr_n\left(\tlconditional{$x$}\right) = \tr_n\circ\Ecal_{n+1}\left(\tldiagram{$x$}\right)
   \]
  \item We first have $i_n\circ{i_{n-1}}\circ\Ecal_n(x)$,
   \[
    \knotsinmath{\begin{scope}[scale=0.75]
     \draw (0,0) rectangle (1,1);
     \draw (0.25,1.5)--(0.25,1);
     \draw (0.25,0)--(0.25,-0.5);
     \draw (0.75,0)--(0.75,-0.5);
     \draw (0.75,1.5)--(0.75,1);
     \draw (0.9,1)--(0.9,1.5);
     \draw (0.9,-0.5)--(0.9,0);
     \draw (0.5,0.5) node[scale=1.33] {$x$};
     \draw[red] (0.9,1.5)--(1.2,1.5)--(1.2,-0.5)--(0.9,-0.5);
     \draw (0.53,1.25) node {$\cdots$};
     \draw (0.53,-0.25) node {$\cdots$};
     \draw (1.5,1.5)--(1.5,-0.5);
     \draw (1.75,1.5)--(1.75,-0.5);
    \end{scope}}
   \]
   and so, multiplying by $E_n$ gives
   \[
    \knotsinmath{\begin{scope}[scale=0.75]
     \draw (0,0) rectangle (1,1);
     \draw (0.25,1.5)--(0.25,1);
     \draw (0.25,0)--(0.25,-0.5);
     \draw (0.75,0)--(0.75,-0.5);
     \draw (0.75,1.5)--(0.75,1);
     \draw (0.9,1)--(0.9,1.5);
     \draw (0.9,-0.5)--(0.9,0);
     \draw (0.5,0.5) node[scale=1.33] {$x$};
     \draw[red] (0.9,1.5)--(1.2,1.5)--(1.2,-0.5)--(0.9,-0.5);
     \draw (0.53,1.25) node {$\cdots$};
     \draw (0.53,-0.25) node {$\cdots$};
     \draw (1.5,-0.5) [dndn=0.5];
     \draw (1.5,1.5) [upup=0.5];
    \end{scope}}
   \]
   This can be seen to be the same as $E_ni_n(x)E_n$,
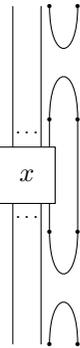
   \[
    \knotsinmath{\begin{scope}[scale=0.75]
     \draw (0,0) rectangle (1,1);
     \draw (0.25,3.5)--(0.25,1);
     \draw (0.25,0)--(0.25,-2.5);
     \draw (0.75,0)--(0.75,-2.5);
     \draw (0.75,3.5)--(0.75,1);
     \draw (0.9,1)--(0.9,1.5);
     \draw (0.9,-0.5)--(0.9,0);
     \draw (0.5,0.5) node[scale=1.33] {$x$};
     \draw (0.53,1.25) node {$\cdots$};
     \draw (0.53,-0.25) node {$\cdots$};
     \draw (0.9,-2.5) [dndn=0.5];
     \draw (0.9,3.5) [upup=0.5];
     \draw (0.9,-0.5) [upup=0.5];
     \draw (0.9,1.5) [dndn=0.5];
     \draw (1.4,-0.5)--(1.4,1.5);
    \end{scope}}
   \]
   Therefore, we can see that for any $x\in\TL_n$, we have $E_ni_n(x)E_n=(i_n\circ{i_{n-1}}\circ\Ecal_n(x))E_n$.
  \item $tr_n(xy)\colon\TL_n\times\TL_n\to\TL_0$,
   \[
    \tr_n\left(
     \knotsinmath{\begin{scope}[scale=0.75]
      \begin{scope}[yshift=-0.5cm]
       \draw (0,0) rectangle (1,1);
       \draw (0.25,1.5)--(0.25,1);
       \draw (0.25,0)--(0.25,-0.5);
       \draw (0.75,0)--(0.75,-0.5);
       \draw (0.75,1.5)--(0.75,1);
       \draw (0.53,1.25) node {$\cdots$};
       \draw (0.53,-0.25) node {$\cdots$};
       \draw (0.5,0.5) node[scale=1.33] {$x$};
      \end{scope}
      \begin{scope}[yshift=-0.5cm,xshift=1.5cm]
       \draw (-0.25,0.5) node[scale=1.33] {$\cdot$};
       \draw (0,0) rectangle (1,1);
       \draw (0.25,1.5)--(0.25,1);
       \draw (0.25,0)--(0.25,-0.5);
       \draw (0.75,0)--(0.75,-0.5);
       \draw (0.75,1.5)--(0.75,1);
       \draw (0.53,1.25) node {$\cdots$};
       \draw (0.53,-0.25) node {$\cdots$};
       \draw (0.5,0.5) node[scale=1.33] {$y$};
      \end{scope}
     \end{scope}}\right) = \tr_n\left(\knotsinmath{\begin{scope}[scale=0.75,yshift=0.25cm]
      \draw (0,0) rectangle (1,1);
      \draw (0.25,1.5)--(0.25,1);
      \draw (0.25,0)--(0.25,-0.5);
      \draw (0.75,0)--(0.75,-0.5);
      \draw (0.75,1.5)--(0.75,1);
      \draw (0.53,1.25) node {$\cdots$};
      \draw (0.53,-0.25) node {$\cdots$};
      \draw (0,-0.5) rectangle (1,-1.5);
      \draw (0.25,-1.5)--(0.25,-2);
      \draw (0.75,-1.5)--(0.75,-2);
      \draw (0.53,-1.75) node {$\cdots$};
      \draw (0.5,0.5) node[scale=1.33] {$x$};
      \draw (0.5,-1) node[scale=1.33] {$y$};
     \end{scope}}\right) = \knotsinmath{\begin{scope}[scale=0.75,yshift=0.25cm]
      \draw (0,0) rectangle (1,1);
      \draw (0.25,1.5)--(0.25,1);
      \draw (0.25,0)--(0.25,-0.5);
      \draw (0.75,0)--(0.75,-0.5);
      \draw (0.75,1.5)--(0.75,1);
      \draw (0.53,1.25) node {$\cdots$};
      \draw (0.53,-0.25) node {$\cdots$};
      \draw (0,-0.5) rectangle (1,-1.5);
      \draw (0.25,-1.5)--(0.25,-2);
      \draw (0.75,-1.5)--(0.75,-2);
      \draw (0.53,-1.75) node {$\cdots$};
      \draw (0.5,0.5) node[scale=1.33] {$x$};
      \draw (0.5,-1) node[scale=1.33] {$y$};
      \draw[red] (0.75,1.5)--(1.25,1.5)--(1.25,-2)--(0.75,-2);
      \draw[red] (0.25,1.5)--(0.25,1.75)--(1.5,1.75)--(1.5,-2.25)--(0.25,-2.25)--(0.25,-2);
     \end{scope}}
   \]
   By moving the $x$ diagram around in a clockwise fashion, we get,
   \[
    = \knotsinmath{\begin{scope}[scale=0.75,yshift=0.25cm]
      \draw (0,0) rectangle (1,1);
      \draw (0.25,1.5)--(0.25,1);
      \draw (0.25,0)--(0.25,-0.5);
      \draw (0.75,0)--(0.75,-0.5);
      \draw (0.75,1.5)--(0.75,1);
      \draw (0.53,1.25) node {$\cdots$};
      \draw (0.53,-0.25) node {$\cdots$};
      \draw (0,-0.5) rectangle (1,-1.5);
      \draw (0.25,-1.5)--(0.25,-2);
      \draw (0.75,-1.5)--(0.75,-2);
      \draw (0.53,-1.75) node {$\cdots$};
      \draw (0.5,0.5) node[scale=1.33] {$y$};
      \draw (0.5,-1) node[scale=1.33] {$x$};
      \draw[red] (0.75,1.5)--(1.25,1.5)--(1.25,-2)--(0.75,-2);
      \draw[red] (0.25,1.5)--(0.25,1.75)--(1.5,1.75)--(1.5,-2.25)--(0.25,-2.25)--(0.25,-2);
     \end{scope}} = \tr_n(yx)
   \]
  \item By (ii), we know $\tr_{n+1}=\tr_n\circ\Ecal_{n+1}$. And we also know $E_n\in\TL_{n+1}$, so $\tr_{n+1}(i_n(x)\cdot{E_n})=\tr_n(\Ecal_{n+1}(i_n(x)\cdot{E_n}))$. And so,
   \[
    \tr_n\left(\Ecal_{n+1}\left(\tlinclusion{$x$}\cdot\,\knotsinmath{\begin{scope}[scale=0.75,yshift=-0.5cm]
     \draw (0.25,-0.5) [dnup=0];
     \draw (1,-0.5) [dndn=0.5];
     \draw (1,1.5) [upup=0.5];
     \draw (0.8,-0.5) [dnup=0];
     \draw (0.55,0.5) node {$\cdots$};
    \end{scope}}\;\;\right)\right) = \tr_n\left(\Ecal_{n+1}\left(\knotsinmath{\begin{scope}[scale=0.75,yshift=-0.3cm]
     \begin{scope}[yshift=1cm]
      \draw (0.625,0.375) node[scale=1.33] {$x$};
      \draw (0,0) rectangle (1.25,0.75);
      \draw (0.25,-0.5)--(0.25,0);
      \draw (0.8,-0.5)--(0.8,0);
      \draw (1,-0.5)--(1,0);
      \draw (0.25,0.75)--(0.25,1.25);
      \draw (0.8,0.75)--(0.8,1.25);
      \draw (1,0.75)--(1,1.25);
      \draw (1.5,-0.5)--(1.5,1.25);
      \draw (0.55,1) node {$\cdots$};
     \end{scope}
     \begin{scope}[yshift=-1cm]
      \draw (0.25,-0.5) [dnup=0];
      \draw (1,-0.5) [dndn=0.5];
      \draw (1,1.5) [upup=0.5];
      \draw (0.8,-0.5) [dnup=0];
      \draw (0.55,0.5) node {$\cdots$};
     \end{scope}
    \end{scope}}\;\right)\right) = \tr_n\left(\knotsinmath{\begin{scope}[scale=0.75,yshift=-0.3cm]
     \begin{scope}[yshift=1cm]
      \draw (0,0) rectangle (1.25,0.75);
      \draw (0.625,0.375) node[scale=1.33] {$x$};
      \draw (0.25,-0.5)--(0.25,0);
      \draw (0.8,-0.5)--(0.8,0);
      \draw (1,-0.5)--(1,0);
      \draw (0.25,0.75)--(0.25,1.25);
      \draw (0.8,0.75)--(0.8,1.25);
      \draw (1,0.75)--(1,1.25);
      \draw (1.5,-0.5)--(1.5,1.25);
      \draw (0.55,1) node {$\cdots$};
     \end{scope}
     \begin{scope}[yshift=-1cm]
      \draw (0.25,-0.5) [dnup=0];
      \draw (1,-0.5) [dndn=0.5];
      \draw (1,1.5) [upup=0.5];
      \draw (0.8,-0.5) [dnup=0];
      \draw (0.55,0.5) node {$\cdots$};
     \end{scope}
     \draw[red] (1.5,2.25)--(1.75,2.25)--(1.75,-1.5)--(1.5,-1.5);
    \end{scope}}\;\right)
   \]
   And by straightening the last string, it is easy to see that we get,
   \[=\tr_n\left(\tldiagram{$x$}\right)\]
  \item From the above property, we know $\tr_n(\Ecal_{n+1}(x)\cdot{y})=\tr_{n+1}(i_n(\Ecal_{n+1}(x)\cdot{y})\cdot{E_n})$. So, we now need to show $i_n(\Ecal_{n+1}(x)\cdot{y})\cdot{E_n}=x\cdot{i_n(y)}$.
   \[
    i_n\left(\tlconditional{$x$}\cdot\,\tldiagram{$y$}\right) = i_n\left(\knotsinmath{\begin{scope}[scale=0.75,yshift=-0.5cm]
     \begin{scope}[yshift=1cm]
      \draw (0.5,0.5) node[scale=1.33] {$x$};
      \draw (0,0) rectangle (1,1);
      \draw (0.25,1.5)--(0.25,1);
      \draw (0.25,0)--(0.25,-0.5);
      \draw (0.75,0)--(0.75,-0.5);
      \draw (0.75,1.5)--(0.75,1);
      \draw (0.9,1)--(0.9,1.5);
      \draw (0.9,-0.5)--(0.9,0);
      \draw[red] (0.9,1.5)--(1.2,1.5)--(1.2,-0.5)--(0.9,-0.5);
      \draw (0.53,1.25) node {$\cdots$};
      \draw (0.53,-0.5) node {$\cdots$};
     \end{scope}
     \begin{scope}[yshift=-1cm]
      \draw (0.5,0.5) node[scale=1.33] {$y$};
      \draw (0,0) rectangle (1,1);
      \draw (0.25,1.5)--(0.25,1);
      \draw (0.25,0)--(0.25,-0.5);
      \draw (0.75,0)--(0.75,-0.5);
      \draw (0.75,1.5)--(0.75,1);
      \draw (0.53,-0.25) node {$\cdots$};
     \end{scope}
    \end{scope}}\;\right) = 
    \knotsinmath{
    \begin{scope}[scale=0.75,yshift=-0.5cm]
     \begin{scope}[yshift=1cm]
      \draw (0.5,0.5) node[scale=1.33] {$x$};
      \draw (0,0) rectangle (1,1);
      \draw (0.25,1.5)--(0.25,1);
      \draw (0.25,0)--(0.25,-0.5);
      \draw (0.75,0)--(0.75,-0.5);
      \draw (0.75,1.5)--(0.75,1);
      \draw (0.9,1)--(0.9,1.5);
      \draw (0.9,-0.5)--(0.9,0);
      \draw[red] (0.9,1.5)--(1.2,1.5)--(1.2,-0.5)--(0.9,-0.5);
      \draw (0.53,1.25) node {$\cdots$};
      \draw (0.53,-0.5) node {$\cdots$};
     \end{scope}
     \begin{scope}[yshift=-1cm]
      \draw (0.5,0.5) node[scale=1.33] {$y$};
      \draw (0,0) rectangle (1,1);
      \draw (0.25,1.5)--(0.25,1);
      \draw (0.25,0)--(0.25,-0.5);
      \draw (0.75,0)--(0.75,-0.5);
      \draw (0.75,1.5)--(0.75,1);
      \draw (0.53,-0.25) node {$\cdots$};
     \end{scope}
     \draw (1.5,2.5)--(1.5,-1.5);
    \end{scope}}
   \]
   And so $i_n(\Ecal_{n+1}(x)\cdot{y})\cdot{E_n}$ becomes
   \[
    \knotsinmath{
    \begin{scope}[scale=0.75,yshift=-0.5cm]
     \begin{scope}[yshift=1cm]
      \draw (0.5,0.5) node[scale=1.33] {$x$};
      \draw (0,0) rectangle (1,1);
      \draw (0.25,1.5)--(0.25,1);
      \draw (0.25,0)--(0.25,-0.5);
      \draw (0.75,0)--(0.75,-0.5);
      \draw (0.75,1.5)--(0.75,1);
      \draw (0.9,1)--(0.9,1.5);
      \draw (0.9,-0.5)--(0.9,0);
      \draw[red] (0.9,1.5)--(1.2,1.5)--(1.2,-0.5)--(0.9,-0.5);
      \draw (0.53,1.25) node {$\cdots$};
      \draw (0.53,-0.5) node {$\cdots$};
     \end{scope}
     \begin{scope}[yshift=-1cm]
      \draw (0.5,0.5) node[scale=1.33] {$y$};
      \draw (0,0) rectangle (1,1);
      \draw (0.25,1.5)--(0.25,1);
      \draw (0.25,0)--(0.25,-0.5);
      \draw (0.75,0)--(0.75,-0.5);
      \draw (0.75,1.5)--(0.75,1);
      \draw (0.53,-0.25) node {$\cdots$};
     \end{scope}
     \draw (1.5,2.5)--(1.5,-1.5);
     \begin{scope}[xshift=1.75cm]
      \draw (0,0.5) node {$\cdot$};
      \draw (0.25,0.5) [dnup=0];
      \draw (0.25,-1.5) [dnup=0];
      \draw (1,-1.5) [dndn=0.5];
      \draw (1,2.5) [upup=0.5];
      \draw (0.8,-1.5) [dnup=0];
      \draw (0.8,0.5) [dnup=0];
      \draw (0.55,0.5) node {$\cdots$};
     \end{scope}
    \end{scope}}=\knotsinmath{
    \begin{scope}[scale=0.75,yshift=0.5cm]
     \begin{scope}[yshift=1cm]
      \draw (0.5,0.5) node[scale=1.33] {$x$};
      \draw (0,0) rectangle (1.1,1);
      \draw (0.25,1.5)--(0.25,1);
      \draw (0.25,0)--(0.25,-0.5);
      \draw (0.8,0)--(0.8,-0.5);
      \draw (0.8,1.5)--(0.8,1);
      \draw (0.9,1)--(0.9,1.5);
      \draw (0.9,-0.5)--(0.9,0);
      \draw[red] (0.9,1.5)--(1.2,1.5)--(1.2,-0.5)--(0.9,-0.5);
      \draw (0.53,1.25) node {$\cdots$};
      \draw (0.53,-0.5) node {$\cdots$};
     \end{scope}
     \begin{scope}[yshift=-1cm]
      \draw (0.5,0.5) node[scale=1.33] {$y$};
      \draw (0,0) rectangle (1.1,1);
      \draw (0.25,1.5)--(0.25,1);
      \draw (0.25,0)--(0.25,-0.5);
      \draw (0.8,0)--(0.8,-0.5);
      \draw (0.8,1.5)--(0.8,1);
      \draw (1,0)--(1,-0.5);
     \end{scope}
     \draw (1.5,2.5)--(1.5,-1.5);
     \begin{scope}[yshift=-3cm]
      \draw (0.25,-0.5) [dnup=0];
      \draw (1,-0.5) [dndn=0.5];
      \draw (1,1.5) [upup=0.5];
      \draw (0.8,-0.5) [dnup=0];
      \draw (0.55,0.5) node {$\cdots$};
     \end{scope}
    \end{scope}}
   \]
   Now from (i), we know $\tr_{n+1}=\tr_n\circ\Ecal_{n+1}$, so
   \[
    \knotsinmath{
    \begin{scope}[scale=0.75,yshift=0.5cm]
     \begin{scope}[yshift=1cm]
      \draw (0.5,0.5) node[scale=1.33] {$x$};
      \draw (0,0) rectangle (1.1,1);
      \draw (0.25,1.5)--(0.25,1);
      \draw (0.25,0)--(0.25,-0.5);
      \draw (0.8,0)--(0.8,-0.5);
      \draw (0.8,1.5)--(0.8,1);
      \draw (0.9,1)--(0.9,1.5);
      \draw (0.9,-0.5)--(0.9,0);
      \draw[red] (0.9,1.5)--(1.2,1.5)--(1.2,-0.5)--(0.9,-0.5);
      \draw (0.53,1.25) node {$\cdots$};
      \draw (0.53,-0.5) node {$\cdots$};
     \end{scope}
     \begin{scope}[yshift=-1cm]
      \draw (0.5,0.5) node[scale=1.33] {$y$};
      \draw (0,0) rectangle (1.1,1);
      \draw (0.25,1.5)--(0.25,1);
      \draw (0.25,0)--(0.25,-0.5);
      \draw (0.8,0)--(0.8,-0.5);
      \draw (0.8,1.5)--(0.8,1);
      \draw (1,0)--(1,-0.5);
     \end{scope}
     \draw (1.5,2.5)--(1.5,-1.5);
     \begin{scope}[yshift=-3cm]
      \draw (0.25,-0.5) [dnup=0];
      \draw (1,-0.5) [dndn=0.5];
      \draw (1,1.5) [upup=0.5];
      \draw (0.8,-0.5) [dnup=0];
      \draw (0.55,0.5) node {$\cdots$};
     \end{scope}
     \draw[red] (1.5,2.5)--(1.75,2.5)--(1.75,-3.5)--(1.5,-3.5);
    \end{scope}} = \knotsinmath{
    \begin{scope}[scale=0.75,yshift=-0.5cm]
     \begin{scope}[yshift=1cm]
      \draw (0.5,0.5) node[scale=1.33] {$x$};
      \draw (0,0) rectangle (1,1);
      \draw (0.25,1.5)--(0.25,1);
      \draw (0.25,0)--(0.25,-0.5);
      \draw (0.75,0)--(0.75,-0.5);
      \draw (0.75,1.5)--(0.75,1);
      \draw (0.9,1)--(0.9,1.5);
      \draw (0.9,-0.5)--(0.9,0);
      \draw[red] (0.9,1.5)--(1.2,1.5)--(1.2,-0.5)--(0.9,-0.5);
      \draw (0.53,1.25) node {$\cdots$};
      \draw (0.53,-0.5) node {$\cdots$};
     \end{scope}
     \begin{scope}[yshift=-1cm]
      \draw (0.5,0.5) node[scale=1.33] {$y$};
      \draw (0,0) rectangle (1,1);
      \draw (0.25,1.5)--(0.25,1);
      \draw (0.25,0)--(0.25,-0.5);
      \draw (0.75,0)--(0.75,-0.5);
      \draw (0.75,1.5)--(0.75,1);
      \draw (0.53,-0.25) node {$\cdots$};
     \end{scope}
    \end{scope}}
   \]
   Now, $\Ecal_{n+1}(x\cdot{i_n(y)})$ is
   \[
    \Ecal_{n+1}\left(\tldiagram{$x$}\cdot\tlinclusion{$y$}\right) = \Ecal_{n+1}\left(\knotsinmath{\begin{scope}[scale=0.75,yshift=-0.5cm]
    \begin{scope}[yshift=1cm]
     \draw (0.75,0.5) node[scale=1.33] {$x$};
     \draw (0,0) rectangle (1.5,1);
     \draw (0.25,1.5)--(0.25,1);
     \draw (0.25,0)--(0.25,-0.5);
     \draw (0.75,0)--(0.75,-0.5);
     \draw (0.75,1.5)--(0.75,1);
     \draw (1.25,1.5)--(1.25,1);
     \draw (1.25,0)--(1.25,-0.5);
     \draw (0.53,1.25) node {$\cdots$};
     \draw (0.53,-0.5) node {$\cdots$};
    \end{scope}
    \begin{scope}[yshift=-1cm]
     \draw (0.5,0.5) node[scale=1.33] {$y$};
     \draw (0,0) rectangle (1,1);
     \draw (0.25,1.5)--(0.25,1);
     \draw (0.25,0)--(0.25,-0.5);
     \draw (0.75,0)--(0.75,-0.5);
     \draw (0.75,1.5)--(0.75,1);
     \draw (0.53,-0.25) node {$\cdots$};
     \draw (1.25,1.5)--(1.25,-0.5);
    \end{scope}
    \end{scope}}\right) = \knotsinmath{\begin{scope}[scale=0.75,yshift=-0.5cm]
    \begin{scope}[yshift=1cm]
     \draw (0.75,0.5) node[scale=1.33] {$x$};
     \draw (0,0) rectangle (1.5,1);
     \draw (0.25,1.5)--(0.25,1);
     \draw (0.25,0)--(0.25,-0.5);
     \draw (0.75,0)--(0.75,-0.5);
     \draw (0.75,1.5)--(0.75,1);
     \draw (1.25,1.5)--(1.25,1);
     \draw (1.25,0)--(1.25,-0.5);
     \draw (0.53,1.25) node {$\cdots$};
     \draw (0.53,-0.5) node {$\cdots$};
    \end{scope}
    \begin{scope}[yshift=-1cm]
     \draw (0.5,0.5) node[scale=1.33] {$y$};
     \draw (0,0) rectangle (1,1);
     \draw (0.25,1.5)--(0.25,1);
     \draw (0.25,0)--(0.25,-0.5);
     \draw (0.75,0)--(0.75,-0.5);
     \draw (0.75,1.5)--(0.75,1);
     \draw (0.53,-0.25) node {$\cdots$};
     \draw (1.25,1.5)--(1.25,-0.5);
    \end{scope}
    \draw[red] (1.25,2.5)--(1.75,2.5)--(1.75,-1.5)--(1.25,-1.5);
    \end{scope}}
   \]
   Thus $\tr_n(\Ecal_{n+1}(x)\cdot{y})=\tr_{n+1}(x\cdot{i_n(y)}$.
  \item And finally,
   \[
    \tr_n\left(\;\knotsinmath{\begin{scope}[scale=0.75,yshift=-1cm]
     \draw (0.25,0) [dnup=0];
     \draw (0.5,1) node {$\ldots$};
     \draw (0.75,0) [dnup=0];
    \end{scope}}\;\right) = \knotsinmath{\begin{scope}[scale=0.75,yshift=-1cm]
     \draw (0.25,0) [dnup=0];
     \draw (0.5,1) node {$\ldots$};
     \draw (0.75,0) [dnup=0];
     \draw[red] (0.75,2)--(1,2)--(1,0)--(0.75,0);
     \draw[red] (0.25,2)--(0.25,2.25)--(1.25,2.25)--(1.25,-0.25)--(0.25,-0.25)--(0.25,0);
    \end{scope}} = d^n
   \]
 \end{enumerate}
\end{proof}



\begin{remark}
 Now comparing the above relations and the relations found in Propositions \ref{proposition:conditional_expectation_factors} and \ref{proposition:Jonesprojections}, we find that they are related. For instance, in factors $N\subset{M}$, we know $\tr_M(E_N(x))=\tr_N(x)$ for all $x\in{N}$ where $E_N\colon{N}\to{M}$. This is the same as (ii) in the above proposition.
 This means that the conditional expectation on factors is related to the conditional expectation on $\TL_n$, and the Jones projections are related to the generators on $\TL_n$. But to relate one to the other, we must be extra careful with the inclusion map on $\TL_n$. For instance, part (iii) above is similar to the relation in part (i) of Proposition \ref{proposition:Jonesprojections}.
\end{remark}

The trace on $\TL_n$ is called the \textit{Markov trace}\index{Markov trace}, and is unique up to a scalar (the scalar being $d^n$).

\begin{remark}
 To get a normal trace, i.e., $\tr_n(1)=1$. We need to define the trace as the number of loops minus $n$. So, then it is clear that we would get: $\tr_n(1)=1$; $\tr_n(xy)=\tr_n(yx)$; and $\tr_{n+1}(xE_{n})=d^{-1}\tr_{n}(x)$ for all $x,y\in\TL_n$.
 We will not use the normal trace, since we will normalise our Jones polynomial in the following chapter.
 So this is only an interesting remark and nothing else.
\end{remark}

\subsection{The Jones-Wenzl idempotents}

In this section we will discuss the Jones-Wenzl idempotents. We say an element is an idempotent if when multiplied by itself it equals itself. So any identity is itself an idempotent.
We now let $[n]_q=[n]$, for better readability.

\begin{proposition}
 Suppose there exists $f^{(n)}\in\TL_n$ such that $f^{(n)}\neq0$, $(f^{(n)})^2=f^{(n)}$ and for all $i=1,\ldots,n-1$ we have $E_if^{(n)}=f^{(n)}E_i$. Then $f^{(n)}$ is unique for each $\TL_n$.
\end{proposition}
\begin{proof}
 Let $f_1^{(n)},f^{(n)}\in\TL_n$ satisfy all properties and $f_1^{(n)}\neq{f^{(n)}}$.
 So by writing them as a linear combination of the generators $\{1,E_1,\ldots,E_{n-1}\}$ we have
 \[f^{(n)}=x+\sum_i{y_iE_i}\qquad\qquad{f_1^{(n)}}=x^\prime+\sum_i{y^\prime_iE_i}\]
 It then follows that $x=x^\prime=1$. Since both are idempotents we then have,
 \begin{align*}
  f^{(n)} &= f^{(n)}f_1^{(n)}\\
  &= \left(1+\sum_i{y_iE_i}\right)f_1^{(n)}
 \end{align*}
 Now $\left(\sum_iy_iE_i\right)f_1^{(n)}=0$, and so the result then follows, i.e., $f^{(n)}=f_1^{(n)}$.
\end{proof}

\begin{definition}[Jones-Wenzl idempotents {\cite[p.~30]{penneys}}]
 The \textit{Jones-Wenzl idempotents}\index{Jones-Wenzl idempotents} are $f^{\left(n\right)}\in\TL_n$ are defined inductively such that
 \begin{enumerate}[(i)]
  \item $f^{\left(0\right)}\in\TL_0$. is the empty diagram
  \item $f^{\left(1\right)}\in\TL_1$ is the single strand 
  \item if $[n+1]\neq0$, then
  \[f^{(n+1)}=i_n\left(f^{(n)}\right)-\frac{[n]}{[n+1]}i_n\left(f^{(n)}\right)E_ni_n\left(f^{(n)}\right).\]
 \end{enumerate}
\end{definition}

To understand this better, we will need to use the diagrams. So, $f^{(n+1)}$ can be rewritten as,
\begin{align*}
 f^{(n+1)} &=
 \knotsinmath{\begin{scope}[scale=0.75,yshift=-0.5cm]
  \draw (0.5,0.5) node[scale=1.33] {$f^{(n)}$};
  \draw (0,0) rectangle (1,1);
  \draw (0.25,1.5)--(0.25,1);
  \draw (0.25,0)--(0.25,-0.5);
  \draw (0.75,0)--(0.75,-0.5);
  \draw (0.75,1.5)--(0.75,1);
  \draw (0.53,1.25) node {$\cdots$};
  \draw (0.53,-0.25) node {$\cdots$};
  \draw (1.25,1.5)--(1.25,-0.5);
 \end{scope}}-\frac{[n]}{[n+1]}\knotsinmath{\begin{scope}[scale=0.75,yshift=-0.5cm]
  \draw (0.5,0.5) node[scale=1.33] {$f^{(n)}$};
  \draw (0,0) rectangle (1,1);
  \draw (0.25,1.5)--(0.25,1);
  \draw (0.25,0)--(0.25,-0.5);
  \draw (0.75,0)--(0.75,-0.5);
  \draw (0.75,1.5)--(0.75,1);
  \draw (0.53,1.25) node {$\cdots$};
  \draw (0.53,-0.25) node {$\cdots$};
  \draw (1.25,1.5)--(1.25,-0.5);
 \end{scope}}\cdot\knotsinmath{\begin{scope}[scale=0.75]\draw (2,-0.85) [dnup=0]; \draw (2.5,-0.85) [dnup=0]; \draw (3,-0.85) [dndn=0.5]; \draw (3,1.15) [upup=0.5]; \draw (2.3,0.15) node {$\cdots$};\end{scope}}\cdot\knotsinmath{\begin{scope}[scale=0.75,yshift=-0.5cm]
  \draw (0.5,0.5) node[scale=1.33] {$f^{(n)}$};
  \draw (0,0) rectangle (1,1);
  \draw (0.25,1.5)--(0.25,1);
  \draw (0.25,0)--(0.25,-0.5);
  \draw (0.75,0)--(0.75,-0.5);
  \draw (0.75,1.5)--(0.75,1);
  \draw (0.53,1.25) node {$\cdots$};
  \draw (0.53,-0.25) node {$\cdots$};
  \draw (1.25,1.5)--(1.25,-0.5);
 \end{scope}}\\
 &=
 \knotsinmath{\begin{scope}[scale=0.75,yshift=-0.5cm]
  \draw (0.5,0.5) node[scale=1.33] {$f^{(n)}$};
  \draw (0,0) rectangle (1,1);
  \draw (0.25,1.5)--(0.25,1);
  \draw (0.25,0)--(0.25,-0.5);
  \draw (0.75,0)--(0.75,-0.5);
  \draw (0.75,1.5)--(0.75,1);
  \draw (0.53,1.25) node {$\cdots$};
  \draw (0.53,-0.25) node {$\cdots$};
  \draw (1.25,1.5)--(1.25,-0.5);
 \end{scope}}-\frac{[n]}{[n+1]}\knotsinmath{
  \begin{scope}[scale=0.75,yshift=1.5cm]
  \draw (0.5,0.5) node[scale=1.33] {$f^{(n)}$};
  \draw (0,0) rectangle (1.1,1);
  \draw (0.25,1.5)--(0.25,1);
  \draw (0.25,0)--(0.25,-0.5);
  \draw (0.8,0)--(0.8,-0.5);
  \draw (0.8,1.5)--(0.8,1);
  \draw (1,0)--(1,-0.5);
  \draw (1,1.5)--(1,1);
  \draw (1.5,1.5)--(1.5,-0.5);
  \draw (0.53,1.25) node {$\cdots$};
  \end{scope}
  \begin{scope}[scale=0.75,yshift=-0.5cm]
  \draw (0.25,-0.5) [dnup=0];
  \draw (1,-0.5) [dndn=0.5];
  \draw (1,1.5) [upup=0.5];
  \draw (0.8,-0.5) [dnup=0];
  \draw (0.55,0.5) node {$\cdots$};
  \end{scope}
  \begin{scope}[scale=0.75,yshift=-2.5cm]
  \draw (0.5,0.5) node[scale=1.33] {$f^{(n)}$};
  \draw (0,0) rectangle (1.1,1);
  \draw (0.25,1.5)--(0.25,1);
  \draw (0.25,0)--(0.25,-0.5);
  \draw (0.8,0)--(0.8,-0.5);
  \draw (0.8,1.5)--(0.8,1);
  \draw (1,1)--(1,1.5);
  \draw (1,-0.5)--(1,0);
  \draw (1.5,1.5)--(1.5,-0.5);
  \draw (0.53,-0.25) node {$\cdots$};
  \end{scope}}
  \tag{JW}\label{equation:jwidempotents}
\end{align*}

\begin{example} {\ }
 \[f^{(1)}=\knotsinmath{\begin{scope}[scale=0.75]\draw (0.25,-1) [dnup=0];\end{scope}}\]
 \[f^{(2)}=\knotsinmath{\begin{scope}[scale=0.75]\draw (0.25,-1) [dnup=0]; \draw (0.75,-1) [dnup=0];\end{scope}}-\frac{1}{[2]}\knotsinmath{\begin{scope}[scale=0.75]\draw (0.25,-1) [dndn=0.5];\draw (0.25,1) [upup=0.5];\end{scope}}=1_2-\frac{1}{[2]}E_1\]
\end{example}

Now we consider mapping $q\in\{Q\cup-Q\}$ to $[2]$, and find that the map is bijective, thus our map uniquely determines $q$ (see \cite{penneys}). So we let $d=[2]$.

\begin{proposition}[{\cite[p.~31]{penneys}}]
 Suppose $n\geq0$ and $[1],\ldots,[n],[n+1]\neq0$ so $f^{\left(0\right)},\ldots,f^{\left(n+1\right)}$ are well-defined. Then, $f^{\left(n+1\right)}$ satisfies the following.
 \begin{enumerate}[(i)]
  \item $f^{\left(n+1\right)}=\left(f^{\left(n+1\right)}\right)^*=\left(f^{\left(n+1\right)}\right)^2$
  \item $\Ecal_{n+1}\left(f^{\left(n+1\right)}\right)=\dfrac{[n+2]}{[n+1]}f^{\left(n\right)}$
  \item $\left(i_n\left(f^{\left(n\right)}\right)\right)f^{\left(n+1\right)}=f^{\left(n+1\right)}\left(i_n\left(f^{\left(n\right)}\right)\right)=f^{\left(n+1\right)}$
  \item $\tr_{n+1}\left(f_{\left(n+1\right)}\right)=[n+2]$
  \item $E_jf^{(n+1)}=f^{(n+1)}E_j=0$ for all $j=\{1,\ldots,n\}$
 \end{enumerate}
\end{proposition}
\begin{proof}
 We show this by induction on $n$.
 This is obvious for $n=0$.
 So suppose this is true for any $n$. Then, we need to show that it is also true for $n+1$.
 \begin{enumerate}[(i)]
  \item
 \begin{align*}
  \left(f^{(n+1)}\right)^2 &= \left( \knotsinmath{\begin{scope}[scale=0.75,yshift=-0.5cm]
  \draw (0.5,0.5) node[scale=1.33] {$f^{(n)}$};
  \draw (0,0) rectangle (1,1);
  \draw (0.25,1.5)--(0.25,1);
  \draw (0.25,0)--(0.25,-0.5);
  \draw (0.75,0)--(0.75,-0.5);
  \draw (0.75,1.5)--(0.75,1);
  \draw (0.53,1.25) node {$\cdots$};
  \draw (0.53,-0.25) node {$\cdots$};
  \draw (1.25,1.5)--(1.25,-0.5);
 \end{scope}}-\frac{[n]}{[n+1]}\knotsinmath{
  \begin{scope}[scale=0.75,yshift=1.5cm]
  \draw (0.5,0.5) node[scale=1.33] {$f^{(n)}$};
  \draw (0,0) rectangle (1.1,1);
  \draw (0.25,1.5)--(0.25,1);
  \draw (0.25,0)--(0.25,-0.5);
  \draw (0.8,0)--(0.8,-0.5);
  \draw (0.8,1.5)--(0.8,1);
  \draw (1,0)--(1,-0.5);
  \draw (1,1.5)--(1,1);
  \draw (1.5,1.5)--(1.5,-0.5);
  \draw (0.53,1.25) node {$\cdots$};
  \end{scope}
  \begin{scope}[scale=0.75,yshift=-0.5cm]
  \draw (0.25,-0.5) [dnup=0];
  \draw (1,-0.5) [dndn=0.5];
  \draw (1,1.5) [upup=0.5];
  \draw (0.8,-0.5) [dnup=0];
  \draw (0.55,0.5) node {$\cdots$};
  \end{scope}
  \begin{scope}[scale=0.75,yshift=-2.5cm]
  \draw (0.5,0.5) node[scale=1.33] {$f^{(n)}$};
  \draw (0,0) rectangle (1.1,1);
  \draw (0.25,1.5)--(0.25,1);
  \draw (0.25,0)--(0.25,-0.5);
  \draw (0.8,0)--(0.8,-0.5);
  \draw (0.8,1.5)--(0.8,1);
  \draw (1,1)--(1,1.5);
  \draw (1,-0.5)--(1,0);
  \draw (1.5,1.5)--(1.5,-0.5);
  \draw (0.53,-0.25) node {$\cdots$};
  \end{scope}}\,\,\right)^2\\
  &=  \knotsinmath{
   \begin{scope}[scale=0.75,yshift=0.6cm]
    \draw (0.5,0.5) node[scale=1.33] {$f^{(n)}$};
    \draw (0,0) rectangle (1.1,1);
    \draw (0.25,1.5)--(0.25,1);
    \draw (0.25,0)--(0.25,-0.5);
    \draw (0.8,0)--(0.8,-0.5);
    \draw (0.8,1.5)--(0.8,1);
    \draw (1,0)--(1,-0.5);
    \draw (1,1.5)--(1,1);
    \draw (1.5,1.5)--(1.5,-0.5);
    \draw (0.53,1.25) node {$\cdots$};
   \end{scope}
   \begin{scope}[scale=0.75,yshift=-1.4cm]
    \draw (0.5,0.5) node[scale=1.33] {$f^{(n)}$};
    \draw (0,0) rectangle (1.1,1);
    \draw (0.25,1.5)--(0.25,1);
    \draw (0.25,0)--(0.25,-0.5);
    \draw (0.8,0)--(0.8,-0.5);
    \draw (0.8,1.5)--(0.8,1);
    \draw (1,0)--(1,-0.5);
    \draw (1,1.5)--(1,1);
    \draw (1.5,1.5)--(1.5,-0.5);
    \draw (0.53,1.4) node {$\cdots$};
   \end{scope}
  }-\frac{[n]}{[n+1]}\knotsinmath{
  \begin{scope}[scale=0.75,yshift=2cm]
  \draw (0.5,0.5) node[scale=1.33] {$f^{(n)}$};
  \draw (0,0) rectangle (1.1,1);
  \draw (0.25,1.5)--(0.25,1);
  \draw (0.25,0)--(0.25,-0.5);
  \draw (0.8,0)--(0.8,-0.5);
  \draw (0.8,1.5)--(0.8,1);
  \draw (1,0)--(1,-0.5);
  \draw (1,1.5)--(1,1);
  \draw (1.5,1.5)--(1.5,-0.5);
  \draw (0.53,1.25) node {$\cdots$};
  \end{scope}
  \begin{scope}[scale=0.75,yshift=0cm]
  \draw (0.25,-0.5) [dnup=0];
  \draw (1,-0.5) [dndn=0.5];
  \draw (1,1.5) [upup=0.5];
  \draw (0.8,-0.5) [dnup=0];
  \draw (0.55,0.5) node {$\cdots$};
  \end{scope}
  \begin{scope}[scale=0.75,yshift=-2cm]
  \draw (0.5,0.5) node[scale=1.33] {$f^{(n)}$};
  \draw (0,0) rectangle (1.1,1);
  \draw (0.25,1.5)--(0.25,1);
  \draw (0.25,0)--(0.25,-0.5);
  \draw (0.8,0)--(0.8,-0.5);
  \draw (0.8,1.5)--(0.8,1);
  \draw (1,1)--(1,1.5);
  \draw (1,-0.5)--(1,0);
  \draw (1.5,1.5)--(1.5,-0.5);
  \draw (0.53,-0.25) node {$\cdots$};
  \end{scope}
  \begin{scope}[scale=0.75,yshift=-4cm]
   \draw (0.5,0.5) node[scale=1.33] {$f^{(n)}$};
   \draw (0,0) rectangle (1.1,1);
   \draw (0.25,1.5)--(0.25,1);
   \draw (0.25,0)--(0.25,-0.5);
   \draw (0.8,0)--(0.8,-0.5);
   \draw (0.8,1.5)--(0.8,1);
   \draw (1,0)--(1,-0.5);
   \draw (1,1.5)--(1,1);
   \draw (1.5,1.5)--(1.5,-0.5);
   \draw (0.53,1.25) node {$\cdots$};
  \end{scope}}-\frac{[n]}{[n+1]}\knotsinmath{
  \begin{scope}[scale=0.75,yshift=-4cm]
  \draw (0.5,0.5) node[scale=1.33] {$f^{(n)}$};
  \draw (0,0) rectangle (1.1,1);
  \draw (0.25,1.5)--(0.25,1);
  \draw (0.25,0)--(0.25,-0.5);
  \draw (0.8,0)--(0.8,-0.5);
  \draw (0.8,1.5)--(0.8,1);
  \draw (1,0)--(1,-0.5);
  \draw (1,1.5)--(1,1);
  \draw (1.5,1.5)--(1.5,-0.5);
  \draw (0.53,1.25) node {$\cdots$};
  \end{scope}
  \begin{scope}[scale=0.75,yshift=-2cm]
  \draw (0.25,-0.5) [dnup=0];
  \draw (1,-0.5) [dndn=0.5];
  \draw (1,1.5) [upup=0.5];
  \draw (0.8,-0.5) [dnup=0];
  \draw (0.55,0.5) node {$\cdots$};
  \end{scope}
  \begin{scope}[scale=0.75,yshift=2cm]
  \draw (0.5,0.5) node[scale=1.33] {$f^{(n)}$};
  \draw (0,0) rectangle (1.1,1);
  \draw (0.25,1.5)--(0.25,1);
  \draw (0.25,0)--(0.25,-0.5);
  \draw (0.8,0)--(0.8,-0.5);
  \draw (0.8,1.5)--(0.8,1);
  \draw (1,1)--(1,1.5);
  \draw (1,-0.5)--(1,0);
  \draw (1.5,1.5)--(1.5,-0.5);
  \draw (0.53,-0.25) node {$\cdots$};
  \end{scope}
  \begin{scope}[scale=0.75,yshift=0cm]
   \draw (0.5,0.5) node[scale=1.33] {$f^{(n)}$};
   \draw (0,0) rectangle (1.1,1);
   \draw (0.25,1.5)--(0.25,1);
   \draw (0.25,0)--(0.25,-0.5);
   \draw (0.8,0)--(0.8,-0.5);
   \draw (0.8,1.5)--(0.8,1);
   \draw (1,0)--(1,-0.5);
   \draw (1,1.5)--(1,1);
   \draw (1.5,1.5)--(1.5,-0.5);
   \draw (0.53,1.25) node {$\cdots$};
  \end{scope}}+\frac{[n]^2}{[n+1]^2}\knotsinmath{\begin{scope}[yshift=2.25cm]
   \begin{scope}[scale=0.75,yshift=1.5cm]
  \draw (0.5,0.5) node[scale=1.33] {$f^{(n)}$};
  \draw (0,0) rectangle (1.1,1);
  \draw (0.25,1.5)--(0.25,1);
  \draw (0.25,0)--(0.25,-0.5);
  \draw (0.8,0)--(0.8,-0.5);
  \draw (0.8,1.5)--(0.8,1);
  \draw (1,0)--(1,-0.5);
  \draw (1,1.5)--(1,1);
  \draw (1.5,1.5)--(1.5,-0.5);
  \draw (0.53,1.25) node {$\cdots$};
  \end{scope}
  \begin{scope}[scale=0.75,yshift=-0.5cm]
  \draw (0.25,-0.5) [dnup=0];
  \draw (1,-0.5) [dndn=0.5];
  \draw (1,1.5) [upup=0.5];
  \draw (0.8,-0.5) [dnup=0];
  \draw (0.55,0.5) node {$\cdots$};
  \end{scope}
  \begin{scope}[scale=0.75,yshift=-2.5cm]
  \draw (0.5,0.5) node[scale=1.33] {$f^{(n)}$};
  \draw (0,0) rectangle (1.1,1);
  \draw (0.25,1.5)--(0.25,1);
  \draw (0.25,0)--(0.25,-0.5);
  \draw (0.8,0)--(0.8,-0.5);
  \draw (0.8,1.5)--(0.8,1);
  \draw (1,1)--(1,1.5);
  \draw (1,-0.5)--(1,0);
  \draw (1.5,1.5)--(1.5,-0.5);
  \draw (0.53,-0.25) node {$\cdots$};
  \end{scope}
  \end{scope}\begin{scope}[yshift=-2.25cm]
   \begin{scope}[scale=0.75,yshift=1.5cm]
  \draw (0.5,0.5) node[scale=1.33] {$f^{(n)}$};
  \draw (0,0) rectangle (1.1,1);
  \draw (0.25,1.5)--(0.25,1);
  \draw (0.25,0)--(0.25,-0.5);
  \draw (0.8,0)--(0.8,-0.5);
  \draw (0.8,1.5)--(0.8,1);
  \draw (1,0)--(1,-0.5);
  \draw (1,1.5)--(1,1);
  \draw (1.5,1.5)--(1.5,-0.5);
  \draw (0.53,1.25) node {$\cdots$};
  \end{scope}
  \begin{scope}[scale=0.75,yshift=-0.5cm]
  \draw (0.25,-0.5) [dnup=0];
  \draw (1,-0.5) [dndn=0.5];
  \draw (1,1.5) [upup=0.5];
  \draw (0.8,-0.5) [dnup=0];
  \draw (0.55,0.5) node {$\cdots$};
  \end{scope}
  \begin{scope}[scale=0.75,yshift=-2.5cm]
  \draw (0.5,0.5) node[scale=1.33] {$f^{(n)}$};
  \draw (0,0) rectangle (1.1,1);
  \draw (0.25,1.5)--(0.25,1);
  \draw (0.25,0)--(0.25,-0.5);
  \draw (0.8,0)--(0.8,-0.5);
  \draw (0.8,1.5)--(0.8,1);
  \draw (1,1)--(1,1.5);
  \draw (1,-0.5)--(1,0);
  \draw (1.5,1.5)--(1.5,-0.5);
  \draw (0.53,-0.25) node {$\cdots$};
  \end{scope}
  \end{scope}}
 \end{align*}
 And by the inductive hypothesis, we have $\left(f^{(n)}\right)^2=f^{(n)}$ and $\Ecal_{n}\left(f^{(n)}\right)=\dfrac{[n+1]}{[n]}f^{(n-1)}$
 \begin{align*}
  =\knotsinmath{\begin{scope}[scale=0.75,yshift=-0.5cm]
  \draw (0.5,0.5) node[scale=1.33] {$f^{(n)}$};
  \draw (0,0) rectangle (1,1);
  \draw (0.25,1.5)--(0.25,1);
  \draw (0.25,0)--(0.25,-0.5);
  \draw (0.75,0)--(0.75,-0.5);
  \draw (0.75,1.5)--(0.75,1);
  \draw (0.53,1.25) node {$\cdots$};
  \draw (0.53,-0.25) node {$\cdots$};
  \draw (1.25,1.5)--(1.25,-0.5);
 \end{scope}}-\frac{2[n]}{[n+1]}\knotsinmath{
  \begin{scope}[scale=0.75,yshift=1.5cm]
  \draw (0.5,0.5) node[scale=1.33] {$f^{(n)}$};
  \draw (0,0) rectangle (1.1,1);
  \draw (0.25,1.5)--(0.25,1);
  \draw (0.25,0)--(0.25,-0.5);
  \draw (0.8,0)--(0.8,-0.5);
  \draw (0.8,1.5)--(0.8,1);
  \draw (1,0)--(1,-0.5);
  \draw (1,1.5)--(1,1);
  \draw (1.5,1.5)--(1.5,-0.5);
  \draw (0.53,1.25) node {$\cdots$};
  \end{scope}
  \begin{scope}[scale=0.75,yshift=-0.5cm]
  \draw (0.25,-0.5) [dnup=0];
  \draw (1,-0.5) [dndn=0.5];
  \draw (1,1.5) [upup=0.5];
  \draw (0.8,-0.5) [dnup=0];
  \draw (0.55,0.5) node {$\cdots$};
  \end{scope}
  \begin{scope}[scale=0.75,yshift=-2.5cm]
  \draw (0.5,0.5) node[scale=1.33] {$f^{(n)}$};
  \draw (0,0) rectangle (1.1,1);
  \draw (0.25,1.5)--(0.25,1);
  \draw (0.25,0)--(0.25,-0.5);
  \draw (0.8,0)--(0.8,-0.5);
  \draw (0.8,1.5)--(0.8,1);
  \draw (1,1)--(1,1.5);
  \draw (1,-0.5)--(1,0);
  \draw (1.5,1.5)--(1.5,-0.5);
  \draw (0.53,-0.25) node {$\cdots$};
  \end{scope}}+\frac{[n]^2}{[n+1]^2}\cdot\frac{[n+1]}{[n]}\,\,\,\knotsinmath{\begin{scope}
   \begin{scope}[scale=0.75,yshift=2cm]
  \draw (0.5,0.5) node[scale=1.33] {$f^{(n)}$};
  \draw (0,0) rectangle (1.1,1);
  \draw (0.25,1.5)--(0.25,1);
  \draw (0.25,0)--(0.25,-0.5);
  \draw (0.8,0)--(0.8,-0.5);
  \draw (0.8,1.5)--(0.8,1);
  \draw (1,0)--(1,-0.5);
  \draw (1,1.5)--(1,1);
  \draw (1.5,1.5)--(1.5,-0.5);
  \draw (0.53,1.25) node {$\cdots$};
  \end{scope}
  \begin{scope}[scale=0.75,yshift=-0.5cm]
  \draw (0.65,0.5) node[scale=1.23] {$f^{(n-1)}$};
  \draw (0,0) rectangle (1.3,1);
  \draw (0.25,2)--(0.25,1);
  \draw (0.25,0)--(0.25,-1);
  \draw (0.8,0)--(0.8,-1);
  \draw (0.8,2)--(0.8,1);
  \draw (0.53,1.25) node {$\cdots$};
  \end{scope}
  \begin{scope}[scale=0.75,yshift=-1cm]
  \draw (1,-0.5) [dndn=0.5];
  \draw (1,2.5) [upup=0.5];
  \draw (0.55,0.5) node {$\cdots$};
  \end{scope}
  \begin{scope}[scale=0.75,yshift=-3cm]
  \draw (0.5,0.5) node[scale=1.33] {$f^{(n)}$};
  \draw (0,0) rectangle (1.1,1);
  \draw (0.25,1.5)--(0.25,1);
  \draw (0.25,0)--(0.25,-0.5);
  \draw (0.8,0)--(0.8,-0.5);
  \draw (0.8,1.5)--(0.8,1);
  \draw (1,1)--(1,1.5);
  \draw (1,-0.5)--(1,0);
  \draw (1.5,1.5)--(1.5,-0.5);
  \draw (0.53,-0.25) node {$\cdots$};
  \end{scope}
  \end{scope}}
 \end{align*}
 And again by the inductive hypothesis, we know $f^{(n)}\left(i_{n-1}\left(f^{(n-1)}\right)\right)=\left(i_{n-1}\left(f^{(n-1)}\right)\right)f^{(n)}=f^{(n)}$. And so by stretching the $E_{n}$ loops, it is easy to see that we get the same diagram as the second term. Equivalently we can use Proposition \ref{proposition:tldiagramrelations} to see (our last diagram)
 $i_n\left(f^{(n)}\right)E_nf^{(n-1)}E_ni_n\left(f^{(n)}\right)$ is equivalent to $i_n\left(f^{(n)}\right)E_ni_n\left(f^{(n)}\right)$. Thus by further simplifying, we find that it is in fact idempotent, i.e., $\left(f^{(n+1)}\right)^2=f^{(n+1)}$.
  \item
 We will now use the results from Proposition \ref{proposition:tldiagramrelations}. So,
 \begin{align*}
  \Ecal_{n+1}\left(f^{(n+1)}\right) &= \Ecal_{n+1}\left(i_n\left(f^{(n)}\right)-\frac{[n]}{[n+1]}i_n\left(f^{(n)}\right)E_ni_n\left(f^{(n)}\right)\right)\\
  &= [2]\id_n\left(f^{(n)}\right)+\frac{[n]}{[n+1]}f^{(n)}f^{(n)}\\
  &= [2]f^{(n)}+\frac{[n]}{[n+1]}f^{(n)}
 \end{align*}
 and using the result ``doubling the quantum integer'' from Proposition \ref{proposition:doubling_quantum_integer}, we find $[2][n+1]=[n+2]+[n]$, and so the result becomes $\dfrac{[n+2]}{[n+1]}f^{(n)}$.
  \item
   We use the results from Proposition \ref{proposition:tldiagramrelations},
   \begin{align*}
  f^{(n+1)}\left(i_n\left(f^{(n)}\right)\right) &= \left(i_n\left(f^{(n)}\right)-\frac{[n]}{[n+1]}i_n\left(f^{(n)}\right)E_ni_n\left(f^{(n)}\right)\right)\left(i_n\left(f^{(n)}\right)\right)\\
  &= \left(i_n\left(f^{(n)}\right)\right)\left(i_n\left(f^{(n)}\right)\right)-\frac{[n]}{[n+1]}i_n\left(f^{(n)}\right)E_ni_n\left(f^{(n)}\right)i_n\left(f^{(n)}\right)\\
  &= i_n\left(f^{(n)}\right)-\frac{[n]}{[n+1]}i_n\left(f^{(n)}\right)E_ni_n\left(f^{(n)}\right)\\
  &= f^{(n+1)}
 \end{align*}
  \item
   Now from (ii), we have $\Ecal_{n+1}\left(f^{(n+1)}\right)=\dfrac{[n+2]}{[n+1]}f^{(n)}$. So, then since $\tr_{n}\circ\Ecal_{n+1}=\tr_{n+1}$, we have
   \[\tr_{n+1}\left(f^{(n+1)}\right)=\tr_n\circ\Ecal_{n+1}\left(f^{(n+1)}\right)=\frac{[n+2]}{[n+1]}\tr_n\left(f^{(n)}\right)\]
   and by the inductive hypothesis, we have $\tr_n\left(f^{(n)}\right)=[n+1]$, and so
   \[
    \tr_{n+1}\left(f^{(n+1)}\right)=[n+2]
   \]
  \item
   Finally, by the inductive hypothesis, for any $j=1,\ldots{n-1}$ we have $E_jf^{(n+1)}=f^{(n+1)}E_j$. So let $j=n$, then
   \[ E_nf^{(n+1)}=E_ni_n\left(f^{(n)}\right)-\frac{[n]}{[n+1]}E_ni_n\left(f^{(n)}\right)E_ni_n\left(f^{(n)}\right) \]
   So, from Proposition \ref{proposition:tldiagramrelations}, we have $E_ni_n\left(f^{(n)}\right)E_n=\left(i_n\circ{i_{n-1}}\circ\Ecal_n\left(f^{(n)}\right)\right)E_n$. And from (ii) we know $\Ecal_n\left(f^{(n)}\right)=\dfrac{[n+1]}{[n]}f^{(n-1)}$, so then
   \begin{align*}
    E_nf^{(n+1)} &= E_ni_n\left(f^{(n)}\right)-\frac{[n]}{[n+1]}\left(i_n\circ{i_{n-1}}\circ\Ecal_n\left(f^{(n)}\right)\right)E_ni_n\left(f^{(n)}\right)\\
    &= E_ni_n\left(f^{(n)}\right)-\left(i_n\circ{i_{n-1}}\left(f^{(n-1)}\right)\right)E_ni_n\left(f^{(n)}\right)
   \end{align*}
   Now we have $i_n\circ{i_{n-1}}\left(f^{(n-1)}\right)E_ni_n\left(f^{(n)}\right)$ is equivalent to $E_ni_n\left(f^{(n)}\right)$, and so the result then follows.
 \end{enumerate}
\end{proof}

\begin{theorem}[{\cite[p.~31]{penneys}}]
 Suppose $\tr_j(y^*x)$ is positive and semi-definite (i.e., $\tr_j(x^*x)\geq0$ and $\tr_j(x^*x)=0$ if and only if $x^*x=0$) for all $j\geq0$, $x,y\in\TL_n$. Then
 \[q+q^{-1}\in\left\{2\cos(\pi/n):n\geq3\right\}\cup[2,\infty)\]
\end{theorem}
\begin{proof}
 We show this by contradiction. We will assume $q+q^{-1}\not\in\{2\cos(\pi/n):n\geq3\}\cup[2,\infty)$. Using our Jones-Wenzl idempotents, we want to show that $\tr_j((f^{(w)})^*f^{(w)})=[w+1]<0$ for some $w$.
 
 If we let $q=e^{i\theta}$ for some $\theta\in(0,\pi/2)$ such that $\theta\neq\pi/n$ for some $\pi\geq3$, then by recalling $\sin(x)=\dfrac{e^{ix}-e^{-ix}}{2i}$, we have
 \[[j]=\frac{(e^{i\theta})^j-(e^{i\theta})^{-j}}{(e^{i\theta})-(e^{-i\theta})}=\frac{\sin(j\theta)}{\sin(\theta)}\]
 Now by letting $k\geq2$ be minimal such that $\theta>\pi/(k+1)$, we have
 \[\frac{\pi}{k}>\theta>\frac{\pi}{k+1}\]
 then since $\sin(\theta),\sin(2\theta),\ldots,\sin(k\theta)>0$ but we have $\sin((k+1)\theta)<0$ so then $[1],[2],\ldots,[k]>0$, but $[k+1]<0$.
 
 So since $k\geq2$ and $[1],[2],\ldots,[k]\neq0$ we have our idempotents $f^{(0)},\ldots,f^{(k)}$ are well-defined. Thus, we get our contradiction by taking the trace $\tr_k\left({(f^{(k)})^*f^{(k)}}\right)=[k+1]<0$.
\end{proof}

\section{Jones' original construction}\label{chapter:jones}
We will now construct our Jones polynomial using braids and the Temperley-Lieb algebra. This discussion is based on that in \cite{Abramsky}, \cite{adams}, \cite{enc}, \cite{jones_quantum}, \cite{birman_survey}, \cite{cromwell}, \cite{jordan}, \cite{jones1985}, \cite{jonesbook}, \cite{jonespoly}, \cite{kauffman1986}, \cite{kauffman1988}, \cite{kauffman_stat_mech}, \cite{kauffman_lomonaco}, \cite{kawauchi}, \cite{manturov}, \cite{murasagibook}, \cite{penneys}, and \cite{speicher}.

To get invariance under conjugation (the Markov $\MI$ move), we need to define a trace. To define a trace, we need a representation of $\Bcal_n$ in a type II$_1$ factor since it is not finite (recall the motivation section \S\ref{section:motivation}).

\begin{proposition}
 The homomorphism $\phi_n\colon\CC[\Bcal_n]\to\TL_n$ which is defined by \[\phi_n(\sigma_i)=A(1_n)+A^{-1}E_i\quad\mbox{ and }\quad\phi_n(\sigma^{-1}_i)=A^{-1}(1_n)+AE_i,\]
 is a representation of $\Bcal_n$ (i.e., the homomorphism preserves the structure of $\Bcal_n$), where $d=([2]_q)_{-A^2=q}$.
\end{proposition}
\begin{proof}
 If $|i-j|\geq2$, then
 \begin{align*}
  \phi_n(\sigma_i\sigma_j) &= \phi_n(\sigma_i)\phi_n(\sigma_j)\\
  &= (A(1_n)+A^{-1}E_i)(A(1_n)+A^{-1}E_j)\\
  &= A^2+E_j+E_i+A^{-2}E_iE_j\\
  &= (A(1_n)+A^{-1}E_j)(A(1_n)+A^{-1}E_i)\\
  &= \phi_n(\sigma_j)\phi_n(\sigma_i)\\
  &= \phi_n(\sigma_j\sigma_i)
 \end{align*}
 since if $|i=j|>1$, then $E_iE_j=E_jE_i$.
 \begin{align*}
  \phi_n(\sigma_i\sigma_{i+1}\sigma_i) &= \phi_n(\sigma_i)\phi_n(\sigma_{i+1})\phi_n(\sigma_{i})\\
  &= (A(1_n)+A^{-1}E_i)(A(1_n)+A^{-1}E_{i+1})(A(1_n)+A^{-1}E_i)\\
  &= (A^2+E_{i+1}+E_i+A^{-2}E_iE_{i+1})(A(1_n)+A^{-1}E_i)\\
  &= A^3+2AE_i+AE_{i+1}+A^{-1}E_{i+1}E_i+A^{-1}E^2_i+A^{-1}E_iE_{i+1}+A^{-3}E_iE_{i+1}E_i\\
  &= A^3+2AE_i+AE_{i+1}+A^{-1}E_{i+1}E_i+A^{-1}d{E_i}+A^{-1}E_iE_{i+1}+A^{-3}{E_i}\\
  &= A^3+E_i(2A+A^{-1}d+A^{-3})+AE_{i+1}+A^{-1}E_{i+1}E_i+A^{-1}E_iE_{i+1}
 \end{align*}
 Note that
 \[\phi_n(\sigma_{i+1}\sigma_i\sigma_{i+1})=A^3+E_{i+1}(2A+A^{-1}d+A^{-3})+AE_i+A^{-1}E_iE_{i+1}+A^{-1}E_{i+1}E_i\]
 So, we need $d=-A^2-A^{-2}=([2]_q)_{-A^2=q}$.
 
 Now
 \begin{align*}
  \phi_n(\sigma_i)\phi_n(\sigma_i^{-1}) &= (A(1_n)+A^{-1}E_i)(A^{-1}(1_n)+AE_i)\\
  &= 1_n+A^2E_i+A^{-2}E_i+E_i^2\\
  &= 1_n+E_i(A^2+A^{-2})+d{E_i} = 1_n
 \end{align*}
\end{proof}

The homomorphism $\phi_n$ defined above can be thought of as resolving crossings in braids to produce $\TL_n$ diagrams, i.e.,
\[\phi_n(\sigma_i)=\phi_n\left(\KPE\right)=A\KPD+A^{-1}\KPC\] and \[\phi_n(\sigma_i^{-1})=\phi_n\left(\KPB\right)=A\KPC+A^{-1}\KPD\]


We need to take a composition of the maps $\phi_n$ and the Markov trace $\tr_n$ to map from $\CC[\Bcal_n]\to\CC(q)$.
Now let us explore the map $\tr_n\circ\phi_n$.
\begin{proposition}
 $\tr_n\circ\phi_n(\KPE)=-A^3d$, while $\tr_n\circ\phi_n(\KPB)=-A^{-3}d$.
\end{proposition}
\begin{proof}
 \begin{align*}
  \tr_n\circ\phi_n\left(\KPE\right) &= A\KPA\KPA+A^{-1}\KPA\\
  &= Ad^2+A^{-1}d=A(d)^2-A^{-1}d\\
  &=A(-A^{-2}-A^{2})^2+A^{-1}(-A^{-2}-A^2)\\
  &= A^5+A=-A^3d
 \end{align*}
 And 
 \begin{align*}
  \tr_n\circ\phi_n\left(\KPB\right) &= A\KPA+A^{-1}\KPA\KPA\\
  &= Ad+A^{-1}d^2\\
  &= A^{-5}+A^{-1}=-A^{-3}d
 \end{align*}
\end{proof}

Now $d^{-1}\tr_n\circ\phi_n$ is in fact the braid analog of the Kauffman bracket\index{Kauffman bracket!via braids}. This can easily be seen with the state sum solution. The states of a braid closure is given by resolving the crossings in the braid, then taking the trace. Thus
\[\langle{\reallywidehat{b}}\rangle=d^{-1}\tr_n\circ\phi_n(b)=\sum_{\text{states }s}A^{\alpha(s)}d^{|s|-1}\]
Thus we can then define the Jones polynomial to be the braid analog of the Kauffman bracket multiplied by $(-A^{-3})^{\omega(\reallywidehat{b})}$.

However, we want to show that it is indeed a link invariant without the Kauffman bracket. 
So we have the identity map of a braid $\id_n\colon\Bcal_n\to\Bcal_n$. Note that we have the following relations,
\begin{enumerate}[(i)]
 \item $\phi_{n+1}\circ\incl_n(\alpha)=i_n\circ\phi_n(\alpha)$, where $\alpha\in\Bcal_n$
 \item $\id_n\circ\phi_n(\alpha)=\phi_n\circ\id_n(\alpha)$, where $\alpha\in\Bcal_n$
\end{enumerate}
This is true since, for instance, in the first relation, the inclusion of a braid would not change the $\TL_n$ representation. We draw the first relation,
 \[
  \phi_{n+1}\circ\incl_n(b)=\braidinclusion{$\phi_n(b)$}=i_n\circ\phi_n(b)
 \]
So we can thus use the relations from Proposition \ref{proposition:tldiagramrelations}.

Since any braid word $\alpha\in\Bcal_n$ can be written with braid generators, we have the \textit{writhe}\index{writhe!braid} defined on braids as the sum of the exponents of the generators.
\[\omega(\sigma_1^{i_1}\sigma_2^{i_2}\cdots\sigma_{n-1}^{i_j})=\sum_zi_z\]
Equivalently, the writhe could also be defined as the writhe of the closure of the braid since an orientation is given with the braid closure.

\begin{theorem}[The original Jones polynomial]\index{the Jones polynomial!original construction}
 For a braid $b\in\Bcal_n$. We can define the Jones polynomial link invariant as the following
 \[
  V(L)=V(\reallywidehat{b})=\frac{(-A^{-3})^{\omega(b)}\tr_n\circ\phi_n(b)}{d}.
 \]
\end{theorem}
\begin{proof}
 We will prove this in two ways for the sake of argument and not necessity.
 
 We will first show that this is in fact an invariant under the Markov moves, which in turn shows that this is invariant under the Reidemeister moves, making it a link invariant. And the second proof will show that this is obviously equivalent to our definitions of the Jones polynomial (Definition \ref{definition:Jones-polynomial-bracket} \& Proposition \ref{proposition:Jones-polynomial-skein}).

 \textbf{Invariance under Markov moves:}
 \begin{description}
  \item[M1]
   Given any braid $\gamma\in\Bcal_n$, we have $\omega(\gamma)=-\omega(\gamma^{-1})$. So $\omega(\gamma{b}\gamma^{-1})=\omega(\gamma)+\omega(b)+\omega(\gamma^{-1})=\omega(b)$. So,
   \begin{align*}
    V(\reallywidehat{\gamma{b}\gamma^{-1}}) &= \frac{(-A^{-3})^{\omega(\gamma{b}\gamma^{-1})}\tr_n\circ\phi_n(\gamma{b}\gamma^{-1})}{d}\\
    &= \frac{(-A^{-3})^{\omega(b)}\tr_n(\phi_n(\gamma)\phi_n(b)\phi_n(\gamma)^{-1})}{d}\\
    &= \frac{(-A^{-3})^{\omega(b)}\tr_n(\phi_n(\gamma)^{-1}\phi_n(\gamma)\phi_n(b))}{d} = \frac{(-A^{-3})^{\omega(b)}\tr_n(\phi_n(b))}{d} = V(\reallywidehat{b})
   \end{align*}
  \item[M2]
   Now given any braid $\sigma_n\in\Bcal_{n+1}$, we have $\omega(\sigma_n)=1$, and $\phi_{n+1}(\sigma_n)=A1_{n+1}+A^{-1}E_n$. We also note $\omega(\incl_n(b))=\omega(b)$, so then
   \begin{align*}
    V(\reallywidehat{\incl_n(b)\sigma_n}) &= \frac{(-A^{-3})^{\omega(b\sigma_n)}\tr_{n+1}\circ\phi_{n+1}(\incl_n(b)\sigma_n)}{d}\\
    &= \frac{(-A^{-3})^{\omega(b)+1}\tr_{n+1}(\phi_{n+1}(\incl_n(b))(A+A^{-1}E_n))}{d}
   \end{align*}
   Recall the above relation $\phi_{n+1}\circ\incl_n=i_n\circ\phi_n$.
   We also know from Proposition \ref{proposition:tldiagramrelations} that the Markov trace $\tr_{n+1}=\tr_n\circ\Ecal_{n+1}$. So, then $\tr_n\circ\Ecal_{n+1}(A\cdot{i_n(\phi_n(b))}+A^{-1}\cdot{i_n(\phi_n(b))}\cdot{E_n})$. And again from Proposition \ref{proposition:tldiagramrelations}, we know $\tr_n\circ\Ecal_{n+1}\circ{i_n}=d\cdot\tr_n\circ\id_n$, and $\tr_n(\Ecal_{n+1}(i_n(\phi_n(b))\cdot{E_n}))=\tr_n(\phi_n(b))$. So,
   \begin{align*}
    &= \frac{(-A^{-3})^{\omega(b)+1}\tr_n(\phi_n(b))(Ad+A^{-1})}{d}\\
    &= \frac{(-A^{-3})^{\omega(b)+1}\tr_n(\phi_n(b))(-A^{-3})}{d}\\
    &= \frac{(-A^{-3})^{\omega(b)}\tr_n(\phi_n(b))}{d} = V(\reallywidehat{b})
   \end{align*}
   Analogously to the above, we find that under M2, where the braid becomes $b\sigma_n^{-1}$, the Jones polynomial $V(\reallywidehat{b\sigma_n^{-1}})=V(\reallywidehat{b})$.
 \end{description}
 Thus, we are done.
 
 \textbf{Skein relation:}
 
 Let us now ensure that this satisfies the skein relation definition.
 Consider single crossings in the braid closure,
 \begin{align*}
  A^{4}V\left(\KPE\right)-A^{-4}V\left(\KPB\right)+(A^2-A^{-2})V\left(\KPD\right) &= A^4\frac{(-A^{-3})(-A^3d)}{d}\\
   &\qquad-A^{-4}\frac{(-A^{-3})^{-1}(-A^{-3}d)}{d}+(A^2-A^{-2})\frac{d^2}{d}\\
  &= A^4-A^{-4}+(A^2-A^{-2})(-A^{-2}-A^2) = 0.
 \end{align*}
 This completes the proof.
\end{proof}

\begin{example}\label{example:og_Jones_3_1}
 For example, recall that the closure of $\sigma_1^3$ is equivalent to the right-handed trefoil knot, while $\sigma_1^{-3}$ is equivalent to the left-handed trefoil knot.
 
 We can see then, that
 \begin{align*}
  V(3_1)&=V(\reallywidehat{\sigma_1^3})\\
  &=\frac{(-A^{-3})^3\tr_2\circ\phi_2(\sigma_1^3)}{d}=\frac{-A^{-9}\tr_2(\phi_2(\sigma_1)^3)}{d}\\
  &=\frac{-A^{-9}\tr_2(A+A^{-1}E_1)^3}{d}\\
  &=\frac{-A^{-9}\tr_2(A^3+3AE_1+3A^{-1}E_1^2+A^{-3}E_1^3)}{d}\\
  &=\frac{-A^{-9}(A^3d^2+3Ad+3A^{-1}d^2+A^{-3}d^3)}{d}\\
  &=-A^{-6}d-3A^{-8}-3A^{-10}d-A^{-12}d^2\\
  &=(-A^2-A^{-2})(-A^{-6}-3A^{-10}-A^{-12}(-A^2-A^{-2}))-3A^{-8}\\
  &=A^{-4}+A^{-12}-A^{-16}
 \end{align*}
 
 It is clear that the left-handed trefoil $V(\overline{3_1})=V(\reallywidehat{\sigma_1^{-3}})$, and the bracket polynomial would be $\phi_2(\sigma_1^{-1})^3=(AE_1+A^{-1})^3=A^3E_1^3+3AE_1^2+3A^{-1}E_1+A^{-3}$, and the Markov trace of the bracket polynomial becomes $A^3d^3+3Ad^2+3A^{-1}d+A^{-3}d^2$. And finally we have the Jones polynomial, $V(\overline{3_1})=-A^{9}(d(A^3d+3A+A^{-3})+3A^{-1}$, which further simplified turns out to be $A^{4}+A^{12}-A^{16}$.
\end{example}


To conclude, we see that the easiest to define (i.e., via the Kauffman bracket) is not necessarily the easiest to work with, while the most difficult to define (i.e., the original construction) is perhaps the easiest to work with.


\subsection{Proving Theorem \ref{theorem:disjoint_links}}\label{section:proving_theorem}

In this section we will prove $V(L_1\cup{L_2})=dV(L_1)V(L_2)$ and $V(L_1\#L_2)=V(L_1)V(L_2)$, where $\#$ is the connected sum of the links. We will prove these results using braid representations instead of the skein relation and Kauffman bracket (where they can be derived easily). The reason we do this with braid representations is so that we can ensure comprehension.

We saw in Chapter \ref{chapter:braids}, that the braid inclusion map essentially adds a strand to the right. So now we define the braid left-inclusion map $\lfincl_n$. The reason for this is to define the tensor product of two braids.
\begin{definition}[braid left-inclusion]
 The \textit{braid left-inclusion map}\index{linear operations on $\Bcal_n$!braid left-inclusion map} $\lfincl_n\colon\Bcal_n\to\Bcal_{n+1}$ for $n\geq1$ is defined by,
 \begin{figure}[H] \centering\[\lfincl_n\left(\knotsinmath{\begin{scope}[yshift=-0.35cm,scale=0.75]
    \draw (0,0) rectangle (1,1);
    \draw (0.25,1.5)--(0.25,1);
    \draw (0.25,0)--(0.25,-0.5);
    \draw (0.75,0)--(0.75,-0.5);
    \draw (0.75,1.5)--(0.75,1);
    \draw (0.53,1.25) node {$\cdots$};
    \draw (0.53,-0.25) node {$\cdots$};
    \draw (0.5,0.5) node[scale=1.33] {$\alpha$};
   \end{scope}}\right) =\,\, \knotsinmath{\begin{scope}[yshift=-0.35cm,scale=0.75]
    \draw (0,0) rectangle (1,1);
    \draw (0.25,1.5)--(0.25,1);
    \draw (0.25,0)--(0.25,-0.5);
    \draw (0.75,0)--(0.75,-0.5);
    \draw (0.75,1.5)--(0.75,1);
    \draw (0.53,1.25) node {$\cdots$};
    \draw (0.53,-0.25) node {$\cdots$};
    \draw (0.5,0.5) node[scale=1.33] {$\alpha$};
    \draw (-0.25,1.5)--(-0.25,-0.5);
   \end{scope}}\]\caption{left inclusion $\lfincl_n$ on $\Bcal_n$}\end{figure}
\end{definition}

We similarly define the \textit{left-inclusion}\index{linear operations on $\TL_n$!left-inclusion!$\TL_n$} on $\TL_n$, i.e., $li\colon\TL_n\to\TL_{n+1}$.

\begin{definition}[braid tensor product]
 Let $\alpha\in\Bcal_n$ and $\beta\in\Bcal_m$. We define the \textit{tensor product}\index{linear operations on $\TL_n$!braid tensor product} $\otimes\colon\Bcal_n\times\Bcal_m\to\Bcal_{n+m}$ by $\alpha\otimes\beta\mapsto\incl_{n\otimes{m}}(\alpha)\cdot\lfincl_{m\otimes{n}}(\beta)$, where the inclusion tensor product of the two integers is given by $\incl_{n\otimes{m}}=\incl_{n+m-1}\circ\incl_{n+m-2}\circ\cdots\circ\incl_{n}$ (the left inclusion is defined analogously). And so we thus have.
  \begin{figure}[H]
  \centering
  \begin{tikzpicture}[scale=0.75]
   \begin{scope}[xshift=-3.5cm]
    \draw (0,0) rectangle (1,1);
    \draw (0.25,1.5)--(0.25,1);
    \draw (0.25,0)--(0.25,-0.5);
    \draw (0.75,0)--(0.75,-0.5);
    \draw (0.75,1.5)--(0.75,1);
    \draw (0.53,1.25) node {$\cdots$};
    \draw (0.53,-0.25) node {$\cdots$};
    \draw (0.5,0.5) node[scale=1.33] {$\alpha$};
   \end{scope}
   \begin{scope}[xshift=-2cm]
    \draw (-0.25,0.5) node[scale=1.33] {$\otimes$};
    \draw (0,0) rectangle (1,1);
    \draw (0.25,1.5)--(0.25,1);
    \draw (0.25,0)--(0.25,-0.5);
    \draw (0.75,0)--(0.75,-0.5);
    \draw (0.75,1.5)--(0.75,1);
    \draw (0.53,1.25) node {$\cdots$};
    \draw (0.53,-0.25) node {$\cdots$};
    \draw (0.5,0.5) node[scale=1.33] {$\beta$};
   \end{scope}
   \draw (-0.5,0.5) node[scale=1.33] {$=$};
   \draw (0,0) rectangle (1,1);
   \draw (0.25,1.5)--(0.25,1);
   \draw (0.25,0)--(0.25,-0.5);
   \draw (0.75,0)--(0.75,-0.5);
   \draw (0.75,1.5)--(0.75,1);
   \draw (0.53,1.25) node {$\cdots$};
   \draw (0.53,-0.25) node {$\cdots$};
   \draw (0.5,0.5) node[scale=1.33] {$\alpha$};
   \begin{scope}[xshift=1.2cm]
    \draw (0,0) rectangle (1,1);
    \draw (0.25,1.5)--(0.25,1);
    \draw (0.25,0)--(0.25,-0.5);
    \draw (0.75,0)--(0.75,-0.5);
    \draw (0.75,1.5)--(0.75,1);
    \draw (0.53,1.25) node {$\cdots$};
    \draw (0.53,-0.25) node {$\cdots$};
    \draw (0.5,0.5) node[scale=1.33] {$\beta$};
   \end{scope}
  \end{tikzpicture}
  \caption{braid tensor product}
 \end{figure}
\end{definition}

So it is easy to see that the closure of the tensor product of two braids is essentially two disjoint links. For instance, let the two links have the following braid representations, $L_1=\reallywidehat{\alpha}$ and $L_2=\reallywidehat{\beta}$, so \[L_1\cup{L_2}=\reallywidehat{\alpha}\cup\reallywidehat{\beta}=\reallywidehat{\alpha\otimes\beta}\]

\begin{remark}
 It is easy to see that the union of two closure braids is well-defined, i.e., $\reallywidehat{\alpha\otimes\beta}\sim\reallywidehat{\beta\otimes\alpha}$.
\end{remark}

We then analogously define the \textit{tensor product}\index{linear operations on $\TL_n$!tensor product} on $\TL_n$, i.e., $\otimes\colon{\TL_n\times{\TL_m}}\to{\TL_{n+m}}$,
 \begin{center}
  \begin{tikzpicture}[scale=0.75]
  \begin{scope}[xshift=-3.5cm]
    \draw (0,0) rectangle (1,1);
    \draw (0.25,1.5)--(0.25,1);
    \draw (0.25,0)--(0.25,-0.5);
    \draw (0.75,0)--(0.75,-0.5);
    \draw (0.75,1.5)--(0.75,1);
    \draw (0.53,1.25) node {$\cdots$};
    \draw (0.53,-0.25) node {$\cdots$};
    \draw (0.5,0.5) node[scale=1.33] {$x$};
  \end{scope}
  \begin{scope}[xshift=-2cm]
    \draw (-0.25,0.5) node[scale=1.33] {$\otimes$};
    \draw (0,0) rectangle (1,1);
    \draw (0.25,1.5)--(0.25,1);
    \draw (0.25,0)--(0.25,-0.5);
    \draw (0.75,0)--(0.75,-0.5);
    \draw (0.75,1.5)--(0.75,1);
    \draw (0.53,1.25) node {$\cdots$};
    \draw (0.53,-0.25) node {$\cdots$};
    \draw (0.5,0.5) node[scale=1.33] {$y$};
  \end{scope}
  \draw (-0.5,0.5) node[scale=1.33] {$=$};
  \draw (0,0) rectangle (1,1);
  \draw (0.25,1.5)--(0.25,1);
  \draw (0.25,0)--(0.25,-0.5);
  \draw (0.75,0)--(0.75,-0.5);
  \draw (0.75,1.5)--(0.75,1);
  \draw (0.53,1.25) node {$\cdots$};
  \draw (0.53,-0.25) node {$\cdots$};
  \draw (0.5,0.5) node[scale=1.33] {$x$};
  \begin{scope}[xshift=1.2cm]
    \draw (0,0) rectangle (1,1);
    \draw (0.25,1.5)--(0.25,1);
    \draw (0.25,0)--(0.25,-0.5);
    \draw (0.75,0)--(0.75,-0.5);
    \draw (0.75,1.5)--(0.75,1);
    \draw (0.53,1.25) node {$\cdots$};
    \draw (0.53,-0.25) node {$\cdots$};
    \draw (0.5,0.5) node[scale=1.33] {$y$};
  \end{scope}
  \end{tikzpicture}
 \end{center}
 
We have the relation $\phi_{n+m}(\alpha\otimes\beta)=\phi_n(\alpha)\otimes\phi_m(\beta)$, where $\alpha\in\Bcal_n$ and $\beta\in\Bcal_m$. This obviously says something along the lines of, ``if two braids are disjoint, then so are their Temperley-Lieb representations''.

\begin{lemma}
 Let $\alpha\in\Bcal_n$ and $\beta\in\Bcal_m$, then we have $\tr_{n+m}\circ\phi_{n+m}(\alpha\otimes\beta)=\tr_n\circ\phi_n(\alpha)\cdot\tr_m\circ\phi_m(\beta)$.
\end{lemma}
\begin{proof}[Proof sketch]
 The diagrammatic sketch of this is pretty straight forward. Essentially, we need to show that the trace of disjoint $\TL$ diagrams are independent of each other. Note the relation above, $\phi_{n+m}(\alpha\otimes\beta)=\phi_n(\alpha)\otimes\phi_m(\beta)$.
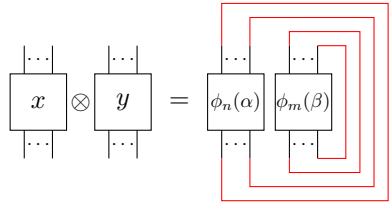
 \begin{center}
  \begin{tikzpicture}[scale=0.75]
   \begin{scope}[xshift=-3.5cm]
    \draw (0,0) rectangle (1,1);
    \draw (0.25,1.5)--(0.25,1);
    \draw (0.25,0)--(0.25,-0.5);
    \draw (0.75,0)--(0.75,-0.5);
    \draw (0.75,1.5)--(0.75,1);
    \draw (0.53,1.25) node {$\cdots$};
    \draw (0.53,-0.25) node {$\cdots$};
    \draw (0.5,0.5) node[scale=1.33] {$x$};
   \end{scope}
   \begin{scope}[xshift=-2cm]
    \draw (-0.25,0.5) node[scale=1.33] {$\otimes$};
    \draw (0,0) rectangle (1,1);
    \draw (0.25,1.5)--(0.25,1);
    \draw (0.25,0)--(0.25,-0.5);
    \draw (0.75,0)--(0.75,-0.5);
    \draw (0.75,1.5)--(0.75,1);
    \draw (0.53,1.25) node {$\cdots$};
    \draw (0.53,-0.25) node {$\cdots$};
    \draw (0.5,0.5) node[scale=1.33] {$y$};
   \end{scope}
   \draw (-0.5,0.5) node[scale=1.33] {$=$};
   \draw (0,0) rectangle (1,1);
   \draw (0.25,1.5)--(0.25,1);
   \draw (0.25,0)--(0.25,-0.5);
   \draw (0.75,0)--(0.75,-0.5);
   \draw (0.75,1.5)--(0.75,1);
   \draw (0.53,1.25) node {$\cdots$};
   \draw (0.53,-0.25) node {$\cdots$};
   \draw (0.5,0.5) node[scale=1] {$\phi_n(\alpha)$};
   \begin{scope}[xshift=1.2cm]
    \draw (0,0) rectangle (1,1);
    \draw (0.25,1.5)--(0.25,1);
    \draw (0.25,0)--(0.25,-0.5);
    \draw (0.75,0)--(0.75,-0.5);
    \draw (0.75,1.5)--(0.75,1);
    \draw (0.53,1.25) node {$\cdots$};
    \draw (0.53,-0.25) node {$\cdots$};
    \draw (0.5,0.5) node[scale=1] {$\phi_m(\beta)$};
   \end{scope}
   \draw[red] (1.95,1.5)--(2.45,1.5)--(2.45,-0.5)--(1.95,-0.5);
   \draw[red] (1.45,1.5)--(1.45,1.75)--(2.7,1.75)--(2.7,-0.75)--(1.45,-0.75)--(1.45,-0.5);
   \draw[red] (0.75,1.5)--(0.75,2)--(2.95,2)--(2.95,-1)--(0.75,-1)--(0.75,-0.5);
   \draw[red] (0.25,1.5)--(0.25,2.25)--(3.2,2.25)--(3.2,-1.25)--(0.25,-1.25)--(0.25,-0.5);
  \end{tikzpicture}
 \end{center}
 Which is clearly the same as $\tr_n\circ\phi_n(\alpha)\cdot\tr_m\circ\phi_m(\beta)$.
\end{proof}

\begin{proof}[Proof of Theorem \ref{theorem:disjoint_links}]
 Since $L_1\cup{L_2}=\reallywidehat{\alpha\otimes\beta}$ for $\alpha\in\Bcal_n$ and $\beta\in\Bcal_m$, we then have
 \begin{align*}
  V\left(\reallywidehat{\alpha\otimes\beta}\right) &= \frac{(-A^{-3})^{\omega(\alpha)+\omega(\beta)}\tr_{n+m}\circ\phi_{n+m}(\alpha\otimes\beta)}{d}\\
  &= \frac{(-A^{-3})^{\omega(\alpha)+\omega(\beta)}\tr_n\circ\phi_n(\alpha)\cdot\tr_m\phi_m(\beta)}{d}
 \end{align*}
 And so it is then clear that multiplying by $dd^{-1}$ will give us the following,
 \[ V(L_1\cup{L_2}) = dV(L_1)V(L_2) \]
 And recall $d=[2]_{q=-A^2=-t^{-1/2}}$, and so the result follows.
\end{proof}


\begin{definition}[connected sum of links]
 $L_1\#L_2$ is the \textit{connected sum}\index{connected sum of links} of the links, and is shown below
 \begin{figure}[H]
  \centering
  \begin{tikzpicture}
   \draw (0,0) rectangle (1,1);
   \draw[<-] (1,0.25)--(1.25,0.25)--(1.25,0.75)--(1,0.75);
   \draw (1.5,0.5) node {$\#$};
   \draw (0.5,0.5) node {$L_1$};
   \begin{scope}[xshift=2cm]
    \draw (0,0) rectangle (1,1);
    \draw[->] (0,0.25)--(-0.25,0.25)--(-0.25,0.75)--(0,0.75);
    \draw (0.5,0.5) node {$L_2$};
   \end{scope}
   \draw (3.5,0.5) node {$=$};
   \draw (4,0) rectangle (5,1);
   \draw (5.5,0) rectangle (6.5,1);
   \draw[<-] (5,0.25)--(5.5,0.25);
   \draw[->] (5,0.75)--(5.5,0.75);
   \draw (4.5,0.5) node {$L_1$};
   \draw (6,0.5) node {$L_2$};
  \end{tikzpicture}
  \caption{connected sum of links $\#$}
 \end{figure}
\end{definition}
It is easy to see that this is well-defined.

An almost direct corollary one can make is the following.
\begin{corollary}
 For links $L_1,L_2$, we have $V(L_1\#L_2)=V(L_1)V(L_2)$.
\end{corollary}
\begin{proof}
 We ``glue'' our link diagrams such that $T_+\sim{T_-}\sim{L_1\#{L_2}}$.
 
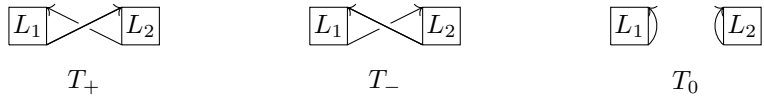
\begin{center}
  \begin{tikzpicture}[scale=1]
  \begin{scope}
    \draw[->] (1.5,0)--(0.5,0.5);
    \draw[ultra thick,draw=white,double=black,double distance=0.5pt] (0.5,0)--(1.5,0.5);
    \draw[->] (0.5,0)--(1.5,0.5);
    \draw (0,0)--(0.5,0)--(0.5,0.5)--(0,0.5)--(0,0);
    \draw (1.5,0)--(2,0)--(2,0.5)--(1.5,0.5)--(1.5,0);
    \draw (0.25,0.25) node {$L_1$};
    \draw (1.75,0.25) node {$L_2$};
    \draw (1,-0.5) node {$T_+$};
  \end{scope}
  \begin{scope}[xshift=4cm]
    \draw[->] (0.5,0)--(1.5,0.5);
    \draw[ultra thick,draw=white,double=black,double distance=0.5pt] (1.5,0)--(0.5,0.5);
    \draw[->] (1.5,0)--(0.5,0.5);
    \draw (0,0)--(0.5,0)--(0.5,0.5)--(0,0.5)--(0,0);
    \draw (1.5,0)--(2,0)--(2,0.5)--(1.5,0.5)--(1.5,0);
    \draw (0.25,0.25) node {$L_1$};
    \draw (1.75,0.25) node {$L_2$};
    \draw (1,-0.5) node {$T_-$};
  \end{scope}
  \begin{scope}[xshift=8cm]
    \draw[->] (0.5,0) to [bend right=50] (0.5,0.5);
    \draw[->] (1.5,0) to [bend left=50] (1.5,0.5);
    \draw (0,0)--(0.5,0)--(0.5,0.5)--(0,0.5)--(0,0);
    \draw (1.5,0)--(2,0)--(2,0.5)--(1.5,0.5)--(1.5,0);
    \draw (0.25,0.25) node {$L_1$};
    \draw (1.75,0.25) node {$L_2$};
    \draw (1,-0.5) node {$T_0$};
  \end{scope}
  \end{tikzpicture}
 \end{center}
 It is clear that we have $T_0\sim{L_1\cup{L_2}}$, so then
 \begin{align*}
  V(L_1\#L_2)(t^{-1}-t)+d(t^{-1/2}-t^{1/2})V(L_1)V(L_2) &= 0
 \end{align*}
 And so $V(L_1\#L_2)=V(L_1)V(L_2)$.
\end{proof}

We now show this result using braid representations. We first need to define what the connected sum of braids is.
\begin{definition}[braid connected sum]\index{linear operations on $\Bcal_n$!braid connected sum}
 If $L_1=\reallywidehat{\alpha}$ and $L_2=\reallywidehat{\beta}$ for $\alpha\in\Bcal_n$ and $\beta\in\Bcal_m$. We have $L_1\#L_2=\reallywidehat{\alpha\clubsuit\beta}$, where $\clubsuit\colon\Bcal_n\times\Bcal_m\to\Bcal_{n+m-1}$ is given by
 \[ \alpha\clubsuit\beta=\incl_{n\otimes{(m-1)}}(\alpha)\cdot\lfincl_{m\otimes{(n-1)}}(\beta) \]
 \begin{figure}[H]
  \centering
  \begin{tikzpicture}[scale=0.75]
   \draw (0,0) rectangle (1,1);
   \draw (1.75,0) rectangle (2.75,1);
   \draw (0.25,0)--(0.25,-0.5);
   \draw (0.75,0)--(0.75,-0.5);
   \draw (0.25,1)--(0.25,1.5);
   \draw (0.75,1)--(0.75,1.5);
   \draw (2,0)--(2,-0.5);
   \draw (2.5,0)--(2.5,-0.5);
   \draw (2,1)--(2,1.5);
   \draw (2.5,1)--(2.5,1.5);
   \draw (1.375,0.5) node[scale=1.33] {$\clubsuit$};
   \draw (3,0.5) node[scale=1.33] {$=$};
   \draw (3.25,2) rectangle (4.25,1);
   \draw (3.75,0) rectangle (4.75,-1);
   \draw (3.5,1)--(3.5,-1.25);
   \draw (4,1)--(4,0);
   \draw (4,-1)--(4,-1.25);
   \draw (4.5,-1)--(4.5,-1.25);
   \draw (4.5,0)--(4.5,2.25);
   \draw (3.5,2)--(3.5,2.25);
   \draw (4,2)--(4,2.25);
   \draw (0.5,0.5) node[scale=1.33] {$\alpha$};
   \draw (2.25,0.5) node[scale=1.33] {$\beta$};
   \draw (3.75,1.5) node[scale=1.33] {$\alpha$};
   \draw (4.25,-0.5) node[scale=1.33] {$\beta$};
   \draw (0.53,1.25) node {$\cdots$};
   \draw (2.28,1.25) node {$\cdots$};
   \draw (3.78,2.2) node {$\cdots$};
   \draw (4.28,-1.2) node {$\cdots$};
   \draw (0.53,-0.25) node {$\cdots$};
   \draw (2.28,-0.25) node {$\cdots$};
   \draw (3.78,0.5) node {$\cdots$};
   \draw (4.28,0.5) node {$\cdots$};
  \end{tikzpicture}
  \caption{braid connected sum $\clubsuit$}
 \end{figure}
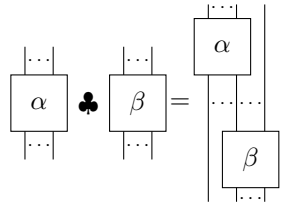
\end{definition}

This is well-defined, as the closure of $\alpha\clubsuit\beta$ can easily be seen to be ambient isotopic to the closure of $\beta\clubsuit\alpha$.
 
So now we have the following theorem.
\begin{theorem}\label{theorem:Jones_clubsuit}
 For $\alpha\in\Bcal_n$ and $\beta\in\Bcal_m$ we have,
 \[V(\reallywidehat{\alpha\clubsuit\beta})=V(\reallywidehat{\alpha})V(\reallywidehat{\beta})\]
\end{theorem}

To prove this, we begin with a lemma.
\begin{lemma}
 The braid closure of $\alpha\clubsuit\beta$ is equivalent to the braid closure of $\alpha\heartsuit\beta$ where $\heartsuit\colon\Bcal_n\times\Bcal_m\to\Bcal_{n+m}$ is given by $\alpha\heartsuit\beta=\sigma_{n}\cdot(\alpha\otimes\beta)$, where $\sigma_{n}\in\Bcal_{n+m}$, $\alpha\in\Bcal_n$ and $\beta\in\Bcal_m$,
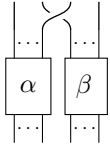
   \begin{center}
    \begin{tikzpicture}[scale=0.75]
     \braid[strands=4] {\sigma_2}
     \draw (0.35,-1) rectangle (1.15,-2);
     \draw (0.5,-0.5)--(0.5,-1);
     \draw (1,-0.5)--(1,-1);
     \draw (0.5,-2)--(0.5,-2.5);
     \draw (1,-2)--(1,-2.5);
     \draw (1.4,-1) rectangle (2.15,-2);
     \draw (1.5,-0.5)--(1.5,-1);
     \draw (2,-0.5)--(2,-1);
     \draw (1.5,-2.5)--(1.5,-2);
     \draw (2,-2.5)--(2,-2);
     \draw (0.75,-1.5) node[scale=1.33] {$\alpha$};
     \draw (1.735,-1.5) node[scale=1.33] {$\beta$};
     \draw (0.78,-0.75) node {$\cdots$};
     \draw (0.78,-2.25) node {$\cdots$};
     \draw (1.78,-0.75) node {$\cdots$};
     \draw (1.78,-2.25) node {$\cdots$};
    \end{tikzpicture}
   \end{center}
\end{lemma}
\begin{proof}
 This is obviously well defined.
 We will show this diagrammatically to ensure that the braid closures are in fact ambient isotopic.
 
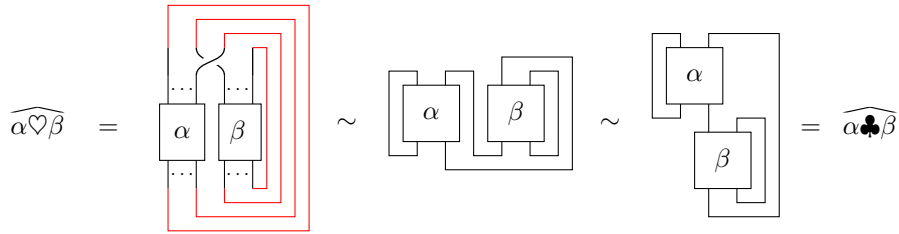
\begin{center}
  \begin{tikzpicture}
   \begin{scope}[scale=0.75]
   \braid[strands=4] {\sigma_2}
   \draw (0.35,-1) rectangle (1.15,-2);
   \draw (0.5,-0.5)--(0.5,-1);
   \draw (1,-0.5)--(1,-1);
   \draw (0.5,-2)--(0.5,-2.5);
   \draw (1,-2)--(1,-2.5);
   \draw (1.4,-1) rectangle (2.15,-2);
   \draw (1.5,-0.5)--(1.5,-1);
   \draw (2,-0.5)--(2,-1);
   \draw (1.5,-2.5)--(1.5,-2);
   \draw (2,-2.5)--(2,-2);
   \draw (0.75,-1.5) node[scale=1.33] {$\alpha$};
   \draw (1.735,-1.5) node[scale=1.33] {$\beta$};
   \draw (0.78,-2.25) node {$\cdots$};
   \draw (0.78,-0.75) node {$\cdots$};
   \draw (1.765,-2.25) node {$\cdots$};
   \draw (1.765,-0.75) node {$\cdots$};
   \draw[red] (2,0)--(2.25,0)--(2.25,-2.5)--(2,-2.5);
   \draw[red] (1.5,0)--(1.5,0.25)--(2.5,0.25)--(2.5,-2.75)--(1.5,-2.75)--(1.5,-2.5);
   \draw[red] (1,0)--(1,0.5)--(2.75,0.5)--(2.75,-3)--(1,-3)--(1,-2.5);
   \draw[red] (0.5,0)--(0.5,0.75)--(3,0.75)--(3,-3.25)--(0.5,-3.25)--(0.5,-2.5);
   \end{scope}
   \draw (2.75,-1) node {$\sim$};
   \begin{scope}[xshift=3.5cm,yshift=-0.5cm,scale=0.75]
    \draw (0,0) rectangle (1,-1);
    \draw (1.5,0) rectangle (2.5,-1);
    \draw (0.75,0)--(0.75,0.25)--(1.25,0.25)--(1.25,-1.25)--(1.75,-1.25)--(1.75,-1);
    \draw (2.25,0)--(2.25,0.25)--(2.75,0.25)--(2.75,-1.25)--(2.25,-1.25)--(2.25,-1);
    \draw (0.25,0)--(0.25,0.25)--(-0.25,0.25)--(-0.25,-1.25)--(0.25,-1.25)--(0.25,-1);
    \draw (1.75,0)--(1.75,0.5)--(3,0.5)--(3,-1.5)--(0.75,-1.5)--(0.75,-1);
    \draw (0.5,-0.5) node[scale=1.33] {$\alpha$};
    \draw (2,-0.5) node[scale=1.33] {$\beta$};
   \end{scope}
   \draw (6.25,-1) node {$\sim$};
   \begin{scope}[xshift=7cm,scale=0.75]
    \draw (0,0) rectangle (1,-1);
    \draw (0.5,-1.5) rectangle (1.5,-2.5);
    \draw (0.25,0)--(0.25,0.25)--(-0.25,0.25)--(-0.25,-1.25)--(0.25,-1.25)--(0.25,-1);
    \draw (1.25,-1.5)--(1.25,-1.25)--(1.75,-1.25)--(1.75,-2.75)--(1.25,-2.75)--(1.25,-2.5);
    \draw (0.75,-1)--(0.75,-1.5);
    \draw (0.75,0)--(0.75,0.25)--(2,0.25)--(2,-3)--(0.75,-3)--(0.75,-2.5);
    \draw (0.5,-0.5) node[scale=1.33] {$\alpha$};
    \draw (1,-2) node[scale=1.33] {$\beta$};
   \end{scope}
   \draw (-1,-1) node {$\reallywidehat{\alpha\heartsuit\beta}\quad=$};
   \draw (9.25,-1) node {$\quad=\;\;\reallywidehat{\alpha\clubsuit\beta}$};
  \end{tikzpicture}
 \end{center}
\end{proof}

\begin{proof}[Proof sketch of Theorem \ref{theorem:Jones_clubsuit}]
 By the lemma above, we know that finding the Jones polynomial of the closure of $\alpha\heartsuit\beta$ is enough.
 
 So we show this by induction on $m$. For $m=1$, $\beta$ is the one strand (and so $V(\reallywidehat{\beta})=1$). So $\alpha\otimes\beta=\incl_n(\alpha)$, and we know $\sigma_n\incl_n(\alpha)$ is Markov equivalent to $\alpha$. So the result then follows.
 
 Now suppose that this is true for $\alpha\in\Bcal_n$ and $\beta\in\Bcal_{m-1}$. So for $\alpha\in\Bcal_n$ and $\beta\in\Bcal_m$ we have
 \begin{align*}
  \phi_{n+m}(\sigma_{n}\cdot(\alpha\otimes\beta)) &= (A+A^{-1}E_{n})\phi_{n+m}(\alpha\otimes\beta)
 \end{align*}
 And so we now take the trace of $E_n\phi_{n+m}(\alpha\otimes\beta)$ and find
  \[
   \tr_{n+m}\left(\knotsinmath{\begin{scope}[yshift=0.55cm,scale=0.75]
     \draw (1,1) [upup=0.5];
     \draw (1,-1) [dndn=0.5];
     \draw (0.35,-1) rectangle (1.15,-2);
     \draw (0.5,1)--(0.5,-1);
     \draw (0.5,-2)--(0.5,-2.5);
     \draw (1,-2)--(1,-2.5);
     \draw (1.4,-1) rectangle (2.15,-2);
     \draw (2,1)--(2,-1);
     \draw (1.5,-2.5)--(1.5,-2);
     \draw (2,-2.5)--(2,-2);
     \draw (0.75,-1.5) node[scale=1] {$\phi(\alpha)$};
     \draw (1.78,-1.5) node[scale=1] {$\phi(\beta)$};
     \draw (0.78,-0.75) node {$\cdots$};
     \draw (0.78,-2.25) node {$\cdots$};
     \draw (1.78,-0.75) node {$\cdots$};
     \draw (1.78,-2.25) node {$\cdots$};
    \end{scope}}\right)=\knotsinmath{\begin{scope}[yshift=0.55cm,scale=0.75]
     \draw (1,1) [upup=0.5];
     \draw (1,-1) [dndn=0.5];
     \draw (0.35,-1) rectangle (1.15,-2);
     \draw (0.5,1)--(0.5,-1);
     \draw (0.5,-2)--(0.5,-2.5);
     \draw (1,-2)--(1,-2.5);
     \draw (1.4,-1) rectangle (2.15,-2);
     \draw (2,1)--(2,-1);
     \draw (1.5,-2.5)--(1.5,-2);
     \draw (2,-2.5)--(2,-2);
     \draw (0.75,-1.5) node[scale=1] {$\phi(\alpha)$};
     \draw (1.78,-1.5) node[scale=1] {$\phi(\beta)$};
     \draw (0.78,-0.75) node {$\cdots$};
     \draw (0.78,-2.25) node {$\cdots$};
     \draw (1.78,-0.75) node {$\cdots$};
     \draw (1.78,-2.25) node {$\cdots$};
     \draw[red] (2,1)--(2.25,1)--(2.25,-2.5)--(2,-2.5);
     \draw[red] (1.5,1)--(1.5,1.25)--(2.5,1.25)--(2.5,-2.75)--(1.5,-2.75)--(1.5,-2.5);
     \draw[red] (1,1)--(1,1.5)--(2.75,1.5)--(2.75,-3)--(1,-3)--(1,-2.5);
     \draw[red] (0.5,1)--(0.5,1.75)--(3,1.75)--(3,-3.25)--(0.5,-3.25)--(0.5,-2.5);
    \end{scope}}=\knotsinmath{\begin{scope}[yshift=1cm,scale=0.75]
     \draw (1,-2) [upup=0.5];
     \draw (1,-1) [dndn=0.5];
     \draw (0.35,-1) rectangle (1.15,-2);
     \draw (1.4,-1) rectangle (2.15,-2);
     \draw (0.75,-1.5) node[scale=1] {$\phi(\alpha)$};
     \draw (1.78,-1.5) node[scale=1] {$\phi(\beta)$};
     \draw (0.78,-0.75) node {$\cdots$};
     \draw (0.78,-2.25) node {$\cdots$};
     \draw (1.78,-0.75) node {$\cdots$};
     \draw (1.78,-2.25) node {$\cdots$};
     \draw (2,-1)--(2,-0.75)--(2.45,-0.75)--(2.45,-2.25)--(2,-2.25)--(2,-2);
     \draw (0.5,-1)--(0.5,-0.75)--(0.15,-0.75)--(0.15,-2.25)--(0.5,-2.25)--(0.5,-2);
    \end{scope}}
   \]
  Now we will see all the different possibilities connecting $E_n$, by only looking at the connection between the four points. The first possibility, for instance, gives
  
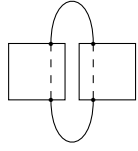
\begin{center}
   \begin{tikzpicture}[scale=0.75]
    \draw (0,0) rectangle (1,1);
    \draw (1.25,0) rectangle (2.25,1);
    \draw (0.75,0) [upup=0.75];
    \draw (0.75,1) [dndn=0.75];
    \draw[dashed] (0.75,1)--(0.75,0);
    \draw[dashed] (1.5,1)--(1.5,0);
   \end{tikzpicture}
  \end{center}
  So, the trace of $\phi_{n+m}(\alpha\otimes\beta)$ without $E_n$ would give $d^2$ (multiplied by the rest of the diagram), whereas $E_n\phi_{n+m}(\alpha\otimes\beta)$ gives $d$ (multiplied by the rest of the diagram).
  The second and third possibilities are the same, but with opposite choices. So without loss of generality, we have,
  
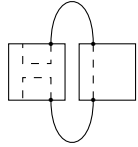
\begin{center}
   \begin{tikzpicture}[scale=0.75]
    \draw (0,0) rectangle (1,1);
    \draw (1.25,0) rectangle (2.25,1);
    \draw (0.75,0) [upup=0.75];
    \draw (0.75,1) [dndn=0.75];
    \draw[dashed] (0.75,1)--(0.75,0.65)--(0.25,0.65)--(0.25,1);
    \draw[dashed] (0.75,0)--(0.75,0.4)--(0.25,0.4)--(0.25,0);
    \draw[dashed] (1.5,1)--(1.5,0);
   \end{tikzpicture}
  \end{center}
  This clearly becomes a single strand.
  So the trace would give $d$. Whereas, the trace of $\phi_{n+m}(\alpha\otimes\beta)$ without $E_n$ would give $d^2$.
  We also have,
  
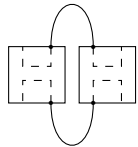
\begin{center}
   \begin{tikzpicture}[scale=0.75]
    \draw (0,0) rectangle (1,1);
    \draw (1.25,0) rectangle (2.25,1);
    \draw (0.75,0) [upup=0.75];
    \draw (0.75,1) [dndn=0.75];
    \draw[dashed] (0.75,1)--(0.75,0.65)--(0.25,0.65)--(0.25,1);
    \draw[dashed] (0.75,0)--(0.75,0.4)--(0.25,0.4)--(0.25,0);
    \draw[dashed] (1.5,1)--(1.5,0.65)--(2,0.65)--(2,1);
    \draw[dashed] (1.5,0)--(1.5,0.4)--(2,0.4)--(2,0);
   \end{tikzpicture}
  \end{center}
  Similarly to the previous possibilities, we have the trace of $E_n\phi_{n+m}(\alpha\otimes\beta)$ giving $d$, and $d^2$ for the trace of $\phi_{n+m}(\alpha\otimes\beta)$ without $E_n$. We then have two other possibilities (i.e., both strands in a diagram are going down/up and all strands are going down/up) - the reader is advised to try them.
  
  We see that it will always give us $\tr_{n+m}(E_n\phi_{n+m}(\alpha\otimes\beta))=d^{-1}\tr_{n+m}\circ\phi_{n+m}(\alpha\otimes\beta)$. So then we have
  \[\tr_{n+m}(E_n\phi_{n+m}(\alpha\otimes\beta))=d^{-1}\cdot\tr_{n+m}\circ\phi_{n+m}(\alpha\otimes\beta)\]
  The result then follows,
  \begin{align*}
   V(\reallywidehat{\alpha\heartsuit\beta}) &= \frac{(-A^{-3})^{\omega(\alpha)+\omega(\beta)+1}\tr_{n+m}\circ\phi_{n+m}(\alpha\otimes\beta)(A+A^{-1}d^{-1})}{d}\\
   &= \frac{(-A^{-3})^{\omega(\alpha)+\omega(\beta)}\tr_{n+m}\circ\phi_{n+m}(\alpha\otimes\beta)}{d^2} = V(\reallywidehat{\alpha})V(\reallywidehat{\beta})
  \end{align*}
\end{proof}

\printindex

\end{document}